\documentclass{article}
\usepackage[T1]{fontenc}
\usepackage[utf8]{inputenc}
\usepackage[english]{babel}
\usepackage{lmodern}
\usepackage[margin=3cm]{geometry}
\usepackage[backend=biber,style=alphabetic,sorting=nty,doi=false,isbn=false,url=false,minalphanames=3]{biblatex}
\AtEveryBibitem{\clearlist{language}}

\usepackage{empheq}

\usepackage{authblk}
\usepackage[hidelinks]{hyperref}
\usepackage{enumitem}
\usepackage{tikz-cd} 	
\usepackage{mathtools}
\usepackage{alltt}
\usepackage{amsfonts}
\usepackage{amsmath}
\usepackage{amssymb}
\usepackage{amsthm}
\usepackage{tikz-cd} 	
\usepackage[nottoc]{tocbibind}
\usepackage{graphicx}
\usepackage[font={small,it}]{caption}
\usepackage{subcaption}
\usepackage{aliascnt}
\usepackage{cases}
\usepackage{comment}	
\usepackage[capitalize,nameinlink]{cleveref}
\usepackage{tikz}
\usetikzlibrary{tikzmark}
\usepackage{slashed}

\crefname{figure}{Fig.}{Fig.}
\Crefname{figure}{Fig.}{Fig.}
\crefname{subfigure}{Fig.}{Fig.}
\Crefname{subfigure}{Fig.}{Fig.}

\usepackage{xargs}   
\usepackage[colorinlistoftodos,prependcaption,textsize=tiny]{todonotes}
\newcommandx{\typo}[2][1=]{\todo[linecolor=red,backgroundcolor=red!25,bordercolor=red,#1]{#2}}
\newcommandx{\change}[2][1=]{\todo[linecolor=blue,backgroundcolor=blue!25,bordercolor=blue,#1]{#2}}
\newcommandx{\answer}[1]{\todo[linecolor=pink,backgroundcolor=pink!25,bordercolor=pink]{#1}}
\newcommandx{\unsure}[2][1=]{\todo[linecolor=green,backgroundcolor=green!25,bordercolor=green,#1]{#2}}
\newcommandx{\improve}[2][1=]{\todo[linecolor=violet,backgroundcolor=violet!25,bordercolor=violet,#1]{#2}}
\newcommandx{\thiswillnotshow}[2][1=]{\todo[disable,#1]{#2}}

\usepackage{tocloft}
\crefformat{equation}{(#2#1#3)}
\numberwithin{equation}{section}

\theoremstyle{definition}

\newtheorem*{theorem*}{Theorem}
\newtheorem*{conjecture*}{Conjecture}

\theoremstyle{plain}
\newtheorem{theorem}{Theorem}[section]
\newtheorem{corollary}{Corollary}[section]
\newtheorem{lemma}{Lemma}[section]
\newtheorem{prop}{Proposition}[section]
\newtheorem{conjecture}{Conjecture}[section]
\newtheorem{cor}{Corollary}[section]

\newtheorem{definition}{Definition}[section]

\theoremstyle{remark}
\newtheorem{remark}{Remark}[section]

\crefname{lemma}{Lemma}{Lemmas}
\crefname{prop}{Proposition}{Proposition}
\crefname{conjecture}{Conjecture}{Conjecture}
\crefname{cor}{Corollary}{Corollary}
\crefname{remark}{Remark}{Remark}
\crefname{defi}{Definition}{Definition}
\crefname{equation}{}{}
\crefname{enumi}{}{}
\crefname{appendix}{}{}

\newcommand{\dd}{\mathop{}\!\mathrm{d}}

\newenvironment{nalign}{
	\begin{equation}
		\begin{aligned}
		}{
		\end{aligned}
	\end{equation}
	\ignorespacesafterend
}

\newcommand{\C}{\mathcal{C}}

\newcommand{\N}{\mathbb{N}}

\newcommand{\R}{\mathbb{R}}

\newcommand{\T}{\mathbb{T}}
\newcommand{\Z}{\mathbb{Z}}
\renewcommand{\H}{\mathbb{H}}
\newcommand{\scri}{\mathcal{I}}

\newcommand{\abs}[1]{\left\lvert #1\right\rvert}

\newcommand{\jpns}[1]{\langle #1 \rangle}

\newcommand{\norm}[1]{\left\lVert #1\right\rVert}

\newcommand{\E}{\mathbb{E}}

\newcommand{\B}{\mathcal{B}}

\newcommand{\D}{\mathcal{D}}

\renewcommand{\d}{\mathrm{d}}

\DeclareMathOperator{\diag}{diag}

\DeclareMathOperator{\Diff}{Diff}

\DeclareMathOperator{\supp}{supp}



\title{Construction of multi-soliton solutions for the energy critical wave equation in dimension 3}
\author[1,2]{Istvan Kadar\thanks{ik338@cam.ac.uk, ik4946@princeton.edu}}
\affil[1]{University of Cambridge, Department of Applied Mathematics and Theoretical Physics, Wilberforce Road,
	Cambridge CB3 0WA, United Kingdom}
\affil[2]{Princeton Gravity Initiative, Jadwin Hall, Washington Road, Princeton, NJ 08544, United States}

\usepackage[makeroom]{cancel}
\newcommand{\Vb}{\mathcal{V}_{\mathrm{b}}}
\newcommand{\Vbe}{\mathcal{V}_{\mathrm{b},\epsilon}}
\newcommand{\un}{{\mathrm{un}}}

\newcommand{\Vbold}{\boldsymbol{V}}
\newcommand{\Vlin}{\stackrel{\scalebox{.6}{\mbox{\tiny (1)}}}{V}}
\newcommand{\phic}{ \prescript{}{\tiny{c}}{\phi} }

\newcommand{\Tlin}{\stackrel{\scalebox{.6}{\mbox{\tiny (1)}}}{\T}}
\newcommand{\Tmod}{\stackrel{\scalebox{.6}{\mbox{\tiny (2)}}}{\T}}
\usepackage{accents}
\newcommand{\master}{\underaccent{\scalebox{.8}{\mbox{\tiny $k$}}}{\mathcal{X}}}
\newcommand{\masterZero}{\underaccent{\scalebox{.8}{\mbox{\tiny $0$}}}{\mathcal{X}}}
\newcommand{\inhom}{\underaccent{\scalebox{.8}{\mbox{\tiny $k$}}}{\mathcal{Y}}}

\newcommand{\Region}{\mathcal{R}_{\tau_1,\tau_2}}
\renewcommand{\E}{\mathcal{E}}
\renewcommand{\T}{\mathbb{T}}
\renewcommand{\O}{\mathcal{O}}

\newcommand{\tT}{\tilde{\mathbb{T}}}
\newcommand{\Hb}{H_{\mathrm{b}}}
\newcommand{\A}[1]{\mathcal{A}_{\mathrm{#1}}}
\newcommand{\e}{\mathrm{e}}
\newcommand{\g}{\mathrm{g}}
\newcommand{\ext}{\mathrm{ext}}
\newcommand{\internal}{\mathrm{int}}
\renewcommand{\b}{\mathrm{b}}
\newcommand{\loc}{\mathrm{loc}}
\newcommand{\lin}{\mathrm{lin}}
\newcommand{\Hbloc}{H_{\mathrm{b},\loc}}
\renewcommand{\r}{\mathrm{r}}
\renewcommand{\Hb}{H_{\mathrm{b}}}

\DeclareFontFamily{U}{rcjhbltx}{}
\DeclareFontShape{U}{rcjhbltx}{m}{n}{<->rcjhbltx}{}
\DeclareSymbolFont{hebrewletters}{U}{rcjhbltx}{m}{n}

\DeclareMathSymbol{\lamed}{\mathord}{hebrewletters}{108}

\graphicspath{{figure}}

\begin{document}
	\maketitle
	\begin{abstract}
		
		We study the energy-critical wave equation in three dimensions, focusing on its ground state soliton, denoted by $W$.
		Using the Poincaré symmetry inherent in the equation, boosting $W$ along any timelike geodesic yields another solution.
		The slow decay behavior of $W$, $W\sim r^{-1}$, indicates a strong interaction among potential multi-soliton solutions. 
		
		In this paper, for arbitrary $N\geq0$, we provide an algorithmic procedure to construct approximate solutions to the energy critical wave equation that: (1) converge to a superposition of solitons, (2) have no outgoing radiation, (3) their error to solve the equation decays like $(t-r)^{-N}$.
		Then, we show that this approximate solution can be corrected to a real solution.
	\end{abstract}
	
	\setcounter{tocdepth}{2}
	\tableofcontents

	\section{Introduction}\label{sec:introduction}
	This paper is the second in a series of works devoted to the construction of multi soliton solutions for nonlinear wave equations of high regularity and polyhomogeneous expansions.
	
	Here, we study solutions to the energy critical wave equation on Minkowski space $\mathring{\mathcal{M}}:=\R^{3+1}$
	\begin{equation}\label{i:eq:main}
			\Box\phi:=(-\partial_t^2+\Delta)\phi=-\phi^5,\qquad
			\phi	:\R^{3+1}\to\R
	\end{equation}
	that admits a stationary solution, called the ground state soliton
	\begin{equation}\label{i:eq:soliton}
		\begin{gathered}
			W(x)=\sqrt{3}\Big(1+3\abs{x}^2\Big)^{-1/2}.
		\end{gathered}
	\end{equation}
	Via the Lorentz boost symmetries of \cref{i:eq:main}, we can construct boosted solutions corresponding to $W$ moving along a straight line at a subluminal speed.
	In this paper, we are concerned with the construction of solutions that settle down to a sum of solitons moving away from one another.
	We call these \emph{multi-soliton} solutions.
	Such solutions were already studied for the higher dimensional analogue of \cref{i:eq:main} in \cite{martel_construction_2016,martel_inelasticity_2018,yuan_construction_2020,yuan_construction_2022}.
	
	Similarly, we have that the energy supercritical wave equation
	\begin{equation}\label{i:eq:supercritical}
		\Box\phi=-\phi^7+\phi^9
	\end{equation}
	admits a stationary solution, called the ground state soliton $W^{\mathrm{sup}}$.
	Since \cref{i:eq:supercritical} also has Poincare symmetries, $W^{\mathrm{sup}}$ may be boosted as well.
	To the authors knowledge, multi-soliton solutions for \cref{i:eq:supercritical}, or other energy supercritical wave equations have not been constructed so far.
	
	We already present a simplistic version of our main theorem.
	\begin{theorem}[Rough form of   \cref{i:thm:existence_of_ansatz,i:thm:existence_of_scattering_solution,i:thm:trivial_scattering}]\label{i:thm:prosaic}\ \\
		a) Fix $N$ arbitrary.
		There exists polyhomogeneous functions $\bar{\phi}$ in a neighbourhood of timelike infinity ($t-r\gg1$) such that $\bar{\phi}$ has no radiation through null infinity ($t\sim r\gg 1$), it settles down to a sum of boosted solitons $W$, and solves \cref{i:eq:main} or \cref{i:eq:supercritical} with an error bounded by $r^{-5}(t-r)^{-N}$.
		\\
		b) For $N$ sufficiently large, there exists a solution to \cref{i:eq:main} or \cref{i:eq:supercritical} of the form $\hat\phi=\bar{\phi}+\phi$, where $\abs{\phi}\lesssim (t-r)^{-N+10}r^{-1/2}$. \\
		c) Together with a scattering result around spacelike infinity ($r-\abs{t}\gg 1$), these imply the existence of a solution $\hat{\phi}$ for $t\geq0$ of \cref{i:eq:main} that settle down to a sum of solitons without radiation: there exists signs $\sigma_a\in\{\pm1\}$, scales $\lambda_a\in\R^+$, velocities $z_a\in\{x\in\R^3:\abs{x}<1\}$ and corrections $z_a^{0,1}\in\R^3$ such that  
		\begin{nalign}
			\lim_{t\to\infty}\norm{				\hat{\phi}(t,x)-\sum_{a}\sigma_a\lambda_a^{1/2}W(x_a)}_{\dot{H}_x^1}=0,\quad x_a=\lambda_a(x-z_at-z_a^{0,1}\log t).
		\end{nalign}
	\end{theorem}
	
	\paragraph{Overview of the introduction:}
	We begin by a brief overview regarding multi-soliton solutions for \cref{i:eq:main} in \cref{i:sec:previous_works}.
	We also include some discussion on problems about the regularity of the linear wave equations on perturbations of Minkowski space that serve as intuition for the ideas leading to \cref{i:thm:prosaic}.
	Next, we discuss some heuristics regarding possible multi-soliton solutions in \cref{i:sec:heuristic}, and connect these to the behaviour of point particles under Newtonian gravity.
	In \cref{i:sec:main_theorems} we give a more precise, though still only schematic, version of the main theorems and give references to their precise versions.
	We also discuss the main difficulties and the approach we take to resolve them.
	Finally in \cref{i:sec:outline}, we provide an outline for the rest of the paper.

	\subsection{Previous works}\label{i:sec:previous_works}
	For a broader overview of the motivation for this project, we refer to \cite{kadar_scattering_2024}.
	Here we only give references specifically concerned with \cref{i:thm:prosaic}.
	The study of solitons for \cref{i:eq:main} has an extensive literature, see \cite{tao_nonlinear_2006,kenig_lectures_2015-1,martel_construction_2016} and references therein.
	It includes both global and local theory as well as singularity formation and late time behaviour.
	Due to the energy critical nature of \cref{i:eq:main}, low regularity theory can yield strong control over the solutions.
	This is especially important for classification results, see \cite{collot_classification_2022,jendrej_soliton_2022} and references therein.
	
	\paragraph{Multi-soliton for the energy critical wave equation}
	Before giving more detail, we present some notation.
	As already mentioned, the Poincare symmetry of \cref{i:eq:main}, yields solutions moving along straight lines $x=z t$, $\abs{z}<1$.
	A further symmetry of \cref{i:eq:main} is scaling, that together with Poincare group yields the family of solutions
	\begin{equation}\label{i:eq:family_of_solitons}
		W_a:=W^{\lambda_a}(y_a),\quad W^\lambda(x):=\lambda^{1/2}W(x\lambda),\quad y_a=\gamma_a(x-z_at)-z^{0,0}_a,\gamma_a=(1-\abs{z_a}^2)^{-1/2},
	\end{equation}
	where $z_a\in B:=\{\abs{x}<1\}$, $z^{0,0}_a\in\R^3$, $\lambda_a\in\R^+$.

	In \cite{martel_construction_2016}, Martel-Merle introduced the methods developed for the study of multi-soliton solutions for dispersive equations\footnote{
		For literature on multi-soliton solutions to other dispersive equations, we refer the reader to the introduction of \cite{martel_construction_2016}.} for
	\begin{equation}\label{i:eq:critical_higher_dim}
		\Box_{\R^{n+1}}\phi=-\phi\abs{\phi}^{\frac{4}{n-2}},
	\end{equation}
	the higher dimensional analogue of \cref{i:eq:main}.
	In $n=5$, they proved
	\begin{theorem}[\cite{martel_construction_2016}]\label{i:thm:Martel_Merle}
		Given any finite number of collinear velocities $z_a\in B$ and scales $\lambda_a$, there exists a solution $\phi$ to \cref{i:eq:critical_higher_dim} with $n=5$ satisfying
		\begin{subequations}\label{i:eq:no_outgoing_multi-soliton}
			\begin{align}
				 &\tilde{\phi}=\phi-\sum_a W_{a}\label{i:eq:multi_soliton}\\
				 &\lim_{t'\to\infty}\norm{(\tilde{\phi},\partial_t\tilde{\phi})}_{\dot{H}^1\times L^2(\{t=t'\})}=0\label{i:eq:no_outgoing}
			\end{align}
		\end{subequations}
		where the scaling factor is $\lambda^{3/2}$ instead $\lambda^{1/2}$ in \cref{i:eq:family_of_solitons}.
	\end{theorem}
	In the theorem above, \cref{i:eq:multi_soliton} indicates that the solution settles down to a state approaching multiple solitons, while \cref{i:eq:no_outgoing} captures that there is no outgoing radiation.
	Indeed, if we modify \cref{i:eq:no_outgoing} so that 
	\begin{equation}\label{i:eq:outgoing}
		\lim_{t'\to\infty}\norm{(\tilde{\phi},\partial_t\tilde{\phi})-(\psi,\partial_t\psi)}_{\dot{H}^1\times L^2(\{t=t'\})}=0
	\end{equation}
	for a solution $\Box\psi=0$, we would say that $\phi$ settles down to multi-solitons $W_a$ plus \emph{radiation} $\psi$.
	
	\cref{i:thm:Martel_Merle} was later improved in \cite{yuan_multi-solitons_2019} to extend the allowed values of velocities.
	An important aspect of all works concerning solutions near one or multiple solitons is the linearised problem around a single one, which for the 3 dimensional case reads
	\begin{equation}\label{i:eq:lin}
		(\Box+V)\phi=0,\qquad V=5W^4.
	\end{equation}
	To study the nonlinear problem, it is important to understand decay and boundedness properties of \cref{i:eq:lin}.
	The eigenvalues of the stationary problem are a clear obstruction to this, but luckily they are explicitly characterised (see \cite{duyckaerts_dynamics_2008})
	\begin{equation}
		(\Delta+V)u=\mu u\in\dot{H}^1(\R^3)\implies u\in\{\Lambda W, \partial_i W, Y\},\text{ and } \mu\in\{0,\lamed^2\}
	\end{equation}
	where $\lamed>0$ is the unique nonzero eigenvalue and $\Lambda=1/2+x_i\partial_i$.
	The issues regarding the eigenfunction $Y$ are standard and treated via a topological argument\footnote{see \cref{non:lemma:unstable} for the application of Brouwer fixed point theorem to control the unstable modes}.
	For the kernel elements, Martel-Merle take the standard approach of constructing a solutions by modulating the location ($z_a^{0,0}$) and scales ($\lambda_a$) of the solitons by making these parameters time dependent.
	
	An important aspect of the proof in \cite{martel_construction_2016} is that the solitons have a weak far field region, they decay like $r^{-3}$.
	Indeed, this decay rate in any dimension ($n\geq3$) is the same as the Green's function of $\Delta$, i.e. $\sim r^{-n+2}$.
	This decay, via the modulation of the soliton parameters ($z^{0,0}_a,\lambda_a$) introduce weaker interaction in higher dimensions.
	An alternative approach to obtain such weak decay is by studying the non ground state solitons of \cref{i:eq:critical_higher_dim}.
	This approach was taken in \cite{yuan_construction_2020,yuan_construction_2022}, where Yuan studied excited multi-soliton solutions in dimension 4 and 5.
	
	Other than proving \cref{i:thm:Martel_Merle}, an important contribution of \cite{martel_construction_2016} is the insight that Martel-Merle make regarding the issues just mentioned.
	They highlight the importance of decay of $W$ in the construction as well as a heuristic for their interaction.
	They also conjecture that multi-soliton solutions satisfying \cref{i:eq:no_outgoing_multi-soliton} also exist in higher dimension as well as their nonexistence in lower dimensions:
	\begin{conjecture}[\cite{martel_construction_2016}]\label{i:conj_martel}
		There exists no solutions to \cref{i:eq:main} satisfying \cref{i:eq:no_outgoing_multi-soliton} with more than 1 ground state soliton.
	\end{conjecture}
	
	\paragraph{Geometric radiation}
	The notion of scattering used in \cref{i:eq:no_outgoing} is that of Lax-Philips scattering \cite{lax_scattering_1964}.
	This notion played and still plays a fundamental role for a large class of equations that settle down to a linear behaviour (ignoring the solitons) at the infinite past and/or future.
	Importantly, this notion applies equally well for nonlinear wave, Klein-Gordon, Schrodinger and many other dispersive equations (see \cite{tao_nonlinear_2006}).
	
	Importantly, the wave equation $\Box\phi=0$, admits an alternative characterisation of radiation as was observed by Friedlander \cite{friedlander_radiation_1980}.
	For sufficiently smooth solutions we have the following equivalence
	\begin{lemma}
		Let $\tilde{\phi}$ be a smooth solution to $\Box\tilde{\phi}=f$, with $f$ smooth and decaying sufficiently fast.
		Then, we have that \cref{i:eq:no_outgoing} is equivalent to the vanishing of the \emph{radiation field} $\partial_u\psi(u,\omega)=0$ for all $u\in\R, \omega\in S^2$, where
		\begin{equation}\label{i:eq:radiation_field}
			\psi(u,\omega):=\lim_{v\to\infty}\Big((\partial_t-\partial_r)\phi\Big)(u+v,\omega(v-u)).
		\end{equation}
	\end{lemma}
	In a similar manner, we can also characterise the outgoing radiation \cref{i:eq:outgoing}.
	
	An important feature of using the alternative characterisation of the radiation field, is that it allows to study solutions localised to different regions of spacetimes instead of studying them globally.
	Following the study of wave type equations connected to General Relativity, we introduce a conformal compactification of Minkowski space called the Penrose diagram\footnote{see \cite{dafermos_lectures_2008} for a gentle introduction to its use in more complicated settings.} in \cref{i:fig:pen}.
	In this diagram, under a spherical symmetry assumption, light rays travel along diagonal lines and this allows for an easy understanding of the causal relations that arise from domain of dependence properties for \cref{i:eq:main}.
	In particular, most of the analysis in this paper is performed in the region $t-r>0$ show as $\mathcal{R}$ in \cref{i:fig:pen}.

	\begin{figure}[htpb]
		\centering
		\begin{subfigure}{0.3\textwidth}
			\centering
			\includegraphics[width=120pt]{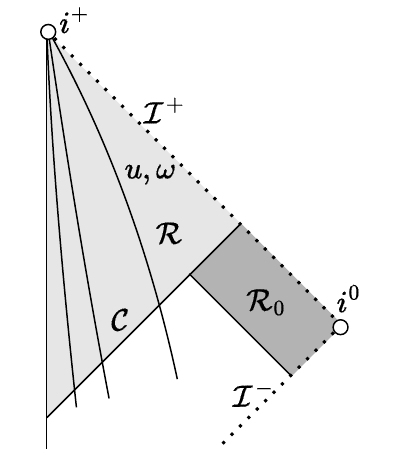}
			\caption{Penrose (conformal) compactification.}
			\label{i:fig:pen}
		\end{subfigure}
		\hspace{0.01\textwidth}
		\begin{subfigure}{0.3\textwidth}
			\centering
			\includegraphics[width=125pt]{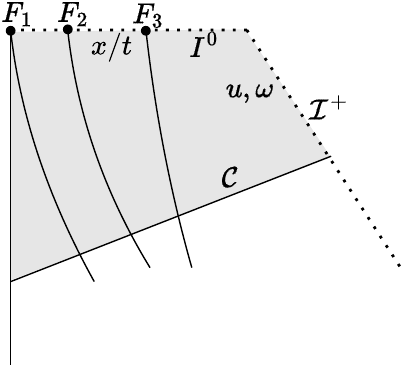}
			\caption{Radial compactification of $\mathcal{M}$ with blown up $\scri$.}
			\label{i:fig:baskin}
		\end{subfigure}
		\hspace{0.01\textwidth}
		\begin{subfigure}{0.3\textwidth}
			\centering
			\includegraphics[width=130pt]{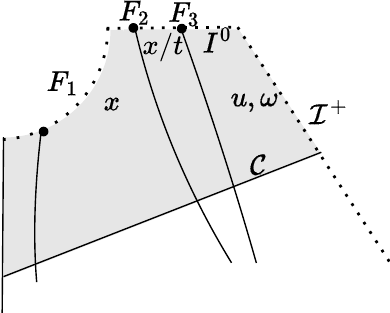}
			\caption{Radial compactification of $\mathcal{M}$ with blown up $\scri,F_1$.}
			\label{i:fig:hintz}
		\end{subfigure}
		\caption{Different compactification of Minkowski space.
			The point $F_1,F_2,F_3$ denote endpoints of timelike geodesics, along which solitons travel.
		We indicated the coordinate functions on the faces $\scri,I^+,F_1$ by $(u,\omega),x/t,x$, respectively.
	The results \cref{i:thm:existence_of_ansatz,i:thm:existence_of_scattering_solution} are concerned about the future of the lightcone $\C$, while the result \cref{i:thm:trivial_scattering} is about the region $\mathcal{R}_0$.}\label{i:fig:compactifications}
	\end{figure}

	\paragraph{Compactification suited to solitons}
	We now give some background on the linear wave equation on asymptotically flat spacetimes.
	The reference are of course not exhaustive, but see the introduction of \cite{hintz_linear_2023} for a thorough review.
	
	In \cite{baskin_asymptotics_2015}, Baskin-Wunsch-Vasy introduced an alternative compactification of Minkowski space ($\mathcal{M}$), that highlights some other important geometric and analytic properties of solutions to $\Box\phi=f$ and perturbation thereof, see \cref{i:fig:baskin}.
	\begin{theorem}[\cite{baskin_asymptotics_2015}]\label{i:thm:baskin}
		For $f\in\mathcal{C}_c^\infty(\mathring{\mathcal{M}})$, and $\Box\phi=f$ with $\phi=0$ for $t<T_0$, $\phi$ is polyhomogeneous on a compactification ($\mathcal{M}$) of $\mathring{\mathcal{M}}$.
		In particular, we have that for any $N$ and $c\in(0,1)$
		\begin{equation}\label{i:eq:baskin_polyhom}
			\phi(t,x)=\sum_{(z,k)\in\E}t^{-z}\log^kt \phi_{z,k}(x/t)+\O(t^{-N}),\qquad \abs{x}<tc
		\end{equation}
		for some discrete set $\E\subset\R\times\N$ and smooth function $\phi_{z,k}\in\C^\infty(B)$.
		Similar expansions hold at other regions.
	\end{theorem}
	The results proven in \cite{baskin_asymptotics_2015} are of course much more general than the one presented above, but the polyhomogeneous expansion \cref{i:eq:baskin_polyhom} is robust.
	Even the extended results however do not cover solutions of $(\Box+w(x))\phi=0$ for smooth functions $w$.
	The reason is, that the potential introduces extra structure at a scale not well resolved by $\mathcal{M}$.
	To remedy this, Hintz introduced an alternative compactification $\dot{\mathcal{M}}$, see \cref{i:fig:hintz} and proved the following
	\begin{theorem}[\cite{hintz_linear_2023}]\label{i:thm:hintz}
		Let $w\in\C^\infty(\R^3)$ be a smooth potential that decays sufficiently fast and satisfies certain spectral properties
		For $f\in\C^\infty_c(\mathcal{M})$ the solution to $(\Box+w)\phi=0$ satisfying $\phi=0$ for $t<T_0$ are conormal on $\dot{\mathcal{M}}$.
		In particular, $\phi$ is in a polynomially weighted $L^2$ space after arbitrary application of derivatives of the form $\jpns{r}(\partial_i-\hat{x}_i\partial_t),u\partial_t, x_i\partial_j-x_j\partial_i$, where the weight does not change after commutation.
	\end{theorem}
	
	The regularity in \cref{i:thm:hintz} is a weaker version of having an expansion of the form \cref{i:eq:baskin_polyhom}, in the more refined space $\dot{\mathcal{M}}$.
	The regularity is measured with respect $\b$-derivatives.
	These are characterised by a degenerating factor towards the boundary for normal derivatives  on $\dot{\mathcal{M}}$.
	The regularity results for linear equations in \cref{i:thm:baskin,i:thm:hintz} form a crucial heuristic for the proof of \cref{i:thm:prosaic}.
	We note, that none of the compactifications in \cref{i:fig:compactifications} play an analytic role in the paper, as the estimates take place in $\mathring{\mathcal{M}}$, however we find all of them useful for structuring the estimates.

	\subsection{Heuristic}\label{i:sec:heuristic}
	Before presenting the main theorem, we find it useful to discuss the weak decay present in the soliton $W$ and its implication for multi-soliton solutions.
	We of course know that as \cref{i:eq:main} is nonlinear, $\phi=W(x)+W(x-x_1)$ is not an solution.
	Nonetheless, as \cref{i:eq:main} comes from a Lagrangian, we may compute associated Hamiltonian of the function $\phi$ to find
	\begin{equation}
		H[\phi]=\int \abs{\nabla \phi}^2-\phi^6/3\sim x_1^{-1}.
	\end{equation}
	From classical mechanics, we know that force is simply the gradient of the energy.
	Thus, we expect that the solitons will have a inverse square force on one another.
	This is the same behaviour as the attraction of point particles in Newtonian gravity in $\R^{3+1}$.
	As we noted already, the tail of $W$ comes from the Green's function of $\Delta$, just as the far field region of the gravitational field of a localised source given by the solution to Poisson equation
	\begin{equation}
		\Delta\phi=\rho
	\end{equation}
	where $\phi$ represents the gravitational potential energy and $\rho$ the mass density.

	The inverse square force law has an important consequence for the location of the point particles in Newtonian gravity.
	Denoting by $d$ the distance between the two particles, we get
	\begin{equation}\label{i:eq:d_ODE}
		\ddot{d}=c\frac{1}{d^2}+\mathcal{O}(d^{-3}).
	\end{equation}
	Solutions describing particles moving apart can have the following two types of behaviour
	\begin{itemize}
		\item \emph{hyperbolic orbit:} this corresponds to $d=c_1t+\O(t^{1/2})$ for free parameter $c_1\neq0$.
		In this case, we can write the position of the two particle as $z_at+\O(t^{1/2})$ for some $z_a\in\R^3$, as shown on \cref{i:fig:compactifications}.
		Importantly, we can find the leading correction to $d$ from \cref{i:eq:d_ODE} to be $d=c_1t+c_2\log t+c_1'+\O(t^{-1/2})$, where $c_2$ is a function of $c_1$ and $c_1'$ is another free parameter.
		The corresponding location for the solitons is $z_at+z_a^{0,1}\log t+z_a^{0,0}+\O(t^{-1/2})$, where $a\in\{1,2\}$,  $z_a^{0,1}$ is determined by $z_a$ and $z_a^{0,0}$ are free parameters.
		This logarithmic correction plays a crucial role in our construction, see already \cref{i:sec:ansatz}.
		\item \emph{parabolic orbit:} in this case $d=ct^{2/3}$ where $c$ is determined by the coefficient in \cref{i:eq:d_ODE}.
		We \emph{do not} consider this case in the present paper, but note \cref{i:conj:parabolic_orbit}.
	\end{itemize}

	In principle, we could give a more convincing heuristic by introducing an approximate Lagrangian corresponding to the free time dependent parameters $\lambda_a,z_a,z_a^{0,0}$ in a function $\phi=\sum_a W_a$.
	However, we find that the rigorous construction discussed already in \cref{i:sec:ansatz}, is just as enlightening.
	Furthermore, as we are not aware of any work on computational heuristics\footnote{there are many important physical equations where such heuristic results do exist, most notably for the Einstein equations, see \cite{kehrberger_case_2024-1} and references therein} for the radiation produced by multi-soliton configuration of \cref{i:eq:main}, we decided to not present further discussion on these issues.
	
	\subsection{Main theorems and ideas of the proof}\label{i:sec:main_theorems}
	In this section we discuss the main theorems together with the relevant ideas and difficulties in the proofs.
	The formulae and presentation is only schematic and we only give references to precise statements.
	
	As already discussed, we study high regularity solutions to \cref{i:eq:main} that settle down to multiple solitons.
	Following \cite{hintz_linear_2023}, we first introduce a compactification on which we expect to construct a smooth solution.
	For soliton velocities $z_a$, we introduce coordinates $y_a$ as in \cref{i:eq:family_of_solitons}, and rescaled coordinate\footnote{these have to be logarithmically corrected, but as this is a minor technical point, we de not include it for the present discussion. See \cref{not:def:coordinates,not:def:compactification} for precise definitions.} $\rho_a=\jpns{y_a}/t$ for the boundary defining function of the faces $F_a$.
	We also need the boundary defining function of timelike ($I^+$) and null infinity ($\scri$) given by $\rho_\scri=(t-r)/t$ and  $\rho_+=t\rho_\scri^{-1}(\sum_a\rho_a^{-1})$, respectively.
	Let's set $\mathring{\D}=\cap_\bullet \{\rho_\bullet<1\}$.  
	We call $\D^\g$, the compactification of $\mathring{\D}$ defined by smoothly extending the functions $\rho_\bullet$ to 0, see \cref{not:sec:compactification} for more precise statement.\footnote{
		One can view $\D^\g$ as the blow-up of $\mathcal{M}$ at the points $\jpns{y_a}/t=0$, for further details on blow-ups see \cite{grieser_basics_2001}.}
	
	We split the study of solutions \cref{i:eq:main} into two distinct parts.
	Firstly, the existence of an approximate solution discussed in \cref{i:sec:ansatz},
	secondly their correction to exact solutions in \cref{i:sec:nonlinear}.
	We finally present some further results on a more general class of equations to which our methods apply in \cref{i:sec:applications}.
	
	\subsubsection{Ansatz}\label{i:sec:ansatz}
	Polyhomogeneous functions are a generalisation of Taylor expansion appropriate for \cref{i:thm:prosaic}.
	We say that $f:[0,1)_x\to\R$ has a polyhomogeneous expansion if it can be written as a sum of terms $x^z\log^k$ plus a faster decaying error term, see \cref{sec:notation} for precise statements.
	A similar construction also works for a manifold with corners.
	In particular, we write $\mathcal{O}_{\scri}^{a,b,c}$ for the space of functions that have such an expansion with $t^{-a},t^{-b},t^{-c}$ leading behaviour towards the boundaries $\scri,I^+,F_a$, respectively as shown on \cref{i:fig:foliation_mine}.
	
	The main theorem regarding approximate solutions can be stated as follows.
	\begin{theorem}[Rough version of \cref{an:thm:existence_of_ansatz} ]\label{i:thm:existence_of_ansatz}\ \\
			
			\begin{enumerate}[label=\alph*),wide=0.5em, leftmargin =*, nosep, before = \leavevmode\vspace{-2\baselineskip}]
				\item\label{i:item:ansatz_a} There exists soliton velocities such that for every $N\geq1$ we can find modulated solitons $W_a$, and a polyhomogeneous $\tilde{\phi}\in\O^{2,2,1}_{\scri}$, expressing that $\phi\sim t^{-1}$ towards $F_a$ and $\phi\sim t^{-2}$ toward $\scri,I^+$ such that the following holds.
				For $\bar{\phi}=\sum_a W_a+\tilde\phi$ we have 
				\begin{equation}\label{an:eq:bar_phi}
					\begin{gathered}
						(\Box+\bar{\phi}^4)\bar{\phi}\in\mathcal{O}_{\scri}^{5,N,N},\qquad \mathfrak{E}^{\lin}[\bar{\phi}]:=5\bar{\phi}^4-\sum_a 5W_a^4\in\O_{\scri}^{4,4,1}\\
						\partial_u(r\bar{\phi})|_{\scri}=0
					\end{gathered}
				\end{equation}
				\item\label{i:item:ansatz_b} For any distinct soliton velocities and $N\geq1$ there exist an ansatz $\tilde\phi\in\O_{\scri}^{1,1,1}$ polyhomogeneous on $\D^\g$ and unmodulated solitons such that $\bar{\phi}=\sum_a W_a+\tilde\phi$ satisfies $(\Box+\bar{\phi}^4)\bar{\phi}=\mathcal{O}_\scri^{5,N,N}$ and $ \mathfrak{E}^{\lin}[\bar{\phi}]\in\O_{\scri}^{4,4,2}$.
				In general, $\tilde\phi$ has nonzero outgoing radiation, i.e. $\partial_u(r\bar{\phi})|_{\scri}\neq0$.
			\end{enumerate}
			
	\end{theorem}
	
	We make the following remarks:
	\begin{remark}[Polyhomogeneity]
		In \cref{i:item:ansatz_a}, we omitted the space on which the function $\tilde{\phi}$ are polyhomogeneous.
		This is because the manifold on which polyhomogeneity holds is not $\D^\g$ as discussed above, but a logarithmic correction of that, see \cref{not:def:compactification}.
		
		In \cref{i:item:ansatz_b}, the soliton velocities no longer have logarithmic or any lower order correction to their position.
		This is also true for their scale.
	\end{remark}
	
	\begin{remark}[Free parameters]
		In \cref{i:item:ansatz_a}, the velocity of the solitons is not freely prescribable.
		Indeed, we need the worst error terms to be not present for the the simple ansatz $\sum_a W_a$, see already \cref{i:rem:startin_cond}.
		On the other hand, in \cref{i:item:ansatz_b}, we can allow arbitrary position for the solitons. We use the radiation leaving the spacetime to shift the position and scales in a way so as to exactly cancel the effect of the solitons on one another.
	\end{remark}
	
	\begin{remark}[Procedure]
		The proof of \cref{i:thm:existence_of_ansatz} is constructive.
		We describe an algorithm to construct the ansatz in an iterative way.
		Importantly, the correction to the soliton velocities and scales can be determined by an explicit computation.
		In the absence of expectation on these corrections, we do not pursue this further in the present paper, however the logarithmic corrections are given as explicit integrals in \cref{an:eq:Newtonian_path_correction}.
	\end{remark}
	
	\paragraph{Ideas of the proof}
	
	The proof of \cref{i:thm:existence_of_ansatz} is largely motivated by the techniques discussed in \cite{hintz_lectures_2023} and \cite{hintz_gluing_2023}.
	We don't use any one particular result from these works, but the ideas discussed here have origins in these works.
	Most of the technical difficulties for \cref{i:thm:existence_of_ansatz} were already dealt with in our previous work, \cite{kadar_scattering_2024}.

	We start, by assuming that a polyhomogeneous solution exists.\footnote{For the discussion, we omit the details of how to compactify with the logarithmic corrections to the path of the solitons, as these only seem to be a technical difficulty.}
	Then, we can restrict the leading terms of both sides of \cref{i:eq:main} to each of the faces $I^+,F_a$.
	The linear part of the operator has leading behaviour, called \emph{model operator}, $\Box$ and $\Delta+5W_a^4$ on $I^+$ and $F_a$, respectively.
	Let's call the error terms generated by the nonlinearity and model operators on the faces $I^+,F_a$ at a particular step in the iteration scheme $t^{-p^+}f^+(x/t),t^{-p_a}f_a(y_a)$.
	We start with the naive ansatz of sum of unmodulated solitons yielding an error
	\begin{equation}\label{i:eq:f_start}
		f_{\mathrm{start}}:=\Box \sum_a W_a+\Big(\sum_a W_a\Big)^5=\Big(\sum_a W_a\Big)^5-\sum_a W_a^5.
	\end{equation}
	We correct for these error terms according to the following principle:
	\begin{enumerate}[label=Step \arabic*),leftmargin=37pt]
		\item\label{i:item:step1} When $p^+\leq p_a+2$, we solve a model operator on $I^+$.
		The correction to the ansatz takes the form $t^{-p_++2}\phi_{p_+}^+(x/t)$.
		Fixing the decay rate turns $\Box$ into an elliptic operator, $t^{-p_+}N_{p_+}\phi_{p_+}^+=\Box t^{-p_++2}\phi_{p_+}^+(x/t)$, where $N_\sigma$ is the shifted Laplacian on hyperbolic space, in our approach identified with the Poincare disk ($B$).
		The error term ($f^+$) in turn is polyhomogeneous towards the boundary of the disk ($\partial B$) and towards the location of the solitons, with at most quadratic divergences towards the latter.
		We prescribe boundary conditions at $\partial B$ such that the corresponding solution in $\mathcal{M}$ has no outgoing radiation.
		Using elliptic theory, $N_{p_+}\phi=f^+$ always has a polyhomogeneous solution with at most $\O(1)$ behaviour towards the soliton locations.
		This procedure introduces an error towards $F_a$ decaying at least $t^{-(p^+-2)}$ fast, thereby not ruining the already obtained behaviour there.
		\item\label{i:item:step2} When $p^+\geq p_a+3$, we solve a model operator on $F_a$, with correction to the ansatz given by $t^{-p_a}\phi_{p_a}^a(y_a)$.
		The corresponding elliptic problem is $(\Delta+5W_a^4)\phi^a_{p_a}=f_a$, where $f_a$ is polyhomogeneous and decays at least like $r^{-3}$.
		Since the linear operator has a kernel, this is not solvable in general.
		We discuss the correction for the spherically symmetric part here, a change in the $\ell=1$ spherical mode follows similarly with 2 powers of $t$ losses instead 1.
		We add $ct^{-p_a+1}\Lambda W_a$ to the ansatz.
		This yields an extra error at decay order $t^{-p_a-2}$ at $I^+$.
		Correcting this as described in \cref{i:item:step1}, yields an error term proportional to $t^{-p_a} 5W_a^4$.
		We use this to cancel the part of $f_a$ not orthogonal to $\Lambda W_a$.
		Alternatively, we can use a radiative solution\footnote{this amount to solving $N_{p_a}\phi^+_{p_a}=0$ form \cref{i:item:step1} with no error terms and well chosen boundary conditions at $\partial B$ so as to cancel the kernel elements of $f_a$} to arrive to the same modification.
		Then, we simply invert $(\Delta+5W_a^4)$ to get a solution that decays like $r^{-1}$.
		This in turn generates an error term towards $I^+$ with decay at least $t^{-(p_a+3)}$, not ruining the previous structure.
	\end{enumerate}
	
	Iterating this construction yields the result, up to the following caveats:
	
	\begin{remark}[Modulation]
		As discussed in \cref{i:item:step2}, we modulate the scaling of the solitons by explicitly adding corrections of the form $t^{-p}\Lambda W$.
		Using a series expansions, one can see that this is equivalent of implicitly modulating the scales $\lambda_a$.
		For the position (and velocity) we choose an implicit modulation, as in that case, this significantly improves the handle of error terms.
	\end{remark}
	
	\begin{remark}[Starting conditions]\label{i:rem:startin_cond}
		As we already seen in \cref{i:item:step2}, solving for an error with leading order $t^{-p_a}$, we need in general a correction of size $t^{-p_a+1}\Lambda W_a$ for the scaling.
		Since the leading order term in $f_{\mathrm{start}}$ at the location of each soliton is $t^{-1}$, this would require $\mathcal{O}(1)$ correction.
		Due to a cancellation at this order, the actual correction would be of size $\mathcal{O}(\log t)$
		This would dominate the leading order behaviour, and thus the approximate solution would not converge even in $L^\infty$ to the sum of fixed size solitons.
		Therefore, we impose the condition that $f_{\mathrm{start}}$ at each $F_a$ decays at least like $t^{-2}$.\footnote{This condition is easy to satisfy, see \cref{an:lemma:existence of admissible points}}
	\end{remark}
	
	\begin{remark}[Logarithmic correction to paths]
		As already discussed in \cref{i:sec:heuristic}, we expect the soliton paths to follow a logarithmically corrected path.
		In \cref{i:item:step2}, we also mentioned that an error with $p_a=2$ leads to a path correction at order $\O(1)$, or more precisely to $\O(\log t)$. 
		Indeed, since we restrict to initial conditions such that $p_a=2$ is the leading correction, we can read off the coefficients of the log terms precisely.
	\end{remark}
	
	Let us emphasis, that the above construction crucially relies on the fact that the leading order errors between the solitons cancel.
	Indeed, it appears that provided a polyhomogeneous solution exists on the compactification $\D^\g$, where this leading nonlinear effect is non-zero, the modulated scaling modes will not converge.
	This motivates the following weaker version of \cref{i:conj_martel}
	
	\begin{conjecture}\label{i:conj:non_existence_of_solitons}
		Fix distinct soliton speeds such that $f_{\mathrm{start}}$ from \cref{i:eq:f_start} has leading $t^{-1}$ behaviour at one of the solitons $F_a$.
		Then, there exists no \emph{approximate} modulated multi-soliton\footnote{we of course require these solitons to be the one corresponding to the ground state $W$} solution to \cref{i:eq:main} with no outgoing radiation such that the scaling of the solitons is also convergent.
		In particular, there exist $N>0$ such that no smooth function without outgoing radiation of the form $\bar{\phi}=\sum_a W_a+\tilde{\phi}$ with $\abs{\tilde{\phi}}\lesssim t^{-1}$ satisfies
		\begin{equation}
			\abs{\Box \bar{\phi}+\bar{\phi}^5}\lesssim t^{-5}(t-r)^{-N}.
		\end{equation}
	\end{conjecture}
	
	\begin{remark}
		We highlight, that we do not conjecture that there is no solution to \cref{i:eq:main} approaching a sum of solitons and no outgoing radiation, only that if such a solution exists, it will have a non convergent scaling factor.  
	\end{remark}
	
	\paragraph{Energy super critical problems}
	The scaling mode, as we've seen is a crucial obstruction of constructing solutions without special initial configuration.
	Since \cref{i:eq:supercritical} breaks the scaling invariance, we have no problem inverting the normal operator $F_a$ at the soliton faces when the error is spherically symmetric.
	As is easy to show, the leading $t^{-1}$ error discussed in \cref{i:rem:startin_cond} is spherically symmetric.
	This leads us to our second theorem.
	\begin{theorem}[Rough version of \cref{an:thm:supercritical}]\label{i:thm:ansatz_super}
		Given any distinct soliton velocities and $N\geq0$, there exists an ansatz $\tilde{\phi}\in\O_\scri^{1,1,1}$, such that for velocity-modulated solitons $W^{\mathrm{sup}}_a$, we have $\bar{\phi}=\sum_a W^{\mathrm{sup}}_a+\tilde{\phi}$ satisfies
		\begin{equation}
			\Box\bar{\phi}+\bar{\phi}^7-\bar{\phi}^9\in\O_{\scri}^{5,N,N}.
		\end{equation}
	\end{theorem}
	Note, that in this case, we require no special condition on the solitons location.
	Indeed, the proof follows the exact same steps as for \cref{i:thm:existence_of_ansatz}, with no need for correcting for spherically symmetric part of the kernel, but still modulating the velocities.
	The only feature of the soliton $W^{\mathrm{sup}}$, necessary for the construction is that its linearised operator satisfies\footnote{this condition on the kernel elements is also called non-degenerecy and it plays an important role in many stability problems, see \cite{lewin_double-power_2020} and references therein for further discussion}
	\begin{equation}
		\ker(\Delta+7(W^{\sup})^6-9(W^{\sup})^8)=\mathrm{span}\{\partial_i W^{\mathrm{sup}}\}.
	\end{equation}
	Motivated by the construction \cref{i:thm:ansatz_super}, we make the following conjecture	
	\begin{conjecture}[Parabolic orbit]\label{i:conj:parabolic_orbit}
		There exists a compactification of Minkowski space, $\D^{\mathrm{par}}$, and a polyhomogeneous function $\bar{\phi}$ on $\D^{\mathrm{par}}$ solving \cref{i:eq:supercritical} to $(t-r)^{-N}$ errors and satisfying
		\begin{equation}
			\lim_{t\to\infty}\abs{\bar{\phi}(x,t)-W^{\mathrm{sup}}(y^+_{\mathrm{par}})-W^{\mathrm{sup}}(y^-_{\mathrm{par}})}_{L^\infty}=0
		\end{equation}
		where $y^{\pm}_{\mathrm{par}}=x\pm z t^{2/3}$, for some $z\in\R^3$.
		Furthermore the value of $z$, up to rotational symmetry is unique.
	\end{conjecture}

	\subsubsection{Nonlinear solution}\label{i:sec:nonlinear}
	Next, we discuss the steps necessary to correct the approximate solution $\bar{\phi}$ to a true solution.
	We follow the dyadic approach\footnote{although our foliation time is not going to double between iterations, but only be multiplied by a factor $(1+\delta)$ for $\delta>0$, we still call it a \emph{dyadic} sequence as we find the name $(1+\delta)$-\emph{adic} is no more descriptive} outlined in \cite{dafermos_quasilinear_2022}, therefore we firstly focus on the linearised problem
	\begin{equation}\label{i:eq:lin_bar_phi}
		(\Box+5\bar{\phi}^4)\phi=f
	\end{equation}
	where $\bar{\phi}$ is as in \cref{i:thm:existence_of_ansatz}.
	
	\paragraph{Linearised problem}
	We first discuss the case, when $\bar{\phi}=\sum W_a$ for unmodulated solitons, and comment on the error terms later.
	We build a global solution by a compactness arguments, applied to local solutions, therefore, we first study the solution between two slices of spacetime as indicated on \cref{i:fig:foliation_mine}.
	Each slice ($\Sigma_\tau$) is composed of the following 4 regions:
	\begin{itemize}
		\item A null portion starting from $r=\delta_\scri t=(1-\delta_4)t$ and going out to infinity. This is spherically symmetric with respect to the $x$ coordinate.
		\item A spacelike part far from the solitons.
		This coincides with level sets of $t$, but transitions into level sets of the local time functions $t_a=\gamma_a(t-z_a\cdot x)$, up to an overall $\gamma_a$ factor corresponding to relativistic time dilation.
		The transition region is $\abs{y_a}\sim \delta_4 t_a$.
		Together with the previous null portion, we call this $\Sigma^{\mathrm{ext}}$.
		\item A null part, close to each soliton.
		The transition into null parts happens at $\abs{y_a}\sim \delta_3 t$.
		These null cones are spherically symmetric with respect to the local coordinates $y_a$.
		\item  A spacelike part, which coincides with the level set of $t_a$.
		This transition happens at $\abs{y_a}\sim R_2=c_2\log(t)$.
		Together with the previous null portion, we call this $\Sigma^{\mathrm{int}}$.
	\end{itemize}
	
	\begin{figure}[htbp]
		\centering
		\includegraphics[width=400pt]{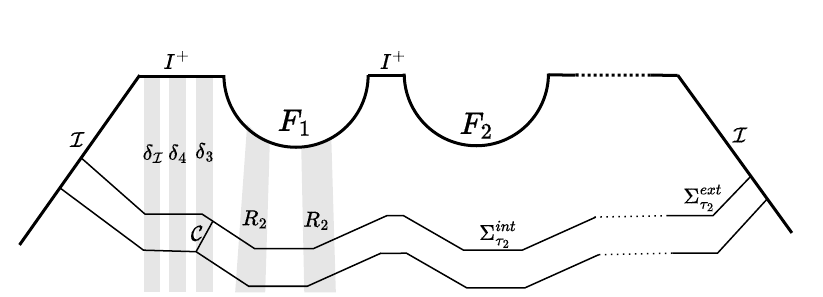}
		\caption{The compactification on which we prove \cref{i:thm:existence_of_ansatz} and the foliation used for \cref{i:thm:existence_of_scattering_solution}.}
		\label{i:fig:foliation_mine}
	\end{figure}
	
	Another important component of the spacetime are the null hypersurfaces $\C_a$, around each soliton.
	These are the transitional regions between interior and exterior parts of the spacetime.
	We are interested in proving energy boundedness for \cref{i:eq:lin_bar_phi} between two slices $\Sigma_{\tau_1}$ and $\Sigma_{\tau_2}$, where $\tau_1\in[(1-\delta_2)\tau_2,\tau_2]$.
	Due to the positive eigenvalue of \cref{i:eq:lin}, this is in general not possible.
	Therefore, we must restrict to the study of solutions that have appropriately small projections to these unstable modes.
	Then a solution is constructed satisfying the assumption by a standard topological argument.
	Indeed, the reason to enlarge the flat parts around the solitons in a logarithmic way is to guarantee that these unstable modes remain sufficiently small.
	
	The main result on the linearised problem is the following
	\begin{theorem}[Rough version of \cref{lin:prop:main,lin:prop:main_weak}]\label{i:thm:lin}
		Let $\phi$ be a solution of \cref{i:eq:lin_bar_phi} between $\Sigma_{\tau_1}$ and $\Sigma_{\tau_2}$ for some $\tau_1\in[\tau_2(1-\delta_3/4),\tau_2]$, $\tau_2\gg1$, with sufficiently small projection onto the unstable modes and no outgoing radiation.
		There exists coercive (polynomially weighted) norm $\master_\tau[\phi]$ on $\Sigma_\tau$, controlling $k$ $\b$-derivatives of $\phi$ and a universal constant $C$ (depending only on $k$ and the location of the solitons) such that 
		\begin{equation}\label{i:eq:dyadic_estimate}
			\master_{\tau_1}[\phi]\leq C\master_{\tau_2}[\phi]+\inhom_{\tau_1,\tau_2}[f],
		\end{equation}
		where $\inhom$ is a polynomially weighted norm.
	\end{theorem}

	\begin{remark}[Error term]
		The reason we need to split \cref{i:thm:lin} is due to the error terms
		\begin{equation}
			\mathfrak{Err}^{\mathrm{\lin},a}:=5\bar{\phi}^4-5W_a^4.
		\end{equation}
		In case $\mathfrak{Err}^{\mathrm{\lin},a}$ decays faster than $1/t^2$ close to $F_a$, we can treat it as an error term irrespective of its structure.
		However, the leading error both in \cref{i:thm:existence_of_ansatz,i:thm:ansatz_super} behave like $1/t$.
		These terms need to satisfy a certain orthogonality condition to treat them perturbatively, as shown in \cref{lin:sec:error_term_cancellation}.
	\end{remark}
	
	\begin{remark}[Loss for conormal regularity]
		Although, we say in \cref{i:thm:lin} that the norms control $\b$-derivatives, in reality, we are not able to obtain such a strong control.
		In particular, each derivative looses a weight $t^{-\epsilon}$, for an arbitrary small $\epsilon>0$.
	\end{remark}
	
	\subparagraph{Ideas of proof}
	The proof of \cref{i:thm:lin} is given in two steps.
	First, we study the solution in the exterior of the cones $\C$.
	In this region, the equation is sufficiently close to the Minkowski wave equation, so that all other terms in the linearisation are perturbative.
	To state the result here, let us introduce $\T[\phi]$ for the energy momentum tensor of $\phi$, see \cref{not:eq:energy_mom_tensor}.
	An application of the divergence theorem yields that the energy $\E_{\Sigma_{\tau}^{\mathrm{ext}}}[\phi]:=\Sigma^{\mathrm{ext}}[J^\E]$ defined by the flux of the current $J^\E=\partial_t \cdot \T[\phi]$ through $\Sigma^{\ext}_\tau$ satisfies
	\begin{equation}
		\E_{\Sigma_{\tau_1}^{\mathrm{ext}}}[\phi]+\sum_a \E_{\C_{a}}[\phi]\lesssim \E_{\Sigma_{\tau_2}^{\mathrm{ext}}}[\phi].
	\end{equation}
	Here $\E_{\Sigma}$ is a quadratic coercive quantity that controls all derivatives on a spacelike hypersurface and tangential derivatives on null hypersurfaces.
	We next, use the control of the energy on the cones to prove boundedness in the interior regions.
	
	Just as for the creation of an ansatz, the main difficulty in proving \cref{i:thm:lin}
	is the kernel elements associated to scaling and translation.
	Building on the constructions in \cite{kadar_scattering_2024}, we obtain orthogonality with respect to these elements using conservation laws.
	We discuss the one corresponding to $\Lambda W$, as the one for $\partial_i W$ follows similarly.
	
	For the equation \cref{i:eq:lin}, there is a divergent free current $J^{\E,V}[\phi]$, which controls the energy ($\E_{\Sigma_\tau}[\phi]$) of $\phi$ up to unstable modes and kernel elements.
	From the divergence theorem, we already have that the flux of $J^{\E,V}$ through $\Sigma_\tau^{\internal}$ is controlled by $\E_{\C}[\phi]$.
	
	There also exists a current $J^\Lambda[\phi]$, that has nonzero flux through hypersurfaces terminating on $\Sigma^{\mathrm{int}}$ for $\phi=\Lambda W$.
	The flux $\Sigma_\tau[J^\Lambda[\phi]]$ is  linear in $\phi$.
	We can use the divergence theorem for this current to enforce orthogonality of a solution \cref{i:eq:lin}, $\phi$, with respect to $\Lambda W$, i.e.\footnote{of course, we need extra control over the unstable modes as well as the $\partial_i W$ kernel elements, but we ignore these for the sake of presentation here}
	\begin{equation}\label{i:eq:coercivity}
		\E_{\Sigma^{\internal}_\tau}[\phi]\lesssim\Sigma^{\internal}_\tau[J^{\E,V}[\phi]]+\abs{\Sigma_{\tau}^{\internal}[J^{\Lambda}[\phi]]}^2.
	\end{equation}
	However, the flux through the hypersurface $\C$ is
	\begin{equation}
		\C[J^\Lambda]\sim \int_{\C} \dd\tau \dd g_{S^2} (\partial_t-\partial_r)r\phi.
	\end{equation}
	Crucially, this is only bounded by $(\tau_2-\tau_1)^{1/2}\E_{\C}[\phi]^{1/2}$, and thus only can be controlled in a weaker decaying space than the energy $\E_{\C}[\phi]$.
	We resolve this, by modifying \cref{i:eq:coercivity} to\footnote{see \cref{lin:lemma:localisation} for a precise statement}
	\begin{equation}\label{i:eq:control_local_energy}
		\E_{\Sigma^{\internal}_\tau\cap\{\abs{y}>\delta_2\tau\}}+\frac{1}{\tau}\E_{\Sigma^{\internal}_\tau}[\phi]\lesssim\Sigma^{\internal}_\tau[J^{\E,V}[\phi]]+\frac{1}{\tau}\abs{\Sigma_{\tau}^{\internal}[J^{\Lambda}[\phi]]}^2.
	\end{equation}
	Thereby, we obtain good control of the energy in far part of the interior regions, and only loose control very close to the solitons.
	
	We finally comment on the error terms \cref{i:eq:lin_bar_phi}, that come from $\bar{\phi}$ not being explicitly equal to a single soliton in the local regions.
	The bulk part of \cref{sec:linear_theory} is devoted to estimating error terms arising from $\bar{\phi}$ that have linear error $\mathfrak{Err}^{\mathrm{\lin},a}\sim t^{-2}$  near $F_a$.
	This decay rate is sufficient to obtain estimates that still have plenty of room, indeed $t^{-1}$ decay is borderline for our methods.
	The leading term in $\mathfrak{Err}^{\mathrm{\lin},a}$ is $t^{-1}W^3\Lambda W$.
	As clear from \cref{i:eq:control_local_energy}, we control the coercive energy in a weaker decaying space that the current of $J^{\E,V}$.
	However, due to the orthogonality
	\begin{equation}
		\int_{\R^3}W^3(\Lambda W)^3=0,
	\end{equation}
	we can control the error terms coming from $t^{-1}W^3\Lambda W$ better than just the control provided by $\E_{\Sigma^{\internal}_\tau}[\phi]$.
	
	\paragraph{Nonlinear solutions:}
	Given the strong control in \cref{i:thm:lin}, it is straightforward to upgrade the result to the nonlinear equation
	\begin{equation}\label{i:eq:nonlinear}
		(\Box+5\bar{\phi}^4\phi)=f+\mathcal{N}[\phi]
	\end{equation}
	where $\mathcal{N}$ is some polynomial nonlinearity.
	\begin{cor}[Rough version of \cref{non:cor:slab}]\label{i:cor:nonlinear_slab}
		Let $\phi$ be a solution to \cref{i:eq:nonlinear} with no outgoing radiation in the same region as in \cref{i:thm:lin}.
		Assume that the we have sufficiently small projection onto the unstable modes.
		Then, for $\master_{\tau_2}[\phi]$ and $\inhom_{\tau_1,\tau_2}[f]$ inverse polynomially small in $\tau_2$, there exist a constant $C$ depending only on $k$ such that \cref{i:eq:dyadic_estimate} holds.
	\end{cor}
	The proof is reminiscent to a local existence, and we get the $\delta_3\tau_2$ size existence by suppressing the data and the inhomogeneity by a power of $\tau_2$.
	In fact, we can take the power to be $\tau_2^{-6}$.
	
	Next, we apply an iterative construction, reminiscent to \cite{dafermos_quasilinear_2022}, to construct a solution in an arbitrarily large unbounded region.
	
	\begin{corollary}
		Let $\phi$ be a solution to \cref{i:eq:nonlinear} between two slice $\Sigma_{\tau_2},\Sigma_{\tau_1}$, with $\tau_2>\tau_1\gg1$ and sufficiently small projection onto the unstable modes.
		Then, there exists $N$ depending only on the location of the solitons and $C$ from \cref{i:cor:nonlinear_slab} such that for $\tau_2^N\master_{\tau_2}[\phi]\leq1$ and $f\in\O^{5,4N,4N}_\scri$ we have $\tau_1^N\master_{\tau_1}[\phi]\lesssim_C 1$.
	\end{corollary}
	The proof is a simple iterative application of \cref{i:cor:nonlinear_slab}.
	Finally, using a topological argument to control the unstable mode, and a compactness argument yields
	\begin{theorem}[Rough version of \cref{non:thm:existence of scattering solution}]\label{i:thm:existence_of_scattering_solution}
		There exists $N$ sufficiently large, such that the approximate solution $\bar{\phi}$ of \cref{i:thm:existence_of_ansatz} admits a nonlinear correction with no outgoing radiation in the region $t-r\gg1$.
		In particular, there exist a solution $\phi$ of \cref{i:eq:main} such that
		\begin{eqnarray}
			\phi-\bar{\phi}\in\O^{2,N/2,N/2},\qquad \partial_r(r\phi)_{\scri}=0.
		\end{eqnarray}
	\end{theorem}
	
	\subsubsection{Further applications}\label{i:sec:applications}
	
	\paragraph{Larger class of problems}
	We note, that the proofs of the results in \cref{i:sec:ansatz,i:sec:nonlinear} are robust enough to apply to a class of semilinear problems.
	
	\begin{definition}[Supercritical class]
		We say that a polynomial $f(x)=\sum c_p x^p$ yields an \emph{admissible polynomial} if the following  holds:
		\begin{itemize}
			\item $c_p=0$ for all $p\leq 4$.
			\item $\Delta u+f(u)=0$ admits a polyhomogeneous solution $W^{\mathrm{sup}}$, decaying at least like $r^{-1}$.
			\item $W^{\mathrm{sup}}$  satisfies $\ker(\Delta+f'(W))=\{\partial_i W\}$.
		\end{itemize}
	\end{definition}

	Using the work of \cite{lewin_double-power_2020}, we know that the nonlinearity in \cref{i:eq:supercritical} is admissible, and indeed \cite{lewin_double-power_2020} provides a much larger class of admissible polynomials.
	
	\begin{theorem}\label{i:thm:energy_supercritical}
		For any admissible polynomial $f$ the results of \cref{i:thm:ansatz_super,i:thm:existence_of_scattering_solution} hold for
		\begin{equation}
			\Box\phi=f(\phi).	
		\end{equation}
	\end{theorem}
	
	\paragraph{Region near spacelike infinity}
	
	We quote the following theorem from \cite{kadar_scattering_2024}, which also follows from \cite{kadar_case_2024}.
	\begin{theorem}\label{i:thm:trivial_scattering}
		Given scattering data on a finite outgoing cone $\phi^{\mathcal{C}}\in\Hb^{1/2;10}(\{t-r=u_0\})$, there exists $d$ sufficiently large such that \cref{i:eq:main} admits a scattering solution with no outgoing radiation in $\{t+r>d, t-r<u_0\}$ such that
		\begin{equation}
				\phi|_{u=u_0}=\phi^{\mathcal{C}}.
		\end{equation}
	\end{theorem}
	
	Together with \cref{i:thm:existence_of_scattering_solution} this implies the existence of multi-soliton solutions in the future of a slice $\{t=T\}$ for $T$ sufficiently large, 	therefore \cref{i:conj_martel} is false.

	\subsection{Outline of the paper}\label{i:sec:outline}
	The paper has 3 different parts:
	
	In \cref{sec:notation}, we introduce the notation necessary to the study of approximate solutions as well as the nonlinear corrections.
	It introduces the analytic toolkit of conormal and polyhomogeneous functions as well as the different geometric coordinates foliations and other hypersurfaces.
	
	In \cref{sec:ansatz}, we prove \cref{i:thm:existence_of_ansatz,i:thm:ansatz_super}.
	We first study the necessary criteria for the first iterate in the construction to exists, then we prove the induction step.
	
	The rest of the paper is devoted to the proof of \cref{i:thm:existence_of_scattering_solution}.
	In \cref{sec:linear_theory}, we study the linear problem and prove \cref{i:thm:lin}.
	In \cref{sec:non-linear}, we use \cref{i:thm:lin} to prove \cref{i:thm:existence_of_scattering_solution} via the dyadic approach outlined above.
	
	\newpage

	\section{Geometric set up and notation}\label{sec:notation}
	The purpose of this section is to introduce the notation, analytic and geometric concepts used in the proof of \cref{i:thm:existence_of_ansatz,i:thm:existence_of_scattering_solution}.
	The analytic side, presented in \cref{not:sec:analytic}, contains polyhomogeneous and conormal spaces.
	Although this treatment is standard (see \cite{grieser_basics_2001}), we include it for the ease of the reader.
	In \cref{not:sec:geometry} we present the geometry vaguely introduced on \cref{i:fig:foliation_mine}.
	In particular, we will introduce coordinates relevant to the geometry.
	We will highlight which set of coordinates is useful for \cref{i:thm:existence_of_ansatz} and which are used for \cref{i:thm:existence_of_scattering_solution}, as they differ significantly.
	Let us already emphasis here, that due to the disjoint nature of their proofs, the buffer regions $\abs{y_a}\sim\delta t_a$ will be used for different purposes in the ansatz creation and the correction.
	Finally, we collect the notation for error terms in \cref{not:sec:errors} and all symbols used in \cref{not:sec:symbols}.

	\subsection{Analytic notation}\label{not:sec:analytic}
	
	In this subsection, we introduce the function spaces and other analytic tools used in the paper.
	The material presented here is standard and is almost identical to that in \cite{kadar_scattering_2024,kadar_case_2024}, for more details on these tools, see \cite{grieser_basics_2001}.
	
	Let's fix a manifold with corners $X=[0,1)_{x_1}\times...\times[0,1)_{x_n}\times Y$ for some smooth manifold $Y$. First we define the vector fields with respect to which we measure smoothness.
	\begin{definition}[b vector fields]
		Let 		
		\begin{equation}
			\begin{gathered}
				V=\{x_i \partial_{x_i},Y_i\}.
			\end{gathered}
		\end{equation}
		where $Y_i$ are smooth vector fields on $Y$ spanning the tangent space at each point. Furthermore, let's define
		\begin{equation}
			\begin{gathered}
				\Diff_b^1(X)=\sum_i f(x,y)V_i, \quad V_i\in V
			\end{gathered}
		\end{equation}	
		with $f_i\in\mathcal{C}^\infty(X)$. Also, let $\Diff^k_b(X)$ denote a $k$-fold product of elements in $\Diff^1_b(X)$.
	\end{definition}
	
	\begin{definition}[Multi-index notation]
		Let $\Gamma$ be a finite set of operators and $H$ a norm on $X$.  We write
		\begin{equation}
			\norm{\Gamma^kf}_{H(X)}=\sum_{\abs{\alpha}\leq k}\norm{\Gamma^{\alpha}f}_{H(X)}.
		\end{equation}
	\end{definition}
	
	\begin{definition}[$H_b$ norm]
		Given a measure $dy$ on $Y$, we define a naturally weighted norm ($L^2_\b$), higher order variants ($\Hb^{;k}$), and higher order variants with extra weights ($\Hb^{\vec{a};k}$)
		\begin{nalign}
			\norm{f}_{L^2_{\b}(X)}&:=\int f^2 \frac{d x_1}{x_1}...\frac{d x_n}{x_n}d y\\
			\norm{f}_{H^{;k}_b(X)}&:=\norm{V^k f}_{L^2_b}\\	
			\norm{f}_{\Hb^{\vec{a};k}(X)}&:=\norm{f}_{H^{a_1,...,a_n;s}_b(X)}:=\norm{{\prod_i x_i^{-a_i}f}}_{H^{;k}(X)}.
		\end{nalign}
	\end{definition}
	
	\begin{remark}
		Away from the boundary, the vector fields span the tangent space of $X$.
		Indeed, these spaces agree with the usual $L^2$ and $H^k$ spaces on compact subsets of $\text{int}(X)$.
		The normalisation for $L^2_{\b}$ is motivated by the observation that 
		\begin{equation}\label{eq:notation:decayHBLinfty}
			\Hb^{\vec{a};k+\dim(X)/2+1}(X)\subset x_{1}^{a_1}...x_n^{a_n}\mathcal{C}^{k}(X)\subset \Hb^{\vec{a}-;\infty}(X)
		\end{equation}
		where the first inclusion follows from Sobolev embedding.
		This already implies the following product rule
		\begin{equation}\label{not:eq:product_Hb}
			\norm{fg}_{\Hb^{\vec{a};k}}\lesssim_{\vec{a},k}\norm{f}_{\Hb^{\vec{a}^f;k}}\norm{g}_{\Hb^{\vec{a}^g;k}}
		\end{equation}
		for $k\geq \dim(X)/2+1$ and $\vec{a}=\vec{a}^f+\vec{a}^g$.
	\end{remark}
	
	The solutions studied in this paper have partial expansions towards different boundaries.
	These expansions are captured by \emph{index sets}.
	
	\begin{definition}
		A discrete subset of $\mathcal{E}\subset\R\times\N$ is called an index set if
		\begin{itemize}
			\item $(z,k)\in\mathcal{E}$ and $k\geq1$ implies $(z,k-1)\in\mathcal{E}$
			\item $(z,k)\in\mathcal{E}$ and $k\geq0$ implies $(z+1,k)\in\mathcal{E}$
			\item $\mathcal{E}_b:=\{(z,k)\in\mathcal{E}| z<b\}$ is finite for all $b\in\R$.
		\end{itemize}
		Furthermore, let's introduce the following notations
		
		\begin{itemize}
			\item $(z,k)\leq(z',k')$ whenever $z<z'$ or $z=z'$ and $k\geq k'$.
			We also write $z\geq(z',k')$ whenever $z\geq z'$.
			\item 
			$\min(\mathcal{E})=(z,k)\in\mathcal{E}$  such that $\forall(z',k')\in\mathcal{E}, (z,k)\leq(z',k')$.
		\end{itemize}
		
	\end{definition}

	\begin{definition}[Polyhomogeneity]
		Let $X=[0,1)_{x_1}\times Y$ be a manifold with boundary. Given an index set $\mathcal{E}$, we define the corresponding polyhomogeneous space $\A{phg}^{\mathcal{E}}(X)$. For $u\in x_1^{-\infty}\Hb^{\infty}(X)$ we say $u\in\A{phg}^{\mathcal{E}}(X)$ if there exist $a_{z,k}\in H^{\infty}(Y)$ for $(z,k)\in\mathcal{E}$ such that for all $c\in \R$ 
		\begin{equation}\label{eq:notation:polyhom def 1}
			\begin{gathered}
				u-\sum_{(z,k)\in\mathcal{E}_c} a_{z,k}x_1^{z}\log^kx_1\in x_1^c \Hb^{\infty}(X).
			\end{gathered}
		\end{equation}
		
		At a corner, we define mixed $b-$ polyhomogeneous spaces as follows. Let $X=[0,1)_{x_1}\times[0,1)_{x_2}\times Y$ be a manifold with corners, $b_2\in\R$ and $\mathcal{E}_1$ an index set. For $u\in x_1^{-\infty}x_2^{-\infty}H^{\infty}_b(X)$ we say $u\in \A{phg,b}^{\mathcal{E}_1,b_2}(X)$ if there exists $a_{z,k}\in x_2^{b_2}H^{\infty}_b([0,1)_{x_2}\times Y)$ such that
		\begin{equation}
			\begin{gathered}
				u-\sum_{(z,k)\in\mathcal{E}_c} a_{z,k}x_1^{z}\log^kx_1\in x_1^c x_2^{b_2}H^{\infty}_b(X).
			\end{gathered}
		\end{equation}
		
		We define polyhomogeneous space at the corner for a manifold with corners $X=[0,1)_{x_1}\times[0,1)_{x_2}\times Y$. For $u\in x_1^{-\infty}x_2^{-\infty}H^{\infty}_b(X)$ we say $u\in\A{phg}^{\mathcal{E}_1,\mathcal{E}_2}(X)$ if there exists $a_{(z,k)}\in\A{phg,b}^{\mathcal{E}_2}([0,1)_{x_2}\times Y)$ such that
		\begin{equation}
			\begin{gathered}
				u-\sum_{(z,k)\in(\mathcal{E}_1)_c} a_{z,k}x_1^{z}\log^kx_1\in \A{b,phg}^{c,\mathcal{E}_2}(X).
			\end{gathered}
		\end{equation}
		
	\end{definition}
	
	\begin{remark}
		Note that we may give an alternative, more geometric characterisation of a polyhomogenous function $u\in\A{phg}^{\mathcal{E}}([0,1)_{x}\times Y)$ as
		\begin{equation}\label{eq:notation:polyhom def 2}
			\begin{gathered}
				\Big(\prod_{(z,k)\in\mathcal{E}_c}(x\partial_x+z) \Big)u\in x^c H^\infty_b([0,1)_{x}\times Y).
			\end{gathered}
		\end{equation}
		From this, it is easy to see that the definition of $\mathcal{E}$ only depends on the smooth structure of $[0,1)_{x}\times Y$, that is given any coordinate change $x\to z$ such that $z'(x)$ is bounded away from 0 on $[0,1)$, the index set with respect to $x$ and $z$ coincide. We use this freedom to detect polyhomogeneity with respect to different choices of coordinates to suit our need.
	\end{remark}	
	
	\begin{definition}
		For index sets $\mathcal{E}_1,\mathcal{E}_2$, we define the index sets
		\begin{equation}\label{index set operations}
			\begin{gathered}
				\mathcal{E}_1\bar{\cup}\mathcal{E}_2:=\{(z,k+1)| \exists(z,k_i)\in\mathcal{E}_i,\, k_1+k_2\geq k\}\cup\mathcal{E}_1\cup\mathcal{E}_2\\\
				\mathcal{E}_1+\mathcal{E}_2:=\{(z,k)|\exists (z_i,k_i)\in\mathcal{E}_i,\, z_1+z_2=z,\,k_1+k_2=k\}\\
				\mathcal{E}_1-(z_2,k_2):=\{(z,k)|\exists (z_1,k_1)\in\mathcal{E}_1,\, z_1-z_2=z,\,k_1-k_2=k\}
			\end{gathered}
		\end{equation}
	\end{definition}

	Following \cref{not:eq:product_Hb}, let us also observe that for $f_\bullet\in\A{phg}^{\E^\bullet}$ with $\bullet\in\{1,2\}$ we have
	\begin{equation}\label{not:eq:product_phg}
		f_1f_2\in\A{phg}^{\E^1+\E^2}.
	\end{equation}

	We will encounter many index sets in the construction, so to ease notation we also introduce the following shorthand
	\begin{definition}
		For a discrete subset of $X\subset\R\times\N$ we write
		\begin{equation}
			\begin{gathered}
				\overline{X}=\cap_{X\subset\mathcal{E}}\mathcal{E}
			\end{gathered}
		\end{equation}
		for the smallest index set containing $X$. When $X$ has a single element we use the shorthand $\overline{(a,b)}:=\overline{\{(a,b)\}}=\{(a+n,l):n\in\N,l\leq b\}$.
	\end{definition}
	
	We also introduce spaces corresponding to the boundary $I^+,F_a$
	\begin{definition}[Standard compactifications]
		For Euclidean space $\R^n$, we introduce its radial compactification $\overline{\R^n}$. defined by extending the coordinates
		\begin{equation}
			\rho=\jpns{\abs{x}}^{-1},\qquad\omega=x/\abs{x}
		\end{equation}
		from the range $\rho\in[1/2,0)$ to $\rho\in [1/2,0]$\footnote{more explicitly, we introduce $\rho,\omega$ in an exterior region, say $\abs{x}>2$. Then we define the exterior region of $\overline{\R^n}$ by extending to $\rho=0$ in coordinates $\rho,x/\abs{x}$. Finally, we define $\overline{\R^n}$ via coordinate patches.}.
		
		For a finite set of point $\{z_a\}=A\subset B:=\{x\in\R^3:\abs{x}\leq1\}$ in the unit ball, we define the punctured and multi-punctured balls $\dot{B}_a=B\setminus\{z_a\}$, $\dot{B}=\B\setminus A$.
		The corresponding manifolds with boundary $\overline{\dot{B}_a},\overline{\dot{B}}$ are defined by extending the radial coordinates to $0$ around each removed point.
	\end{definition}
	
	\begin{remark}
			By an abuse of notation, we will write $\R^n,\dot{B}$ instead of   $\overline{\R^n}$ and $\overline{\dot{B}}$ when we write polyhomogeneous and $\Hb$ spaces.
			We use the notation $\Hb^{\vec{c};k}(\dot{B}_A)$ with $\vec{c}=(c_\scri,c_1,...,c_n)$ to denote the behaviour around $\partial B$ and $\partial B_a\setminus\partial B$ respectively.
			On $\R^n$, we write $\Hb^{c;k}(\R^n)$ b-Sobolev space corresponding to $\overline{\R^n}$.
			Similarly $\A{phg}$ spaces.
	\end{remark}
		
	\begin{definition}
		For a function $f\in L^2(\R^3)$ we introduce $P^{S^2}_lf$ to be the projection to the $l$-th spherical harmonic.
	\end{definition}
	
	\subsubsection{Eigenvalues and kernel elements}
	
	As discussed in the introduction the elliptic operator $\Delta+5W^4$ has the following eigenfunctions
	\begin{equation}
		(\Delta+5W^4)\Lambda W=0,\quad (\Delta+5W^4)\partial_i W=0,\quad (\Delta+5W^4)Y=\lamed Y
	\end{equation}
	for $\lamed>0$, where $Y$ is spherically symmetric.
	We record their explicit behaviour
	\begin{lemma}
		We have
		\begin{nalign}
			\Lambda W,\jpns{r}\partial_i W\in\A{phg}^{\overline{(1,0)}}(\R^3)\\
			e^{\lamed r}Y(r) \in\A{phg}^{\overline{(1,0)}}(\R^3)+\Hb^{2;\infty}.
		\end{nalign}
	\end{lemma}
	\begin{proof}
		The first two inclusions follow from explicit formulae.
		We prove the second statement.
		Let $g=e^{\lamed r}Y(r)r$.
		From elliptic regularity we have $Y\in\C^\infty$, from Proposition 3.9 \cite{duyckaerts_solutions_2016}, we have $\abs{g-c}\lesssim \jpns{r}^{-1/2}$ for some $c\neq0$.
		We compute that $g$ solve
		\begin{equation}\label{not:eq:g_ODE}
			(\partial_r^2+2\lamed\partial_r+V(r))g=0.
		\end{equation}
		We can not conclude polyhomogeneity for $g$ using the Frobenius method, as the ODE has an irregular singular point.
		We proceed in an explicit fashion.
		Multiplying both side by $\partial_r g$ and in integrating we get
		\begin{multline}
			0=\int_{r=0}^R \partial_r(\partial_r g)^2+2\lamed (\partial_r g)^2+V\partial_r g^2=\big((\partial_r g)^2-Vg^2\big)|_{r=R}-\big((\partial_r g)^2-Vg^2\big)|_{r=0}+\\
			\int_{r=0}^R2\lamed (\partial_r g)^2-g^2V'.
		\end{multline}
		Using that $\lamed>0$, we get that $\partial_rg\in L^2\cap L^\infty$.
		Multiplying with $r^\alpha\partial_r g$ for $\alpha\in\{1,2,3,4-\epsilon\}$ we get that $\partial_r g\in r^{-2+\epsilon}L^2$.
		Now, we multiply \cref{not:eq:g_ODE} by $r$ and commute with $r\partial_r$ to get
		\begin{equation}
			(\partial_r^2+2\lamed\partial_r+V(r))r\partial_rg=(\partial_r^2-V-V'r)g.
		\end{equation}
		Using the multipliers $r^\alpha\partial_r (r\partial_r g)$ for $\alpha\in\{0,1,2,3,4-\epsilon\}$, we find that $\partial_rg\in \Hb^{2-\epsilon;1}(\R)$.
		Commuting with $r\partial_r$ further, we obtain the infinite regularity statement.
	\end{proof}
	
	\subsection{Geometry}\label{not:sec:geometry}
	The geometric approach to energy estimates, and their behaviour on null cones is standard, and we refer to \cite{dafermos_lectures_2008,dafermos_quasilinear_2022} for further explanations.
	
	\subsubsection{Basic}
	We start with introducing geometric qunatities related to solutions of the linearised equation \cref{i:eq:lin}.
	We define the energy momentum tensor
	\begin{equation}\label{not:eq:energy_mom_tensor}
		\begin{gathered}
			\T^w_{\mu\nu}[\phi]=\partial_\mu\phi\partial_\nu\phi-\frac{\eta_{\nu\mu}}{2}(\partial\phi\cdot\partial\phi-w\phi^2),\qquad \partial\phi\cdot\partial\phi=\partial_\sigma\phi\partial^\sigma\phi.
		\end{gathered}
	\end{equation}
	Note that for $\phi$ a solution of \cref{i:eq:lin}, $\text{div}(\T^V[\phi])=\phi^2\text{grad}(V)$ and so $\text{div}(T\cdot\T^V[\phi])=0$.
	We also introduce the bilinear energy momentum tensor
	\begin{equation}\label{not:eq:bilinear energy mom tensor}
		\begin{gathered}
			\T^w_{\mu\nu}[f,g]=\partial_{(\mu}f\partial_{\nu)}g-\frac{\eta_{\mu\nu}}{2}(\partial f\cdot\partial g-wfg).
		\end{gathered}
	\end{equation}
	
	Following \cite{holzegel_boundedness_2014} and \cite{dafermos_quasilinear_2022} we introduce the twisted energy momentum tensor
	\begin{equation}\label{not:eq:twisted_energy_mom_tensor}
		\begin{gathered}
			\tilde{\T}^w_{\mu\nu}[\phi]=\tilde{\partial}_\nu\phi\tilde{\partial}_\mu\phi-\frac{\eta_{\mu\nu}}{2}\big(\tilde{\partial}\phi\cdot\tilde{\partial}\phi-w\phi^2+w'\phi^2\big)\\
			\tilde{\partial}(\cdot)=\beta\partial(\beta^{-1}\cdot ),\quad w'=-\frac{\Box\beta}{\beta}
		\end{gathered}
	\end{equation}
	Although $\tilde{\T}$ is not divergence free, we have the following result from \cite{holzegel_boundedness_2014} Proposition 3
	\begin{equation}\label{not:eq:twisted_conservation}
		\begin{gathered}
			\partial^\mu (X^\nu \tilde{\T}^w_{\mu\nu}[\phi])=X^\nu S_\nu[\phi]+\frac{1}{2}\phi^2X(w)\\
			S_\nu[\phi]=-\frac{\beta^{-1}\partial_\nu (\beta w')}{2\beta}\phi^2+\frac{\beta^{-1}\partial_\nu\beta^2}{2\beta}\tilde{\partial}_\sigma\phi\cdot\tilde{\partial}^\sigma\phi
		\end{gathered}
	\end{equation} 
	where $X$ is a Killing vector field. When using $\tilde{\T}$, we will take $X=T$, $\beta=\jpns{r}^{-1}$ which yields $X^\nu S_\nu=0$ and $w'=3\jpns{r}^{-4}$. 
	We use $\bar{\T}=\T+\tilde{\T}$.
	
	For a hypersurface $\Sigma$ and a current (1-form) $J$, let's denote
	\begin{equation}
			\Sigma[J]=\int_\Sigma J\cdot n
	\end{equation}
	the flux of the current through $\Sigma$.
	For instance, we may compute, see \cref{app:sec:flux calculation}, that for $\tilde{\Sigma}=\{t=0\}$ we have that the energy $\E$ of $\phi$ is given by
	\begin{equation}\label{not:eq:energy integral}
			\mathcal{E}^{w}_{\tilde{\Sigma}}[\phi]:=-\tilde{\Sigma}[(\partial_t)\cdot \T^w[\phi]]=\frac{1}{2}\int_{\tilde{\Sigma}} \Big((T_a\phi)^2+\abs{X_a\phi}^2-w\phi^2\Big),
	\end{equation}
	where the measure it the one given by the diffeomorphism $\tilde{\Sigma}\ni(t,x)\mapsto x\in\R^3$.
	Note, that this quantity is is not coercive for $w=V$, so we cannot use it to control the solution.
	Indeed, an important point of \cite{kadar_scattering_2024} was to obtain control over $\mathcal{E}^0$ in terms of $\mathcal{E}^V$.

	\subsubsection{Coordinates}
	Our foliations are going to be composed of multiple flat and null pieces.
	They are constructed between two time steps $\tau_1\in[(1-\delta_3/4)\tau_2,\tau_2]$ and have $\tau_2$ dependent parameters.
	In a neighbourhood of each soliton, we use a single flat region, that smoothly turns null after a finite $R_2=c_2\log\tau_2$ distance away from the center of the soliton following the framework of \cite{dafermos_quasilinear_2022}.
	Far away from the solitons, we depart from the usual hyperboloidal setting by introducing a new flat region growing linearly in $\tau$ as we increase the foliation time.
	The corresponding cutoff will be denoted by $\delta_\scri=(1-\delta_4)$.
	We will also use constants $d$, $c_1<c_2$, $R_i(\tau)=c_i\log\tau$ and $\delta_1<\delta_2/d<\delta_3/d^2<\delta_4/d^3$ to build our foliations.
	
	The geometric set-up we use in the rest of the paper will be valid if $\delta,d$ are chosen sufficiently small or large \emph{and} $\tau_2$ is sufficiently large.
	In particular, we only construct a foliation and a solution to \cref{i:eq:main} in a neighbourhood of timelike infinity.
	In turn, the estimates and theorems we prove have implicit dependence on $d,\delta$, but the dependence on $c_1,c_2$ will always be explicit.
	We find it helpful to introduce these constants separately, but one could replace them with powers of a sufficiently small $\epsilon>0$.
	
	We use the usual null coordinates
	\begin{equation}
		\begin{gathered}
			v=\frac{t+r}{2}\qquad u=\frac{t-r}{2}.
		\end{gathered}
	\end{equation}
	To define our foliation, we introduce a global cutoff function localising to the region $x>d$
	\begin{lemma}\label{not:lemma:cutoff_existence}
		There exists $\bar{\chi}\in\mathcal{C}^\infty(\R)$ cutoff function such that 
		\begin{nalign}
			\bar{\chi}|_{\{x<1\}}=0,\quad \bar{\chi}_{\{x>d\}}=1,\quad \bar{\chi}'\in(0,1),\quad
			\abs{x\bar\chi'(x)}\leq2/\log(d)
		\end{nalign} 
		
		We set $\chi_R=\bar{\chi}(\frac{x}{R})$ and $\chi^c_R=1-\chi_R$.
	\end{lemma}
	\begin{proof}
		Satisfying the first three conditions is standard, while the last one follows from the change of coordinates $s=\log x$.
	\end{proof}

	Next, we introduce the smooth function transitioning the flat part into null. 
	The function $h(x)=\int_{-\infty}^xdy \bar{\chi}(y)$ satisfies: $h|_{y\leq1}=0,\, h|_{y\geq d}=h(d)-d+y,\, h'\in[0,1]$.
	Define new time coordinates $t^{R}_\star:=t+h_R$ for $h_R:=Rh(\frac{r-R}{R})$.
	Using the implicit function theorem, we also define $t^{\e}_\star=t^{(1-\delta_4)t_\star^{\e}}_\star$, where we set the $R$ value to increase with $t^{\e}_\star$ ($\e$ stands for external). We fix the vector fields associated to these coordinates
	\begin{equation}
		\begin{gathered}
			T=\partial_t|_r=\partial_{t_\star}|_r,\quad \Omega_{ij}=x_i\partial_j-x_j\partial_i,\quad  S=t\partial_t|_r+r\partial_r|_t,\quad X_\star=\partial_i|_{t_\star},\quad X=\partial_i|_t=X_\star-\hat{x}_ih'T
		\end{gathered}
	\end{equation}
	where the index $i$ is kept implicit in $X_\star$ which denotes a 3-tuple of vector fields. 
	We also set $X^{\mathrm{r}}_\star=\hat{x}\cdot X_\star,\, X^{\mathrm{r}}=\hat{x}\cdot X$.
	We will also use $\Lambda=1/2+S$, but only as it acts on $W$. 
	In $(t^{R}_\star,x)$ coordinates the wave operator and partial derivatives take the form
	\begin{equation}\label{not:eq:wave__operator_in_t_star}
		\begin{gathered}
			\Box=-(1-h'^2)T^2+(-2h'T+X^\r_\star)X^\r_\star-R^{-1}h''T+\frac{2}{r}(X^\r_\star-h'T)+\frac{\slashed{\Delta}}{r^2},\\
			\partial_u|_v=T-X^\r=(1+h')T-X^\r_\star,\quad \partial_v|_u=T+X^\r=(1-h')T+X^\r_\star.
		\end{gathered}
	\end{equation}
	The Minkowski metric and its inverse in ($t_\star^R,x$) coordinates are given by
	\begin{nalign}\label{not:eq:metric}
		\eta=-(\dd t_\star)^2+2h'\dd r\dd t_\star+(1-(h')^2)\dd r^2+r^2\dd g_{S^2}\\
		\eta^{-1}=-(1-(h')^2) T^2-2h'T\otimes X^\r_\star+(X^\r_\star)^2+\frac{1}{r^2}g^{-1}_{S^2}
	\end{nalign}
	
	We can compute the energy (\cref{not:eq:energy integral}) on $\Sigma=\{t_\star=0\}$ hypersurface to be
	\begin{equation}
		\mathcal{E}^{w}_{\Sigma}[\phi]:=-\Sigma[(\partial_t)\cdot \T^w[\phi]]=\frac{1}{2}\int_{\Sigma} \Big((1-h_R'^2)(T\phi)^2+\abs{X_\star\phi}^2-w\phi^2\Big),
	\end{equation}
	where again the measure is the one from the diffeomorphism $\Sigma\ni(t,x)\mapsto x\in\R^3$.
	
	\begin{definition}
		For $f\in\C^{\infty}(\R^3)$ we define the seminorm, implicitly depending on $R$,
		\begin{equation}
			\norm{f}_{\mathfrak{X}_1}:=\int(1-(h'_R)^2)f^2.
		\end{equation}
	\end{definition}
	
	We will consider finite number of solitons moving with subluminal velocities $z_a\in B\setminus\partial B$. 
	We call $A\subset B\setminus\partial B$ with $\abs{A}<\infty$ \textit{soliton velocities}.
	Without loss of generality we will always take $z_1=0\in A$. 
	For soliton velocities $A$ we write 
	\begin{equation}\label{not:eq:delta_definition}
		\begin{gathered}
			\delta_4=\frac{1}{10\max(100,d)}\min_{a\neq b}\{\abs{z_a-z_b}\}\min_a\{1-\abs{z_a}\},\quad A=\{z_a\}
		\end{gathered}
	\end{equation}
	with the dependence on $A$ implicit.
	The solitons will experience a logarithmic correction coming from a $1/r^2$ Newtonian force law, and we call the corresponding directions $A^l=\{z_a^{0,1}\}\subset{\R^3}$ the \textit{velocity corrections}, see already \cref{an:prop:starting}.
	These are uniquely determined from $A$, see \cref{an:eq:Newtonian_path_correction}.
	Here, we keep them arbitrary and only see in \cref{sec:ansatz} how their values get fixed.
	For brevity, we will sometimes refer to $A,A^l$ simply as soliton velocities.
	
	We introduce coordinates around the solitons.
	\begin{definition}[Local coordinates]\label{not:def:coordinates}
		For soliton velocities $A,A^l$ , we define local coordinates
		\begin{equation}\label{not:eq:coordinates}
			y_a=\gamma_a(x-z_a t),\quad t_a=\gamma_a (t-z_a\cdot x),\quad \gamma_a=(1-\abs{z_a^2})^{-1/2},\quad \tilde{y}_a=y_a-z_a^{0,1}\log t_a.
		\end{equation}
		We also define interpolating local and global coordinates that will be used in \cref{sec:ansatz}
	\begin{subequations}\label{not:eq:bared_coords}
		\begin{align}
			&\overline{y}_a=\tilde{y}_a\chi^c_{a,\delta_3}+\chi_{a,\delta_3}y_a,\quad \chi_{a,\delta}=\bar{\chi}_{\delta t_a}\big(\abs{y_a}\big),\label{not:eq:bared_coords1}\\
			&\bar{t}=\sum_a t^{\e}_\star-\gamma (z_a\cdot y_a)\chi^c_{a,\delta_4},\quad \overline{x}=\sum_a \gamma_a(\overline{y}_a+z_a t)\chi_{a,\delta_4}^c+\chi_{a,\delta_4}x\label{not:eq:bared_coords2}.
		\end{align}
	\end{subequations}
		Finally, we define the foliation dependent coordinates that are used in the energy estimates in \cref{sec:linear_theory}
		\begin{equation}\label{not:eq:foliation_dependent_coords}
			y_{a;\tau_2}=y_a-z_a^{0,1}\log(\gamma_a^{-1}\tau_2),\quad t_{a,\star}^{R,\tau_2}=t_a+h^{\tau_2}_{a,R}, \quad  h^{\tau_2}_a=Rh\Big(\frac{\abs{y_{a,\tau_2}}-R}{R}\Big).
		\end{equation}
		Unless otherwise stated we will take $t_{a,\star}^{\tau_2}=t_{a,\star}^{R_2,\tau_2}$ with $R_2=c_2\tau_2$ and keep the $R$ dependence implicit.
		Similarly, when clear from context, we drop the $\tau_2$ superscript.
		We define the corresponding vector fields
		\begin{nalign}
			T^a:=\partial_{t_a}|_{y_a}=\partial_{t_\star^a}|_{y_a},\quad X^a:=\partial_{y_a}|_{t_a},\quad S^a:=y_a\cdot X^a+t_a T^a,\quad\Omega^a_{ij}:=(y_a)_iX^a_j-(y_a)_jX^a_i\\
			\tilde{T}^a:=\partial_{t_a}|_{\tilde{y}_a}=T^a+\frac{z^{0,1}_a}{t_a}\cdot X^a,\quad \tilde{X}^a:=\partial_{\tilde{y}_a}|_{t_a}=X^a,\quad
			X_\star^{\tau_2}=\partial_{y_a}|_{t_\star^{a,\tau_2}}.
		\end{nalign}
		and relative speeds 
		\begin{equation}\label{eq:not:relative_speeds}
			z_{ab}=	\frac{1}{1-z_a\cdot z_b}\Big(\gamma_{ab}^{-1}z_a-z_b+\frac{\gamma_{ab}}{1+\gamma_{ab}}(z_a\cdot z_b)z_b\Big),\quad \gamma_{ab}:=(1-\abs{z_{ab}}^2)^{-1/2}.
		\end{equation}
	\end{definition}
	
	\begin{remark}[Relativistic velocity addition]
		The equation \cref{eq:not:relative_speeds} captures the speed of an observer $a$ as measured in the reference frame of $b$.
		More precisely, we have
		\begin{equation}
			t_a=\gamma_{ab}(t_b-z_{ab}\cdot y_b),\qquad 
			y_a=\gamma_{ab}(y_b-z_{ab}t_b).
		\end{equation}
	\end{remark}

	\begin{remark}
	We note that the global foliation $\{\bar{t}=\tau\}$ at the position $y_a=0$ corresponds to $\{t=\tau\}\cap\{y_a=0\}=\{\gamma_at_a=\tau\}$.
	That is, a global foliation necessarily experience time dilation as measured in the local frames of reference.
	\end{remark}
	
	In the case of colinear velocities $z_a$, we have the following simplification to \cref{eq:not:relative_speeds}
	\begin{equation}\label{eq:not:colinear_relative_speeds}
		z_{ab}=\frac{z_a-z_b}{1-z_az_b}.
	\end{equation}
	
	Using the assumption $z_1=0$, we will always write $x=y_1, t=t_1$ and we also introduce $\tilde{x}=\tilde{y}_1$. 
	Indeed, much of the local estimates in the region around $F_a$ will only be presented near $F_1$, as the other locations follow by doing the analysis in boosted coordinates up to an overall time dilation as mentioned above.
	In $t,\tilde{x}$ coordinates, we compute $T$
	\begin{equation}\label{not:eq:box_in_tilde}
		\Box=-\tilde{T}^2+\tilde{X}\cdot\tilde{X}-\frac{2}{t}z_1^{0,1}\cdot\tilde{X}\tilde{T}-\frac{1}{t^2}\big((z_1^{0,1}\cdot\tilde{X})^2-(z_1^{0,1}\cdot\tilde{X})\big)
	\end{equation}
	
	The different regions of spacetime, also require us to work with multiple twisted energy momentum tensors.
	Indeed, whenever we work near $F_a$, we will use the notation $\tilde{\T}_a$ to refer to the twisted energy momentum tensor where we use $\jpns{\tilde{y}_{a}}$ as the twisting function instead $\jpns{r}$ as in \cref{not:eq:twisted_conservation}.
	Similarly, we call $\bar{\T}_a=\T+\tilde{\T}_a$.
	
	\begin{lemma}
		a) For $t^{\e}_\star$ sufficiently large, $\bar{t}$ is a a non-spacelike hypersurfaces which is strictly timelike for $r<(1-\delta_4)t$.
		
		b) On $\Sigma_\tau$, $x\to\overline{x}$ is a diffeomorphism for $t_\star^\e$ sufficiently large.
	\end{lemma}
	\begin{proof}
		a) We first note, that 
		$\norm{\d t^{R}_\star}_g^2=g(\d t_\star^R,\d t_\star^R)=-1+h'^2\leq0$.
		Therefore, for any $R$, $\d t_\star^R$ is non spacelike and therefore for $t$ sufficiently large, also $t_\star$.
		
		We also notice, for $\delta_3$ sufficiently small, $\bar{t}$ is either equal to $t_\star$ or $t_\star-\gamma_a z_a\cdot y_a\chi_a^c=\gamma_a(t_a+\chi_a z_a\cdot y_a)$.
		In the latter case, we have
		\begin{equation}
			\eta\Big(\d (z_a\cdot y_a\chi_a),\d (z_a\cdot y_a\chi_a)\Big)=-\chi'^2\frac{(z_a\cdot y_a\abs{y_a})^2}{\delta_3 t^4}+\underbrace{\eta\Big(z_a\chi_a\cdot \d y_a+\chi_a'\frac{z_a\cdot y_a\d \abs{y_a}}{\delta_3 t}\Big)}_{\leq \abs{z_a}^2\big(1+\sup(x\bar\chi(x))\big)}
		\end{equation}
		Using \cref{not:lemma:cutoff_existence}, we get that $\norm{\d t_\star}_g^2\leq \gamma_a^2(-1+c\delta_3+\abs{z_a}^2(1+\frac{1}{\log d}))$ close to a $F_a$.
		Choosing $d^{-1},\delta_3$ sufficiently small, we can guarantee that $\norm{\d t_\star}_g^2\leq-1/2$ close to a $F_a$.
		
		b) Similarly as before, it suffices to study the map $x\mapsto \bar{x}$ around the solitons, otherwise it is the identity map.
		Moreover, we note that the transformation from $x$ to $\bar{x}$ in \cref{not:eq:bared_coords2}
		in $\supp \chi_{a,\delta_4}'$ is the identity map.
		Therefore, we only need to study the map in $\supp \chi^c_{a,\delta_3}$.
		Due to the logarithmic change, the map is clearly a diffeomorphism for $t_a$, equivalently to $t_\star^\e$ sufficiently large.			
	\end{proof}

	\subsubsection{Regions and foliations}
	
	Corresponding to the solitons, we introduce spacetime regions and hypersurfaces around them, as shown on \cref{i:fig:foliation_mine,not:fig:local_regions}.
	
	\begin{definition}[Regions of spacetime]\label{not:def:spacetime_regions}
		We introduce $\tau^\Delta:=(1-\delta_3/4)\tau$.
		For $\tau_1\in[\tau_2^\Delta,\tau_2]$ we introduce the following regions of spacetime
		\begin{equation}
			\begin{gathered}
				 \mathcal{D}^{\mathrm{c},a,\delta}_{\tau_2}=\Big\{\abs{y_{a,\tau_2}}<\abs{t_a-\gamma_a^{-1}\tau_2(1-\delta)}\Big\},\quad\Sigma^{\mathrm{g}}_\tau=\{t^{\e}_\star=\tau\},\\
				\Sigma^{a}_{\tau;\tau_2}=\Big\{t^{a,\tau_2}_\star(y_{a,\tau_2},t_a)=t^{a,\tau_2}_\star\big(\gamma_a^{-1}(\delta_3\tau_2+\tau_2-\tau),\gamma_a^{-1}\tau\big)\Big\},\quad
				\Sigma^{a,\delta}_{\tau;\tau_2}=\Sigma^{a}_{\tau;\tau_2}\cap \mathcal{D}^{a,\delta}_{\tau_2},\\
				\mathcal{R}^a_{\tau_1,\tau_2}=\bigcup_{\tau\in(\tau_1,\tau_2)}\Sigma_{\tau;\tau_2}^{\mathrm{c},a},\quad\mathcal{R}^{\e}_{\tau_1,\tau_2}=\bigcup_{\tau\in{(\tau_1,\tau_2)}}\Sigma^{\mathrm{g}}_{\tau}\setminus \bigcup_a \D^{\mathrm{c},a,\delta_3}_{\tau_2},\quad\Region=\bigcup_a\Region^a\cup\Region^\e,
				\\
				\mathcal{R}_{\tau_1,\tau_2}^{\mathrm{t},a}=\D^{\mathrm{c},a,\delta_3}_{\tau_2}\cap\mathcal{R}^a_{\tau_1,\tau_2}\cap\{\gamma_at_a\geq \tau_1\},\quad \Sigma^{\mathrm{t}}_{\tau_2}=\{t^{\e}_\star=\tau^\Delta\}\cap\mathcal{R}_{\tau_2}^{\mathrm{t},a},\\
				\mathcal{C}^{a,\delta}_{\tau_1,\tau_2}= \partial\Big(\big\{t^{a,\tau_2}_\star\in[\tau_1,\tau_2]\big\}\cap \mathcal{D}^{\mathrm{c},a,\delta}_{\tau_2}\Big)\backslash(\Sigma^{a}_{\tau_1}\cup\Sigma^{a}_{\tau_2}).
			\end{gathered}
		\end{equation}
		The superscripts $\mathrm{e},\mathrm{t},\mathrm{c}$ stand for \emph{external, transitional, conic} respectively.
	\end{definition}
	
	\begin{figure}[htbp]
		\centering
		\includegraphics[width=0.6\textwidth]{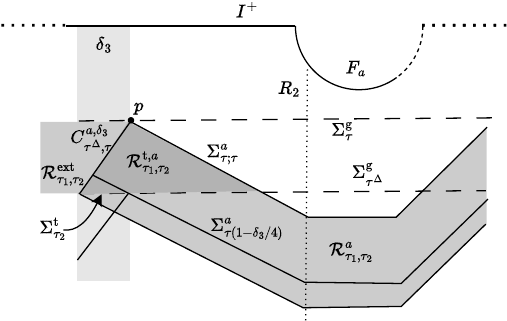}
		\caption{Visual representation of spacetime regions define in \cref{not:def:spacetime_regions}.
		The sphere denoted by $p$ is the set of points $\{\gamma_1t=\tau_2,\abs{y_{1,\tau_2}}=\delta_3\gamma_1\tau_2\}$.}
		\label{not:fig:local_regions}
	\end{figure}

	We summarise some important properties of the foliation in the remark and lemma below.
	
	\begin{remark}
		Since the leading order correction to the solitons path is logarithmic, their location from $\Sigma^{\g}_{\tau_2}$ to $\Sigma^{\g}_{\tau^{\Delta}_2}$ changes by $z^{0,1}_a\log(1-\delta_3)$.
		As we choose $\delta_3$ sufficiently small, it is sufficient to relocate the foliation at the end of a dyadic iteration.
	\end{remark}
	
	\begin{lemma}\label{not:lemma:spacetime_regions}
		Provided that we choose $\delta_4$ according to \cref{not:eq:delta_definition} the following holds for $\tau_2$ sufficiently large and $\tau_1\in[\tau_2^{\Delta},\tau_2]$
		\begin{enumerate}
			\item $\supp \partial\chi_{a,\delta_4}$ are disjoint, and therefore so are $\Region^a$.
			\item $\Sigma^{a}_{\tau^\Delta_2;\tau^\Delta_2}\cap\D^{\mathrm{c},a,\delta_3}_{\tau_2}\subset\mathcal{R}_{\tau_2^\Delta,\tau_2}^a$.
			\item For $\tau_1\in[\tau_2^\Delta,\tau_2]$ we have $\partial\mathcal{R}^{\mathrm{t},a}_{\tau_1,\tau_2}=\big(\C^a_{\tau_1,\tau_2;\tau_2}\cup \Sigma^{\g}_{\tau_1}\cup\Sigma^{a,\delta_3}_{\tau_2,\tau_2}\big)\cap\mathcal{R}^{\mathrm{t},a}_{\tau_1,\tau_2}$.
			\item $\mathcal{R}_{\tau_2^\Delta,\tau_2}\supset\mathcal{R}_{\tau_2^\Delta+1/2,\tau^\Delta_2+1}$.
		\end{enumerate}
	\end{lemma}
	
	\begin{proof}
		We proceed in order.
		In ${\D^{a,\delta_4}}$ we have $\abs{\frac{x}{t}-z_a}\leq \gamma_a^2 \delta_4/2$, indeed
		\begin{equation}
			\abs{\frac{x}{t}-z_a}=\abs{\frac{\gamma_a^{-2}y_a}{t_a+z_a\cdot y_a}}\leq \gamma_a^{2}\frac{\delta_4}{1-\abs{z_a}\delta}\leq \gamma_a^{2}\delta_4/2.
		\end{equation}
		Now, we use \cref{not:eq:delta_definition} to get
		\begin{equation}
			\gamma_a^{2}\delta_4/2\leq \frac{1(1-\abs{z_a})}{100(1-\abs{z_a}^2)}\min_{a\neq b}\{\abs{z_a-z_b}\}\leq \frac{1}{100}\min_{a\neq b}\{\abs{z_a-z_b}\}.
		\end{equation}
		Which proves the claim.
		
		2) and 4) are a consequence of the finite shift for the center of the solitons.
		
		3) This follows from choosing $\tau_2^\Delta>(1-\delta_3/4)\tau_2$, and the fact that the flat part of $\Sigma^a_{\tau,\tau_2}$ grows logarithmically.		
	\end{proof}

	\subsubsection{Compactification}\label{not:sec:compactification}
	We introduce a compactified version of the spacetime we are working on. This will be a manifold with corners. All of our estimates can be restricted to the interior, however we find it extremely useful for bookkeeping and conceptual purposes to have a compactification as well. 
	
	\begin{definition}[Logification]\label{not:def:compactification}
		Let's introduce global coordinates $\rho_\scri=\frac{t_\star}{t},\rho_{a}=\frac{\jpns{\overline{y}_a}}{t}, \rho_+=\frac{t}{t_\star}\sum_a\frac{1}{\jpns{\overline{y}_a}}$ that serve as defining function of null infinity ($\scri$), soliton faces ($F_a$), and timelike infinity ($I^+$).
		These satisfy $\rho_\bullet(\{u\geq u_0\})\subset(0,1]$ for $u_0$ sufficiently large. 
		Let's write $\D^{\g}$ for the compactification of $\{u\geq1\}$ with $\rho_\bullet$ smoothly extending to $0$.\footnote{see \cref{app:compactification} for more details.}
		
		On $\D^{\g}$, we will denote by $\Hb^{\vec{a};k}(\D^{\g}),\A{phg}^{\vec{\mathcal{E}}}(\D^{\g})$ with $\vec{a}=(a_\scri,a_+,a_1,...,a_n),\vec{\mathcal{E}}=(\mathcal{E}_\scri,\mathcal{E}_+,\mathcal{E}_1,...,\mathcal{E}_n)$ the corresponding conormal and polyhomogeneous spaces where the subscript denotes the behaviour close to the given boundary.
		We introduce the shorthand $\Hbloc^{a_\scri,a_+,a_{\loc}}$ for $a_1=a_2=...=a_n=a_{\loc}$.
	\end{definition}
	
	We introduce $\Vb$ to be the union of the function $1$ and a finite set of vector fields spanning $\Diff^1_b(\D^\g)$.
	
	\begin{lemma}[Logified projections]\label{not:lemma: projection operators}
		Let's fix a function $f\in\A{phg}^{\vec{\mathcal{E}}}(\D^\g)$ with $(p_\bullet,k_\bullet)=\min(\mathcal{E}_\bullet)$ for  $\bullet\in\{\scri,+,1,..,n\}$. We can write 
		\begin{equation}
			\begin{gathered}
				f=\Big(1-\sum_a\bar{\chi}\big(\abs{\bar{y}_a}/R\big)\Big)t^{-p_+}\log^{k_+}(t) (P^+_{p_+,k_+} f)(\overline{x}/t)+f',\quad f'\in \A{phg}^{\vec{\mathcal{E}'}}(\D^\g),P^+_{p_+,k_+}f\in\A{phg}^{\vec{\mathcal{E}}^B}(\dot{B})\\
				f=\bar{\chi}\big(\abs{\bar{y}_a}/(\delta_1 t)\big)t_a^{-p_i}\log^{k_a}(t_a)(P^a_{p_a,k_a} f)(\tilde{y}_a)+f'_a,\quad f'_a\in \A{phg}^{\vec{\mathcal{E}}^{\prime;a}}(\D^\g),P^a_{p_a,k_a}f\in\A{phg}^{\mathcal{E}^\R}(\R^3),
			\end{gathered}
		\end{equation}
		where we have
		\begin{equation}
			\begin{gathered}
				\mathcal{E}'_+=\mathcal{E}_+\backslash\{(p_+,k_+)\}, \quad \mathcal{E}'_a=\mathcal{E}_a+ \overline{(0,k_+)}, \E'_\scri=\E_\scri+\overline{(0,k_+)}\\
				\mathcal{E}^{\prime;a}_+=\mathcal{E}_++\overline{(0,k_a)},\quad  \mathcal{E}^{\prime;a}_a=\mathcal{E}_a\backslash\{(p_a,k_a)\},\quad \mathcal{E}^{\prime;a}_b=\mathcal{E}_b \text{ for } b\neq a,\quad  \E{\prime,a}_\scri=\E_\scri \\
				\mathcal{E}^{\dot{B}}_{a}=\mathcal{E}_a-(p_+,0),\quad\mathcal{E}^{\dot{B}}_{\scri}=\mathcal{E}_\scri-(p_+,0),\quad
				\mathcal{E}^{\R}=\mathcal{E}_+-(p_a,0).
			\end{gathered}
		\end{equation}
	\end{lemma}
	\begin{proof}
		This is standard and follows similarly as Lemma 12 in \cite{kadar_scattering_2024}.
	\end{proof}
	
	\begin{remark}
		Even though in \cref{not:lemma: projection operators}, we only projected to the leading error term, e.g. to decay $(p_+,k_+)$, the projections are also well defined for $(p_+,k)$ with $k\leq k_+$.
	\end{remark}
	
	Finally, let us introduce the modulated solitons that are going to form the basis of our ansatz in \cref{sec:ansatz}.
	\begin{definition}[Modulated solitons]\label{not:def:modulated_solitons}
		Let $(A,A^l),\sigma_a,\lambda_a$ be soliton velocities, signs and scales with $\sigma_a\in\{\pm1\},\lambda\in\R^+$. Let us fix further parameters $z_a^{i,j}\in\R$ for $i,j\in\N$ which are nonzero for only finitely many $i,j$. We define the $a-$th soliton as
		\begin{equation}\label{not:eq:modulated_solitons}
			\begin{gathered}
				W_a=\sigma_aW^{\lambda_a}(y^{\mathrm{c}}_a),\quad y^{\mathrm{c}}_a=y_a-z_a^c(t_a),\quad W^{\lambda}=\sigma_a\lambda^{-1/2}W(x/ \lambda)\\
				z_a^c(y_a,t_a)=z_a^{0,1}\log t_a-z^{0,0}_a-\sum_{i\geq1,j\geq0}z_{a}^{i,j}\frac{\log^j t_a}{t_a^i}
			\end{gathered}
		\end{equation}
		where $y_a,t_a$ are implicitly defined in terms of $x,t$ as in \cref{not:eq:coordinates}. 
		We write $W_a^{\mathrm{un}}$ for the soliton with no time correction and $ V_a^{\mathrm{un}}=5(W_a^{\mathrm{un}})^4$ for the corresponding potentials.
	\end{definition}
	
	When studying rescaled solitons, it is important to keep track of the rescaling of all associated quantities.
	We note, that for $\lambda<1$, $W^\lambda$ is a more localised soliton with larger absolute value near the centre but faster fall off.
	The far field behaviour of a rescaled soliton is given by
	\begin{equation}
		\lim_{r\to\infty}r(rW^\lambda-\lambda^{1/2})=0.
	\end{equation}
	Similarly, the eigenvalue corresponding to unstable mode is rescaled as
	\begin{equation}
		\big(\Delta+5(W^\lambda)^4\big)Y(x\lambda)=\lambda^2\lamed^2 Y(x \lambda).
	\end{equation}

	\subsection{Error terms}\label{not:sec:errors}
	
	Extending the notation of $H_b,\A{phg}$ spaces, we further introduce spaces capturing error terms
	\begin{definition}[Polyhomogeneous error]
		For a function $f\in\A{phg}^{\vec{\mathcal{E}}}(\D^{\g})$ for some $\vec{\mathcal{E}}$ satisfying $\min(\mathcal{E}_\bullet)\geq(p_\bullet,k_\bullet)$, we write $f\in\mathcal{O}^{\overrightarrow{(p,k)}}$ as well as the following shorthands
		\begin{equation}
			\begin{gathered}
				\mathcal{O}_a^{(p_+,k_+),(p_a,k_a)}:=\mathcal{O}^{\overrightarrow{(q,j)}},\quad (q_\bullet,j_\bullet)=\begin{cases}
					(p_a,k_a),& \bullet=a\\
					(p_+,k_+),& \text{else}
				\end{cases}\\
				\mathcal{O}_{\scri}^{(p_\scri,k_\scri),(p_+,k_+),(p_{\loc},k_{\loc})}:=\O^{\overrightarrow{(q,j)}},\quad(q_\bullet,j_\bullet):=\begin{cases}
					(p_\scri,k_\scri),& \bullet=\scri\\
					(p_+,k_+),& \bullet=+\\
					(p_{\loc},k_{\loc}),& \text{else}
				\end{cases}\\
				\mathcal{O}_{\loc}^{(p_+,k_+),(p_{l},k_{l})}=\mathcal{O}_{\scri}^{(p_+,k_+),(p_+,k_+),(p_l,k_l)},\quad \mathcal{O}_{all}^{(p,k)}=\mathcal{O}_{\loc}^{(p,k),(p,k)}.
			\end{gathered}
		\end{equation}
		We write $p=(p,\infty)$, so that $f\in\mathcal{O}^{\vec{p}}$ implies $\min(\mathcal{E}_\bullet)\geq p_\bullet$ and set $(p,-1)=(p+1,\infty)$.
		We use similar notation on the ball $\dot{B}$.
		For $f\in\A{phg}^{\vec{\E}(\dot{B})}$ with $\min(\E_\bullet)\geq(p_\bullet,k_\bullet)$ we write $f\in\O^{\overrightarrow{p,k}}_{\dot{B}}$ and
		\begin{nalign}
			\O^{(p_\scri,k_\scri),(p_l,k_l)}_{\loc,\dot{B}}:=\O^{\overrightarrow{q,j}}_{\dot{B}},\quad (q_\bullet,j_\bullet)=\begin{cases}
				(p_a,k_a),& \bullet=a\\
				(p_\scri,k_\scri),& \text{else}
			\end{cases}
		\end{nalign}
		Finally, on $\R^3$, for $f\in\A{phg}^{\E}(\R^3)$, we write $f\in\O^{(p,k)}$ whenever $\E\geq(p,k)$.
	\end{definition}
	
	The product estimates from \cref{not:eq:product_phg} also extend to $\O$ notation.

	\subsection{Symbols}\label{not:sec:symbols}
	We collect the symbols used in the below
	\begin{itemize}\setlength\itemsep{-0.4em}
		\item We will use the notation $A\lesssim_{a,b,c...} B$ to say that there exists a constant $C>0$ depending on $a,b,c...$ such that $A\leq C B$.
		\item $\bar{\chi},h$ cutoff functions
		\item $u,v,t_\star^{\g},t_\star^{\e},y_a,\bar{y}_a,y_{a;\tau_2},\bar{t},\bar{x}$ are coordinates on the spacetime
		\item  $z_a,\lambda_a,\sigma_a$ denote the velocity, scale and sign of the solitons. Without loss of generality we take $z_1=0,\lambda_1=\sigma_1=1$.
		\item $W^\lambda$ is the soliton at a particular scale, $V^\lambda$ potential created by soliton. $W_a^{\mathrm{un}},V_a^{\mathrm{un}}$  are there scaled and boosted solitons with prescribed speed $z_a$ and scale $\lambda_a$, while $W_a$ is the path corrected.
		\item $\T^w$ is the energy momentum tensor with potential $w$, $\T=\T^0$. $\tT^w_a$ 
		is the twisted one and $\bar{\T}^w_a=\tT^w_a+\T^w$. 
		\item $T^a,X^a,\Omega^a,S^a,\Lambda$ vector fields;
		\item $\Sigma^a_{\tau;\tau_2},\mathcal{C}^{a,\delta}_{\tau_1,\tau_2},\scri_{\tau_1,\tau_2}$ hypersurfaces, $\dot{B}$ is the multi-punctured ball. $\D^\g$ is the compactification of Minkowski space.
		\item $\lamed$ eigenvalue of the linearised problem with $Y$ as eigenfunction at unit scale $\lambda=1$.
		\item $P^{S^2}_l$ is the projection operator acting on $L^2$ functions projecting onto the $l-$th spherical harmonic. $P^+_{p,k},P^a_{p,k}$ are projection operators giving the leading term of polyhomogeneous functions on $\D^\g$.
	\end{itemize}
		
	\newpage
	\section{Ansatz}\label{sec:ansatz}
	The main result of the current section is the following
	\begin{theorem}\label{an:thm:existence_of_ansatz}\ \newline
		
		\begin{enumerate}[label=An.\alph*),wide=0.5em, leftmargin =*, nosep, before = \leavevmode\vspace{-2\baselineskip}]
			\item\label{an:item:main1} 	There exists soliton velocities $A,A^l$ such that for every $N\geq1$ we can find modulated solitons $W_a$, as in \cref{not:def:modulated_solitons}, and an ansatz $\tilde{\phi}\in\O^{2,(1,1)}_{\loc}$ such that the following holds.
			For $\bar{\phi}=\sum_a W_a+\tilde\phi$ we have 
			\begin{equation}\label{an:eq:bar_phi}
				\begin{gathered}
					(\Box+\bar{\phi}^4)\bar{\phi}\in\mathcal{O}_{\scri}^{5,N,N},\qquad \mathfrak{E}^{lin,a}[\bar{\phi}]:=5\bar{\phi}^4-5W_a^4\in\O_{a}^{0,(1,1)}\O_{\loc}^{4,0}\\
					\partial_u(r\bar{\phi})|_{\scri}=0
				\end{gathered}
			\end{equation}
			\item \label{an:item:main2} For any soliton velocities $A$ and $N\geq1$ there exist an ansatz $\tilde\phi\in\O_{all}^{1,(1,0)}$ with $z^{i,j}_a=0$ such that $\bar{\phi}=\sum_a W_a+\tilde\phi$ satisfies $(\Box+\bar{\phi}^4)\bar{\phi}=\mathcal{O}_\scri^{(5,N,N)}$ and $ \mathfrak{E}^{lin,a}[\bar{\phi}]\in\O_{a}^{0,(2,0)}\O_{\loc}^{4,0}$.
			In general, $\tilde\phi$ has nonzero outgoing radiation, i.e. $\partial_u(r\bar{\phi})_{\scri}\neq0$.
		\end{enumerate}
	\end{theorem}
	
	\begin{remark}
	As already discussed in the introduction, in \cref{an:item:main1}, we cannot use an arbitrary collection of velocities ($A$), as in that case, the first correction to scaling would be unbounded.
	This we call \emph{admissibility}, see 
	However, there are extra constraints that we must impose on $A$.
	We call these \emph{strong admissibility}, see \cref{an:def:strong_admissible}.
	These are conditions on the non vanishing of determinants of a sequence of matrices ($M^{\Lambda,i}$) that are functions $A$. 
	Although we prove that in the limit $i\to\infty$ it holds that $\det M^{\Lambda,i}\neq0$, we must evaluate the first few cases numerically\footnote{This amounts to the computation of the determinant of $9$ four by four matrices. Whilst one could in principle estimate the determinant after some analytic procedure, we simply computed them with Mathematica \cite{Mathematica}, and observed that their values are \emph{far} more than machine precision away from 0.}.
	\end{remark}
	
	The more difficult part is \cref{an:item:main1}.
	This will take most of the work as done in \cref{an:sec:set_up,an:sec:starting_condition,an:sec:improvements}.
	
	The inclusion of outgoing radiation follows from solving a simple boundary value problem and is treated in \cref{an:sec:outgoing}.
	
	We note, that the proof also applies to admissible supercritical polynomials as well.
	\begin{theorem}\label{an:thm:supercritical}
		Let $f$ be as in \cref{i:thm:energy_supercritical}, and let $A$ be soliton velocities.
		Then, there exists an ansatz $\tilde{\phi}\in\O^{2,(1,0)}$ such that $\bar{\phi}=\sum_a W_a+\tilde\phi$  satisfies $\Box\bar{\phi}+f(\bar{\phi})=0$ and $\mathfrak{E}^{lin,a}[\bar{\phi}]\in\O_{a}^{0,(1,0)}\O_{\loc}^{4,0}$, with leading term given in \cref{an:eq:leading_supercritical}.
	\end{theorem}
	The \hyperlink{an:proof:supercritical}{proof} is at the end of \cref{an:sec:improvements}.
	
	\begin{remark}[Improved supercritical solution]\label{an:rem:supercritical_improved}
		We mention, that provided that the soliton velocities, $A$, satisfy the \emph{admissibility}  constraint of \cref{an:thm:existence_of_ansatz}, than the solution constructed in \cref{an:thm:supercritical} has linear error $\mathfrak{E}^{\lin,a}\in\O_{a}^{0,2}\O_{\loc}^{4,0}$.
	\end{remark}

	We note, that provided that we can find an ansatz as in \cref{an:thm:existence_of_ansatz} of the form \cref{an:eq:ansatz}, we can translate the $\Lambda W$ part into modulation for the scaling of the soliton.
	\begin{theorem}\label{an:thm:modulated_scaling}
		Let $\bar{\phi}$ be an ansatz of the form \cref{an:eq:ansatz} satisfying the conditions in $\cref{an:thm:existence_of_ansatz}$.
		Then, we can write an ansatz of the form
		\begin{nalign}\label{an:eq:alternative_ansatz}
			\bar{\phi}'=\tilde{\phi}'+\sum_a W^{\lambda_a(t_a)}(y^{a,c}),\qquad
			\lambda_a(t_a)=\lambda_a+\sum_{i,j} c^{i,j}_{\Lambda,a}\frac{\log^jt_a}{t_a^{i-1}},
		\end{nalign}
		where $\tilde{\phi}'$ is of the form \cref{an:eq:ansatz} without $\Lambda W$ terms and $y^{a,c}$ is from \cref{not:eq:modulated_solitons}.
		Moreover $\bar{\phi}'$ satisfies
		\begin{equation}
			(\Box+(\bar{\phi}')^4)\bar{\phi}'\in\mathcal{O}_{\scri}^{5,N,N},\qquad \mathfrak{E}^{lin,a}[\bar{\phi}]\in\O_{a}^{0,(2,2)}\O_{\loc}^{4,0}.
		\end{equation}
	\end{theorem}
	\begin{proof}
		This follows from expanding $\lambda_a$ from \cref{an:eq:alternative_ansatz} in $t_a^{-1}$ to recover $\tilde{\phi}$ in \cref{an:eq:ansatz}
		The improvement for the error term follows from the lack of $\Lambda W$ terms in $\tilde{\phi}'$.
	\end{proof}

	\subsection{Set up}\label{an:sec:set_up}
	
	We build $\bar{\phi}$ by induction.
	To do so, we solve model problems in an iterative manner. 
	On the soliton faces ($F_a$), we use $N^{F_a}=\Delta_{\R^3}+V_a^{\mathrm{un}}$ with $V^\mathrm{un}_a=5(W_a^\mathrm{un})^4$ acting on $\mathcal{C}^\infty(\R^3)$.
	On timelike infinity ($I^+$) we will use that $\Box$ acts homogeneously on functions of the form $\abs{t}^{-\sigma}f(y)$ for $y=\frac{x}{t}$
	\begin{equation}\label{an:eq:Nsigma_def}
		\begin{gathered}
			L=\abs{t}^3\Box \abs{t}^{-1}\\
			N_\sigma f:=\abs{t}^\sigma L \abs{t}^{-\sigma} f(y)=-\Big(\sigma^2+3\sigma+2+2\big((\sigma+2)\rho-\frac{1}{\rho}\big)\partial_\rho+(\rho^2-1)\partial_\rho^2-\frac{\slashed{\Delta}}{\rho^2}\Big)f.
		\end{gathered}
	\end{equation}
	\begin{remark}
		The subscript $\sigma$ for $N_\sigma$ is connected to the decay rate of the radiation field.
		This introduces some necessary shifts when writing expansions for $\phi$ instead the radiation field $r\phi$.
	\end{remark}
	We already note that $N_\sigma:\Hbloc^{a_\scri,a_1;k}(\dot{B})\to\Hbloc^{a_\scri-1,a_1-2;k-2}(\dot{B})$,
	and, we recall the appropriate inverses of this statement in more detail in \cref{an:lemma:model_operator_on_i_+,an:lemma:model_operator_on_i_+_punctured,an:lemma:model operator on i_- multi punctured}.
	Furthermore, we have that $N_\sigma$ on fixed angular modes has a regular singular point at $\abs{x/t}=1$ with indicial roots $-\sigma,0$ corresponding to solution \emph{with} and \emph{without} outgoing radiation
	\begin{lemma}[Radiative solutions, Lemma 7.14 in \cite{kadar_scattering_2024}]
		Let's fix a functions on the unit ball
		\begin{equation}
			\begin{gathered}
				f_{rad}(z)=f_r\big(\frac{z}{\abs{z}}\big)(1-\abs{z})	^{-\sigma}+H_b^{-\sigma+}(B_1),\quad f_r\in\mathcal{C}^{\infty}(S^2).
			\end{gathered}
		\end{equation}
		Then for $g=t^{-1-\sigma}f_{rad}(x/t)$
		\begin{equation}
			\begin{gathered}
				u\partial_u(rg)_{\scri}=-\sigma 2^{-\sigma}f_r(\omega) u^{-\sigma}.
			\end{gathered}
		\end{equation}
	\end{lemma}

	We write part of our ansatz $
	\tilde\phi$ from \cref{an:thm:existence_of_ansatz} for $c_{\nabla,a}^{i,j}\in\R^3,c_{\Lambda,a}^{i,j}\in\R$ as
	\begin{equation}\label{an:eq:ansatz}
		\begin{gathered}
			\tilde{\phi}=\sum_{a,i\geq2,j\geq0}\frac{\log^j t_a}{t_a^i} (c_{\nabla,a}^{i,j}\cdot g_{\nabla,a}^{i-1}+c_{\Lambda,a}^{i,j}(\underbrace{i(i+1)\lambda_a^{1/2}\sigma_ag_{\Lambda,a}^{i-1}}_{\bar{g}_{\Lambda,a}^{i-1}}+t_a\Lambda W_a)+g^{i,j}_a(y_a))+\sum_{i\geq2,j\geq0}\frac{\log^jt}{t^i}g^{i,j}_+\big(\frac{\overline{x}}{t}\big)\\
			c^{2,j}_{\nabla,a}=0 \, \forall j>0,\quad c^{2,j}_{\Lambda,c}=0\, \forall j>1,\quad g_a^{2,j}=0 \,\forall j>2,\quad \sup_{\abs{z}\to 1}(1-\abs{z})^{-i+1}g^{i,j}_+(z)<\infty,
		\quad g^{2,j}_+=0 \forall j\geq1,
		\end{gathered}
	\end{equation}
	where $g_{\nabla,a}^{i}\in\mathcal{O}_a^{4,0},g^i_{\Lambda,a}\in\mathcal{O}^{2,0}_a$ are defined in \cref{def:an:g_corrections} and $g^{i,j}_{a}\in\O^{1}_{\R^3},g^{i,j}_{+}\in\O^{\min(0,6-i),0}_{\dot{B}}$ are determined by solving elliptic problems coming from $N^F$ and $N_\sigma$ respectively.
	
	\begin{remark}
		Let us here give a motivation for each term in the ansatz
		\begin{itemize}
			\item $t^{-i}t\Lambda W$ introduces extra term at decay rate $t^{-i-1}\Lambda W$, however these cannot be used to cancel kernel elements of error terms at $F$ due to weak decay at $I^+$.
			\item $t^{-i}g_\Lambda^i$ corrects the error created by $t^{-i+1}\Lambda W$ on $I^+$ and leaves behind a term that allows for the cancellation of kernel elements of the error at $F$.
			\item $t^{-i}g_{\nabla}$ correct the error created by $z^{i,j}$ corrections of the soliton path at $I^+$ and leaves behind a term on $F$ to cancel kernel elements corresponding to $\partial_i W$.
			\item $g^{i,j}_a$ solve error terms on the faces $F_a$ that have no kernel elements
			\item $g^{i,j}_+$ solve error terms on the face $I^+$.
		\end{itemize}
	\end{remark}

	Since we are to build an ansatz that is close to $\sum_a W_a$, we introduce notation to describe this proximity
	\begin{definition}(Nonliear terms)\label{an:def:nonlinear_operators}
		Let's write a solution ($\hat{\phi}$) to \cref{i:eq:main} with soliton speed $z_a$, scales $\lambda_a$, and signs $\sigma_a$ as 
		\begin{equation}
			\begin{gathered}
				\hat{\phi}=\bar{\phi}+\phi,\quad \bar{\phi}=\sum_b W_b+\tilde{\phi}
			\end{gathered}
		\end{equation}
		with $\bar{\phi}$ representing an ansatz and $\phi$ the nonlinear correction. We introduce the following notations
		\begin{equation}
			\begin{gathered}
				\frac{1}{5}\mathfrak{Err}^{\lin}[\bar{\phi}]=(\bar{\phi})^4-\Vbold,\quad \Vbold=\sum_a V_a,\quad V_a=5W_a^4,\\
				\mathfrak{Err}^{\lin}[\bar{\phi},a]:=5\bar{\phi}^4-5W^4_a,\qquad\mathfrak{Err}[\bar{\phi}]=(\Box+\bar{\phi}^4)\bar{\phi},\\
				\mathcal{N}[\phi;\bar{\phi}]:=(\bar{\phi}+\phi)^5-\bar{\phi}^5-5\bar{\phi}^4\phi.
			\end{gathered}
		\end{equation}
		We will use the shorthand $\mathfrak{E}^l=\mathfrak{E}^{\lin}$. Using this notation, we can write
		\begin{equation}
				0=\mathfrak{Err}[\hat{\phi}]=\mathfrak{Err}[\bar{\phi}]+(\Box+\Vbold+\mathfrak{Err}^{\lin}[\bar\phi])\phi+\mathcal{N}[\phi;\bar\phi].
		\end{equation}
		Whenever we refer to the supercritical problem, we adorn the above notation with $\mathrm{s}$ superscript, e.g. $\Vbold^{\mathrm{s}}:=\sum_a f'(W^{\mathrm{sup}})$.
	\end{definition}

	We recall the  following lemmas about the model operators from \cite{kadar_scattering_2024}: 
	
	\begin{lemma}[Lemma 7.17 of \cite{kadar_scattering_2024}]\label{an:lemma:model_operator_on_i_+}
		Let $\sigma>0$, $F\in\mathcal{C}^\infty(S^2)$, and $f\in \A{phg}^{\mathcal{E}^f}(B)+H^{a_\scri;k}_b(B_1)$ with $a_\scri>-\sigma-1$ and $\min(\E^f)>-\sigma-1$.
		Then, there exists $u\in \A{phg}^{\mathcal{E}}(B)+\Hb^{a_\scri+1;k+2}(B)$ with $\mathcal{E}=\overline{(0,0)}\overline{\cup}(\mathcal{E}_f+1)$ to
		\begin{equation}\label{an:eq:Iplus_normal_operator}
			\begin{gathered}
				N_\sigma u=f\\
				((1-\rho)^\sigma u)|_{\partial B_1}=F.
			\end{gathered}
		\end{equation}
	\end{lemma}
	
	We write $N_\sigma^{-1}(f)$ for the solution $u$ for $F=0$.
	
	\begin{lemma}[Lemma 7.18 of \cite{kadar_scattering_2024}]\label{an:lemma:model_operator_on_i_+_punctured}
		a) Let $\sigma>0$ $F\in\mathcal{C}^\infty(S^2)$, and $f\in\A{phg}^{\mathcal{E}^f_{\scri},\mathcal{E}^f_{1}}(\dot{B}_1)+ H^{a_\scri,a_1;k}_b(\dot{B}_1)$ with $a_\scri,\min(\E^f_\scri)>-\sigma-1$ and $a_1>-2$ as well as $\min(\E^f_1)>-2$. Then, there exists $u\in \A{phg}^{\mathcal{E}_\scri,\mathcal{E}_1}(\dot{B}_1)+\Hb^{a_\scri+1,a_1+2;k+2}(\dot{B}_1)$ with $\mathcal{E}_\scri=\overline{(0,0)}\overline{\cup}(\mathcal{E}^f_{\scri}+1)$ and $\mathcal{E}_1=\overline{(0,0)}\overline{\cup}(\mathcal{E}^f_{1}+2)$ to \cref{an:eq:Iplus_normal_operator}.
		
		b) For $F=0,f=(t\Lambda W)_{I^+}=-\frac{1}{2\rho}$ we have 
		\begin{equation}
			\tilde{g}^\sigma_{\Lambda}:=(N_\sigma^{-1}f)=\frac{1}{2\sigma(1+\sigma)\rho}(1-\frac{1}{(1+\rho)^{\sigma}})\in\O_{\dot{B}_1}^{0,0}
		\end{equation}
		 with $\tilde{g}^\sigma_{\Lambda}|_{\rho=0}=\frac{1}{2(1+\sigma)}\neq0$ for $\sigma>0$.
		
		c) For $F=0,f=(t^2\partial_i W)_{I^+}=-\frac{e_i}{\rho^2}$ we have $\tilde{g}^\sigma_{\nabla}:=(N_\sigma^{-1}f)$  with $\tilde{g}^\sigma_{\Lambda}|_{\rho=0}\neq0$ for $\sigma>0$.
	\end{lemma}
	
	The previous theorem trivially extends to a multi punctured ball
	
	\begin{lemma}\label{an:lemma:model operator on i_- multi punctured}
		Let $F\in\mathcal{C}^\infty(S^2)$, and $f\in\A{phg}^{\vec{\mathcal{E}}^{f}}(\dot{B})+ \Hb^{\vec{a}^f;k}(\dot{B})$ with $a_\scri,\min(\E^f_\scri)>-\sigma-1$ and $a^f_b>-2$ as well as $\min(\E^f_b)>-3$ for all $b$.
		Then, there exists $u\in \A{phg}^{\vec{\mathcal{E}}}(\dot{B})+\Hb^{\vec{a};k+2}(\dot{B})$ to \cref{an:eq:Iplus_normal_operator} with the index sets given by $\mathcal{E}_\scri=\overline{(0,0)}\overline{\cup}(\mathcal{E}^f_{\scri}+1)$ and $\mathcal{E}_b=\overline{(0,0)}\overline{\cup}(\mathcal{E}^f_{b}+2)$, while the errors are $a_\scri=a^f_\scri+1$, $a_b=a^f_b+1$.
	\end{lemma}

	Finally, let us also recall the following result concerning the normal operator on $F$.
	\begin{lemma}[Lemma 7.19 of \cite{kadar_scattering_2024}]\label{an:lemma:model operator on F}
		Let $\lambda>0$, $f\in \A{phg}^{\mathcal{E}^f}(\R^3)+\Hb^{a;k}(\R^3)$ with $a,\min(\E_f)>2$, and $\jpns{f,\Lambda W^\lambda}=\jpns{f,\partial_i W^\lambda}=0$. Then, there exists solution $u\in\A{phg}^{\mathcal{E}}(\R^3)+\Hb^{a-2;k+2}(\R^3)$ with $\mathcal{E}=\overline{(1,0)}\overline{\cup}(\mathcal{E}_f-2)\backslash(2,0)$ to 
		\begin{equation}
			\begin{gathered}
				(\Delta+V^\lambda)u=f.
			\end{gathered}
		\end{equation}
	\end{lemma}

	The corrections we will perform at the faces $F_a$ corresponding to the kernel elements will introduce error terms on $I^+$ that needs to be treated imminently. These will be dealt with the elements from the following definition
	\begin{definition}\label{def:an:g_corrections}
		We define the radiation field corrections
		\begin{equation}
			\begin{gathered}
				g^{\sigma,R}_{\nabla,a}:=\bar{\chi}(\tilde{y}_a)\tilde{g}^{\sigma}_{\nabla}\big(\frac{\tilde{y}_a}{t_a}\big)\\
				g^{\sigma,R}_{\Lambda,a}:=\bar{\chi}(\tilde{y}_a)\tilde{g}^{\sigma}_{\Lambda}\big(\frac{\tilde{y}_a}{t_a}\big)+\big(1-\bar{\chi}(\tilde{y}_a)\big)\tilde{g}^{\sigma}_{\Lambda}(0).
			\end{gathered}
		\end{equation}
	\end{definition}

	These quantities are useful because we have the following explicit computation
	\begin{lemma}\label{an:lemma:Box_on_ball_functions}
		Let $h=\bar{h}\big(\frac{\tilde{y}_a}{t_a}\big)$ for $\bar{h}\in\O_{\dot{B}_1}^{p_\scri,(0,0)}$ and $\chi=\bar{\chi}(\tilde{y}_a)$. Then
		\begin{equation}
			\Box \Big(t^{-\sigma-1}_a\chi h\Big)=t^{-\sigma-2}_a\chi(N_\sigma \bar{h})\big(\frac{\tilde{y}_a}{t_a}\big)+h(0)t_a^{-\sigma-2}\Delta\chi+\O^{p_\scri+\sigma+2,\sigma+3,\sigma+3}_{\scri}
		\end{equation}
	\end{lemma}
	\begin{proof}
		This follows from the expression for $\Box$ in $t_a,\tilde{y}_a$ coordinates and the definition of $N_\sigma$ given in \cref{not:eq:box_in_tilde}, \cref{an:eq:Nsigma_def} respectively.
	\end{proof}

	We also need some non-orthogonality properties, as used in \cref{an:eq:Newtonian_path_correction,an:eq:first_scaling_log}
	\begin{lemma}\label{lemma:kernel correctability}
		For $\chi=\bar{\chi}_R(x)$, we have 
		\begin{subequations}
			\begin{gather}
				\delta_{ij}C_\nabla=\int_{\R^3}\partial_iW\partial_jW,\quad C_\Lambda=\int_{\R^3}V \Lambda W, \qquad C_\nabla,C_\Lambda\neq 0\label{an:eq:nablaW_non-orthogonality}\\
				\int_{\R^3}\partial_j W (\Delta+V)\hat{x}_i\chi=0.\label{an:eq:nablaW_orthogonality}
			\end{gather}
		\end{subequations}
	\end{lemma}
	\begin{proof}
		This follows from direct computation and the fact that $\partial_jW\in\O^{2}_{\R^3}$ and $\partial_j W\in\ker (\Delta+V)$.
	\end{proof}
	
	\subsection{Starting conditions}\label{an:sec:starting_condition}
	
	The proof of \cref{an:thm:existence_of_ansatz} proceeds by induction, as already explained in \cref{i:sec:ansatz}.
	At later stages of the iteration, the nonlinear effects are becoming weaker, and so it is considerably easier to prove the induction step, therefore, we begin by doing the first step separately.
	The main result of the current section is 
	\begin{prop}\label{an:prop:starting}
		There exist a choice for the free parameters in $\bar{\phi}$ such that 
		\begin{equation}
			\begin{gathered}
				\mathfrak{Err}[\bar{\phi}]\in\mathcal{O}_{\loc}^{5,3}.
			\end{gathered}
		\end{equation}
	\end{prop}
	
	As discussed in the introduction, we are not able to  treat the case when the first order interaction of the trivial ansatz does not cancel.
	We introduce this condition precisely below.
	
	\begin{definition}[Admissible point]\label{an:def:admissible}
		a) We call $\{W_a\}$ with $z_a^{i,j}=0$ an \textit{admissible} configuration if 
		\begin{equation}\label{an:eq:f_start}
			\begin{gathered}
				f_{\mathrm{start}}:=\Big(\sum_{a} W_a^\un\Big)^5-\sum_{a}(W_a^\un)^5\in\mathcal{O}_{\loc}^{5,2}
			\end{gathered}
		\end{equation}
		
		b) 	We call a $\{W_a\}$ with $z_a^{i,j}=0$ \textit{balanced admissible} configuration if it is admissible and
		\begin{equation}
			\begin{gathered}
				P^{S^2}_{1} P^a_{2,k}f_{start}=0 \quad \forall a,k.
			\end{gathered}
		\end{equation}
		
	\end{definition}
	
	\begin{remark}
		Balanced admissible configurations make the construction of the ansatz significantly easier, as we need not include $\frac{\log t}{t}\Lambda W$ corrections.
		This configuration also rules out the possibility of logarithmic correction for the path of the solitons. Unfortunately, we do not know if such a configuration exists.
	\end{remark}
	
	\begin{lemma}\label{an:lemma:existence of admissible points}
		a) A set of points $z_a\in B$ has corresponding signs and scales such that $\{W_a\}$ is admissible iff
		\begin{equation}
			\begin{gathered}
				\mathbb{A}_{ab}=\begin{cases}
					\frac{\sqrt{1-\abs{z_{ab}}^2}}{\abs{z_{ab}}}& \text{if }a\neq b\\
					0 & \text{else}
				\end{cases}
			\end{gathered}
		\end{equation}
		satisfies $\det(\mathbb{A})=0$ with all entries in the eigenvector corresponding to 0 eigenvalue nonzero. 
		For an eigenvector $\lambda^{1/2}_a\sigma_a$ corresponding to 0 eigenvalue of $\mathbb{A}$, we call $\lambda,\sigma$ admissible scales and signs.
		
		b) There exist admissible configurations.
		
		c) 	 A set of points $z_a\in B_1$ is balanced admissible if for a 0 eigenvalue $\mu$ of $\mathbb{A}$ it holds that
		\begin{equation}\label{eq:balanced admissible}
			\begin{gathered}
				\sum_{a\neq b}\mu_a\frac{z_{ab}\sqrt{1-\abs{z_{ab}}^2}}{\abs{z_{ab}}^3}=0\quad \forall b.
			\end{gathered}
		\end{equation}
		
	\end{lemma}

	\begin{remark}
		Provided that $\{z_a\}\in\B$ satisfy $\det(\mathbb{A})=0$, we have that for some subset of $\{z_a\}$, the eigenvector with zero eigenvalues has only non-zero entries.
	\end{remark}
	
	\begin{remark}
		We note, that for \cref{an:lemma:existence of admissible points}, we use that $z^{i,j}_a=0$.
		In particular, we assumed that there are no logarithmic corrections to the compactification coming from $z^{0,1}_a$.
		Introducing these corrections for admissible configuration only changes the error term from $\O^{5,2}_{\loc}$ to $\O^{4,2}_{\loc}$ due to the extra terms appearing in \cref{not:eq:box_in_tilde}.
		These extra error terms will be included in the proof of \cref{an:prop:starting}, and form an important part of the modulation argument,
		see \cref{eq:an:z_a_contibution,an:eq:com_correction1}.
	\end{remark}
	
	\begin{proof}
		\textit{a):} Using Lorentz invariance and the transformations from \cref{eq:not:relative_speeds} it suffices to compute $f_{\mathrm{start}}$ at $F_1$.
		We calculate using \cref{i:eq:soliton,not:eq:coordinates,not:eq:modulated_solitons} for $a\neq1$
		\begin{equation}
				W^\un_a=\sigma_aW^{\lambda_a}(y^{a})=\sigma_a\lambda_a^{-1/2}W\big(\lambda_a^{-1}\gamma_a(x-z_at)\big)=\frac{\sigma_a(\lambda_a)^{1/2}\sqrt{1-\abs{z_a}^2}}{\abs{z_a}}t^{-1}\qquad \mod \mathcal{O}_1^{1,2}
		\end{equation}
		This implies 
		\begin{equation}
			\begin{gathered}
				f_{\mathrm{start}} = V_1^\un\sum_{a\neq 1} W^\un_a
				=V_1^\un\sum_{a\neq 1}\sigma_a(\lambda_a)^{1/2}\frac{\sqrt{1-\abs{z_a}^2}}{\abs{z_a}}\qquad\mod\mathcal{O}_1^{0,1}\mathcal{O}_{\loc}^{5,1}
			\end{gathered}
		\end{equation}
		We calculate the leading order effect on the other faces. 
		To do so, we note that using the relativistic velocity transformation we get that in the inertial frame moving with velocity $z_a$ the $b$-th soliton has speed $z_{ba}$ given by \cref{eq:not:relative_speeds}.
		Therefore
		\begin{equation}
			\begin{gathered}
				f_{\mathrm{start}}=V_a^\un\sum_{b\neq a} W^\un_b+\mathcal{O}_a^{0,1}\mathcal{O}_{\loc}^{5,1}
				=V_a^\un\sum_{b\neq a}\sigma_b(\lambda_b)^{1/2}\frac{\sqrt{1-\abs{z_{ab}}^2}}{\abs{z_{ab}}}\qquad\mod\mathcal{O}_a^{0,1}\mathcal{O}_{\loc}^{5,1}
			\end{gathered}
		\end{equation}

		We see, that choosing $\sigma_a \lambda_a^{1/2}$ as 0 eigenvalue of $\mathbb{A}_{ab}$ we get a vanishing leading order term at each face.
		
		\textit{b): } Let's consider collinear initial velocities located at $e_1(-x_1,-x_2,x_2,x_1)$ where $e_1$ it the unit vector in the $x$ direction and $0<x_1<x_2<1$.
		We use \cref{eq:not:colinear_relative_speeds} to compute $\abs{z_{ab}}$, which after a lengthy (\cite{Mathematica}) algebra gives
		\begin{equation}
			\det(\mathbb{A})=\frac{(1-x_1)^2(1-x_2)^2(x_1^4-16x_1^3x_2-2x_1^2x_2^2+x_2^4)(x_1^4-2x_1^2x_2^2-16x_1x_2^3+x_2^4)}{16x_1^2x_2^2(x_1^2-x_2^2)^4}
		\end{equation}
		The equation $(1-16y-2y^2+y^4)=0$ has two real solutions $y_{1}\sim 0.062,y_2\sim2.77$. Therefore $\det(\mathbb{A})=0\iff \{x_2\in(0,1):x_1/x_2\in\{y_1,y_2^{-1}\}\}$.
				
		\textit{c):} We use admissibility of $f_{\mathrm{start}}$ to compute
		\begin{equation}
			\begin{multlined}
				P^{S^2}_1P^a_{2,k}f_{\mathrm{start}}=P^{S^2}_1P^a_{2,k}\Big(V_1^\un{\sum_{a\neq1}W_a^\un}+20(W_1^\un)^3\cancel{\sum_{\substack{a\neq b\\a,b\neq 1}}W_aW_b}+10(W_1^\un)^3\cancel{\sum_{a\neq1}W_a^2}\Big)\\
				=V_1^\un P^{S^2}_1P^a_{2,k}\sum_{a\neq1}W_a
			\end{multlined}
		\end{equation}
		where the cancellation happens due to the projections.
		Therefore, we are led to expend $W_a$ to next order
			\begin{equation}
			W_a=\sigma_a\lambda_a^{-1/2}W\big(\lambda_a^{-1}\gamma_a(x-z_at)\big)=\frac{\sigma_a(\lambda_a)^{1/2}}{\gamma_a\abs{(x-z_at)}}=P^1_{1,0}W_a+t^{-2}\sigma_a(\lambda_a)^{1/2}\frac{z_a\cdot x}{\gamma_a\abs{z_a}^3}\qquad \mod \mathcal{O}_1^{1,3}.
		\end{equation}
		The result follows from using the same expansion at the other faces.
	\end{proof}

	For the rest of this section, we assume that $\{W_a\}$ is an admissible configuration.
	Before proving \cref{an:prop:starting}, let us record the form of $\tilde{\phi}$ that we use for this first iteration, for the sake of clarity
	
	\begin{multline}\label{an:eq:ansatz_upto2}
	\tilde{\phi}=\sum_at_a^{-2}\Big(\log^2 t_a g^{2,2}_a+\log^1 t_a g^{2,1}_a+ g^{2,0}_a\Big)+\sum_a t_a^{-2}c^{2,0}g^1_{\nabla,a}\\
	+\sum_a t_a^{-2}\big(c^{1,1}_{\Lambda,a}\log t_a+c^{1,0}_{\Lambda,a}\big)\big(2\lambda_a^{1/2}\sigma_a g^{1}_{\Lambda,a}+t_a\Lambda W_a\big)+{\sum_a t_a^{-2}g^{2,0}_{+,a}}
	\end{multline}

	\begin{proof}[Proof of \cref{an:prop:starting}]
		\textit{Step 1a) $z_a^{0,1}$ corrections:}
		We compute the error terms arising from $W_a$ not being an exact soliton. 
		\begin{equation}\label{eq:an:z_a_contibution}
			\begin{aligned}[c]
				(\Box+W_a^4)W_a&=-(\partial_{t_a}|_{y_a})^2W_a=-\ddot{z}_a^c\cdot \nabla W_a+(\dot{z}_a^c,\dot{z}_a^c)\cdot\nabla^2W_a\\
				&=-\frac{1}{t_a^2}z_a^{0,1}\nabla W_a-\frac{1}{t_a^2}(z_a^{0,1},z_a^{0,1})\cdot\nabla^2W_a
			\end{aligned}
			\quad
			\begin{aligned}[c]
				\mod \mathcal{O}_{\loc}^{5,3},
			\end{aligned}
		\end{equation}
		where we introduced the notation $(z,z)\nabla^2 W=z_iz_j\partial_i\partial_j W$.
		We split the second term in the last equation to $l=0$ and $l=2$ modes. For arbitrary spherically symmetric function $f$ we have
		\begin{equation}\label{an:eq:proof2}
			\begin{gathered}
				\partial_i\partial_jf=\partial_i (\hat{x}_j f')=\hat{x}_i\hat{x}_j f''+\frac{\delta_{ij}-\hat{x}_i\hat{x}_j}{r}f'=\frac{\delta_{ij}}{3}\Delta f+(\hat{x}_i\hat{x}_j-\frac{\delta_{ij}}{3})(f''-\frac{1}{r}f')
			\end{gathered}
		\end{equation}
		where the last term is supported purely on $l=2$ angular modes.
		
		Note, that to compute a similar correction near the face $F_a$ coming from the terms $z^{i,j}_a$ from $i\geq1$, we have that the first term in right hand side of \cref{eq:an:z_a_contibution} will dominate the second, see \cref{an:eq:zij_dependence}.
	
		\textit{Step 1b) scaling:}
		Let us note, that $\tilde{g}^{1}_{\Lambda}-\tilde{g}^{1}_{\Lambda}(0)\in\O_{\dot{B}_1}^{0,1}$, in particular $\tilde{g}^{1}_{\Lambda}(\rho)-\tilde{g}^{1}_{\Lambda}(0)\sim \rho$ at $\rho=0$.
		Using this, we compute when $\Box+V_a$ act on the scaling part of the ansatz
		\begin{subequations}\label{eq:an:Lambda-g-upto2}
				\begin{gather}
					\begin{multlined}
						(\Box+V_a) \frac{1}{t_a^2}g^{1}_{\Lambda,a}=\Box t_a^{-2}\Big(\bar{\chi}(\tilde{y}_a)\tilde{g}^{1}_{\Lambda}(\frac{\tilde{y}_a}{t_a})+(1-\bar{\chi}(\tilde{y}_a))\tilde{g}^{1}_{\Lambda}(0)\Big)
						+V_a \frac{t_a}{t_a^2}g^{1}_{\Lambda,a}\\
						=\frac{1}{t_a^{4}}\bar{\chi}(\tilde{y}_a)\underbrace{(N_2\tilde{g}_\Lambda^1)}_{=P^+_0(t\Lambda W )}\Big(\frac{\tilde{y}_a}{t_a}\Big)+V_a \frac{1}{t_a^2}\tilde{g}^{1}_{\Lambda}(0) \qquad \mod \mathcal{O}^{5,(3,0)}_a
					\end{multlined}\\
					\begin{multlined}\label{eq:an:Lambda-g-upto2-log}
						(\Box+V_a) \frac{\log t_a}{t_a^2}g^{1}_{\Lambda,a}=\log t_a\Box t_a^{-2}\Big(\bar{\chi}(\tilde{y}_a)\tilde{g}^{1}_{\Lambda}(\frac{\tilde{y}_a}{t_a})+(1-\bar{\chi}^c(\frac{\tilde{y}_a}{R}))\tilde{g}^{1}_{\Lambda}(0)\Big)
						-\frac{2}{t_a}T_a \frac{g^1_{\Lambda,a}}{t_a^2}+V_a \frac{\log t_a}{t_a^2}g^{1}_{\Lambda,a}\\
						=\frac{\log t_a}{t_a^{4}}\bar{\chi}(\tilde{y}_a)\underbrace{(N_2\tilde{g}_\Lambda^1)}_{=P^+_0(t\Lambda W )}\Big(\frac{\tilde{y}_a}{t_a}\Big)+V_a \frac{\log t_a}{t_a^2}\tilde{g}^{1}_{\Lambda}(0)-\frac{2}{t_a}T_a \frac{g^1_{\Lambda,a}}{t_a^2} \qquad \mod \mathcal{O}^{5,(3,1)}_a.
					\end{multlined}
				\end{gather}
			\end{subequations}
		We also compute
		\begin{subequations}\label{eq:an:Lambda-W-upto2}
			\begin{gather}
				\begin{multlined}
					(\Box+V_a) \frac{1}{t_a^{1}}\Lambda W_a=\frac{-2}{t_a^{3}}\Lambda W_a\qquad \mod \mathcal{O}_a^{5,(3,0)}
				\end{multlined}\\
				\begin{multlined}\label{eq:an:Lambda-W-upto2-log}
					(\Box+V_a) \frac{\log t_a}{t_a^{1}}\Lambda W_a=\frac{-1}{t_a^{3}}\Big(2\log t_a+2\Big)\Lambda W_a\qquad \mod \mathcal{O}_a^{5,(3,1)}.
				\end{multlined}
			\end{gather}
		\end{subequations}
		Summing the above we get
		\begin{subequations}\label{an:eq:proof_start1}
			\begin{gather}
				(\Box+V_a)\Big(\frac{1}{t_a^{2}}(t_a\Lambda W_a +2\sigma_a\lambda_a^{1/2}2 g^2_{\Lambda,a})\Big)=\sigma_a\lambda_a^{1/2}2V_a\frac{1}{t_a^2}\tilde{g}^1_\Lambda(0) \qquad \mod \mathcal{O}_a^{5,(3,0)}\\
				\begin{multlined}
					(\Box+V_a)\Big(\frac{\log t_a}{t_a^{2}}(t_a\Lambda W_a +2\sigma_a\lambda_a^{1/2}2 g^2_{\Lambda,a})\Big)=\sigma_a\lambda_a^{1/2}2V_a\frac{\log t_a}{t_a^2}\tilde{g}^1_\Lambda(0)\\
					-\frac{4\sigma_a\lambda_a^{1/2}}{t_a}T_a\frac{g_{\Lambda,a}^1}{t_a^2}-\frac{2}{t^3_a}\Lambda W_a \qquad \mod \mathcal{O}_a^{5,(3,1)}
				\end{multlined}\label{an:eq:proof_log_terms}
			\end{gather}
		\end{subequations}

		The extra $\lambda_a^{1/2}\sigma_a$ in \cref{an:eq:proof_start1} comes from the rescaling of the soliton.
		Note that for \cref{an:eq:proof_log_terms} the error term on the $I^+$ face is still weakly decaying.
		We can solve these by introducing $g^{2,0}_{+,a}\in\O^{0,(0,0)}_{\dot{B}_a}$ similar to $g^2_{\Lambda,a}$, but now we do not keep track of the leading order error term generated on $F_a$ to get
		\begin{equation}
			\begin{gathered}
				(\Box+V_a)\Big(\frac{\log t_a}{t_a^{2}}(t\Lambda W_a +2 \sigma_2\lambda_a^{1/2}g^1_{\Lambda,a})+t_a^{-2}g^{2,0}_{+,a}\Big)=2V_a\sigma_a\lambda_a^{1/2}\frac{\log t_a}{t_a^2}\tilde{g}^1_\Lambda(0)+t^{-2}_aV_ag^{2,0}_{+,a}(0)+\O_a^{5,(3,0)}
			\end{gathered}
		\end{equation}
		
		\textit{Step 1c) center of mass:} For the correction to the centre of mass we proceed similarly
		\begin{subequations}\label{an:eq:com_corr_upto2}
			\begin{gather}
				\begin{multlined}
					(\Box+V_a) \frac{1}{t_a^2}g^{1}_{\nabla,a}=\Box t_a^{-2}\Big(\bar{\chi}(\tilde{y}_a)\tilde{g}^{1}_{\nabla}(\frac{\tilde{y}_a}{t_a})\Big)+V_a \frac{1}{t_a^2}\tilde{g}^{1}_{\nabla}(0)\bar{\chi}(\tilde{y}_a)+\mathcal{O}_a^{(5,0),(3,0)}\\
					=\frac{1}{t_a^{4}}\bar{\chi}(\tilde{y}_a)P_0^+(t^2\nabla W)\Big(\frac{\tilde{y}_a}{t_a}\Big)+ \frac{1}{t_a^2}\tilde{g}^{1}_{\nabla}(0)\Big((V_a+\Delta)\bar{\chi}(\tilde{y}_a)\Big)+\mathcal{O}_a^{(5,0),(3,0)}
				\end{multlined}\\
				\begin{multlined}\label{an:eq:com_correction1}
					\implies \frac{1}{t^{2}}\nabla W_a-(\Box+V_a) \frac{\lambda_a^{1/2}\sigma_a}{t_a^2}g^{1}_{\nabla,a}=\frac{C^2_{\nabla,a}}{t_a^2}+\mathcal{O}_a^{(5,0),(3,0)}\\
					=:\frac{1}{t_a^{2}}\Big(\nabla W_a-\lambda_a^{1/2}\sigma_a\big(\tilde{g}^1_\nabla(0)(V_a+\Delta_{\tilde{y}_a})+\frac{\hat{x}}{r^2}\big)\bar{\chi}^c(\tilde{y}_a/R)\Big)
					+\mathcal{O}_a^{(5,0),(3,0)}.
				\end{multlined}
			\end{gather}
		\end{subequations}
		where we introduced $C^2_{\nabla,a}$ for a shorthand.
		Let us already note that $\Delta \tilde{g}^1_{\nabla}(0)=\frac{\hat{x}}{r^2}$, and so, via \cref{an:eq:nablaW_orthogonality}
		\begin{equation}
			\int_{\R^3}\nabla W^\un_a\big(\tilde{g}^1_\nabla(0)(V^\un_a+\Delta_{\tilde{y}_a})+\frac{\hat{x}}{r^2}\big)\bar{\chi}(\tilde{y}_a/R)
			=\int_{\R^3}\nabla W\big(V_a^\un+\Delta_{\tilde{y}_a}\big)\tilde{g}^1_\nabla(0)\bar{\chi}(\tilde{y}_a/R)
			=0.
		\end{equation}
		Based on the above computation, we set  $c^{2,0}_{\nabla,a}=z^{0,1}_a$. 
		We also note that using \cref{lemma:kernel correctability}, we know that the leading term, $C_{\nabla,a}^2$, on the right hand side of  \cref{an:eq:com_correction1} is not orthogonal to $\nabla W_a$.

		\textit{Step 2) nonlinear errors:}
		Let's start by computing what $W_a$ for $a\neq1$ looks like around the the first soliton, at $F_1$
		
		\begin{equation}
			\begin{aligned}
				W_a&=\sigma_aW^{\lambda_a}\big(\gamma_a(x-z_at)-z^c_a\big)=\sigma_aW^{\lambda_a}\big(\gamma_a(\tilde{x}+z_1^{0,1}\log t+z_1^{0,0}-z_at)-z^c_a\big)\\
				&=(\lambda_a)^{1/2}\sigma_a\jpns{(\gamma_a(\tilde{x}+z_1^{0,1}\log t+z_1^{0,0}-z_at)-z_a^c)\lambda_a}^{-1/2}\\
				&=\frac{(\lambda_a)^{1/2}\sigma_a}{\gamma_a\abs{z_a}t}\Big(1+\frac{\log t}{t}z_a\cdot( z_1^{0,1}-\frac{z_a^{0,1}}{\gamma_a})+\frac{\tilde{x}\cdot z_a+z_a\cdot (z^{0,1}_a\frac{\log\gamma_a}{\gamma_a}-z_a^{0,0}\gamma_a^{-1}+z_1^{0,0})}{t}\Big).
			\end{aligned}
			\mod \mathcal{O}_1^{0,3}\O^{1,0}_{\mathrm{\loc}}
		\end{equation}
		Next, let us compute the error terms arising from the nonlinearity $(\bar{\phi})^5$ at the location of the first soliton
		\begin{equation}
			\begin{gathered}
				(\tilde{\phi})^5=W_1^5+5W_1^4\Bigg(\frac{C_{1,1}\log t+C_{1,2}\cdot \tilde{x}+C_{1,3}}{t^2}+\frac{\log^2 t}{t^2}g^{2,2}_1 +\frac{\log t}{t^2} \Big(c_{\Lambda,1}^{2,1}(\bar{g}_{\Lambda,1}^{1}+t\Lambda W_1)+g^{2,1}_1(x)\Big)\\
				+\frac{1}{t^2}\Big({c}_{\nabla,1}^{2,1}\cdot g_{\nabla,0}^1+c_{\Lambda,1}^{2,0}(g_{\Lambda,1}^{1}+t\Lambda W_1)+g^{2,0}_1(x)+g^{2,0}_+\Big)+\sum_j \frac{c_{\Lambda,1}^{3,j}\log^j t}{t^2}\Lambda W_1\Bigg)\\
				+10W_1^3(\Lambda W_1)^2\frac{1}{t^2}\Big((c_{\Lambda,1}^{2,1})^2\log^2t+2c_{\Lambda,1}^{2,1}c_{\Lambda,1}^{2,0}\log t+(c_{\Lambda,1}^{2,0})^2\Big)+\mathcal{O}_{\loc}^{5,0}\mathcal{O}_{1}^{0,3},
			\end{gathered}
		\end{equation}
		where we get the constants $C_\bullet$ from the contribution of the other solitons
		\begin{equation}
			\begin{gathered}
				\sum_{a\neq1} W_a+\frac{1}{t_a^2}\Big({c}^{2,0}_{\nabla,a}\cdot g_{\nabla,a}^1+c^{2,0}_{\Lambda,a} (g_{\Lambda,a}^1+t_a\Lambda W_a)+\log t_ac^{2,1}_{\Lambda,a} (g_{\Lambda,a}^1+t_a\Lambda W_a)\Big)\\
				\begin{multlined}
					=\sum_{a\neq 1}\frac{(\lambda_a)^{1/2}\sigma_a}{\gamma_a\abs{z_a}t}\Bigg(1+z_a\cdot \frac{\log t( z_1^{0,1}-z_a^{0,1})+\tilde{x}+ (z^{0,1}_a\frac{\log\gamma_a}{\gamma_a}-z_a^{0,0}\gamma_a^{-1}+z_1^{0,0})}{t} 
					\\
					+\frac{1}{\gamma_a t}\bigg(\big(\log tc^{2,1}_{\Lambda,a}+c^{2,0}_{\Lambda,a}-c^{2,1}_{\Lambda,a}\log\gamma_a\big)\Big(\tilde{g}_{\Lambda}^1(-z_a)-\frac{1}{2\abs{z_a}}\Big)\lambda_a^{1/2}\sigma_a+{c}^{2,0}_{\nabla,a}\cdot \tilde{g}_{\nabla}^1(-z_a)\bigg)\Bigg)+\mathcal{O}_{\loc}^{5,2}\mathcal{O}_{1}^{0,1}
				\end{multlined}\\
				=\sum_{a\neq1} \frac{C_{a,1}\log t+C_{a,2}\cdot \tilde{x}+C_{a,3}}{t^2}+\mathcal{O}_{\loc}^{5,2}\mathcal{O}_{1}^{0,1}
			\end{gathered}
		\end{equation}
		We emphasis that $C_{a,2}$ only depends on the soliton positions $\{z_a\}$ and not any other corrections.
		In turn, this will imply that $z_a^{0,1}$ are determined by these, see \cref{an:eq:Newtonian_path_correction}.
		Concluding the computations up to this point, we get
		\begin{equation}\label{an:eq:error_upto2}
			\begin{multlined}
				\mathfrak{Err}[\bar{\phi}]=\sum_a t_a^{-2}(\Delta_a+V_a) \Big(\log^2 t_a g^{2,2}_a+\log^1 t_a g^{2,1}_a+ g^{2,0}_a\Big)\\
				+\sum_a 10W_a^3(\Lambda W_a)^2t_a^{-2}\Big((c_{\Lambda,a}^{2,1})^2\log^2t_a+2c_{\Lambda,a}^{2,1}c_{\Lambda,a}^{2,0}\log t_a+(c_{\Lambda,a}^{2,0})^2\Big)\\
				+\sum_a t_a^{-2} \Big(V_a\tilde{y}_a\cdot C_{a,2}-z^{0,1}_a\cdot C^2_{\nabla,a}\Big)\\
				+\sum_a {V_a t_a^{-2}\log t_a} \Big(C_{a,1}+2c^{2,1}_{\Lambda,1}\sigma_a\lambda_a^{1/2}\tilde{g}^1_\Lambda(0)\Big)\\
				+\sum_a {t_a^{-2}V_a} \Big(C_{a,3}+2c^{2,0}_{\Lambda,1}\sigma_a\lambda_a^{1/2}\tilde{g}^1_\Lambda(0)+c^{2,1}_{\Lambda,a}g^{2,0}_{+,a}(0)\Big)+\sum_a \frac{1}{t_a^2}(z_a^{0,1},z_a^{0,1})\cdot\nabla^2W_a +\O^{5,3}_{\loc}
			\end{multlined}
		\end{equation}

		\textit{Step 3) fixing parameters, $t^{-2}P^{S^2}_1$:}
		We first choose $z_a^{0,1}={c}^{2,0}_{\nabla,a}$. 
		Using \cref{lemma:kernel correctability} \cref{an:eq:error_upto2} we pick $z^{0,1}_a$ such that  
		\begin{equation}\label{an:eq:Newtonian_path_correction}
			\begin{gathered}
				(V_a\tilde{y}_a\cdot C_{a,2}-z^{0,1}_a\cdot C^2_{\nabla,a},\nabla W_a)_{L^2_{y^{a,\mathrm{c}}}}=0.
			\end{gathered}
		\end{equation}
	
		\textit{Step 3b) $t^{-2}\log^2 tP^{S^2}_0$:}
		Next, we choose $g^{2,2}_a$.
		This will depend on ${c}^{2,1}_{\Lambda,a}$, but the argument will not be circular, as $g^{2,2}_a$ does not influence our choice for $c^{2,1}_{\Lambda,a}$. 
		We define it first, as it is the least decaying spherically symmetric contribution. 
		We notice that
		\begin{equation}\label{M1}
			\begin{gathered}
				(W^3(\Lambda W)^2,\Lambda W)=0.
			\end{gathered}
		\end{equation}
		This is an explicit computation, or alternatively, one can differentiate in $\lambda$ the equation $\norm{\Lambda W}_{\dot{H}^1}=(\Lambda W,V \Lambda W)_{L^2}$ to obtain the orthogonality, see \cref{lin:eq:error_Y_equation}.
		In conclusion, we can find a solution
		\begin{equation}\label{an:eq:starting_G}
			\begin{gathered}
				(\Delta+V)G=W^3(\Lambda W)^2,\quad G\in\A{phg}^{3}(\R^3).
			\end{gathered}
		\end{equation}
		We set $g^{2,2}_a=-10(c_{\Lambda,a}^{2,1})^2\lambda_a^{1/2}G(\tilde{y}_a)$.
		
		\textit{Step 3c) $t^{-2}\log tP^{S^2}_0$:}
		Next, we choose ${c}^{2,1}_{\Lambda,a}$. 
		For this we compute the $P^a_{2,1}$ projection of the error
		\begin{equation}\label{an:eq:first_scaling_log}
			\begin{gathered}
				P^1_{2,1}\mathfrak{Err}[\bar{\phi}]=V^{\un}_1			\sum_{a\neq 1}\frac{(\lambda_a)^{1/2}\sigma_a}{\gamma_a}\Bigg(\frac{z_a}{\abs{z_a}}\cdot( z_1^{0,1}-z_a^{0,1})
				+\frac{1}{\gamma_a}{c}^{2,1}_{\Lambda,a} \Big(\tilde{g}_{\Lambda}^1(-z_a)-\frac{\gamma_{ab}}{2\abs{z_a}}\Big)\Bigg)\\
				+20(W^{\un}_1)^3(\Lambda W^{\un}_1)^2c_{\Lambda,1}^{2,1}c_{\Lambda,1}^{2,0}+2V^{\un}_1c_{\Lambda,1}^{2,1}\tilde{g}^1_{\Lambda}(0)+(\Delta+V_1^{\un})g_1^{2,1}
			\end{gathered}
		\end{equation}
		In order for this term to vanish, we need to ensure that the sum of the first three terms is in $\ker(\Delta+V)$. Indeed, we already saw from \cref{M1} that $W^3(\Lambda W)^2\perp \Lambda W$, so the first and the third terms must cancel exactly and we need
		\begin{equation}
			\begin{gathered}
				(\Delta+V^{\un}_1)g_1^{2,1}=-20(W^{\un}_1)^3(\Lambda W_1^{\un})^2 c^{2,1}_{\Lambda,1}c^{2,0}_{\Lambda,1}.
			\end{gathered}
		\end{equation}
		We can find $c^{2,1}_{\Lambda,a}$ such that this cancellation happens if $\det(M^{\Lambda,1})\neq 0$ for
		\begin{equation}\label{eq:an:scale_matrix_2}
				M^{\Lambda,1}_{ab}=\begin{cases}
					\frac{(\lambda_b)^{1/2}\sigma_b}{\gamma_{ab}^2\abs{z_{ab}}}\big(\tilde{g}^1_\Lambda(-z_{ab})-\frac{1}{2\abs{z_{ab}}}\big)& a\neq b\\
					2\lambda_a^{1/2}\sigma_a\tilde{g}^{1}_\Lambda(0) & a=b
				\end{cases}
		\end{equation}
		We prove in \cref{lemma:an:existence_of_strong_1_admissible} below, that admissible configuration with $\det(M^{\Lambda,2})\neq 0$ exist.
		We choose $c^{2,1}_{\Lambda,a}$ accordingly.		
		
		\textit{Step 3d) $t^{-2}P^{S^2}_0$:}
		Next, we compute $c_{\Lambda,a}^{2,0}$.
		\begin{equation}
			\begin{multlined}
				P^{S^2}_0P^1_{2,0}\mathfrak{Err}[\bar{\phi}]=V^{\un}_1c^{2,1}_{\Lambda,1}g^{2,0}_{+,1}(0)+V^{\un}_1\sum_{a\neq 1}\frac{(\lambda_a)^{1/2}\sigma_a}{\gamma_a}\Bigg(\frac{z_a}{\abs{z_a}}\cdot (z^{0,1}_a\frac{\log\gamma_a}{\gamma_a}-z_a^{0,0}\gamma_a^{-1}) \\
				+\frac{1}{\gamma_a }\big(c^{2,0}_{\nabla,a}\cdot \tilde{g}_{\nabla}^2(-z_a)+c^{2,0}_{\Lambda,a} \Big(\tilde{g}_{\Lambda}^2(-z_a)-\frac{\gamma_{ab}}{2\abs{z_a}}\Big)-\log \gamma_a c^{2,1}_{\Lambda,a} \tilde{g}_{\Lambda,a}^2(-z_a)\big)\Bigg)\\
				+(\Delta+V_1^{\un})g^{2,0}_1+10(W_1^{\un})^3(\Lambda W^{\un}_1)^2(c_{\Lambda,1}^{2,0})^2-2\lambda_1^{1/2}\sigma_1V^{\un}_1c^{2,0}_{\Lambda,1}\tilde{g}^{2}_{\Lambda}(0)+(z_1,z_1)\cdot\nabla^2 W^{\un}_a
			\end{multlined}
		\end{equation}
		We notice that using \cref{an:eq:proof2}, we know that the last has exactly $\abs{z_1}^2 V_a$ contribution to the kernel.
		As before, we need to exactly cancel the coefficient of $V_1$. This is possible if $M^{\Lambda,1}_{ab}$ is invertible.
	\end{proof}

	\begin{lemma}\label{lemma:an:existence_of_strong_1_admissible}
		There exists admissible configuration such that $\det(M^{\Lambda,1})\neq0$, where $M^{\Lambda,1}$ is defined in \cref{eq:an:scale_matrix_2}.
	\end{lemma}
	\begin{proof}
		This is an explicit computation given the form of admissible velocities from \cref{an:lemma:existence of admissible points}.
		See \cref{an:lemma:existence_strong_admissible} for a stronger statement.
	\end{proof}
	
	\begin{remark}
			Note, that when $z_a$ are all along the same line, then the contribution of the $z^{0,0}_a$ term to $c^{2,0}_{\Lambda,a}$ is $M^{com}_{ab}z_b^{0,0}$ where
		\begin{equation}
			\begin{gathered}
				M^{com}_{ab}=\begin{cases}
					-\frac{\lambda_b^{-1/2}\sigma_bz_{ab}}{\gamma_{ab}^2\abs{z_{ab}}} & a\neq b\\
					0 & a=b
				\end{cases}.
			\end{gathered}
		\end{equation}
		Therefore, whenever $M^{com}_{ab}$ is invertible, the value of $c^{2,0}_{\Lambda,a}$ is freely prescribed by a choice of $z_b^{0,0}$.
		We emphasis this, to show how intertwined the modulations are and the fact that these are in principle explicitly computable by our algorithm. 
	\end{remark}
	
	We continue the creation of the ansatz inductively.

	\subsection{Improvements}\label{an:sec:improvements}
	In this section, we improve the ansatz ($\bar\phi$) to arbitrary fast decay.
	For this, we need not take such a fastedious care of all the error terms generated, as the nonlinearities will decay much faster.

	\begin{lemma}\label{an:lemma:improving_+}
		Assume that we can choose the free parameters in $\bar{\phi}$ such that $\mathfrak{E}[\bar{\phi}]\in\mathcal{O}_{\scri}^{5,(N+2,j),N}$ for $N\geq2$. Then, we can choose $g^{N,j}_+$ such that
		\begin{equation}
			\begin{gathered}
				\mathfrak{E}[\bar{\phi}]\in\mathcal{O}_{\scri}^{5,(N+2,j-1),N}.
			\end{gathered}
		\end{equation}
	\end{lemma}
	\begin{proof}
		We simply set  
		\begin{equation}
			\begin{gathered}
				g_+^{N,j}=\log(t)^jt^{-N}N_N^{-1}(P_{N+2,j}^+\mathfrak{E}[\bar{\phi}])\big(\frac{\bar{x}}{t}\big)\prod_a\bar{\chi}(\abs{y_a})\\
				g_+^{N,j}\in\mathcal{O}_{\scri}^{\min(N-1,5),(N,j),N},\qquad P_{N+2,j}^+(\mathfrak{E}[\bar\phi])\underbrace{\in}_{\cref{not:lemma: projection operators}}\mathcal{O}^{5-(N+2),-2}_{\loc,\dot{B}}.
			\end{gathered}
		\end{equation}
		Then, we compute using \cref{an:lemma:Box_on_ball_functions} that
		\begin{equation}
			\begin{gathered}
				(\Box+5\bar{\phi}^4)g_+^{N,j}=t^{-N-2}\log^j(t_\star)P^+_{N+2,j}\mathfrak{E}[\bar{\phi}^a]+\mathcal{O}_\scri^{5,(N+2,j-1),N}.
			\end{gathered}
		\end{equation}
		We use \cref{not:eq:product_phg} to get
		\begin{equation}
			\begin{gathered}
				\mathcal{N}[g^{N,j}_+;\bar{\phi}^a]=\mathcal{O}_\scri^{6,2N+3,2N}
			\end{gathered}
		\end{equation}
		yielding the result.
	\end{proof}
	
	Before continuing, we will introduce a concept that is necessary for our proof to work.
	As already discussed in the introduction, unlike \emph{admissibility}, we do not take a stand on whether this extra condition is necessary or not for the existence of multisoliton solutions.
	\begin{definition}[Strong admissible point]\label{an:def:strong_admissible}
		We call a finite set of points $z_i\in B_1$ a \textit{strong admissible} configuration if $\det(M^{\Lambda,i})\neq 0$ for all $i\in\Z_{\geq 2}$ where
		\begin{equation}
			\begin{gathered}
				M^{\Lambda,i}_{ab}=\begin{cases}
					i(i+1)\lambda_a^{1/2}\sigma_a\tilde{g}^{i}_{\Lambda}(0),\quad a=b\\
					\lambda_b^{1/2}\sigma_b\gamma^{-i}_{ab}\big(\frac{-1}{2\abs{z_{ab}}}+i(i+1)\tilde{g}^i_{\Lambda}(z_{ba})\big),\quad a\neq b ,
				\end{cases}
			\end{gathered}
		\end{equation}
	\end{definition}
	
	\begin{lemma}\label{an:lemma:existence_strong_admissible}
		There exists strongly admissible configuration
	\end{lemma}
	\begin{proof}
		Using \cref{an:lemma:model_operator_on_i_+_punctured}, we get that
		\begin{equation}
			M^{\Lambda,i}_{ab}=\begin{cases}
				\frac{(\lambda_b)^{1/2}\sigma_b}{\abs{z_{ab}}\gamma_{ab}^i}\big(\frac{1-(1+\abs{z_{ab}})^{-i}}{2}-\frac{1}{2}\big)& a\neq b\\
				\frac{-\lambda_a^{1/2}\sigma_ai}{2} & a=b
			\end{cases}
		\end{equation}
		
		Let's pick admissible points as in the proof of \cref{an:lemma:existence of admissible points}, i.e. along the $e_1$ direction located at $(-x,-y_2x,y_2x,x)$ with $y_2\sim0.36$ defined in \cref{an:lemma:existence of admissible points}.
		
		We normalise $M^{\Lambda,i}=\frac{-i}{2}\bar{M}^{\Lambda,i}\diag(\lambda_\bullet^{1/2}\sigma_\bullet)$ so that $\prod_{a=b}\bar{M}_{ab}^{\Lambda,i}=1$. Then, we compute
		\begin{nalign}\label{an:eq:sum_off_diag_M}
			\sum_{a\neq b}\abs{\bar{M}^{\Lambda,i}_{ab}}=\sum_{a\neq b}\abs{\frac{1}{i\abs{z_{ab}}\gamma_{ab}^i}\big(1+\abs{z_{ab}}\big)^{-i}}\\
			\leq \frac{24}{i}\max(\gamma_{ab}^{-1})^i\max(\abs{z_{ab}}^{-1})
		\end{nalign}
		Setting $x=0.9$, we may compute that $\max(\gamma_{ab}^{-1})<0.81$ and $\max(\abs{z_{ab}}^{-1})<2$, therefore the right hand side is bounded by 
		\begin{nalign}
			RHS\cref{an:eq:sum_off_diag_M}\leq\frac{42}{i}0.81^i.
		\end{nalign}
		Therefore, for $i\geq9$, we get that $\bar{M}$ is a strictly diagonally dominated matrix and thus invertible.
		The rest of the values may be computed explicitly to find
		\begin{equation}
			\abs{\det (\bar{M}^{\Lambda,a})}>0.9.
		\end{equation}
	\end{proof}

	For the rest of the section, unless otherwise stated, we always assume that we are working with strong admissible configuration.
	
	\begin{lemma}\label{an:lemma:improving_F}
		Assume that we can choose the free parameters in $\bar\phi$ such that $\mathfrak{E}[\bar{\phi}]\in\mathcal{O}_{\scri}^{5,i+3,(i,j)}$ for $i\geq3$.
		Then, we can choose $g^{i,j}_{a},c^{i,j}_{\Lambda,a},c^{i,j}_{\nabla,a},z_a^{i-2,j}$ such that 
		\begin{equation}\label{eq:local improvement}
			\begin{gathered}
				\mathfrak{E}[\bar{\phi}]=\mathcal{O}_{\scri}^{5,i+3,(i,j-1)}.
			\end{gathered}
		\end{equation} 
	\end{lemma}
	\begin{proof}
		Let's write
		\begin{equation}
			\begin{gathered}
				f_a=P^a_{i,j}\mathfrak{E}[\bar{\phi}]\in\mathcal{O}^3_{\R^3}.
			\end{gathered}
		\end{equation}
		We want to invert $\Delta+V^{\un}_a$ on each of these error terms, but as before, this may not be possible for $l=0,1$ spherical harmonics. 
		
		\textit{Step 1:} 
		First, we tune the $l=1$ modes.
		We don't write the $\mathfrak{E}^{\lin}[\bar{\phi}]$ corrections, as they will always be part of the error terms.
		We note, that the $z_a^{i,j}$ term, for $i\geq1$, from the ansatz $\bar{\phi}$ contributes via  $-\partial_t^2 W_a$ a factor
		\begin{equation}\label{an:eq:zij_dependence}
			\begin{gathered}
				z^{i,j}_a\frac{\partial}{\partial z^{i,j}_a}(\Box+W^4_a)W_a=-i(i+1)t_a^{-i-2}\log^j t_az_a^{i,j}\cdot\nabla W_a\quad \mod \mathcal{O}_a^{4+i,(i+2,j-1)}
			\end{gathered}
		\end{equation} 
		similarly as in \cref{eq:an:z_a_contibution}.
		We compute the contribution of $g^{i}_{\nabla,a}$ as in \cref{an:eq:com_corr_upto2}
		\begin{subequations}
				\begin{gather}
					\begin{multlined}
						(\Box+V_a) \frac{\log^jt_a}{t_a^i}g^{i-1}_{\nabla,a}=\log^jt_a\Box t_a^{-i}\Big(\bar{\chi}^c(\frac{\tilde{y}_a}{R})\tilde{g}^{\sigma}_{\nabla}(\frac{\tilde{y}_a}{t_a})\Big)+V_a \frac{\log^jt_a}{t_a^i}\tilde{g}^{i-1}_{\nabla}(0)\bar{\chi}^c\Big(\frac{\tilde{y}_a}{R}\Big)+\mathcal{O}_a^{(i+2,j-1),(i+1,j)}\\
						=\frac{\log^jt_a}{t_a^{i+2}}\bar{\chi}^c\Big(\frac{\tilde{y}_a}{R}\Big)P_0^+(t^2\nabla W)\Big(\frac{\tilde{y}_a}{t_a}\Big)+ \frac{\log^jt_a}{t_a^i}\tilde{g}^{i-1}_{\nabla}(0)\Big((V_a+\Delta)\bar{\chi}^c\big(\frac{\tilde{y}_a}{R}\big)\Big)+\mathcal{O}_a^{(i+2,j-1),(i+1,j)}
					\end{multlined}\\
					\begin{multlined}\label{an:eq:com_correction-improved}
						\implies -\frac{\log^jt_a}{t^{i}}\nabla W_a+(\Box+V_a) \frac{\lambda_a^{1/2}\sigma_a\log^jt_a}{t_a^i}g^{i-}_{\nabla,a}=\frac{\log^jt_a C^{i}_{\nabla,a}}{t_a^{i}}+\mathcal{O}_a^{(i+2,j-1),(i+1,j)}\\
						:=\frac{\log^jt_a}{t_a^{i}}\Big(-\nabla W_a+\lambda_a^{-1/2}\sigma_a\big(\tilde{g}^{i-1}_\nabla(0)(V_a+\Delta_{\tilde{y}_a})+\frac{\hat{x}}{r}\big)\bar{\chi}^c(\tilde{y}_a/R)\Big)
						+\mathcal{O}_a^{(i+2,j-1),(i+1,j)}.
					\end{multlined}
				\end{gather}
		\end{subequations}
		We set $c^{i+2,j}_{\nabla,a}=i(i+1)z^{i,j}_a$ for $i>0$.
		As for the logarithmic terms in \cref{eq:an:Lambda-W-upto2-log,eq:an:Lambda-g-upto2-log}, we need to introduce $g^{i+2,j-1,k}_{+,\nabla,a}\in\O^{(i,j-1),(i,j-1)}_{\dot{B}_a}$ to cancel the leading terms on $I^+$ of \cref{an:eq:com_correction-improved}.
		We simply set
		\begin{equation}
			g^{i+2,j-1,j-1}_{+,\nabla,a}=-\frac{log^{j-1} t}{t^i} N_{i+2}^{-1}\Big(P^+_{i+2,j-1}\big(LHS \cref{an:eq:com_correction-improved}\big)\Big)\prod_a\bar{\chi}^c(\abs{y_a})
		\end{equation}
		to get 
		\begin{equation}
			-\frac{\log^jt_a}{t^{i}}\nabla W_a+(\Box+V_a) \Big(\frac{\lambda_a^{1/2}\sigma_a\log^jt_a}{t_a^i}g^{i-1}_{\nabla,a}+g^{i+2,j-1,j-1}_{+,\nabla,a}\Big)=\frac{\log^jt_a C^{i}_{\nabla,a}}{t_a^{i}}+\mathcal{O}_a^{(i+2,j-2),(i,j-1)}
		\end{equation}
		We similarly introduce $g^{i+2,j-1,k}_{+,\nabla,a}$ for $k\leq j-2$ and call denote their sum by $g^{i+2,j-1}_{+,\nabla,a}$.
		In conclusion, we have 
		\begin{equation}
			-\frac{\log^jt_a}{t^{i}}\nabla W_a+(\Box+V_a) \Big(\frac{\lambda_a^{1/2}\sigma_a\log^jt_a}{t_a^i}g^{i-}_{\nabla,a}+g^{i+2,j-1}_{+,\nabla,a}\Big)=\frac{\log^jt_a C^{i}_{\nabla,a}}{t_a^{i}}+\mathcal{O}_a^{(i+3),(i,j-1)}.
		\end{equation}
		We  use \cref{lemma:kernel correctability} to find $c^{i,j}_{\nabla,a}$ such that
		\begin{equation}
			\begin{gathered}
				\underbrace{c^{i,j}_{\nabla,a}\cdot C^{i}_{\nabla,a} -f_a}_{\tilde{f}_a}\perp\nabla W_a.
			\end{gathered}
		\end{equation}
		
		\textit{Step 2: $\Box+\boldsymbol{V}$.} For this part, we assume that $\mathfrak{Err}^{\lin}[\bar\phi]=0$ and the terms involving $\mathfrak{Err}^{\lin}[\bar\phi]$ will be treated in the next step.
		We can project out the kernel in the $l=0$ part of the solution as in \cref{eq:an:Lambda-g-upto2,eq:an:Lambda-W-upto2}.
		\begin{subequations}\label{an:eq:Lamdba-W}
				\begin{gather}
					\begin{multlined}
						(\Box+V_a) \frac{\log^jt_a}{t_a^i}g^{i-}_{\Lambda,a}=\log^jt_a\Box t_a^{-i}\Big(\bar{\chi}^c(\frac{\tilde{y}_a}{R})\tilde{g}^{i-1}_{\Lambda}(\frac{\tilde{y}_a}{t_a})+(1-\bar{\chi}^c(\frac{\tilde{y}_a}{R}))\tilde{g}^{i-1}_{\Lambda}(0)\Big)\\
						-j\frac{\log^{j-1}t_a}{t_a}\partial_t \frac{g^{i-1}_{\Lambda,a}}{t_a^i}-j(j-1)\frac{\log^{j-2}t_a g_{\Lambda,a}}{t_a^{i+2}}+V_a \frac{\log^jt_a}{t_a^i}g^{i-1}_{\Lambda,a}\\
						=\frac{\log^jt_a}{t_a^{i+2}}\bar{\chi}^c(\frac{\tilde{y}_a}{R})\underbrace{(N_i\tilde{g}_\Lambda^i)}_{=P^+_0(t\Lambda W )}\Big(\frac{\tilde{y}_a}{t_a}\Big)+V_a \frac{\log^jt_a}{t_a^i}\tilde{g}^{i-1}_{\Lambda}(0)\qquad \mod \mathcal{O}^{(i+2,j-1),(i+1,j)}_a
					\end{multlined}\\
					\begin{multlined}
						(\Box+V_a) \frac{\log^jt_a}{t_a^{i-1}}\Lambda W_a=\frac{-1}{t_a^{i+1}}\Big(i(i-1)\log^j t_a+\log^{j-1} t_a2j(i-1)+j(j-1)\log^{j-2} t_a\Big)\Lambda W_a\\
						=\frac{-i(i-1)\log^j t_a}{t^{i+1}_a}\Lambda W_a\qquad \mod \mathcal{O}_a^{(i+2,j-1),(i+1,j)}\\
					\end{multlined}\\
					\implies (\Box+V_a)\Big(\frac{\log^j t_a}{t_a^{i}}(t\Lambda W_a +\sigma_a\lambda_a^{1/2}i(i-1) g^{i-1}_{\Lambda,a})\Big)=\sigma_a\lambda_a^{1/2}i(i-1)V_a\frac{\log^jt_a}{t_a^i}\tilde{g}^{i-1}_\Lambda(0)\qquad \mod \mathcal{O}_a^{(i+2,j-1),(i+1,j)}
				\end{gather}
		\end{subequations}
		However, we again find terms that do not decay sufficiently fast on $I^+$, therefore we add corrections without keeping track of their precise behaviour on $F_a$
		\begin{equation}
			\begin{gathered}
				(\Box+V_a)\Big(\underbrace{\frac{\log^j t_a}{t_a^{i}}(t\Lambda W_a -i(i-1) g^{i-1}_{\Lambda,a})-g^{i,j}_{+,\Lambda,a}}_{I_a}\Big)=i(i-1)\lambda_a^{1/2}\sigma_aV_a\frac{\log^jt_a}{t_a^i}\tilde{g}^{i-1}_\Lambda(0)+\mathcal{O}_a^{(i+3),(i,j-1)}\\
				g^{i,j}_{+,\Lambda,a}=\sum_{k<j}\frac{\log^k t}{t^i}g^{i,j,k}_{+,\Lambda,a}
			\end{gathered}
		\end{equation}
		Therefore, the extra error term from $c^{i,j}_{\Lambda,a}$ at $F_a$ is
		\begin{nalign}
			(\Box+\Vbold)c^{i,j}_{\Lambda,a}I_a= i(i-1)c_{\Lambda,a}^{i,j}{\lambda_a^{1/2}\sigma_a}V_a\frac{\log ^j t_a}{t_a^i}\tilde{g}^{i-1}_\Lambda(0) \qquad \mod \O^{i+3,0}_{\loc}\O_a^{0,(i,j-1)},
		\end{nalign}
		while the error term from $c^{i,j}_{\Lambda,b}$ at $F_a$ for $b\neq a$ is
		\begin{nalign}
			&(\Box+\Vbold) c^{i,j}_{\Lambda,b}I_b=V_a\frac{\log^j t_b}{t_b^i}\Big(t_b\Lambda W_b+i(i-1)\lambda_b^{1/2}\sigma_bg^{i,j}_{\Lambda,b}\Big)\\
			&=c_{\Lambda,b}^{i,j}V_a{\lambda_b^{1/2}\sigma_b}\frac{\log ^jt_a}{\gamma_{ba}^it_a^i}\Big(\frac{-1}{2\abs{z_{ba}}}+i(i-1)\tilde{g}^{i-1}_\Lambda(z_{ba})\Big)\qquad\mod\O^{i+3,0}_{\loc}\O_a^{0,(i,j-1)}.
		\end{nalign}
		Summing the above two, we get
		\begin{equation}
		\sum_{a,b}\frac{\log^j t_a}{t_a^i}V_aM^{\Lambda,i}_{ab}c_{\Lambda,b}^{i,j}\qquad\mod\O^{i+3,(i,j-1)}_{\loc}.
		\end{equation}
		By strong admissibility, we know that $M^{\Lambda,i}_{ab}$ is an invertible matrix in $ab$, thus, we can solve away for the kernel of the error term. 
		In particular, we conclude, that there exists $c_{\Lambda,b}^{i,j}$ such that for each $a$
		\begin{equation}
			\underbrace{\tilde{f}_a-V_a\sum_{b}M^{\Lambda,i}_{ab}c_b^{i,j}}_{\bar{f}_a}\perp \Lambda W_a
		\end{equation}
	
		We also note, that by adding these extra term to the ansatz, we keep orthogonality to $\nabla W_a$. Hence, we can conclude that there exists
		\begin{equation}
			\begin{gathered}
				{g}_a^{i,j}=(\Delta+V_a)^{-1}\bar{f}_a\in\mathcal{O}^{3}_{\R^3}.
			\end{gathered}
		\end{equation}
		In conclusion, we proved that there exists $g^{i,j}_a,c^{i,j}_{\Lambda,a},c^{i,j}_{\nabla,a}$ and $g^{i,j-1}_+\in\O^{i,(i,j-1)}_{\loc}$ such that
		
		\begin{nalign}
			\sum_az^{i,j}_a\frac{\partial}{\partial z^{i,j}_a}(\Box+W^4_a)W_a+(\Box+\Vbold)\Big(c_{\nabla,a}^{i,j}\cdot g_{\nabla,a}^{i-1}+c_{\Lambda,a}^{i,j}({\bar{g}_{\Lambda,a}^{i-1}}+t_a\Lambda W_a)+g^{i,j}_a(y_a)+g_+^{i,j-1}\Big)\\
			=-\mathfrak{Err}[\bar\phi]+\O^{i+3,(i,j-1)}_{\loc}
		\end{nalign}
		
		\textit{Step 3: $\bar\phi$ cross terms.}
		We finally, compute the error terms arising from  linearising around $\bar{\phi}$.
		Starting with the scaling terms, we notice that
		\begin{equation}
			\begin{gathered}
				\mathfrak{E}^{\lin}[\bar{\phi}]\frac{\log^j t_a}{t^i_a}\Big(g_{\Lambda,a}^{i-1},\tilde{g}^{i,j}_a\Big)\in \mathcal{O}^{6+i,i+1}_{\loc},\qquad \mathfrak{E}^{\lin}[\bar{\phi}]\frac{\log^j t_a}{t_a^i}t_a\Lambda W_a\in\mathcal{O}^{5+i,(i,j+1)}.
			\end{gathered}
		\end{equation}
		The second term is not acceptable according to \cref{eq:local improvement} due to the flow decay towards $F_a$.
		This follows from the fact that $\Lambda W$ is no longer a good conserved quantity for $(\Box+V_a+\mathfrak{E}^{\lin,a}[\bar{\phi}])$.
		We now show, that this can be corrected using$G$ defined in \cref{an:eq:starting_G}.
		In particular, we can correct the worst errors created by $\mathfrak{E}^{\lin,1}[\bar{\phi}]$, by adding $-c^{i,j}_{\Lambda,a}G_at_a^{-i-1}\log^{j}t_a(\log t_a c_{\Lambda,1}^{2,1}+c_{\Lambda,1}^{2,0})$ to $g_a^{i,j}$, since
		\begin{equation}
			(\Box+5(\bar{\phi})^4)(\underbrace{t^{-i}\Lambda W_1-\frac{G\log t}{t^{i+1}}c^{2,1}_{\Lambda,1}-\frac{G}{t^{i+1}}c^{2,0}_{\Lambda,1}}_{=I})\in\O^{0,2}_1\O^{i+2,i}_{\loc}.
		\end{equation}
		Therefore, we $I$ instead of $t^{-i}\Lambda W$ when correcting for the kernel related to scaling.
		We already notice, that we can write
		\begin{equation}
			I=t^{-i}\Lambda W^{\lambda_1+c^{1,1}_{\Lambda,1}\frac{\log t}{t}+c^{1,0}_{\Lambda,1}\frac{1}{t}},
		\end{equation}
		which shows again the connection between implicit and explicit modulation.
		In conclusion, we proved that for appropriate choice of $z_a^{i,j}$ as above, we get
		\begin{nalign}
			(\Box+\Vbold+\mathfrak{Err}^{\lin}[\bar\phi])\Big(c_{\nabla,a}^{i,j}\cdot g_{\nabla,a}^{i-1}+c_{\Lambda,a}^{i,j}({\bar{g}_{\Lambda,a}^{i}}+t_a\Lambda W_a)+g^{i,j}_a(y_a)+g_+^{i,j-1}\Big)\\
			+\sum_az^{i,j}_a\frac{\partial}{\partial z^{i,j}_a}(\Box+W^4_a)W_a=-\mathfrak{Err}[\bar\phi]+\O^{i+3,i+3,(i,j-1)}_{\loc}
		\end{nalign}
		We can also bound the nonlinear terms as 
		\begin{equation}
			\begin{gathered}
				\mathcal{N}[\phi_{corr};\bar{\phi}]=\mathcal{O}^{2i+1,(2i-2)}_{\loc},
			\end{gathered}
		\end{equation}
		and \cref{eq:local improvement} follows.
	\end{proof}

	We can combine these improvements in an iteration scheme to construct an ansatz with arbitrary fast decaying error.
	\begin{proof}[\hypertarget{an:proof:main1}{Proof of \cref{an:item:main1}} of \cref{an:thm:existence_of_ansatz}]
		We start with \cref{an:prop:starting}.
		Given $\mathfrak{Err}[\bar{\phi}]\in\O^{5,i,j}_{\scri}$ we do the following iteration
		\begin{itemize}
			\item for $i=j+2$, we use \cref{an:lemma:improving_+} to improve the error close to $I^+$ to $\O^{5,i+1,j}_{\scri}$,
			\item for $i=j+3$, we use \cref{an:lemma:improving_F} to improve the error close to the solitons ($F_a$) to $\O^{5,i,j+1}_{\scri}$.
		\end{itemize}
	\end{proof}

	\begin{remark}
		In case, there were non integer terms in the index set, e.g. $\mathfrak{Err}[\bar{\phi}]\in\A{phg}^{\E_{\scri},\E_{+},\E_{\loc}}$, we would replace the above algorithm with
		\begin{itemize}
			\item for $(i,j)=\min(\E_+)\leq \min(\E_{\loc})+2$, we use \cref{an:lemma:improving_+} to improve the error to $\A{phg}^{\E_{\scri},(\E_{+})_{>i},\E_{\loc}\cup \overline{(i-2,N)}}$ for some $N$ depending on the minimum in $j$.
			\item for $(i,j)=\min(\E_{\loc})\leq \min(\E_+)-3$, we use \cref{an:lemma:improving_F} to improve the error to $\A{phg}^{\E_{\scri},(\E_{+})\cup \overline{(i+3,N)},(\E_{\loc})_{>i}}$.
		\end{itemize}
	\end{remark}
	
	We finish this section with the supercritical case.
	
	\begin{proof}[\hypertarget{an:proof:supercritical}{Proof of \cref{an:thm:supercritical}}]\label{key}
		We highlight the differences compared to the proof of \cref{an:item:main1}.
		
		\emph{Step 1) start:}
		The error term ($f_{\mathrm{start}}$ as in \cref{an:eq:f_start}) of the trivial ansatz $\sum_a W^{\mathrm{sup}}_a$ satisfies $f_{\mathrm{start}}\in\O^{5,(1,0)}_{\loc}$ as in \cref{an:lemma:existence of admissible points}.
		Just as in \cref{an:lemma:existence of admissible points}, we have that $P^a_{1,0} f_{\mathrm{start}}=c_af'(W_f)$ is supported on $\ell=0$ spherical harmonic.
		Using that the nonlinearity $f$ is admissible, in particular that
		\begin{equation}
			\ker(\Delta+f'(W^\mathrm{sup}))=\{\partial_i W\},
		\end{equation}
		we can solve away for the leading order errors with
		\begin{equation}\label{an:eq:supercritical_correction_1}
			g_a^{1,0}=\frac{1}{t_a}(\Delta+f'(W^\mathrm{sup}))^{-1}P^a_{1,0}f_{\mathrm{start}}\in\O^{(2,0),(1,0)}_{a}.
		\end{equation}
		
		$g_a^{1,0}$ creates $\O^{(4,0),(3,0)}_{\loc}$ error terms that we solve away on $I^+$ via \cref{an:lemma:improving_+}:
		\begin{equation}
			g_+^{2,0}=\frac{1}{t^2}N_2^{-1}\Big(P^+_{4,0}\sum_a \frac{1}{t_a^3}g_a^{1,0}\Big)\Big(\frac{\bar{x}}{t}\Big).
		\end{equation}
		We note, that the nonlinear error terms from $g_+^{2,0},g_a^{1,0}$ are in $\O^{5,(2,0)}_{\loc}$, while the linear error terms may be computed explicitly, similar to the proof of \cref{an:prop:starting}.
		Therefore, we obtain that $\bar{\phi}=g_+^{2,0}+\sum_a W_a+g_a^{1,0}$ satisfies
		\begin{equation}
			\Box \bar{\phi}+f(\bar{\phi})\in\O^{5,(2,0)}_{\loc}.
		\end{equation}
		From here on, the proof is the same as for \cref{an:item:main1}, with the exception that there are no kernel elements connected to scaling $\Lambda W$, therefore we do not require strong admissibility for the location of the solitons.		
	\end{proof}

	\paragraph{Orthogonality}
	
	Although, we do not require any extra orthogonality for the supercritical problem to construct the ansatz, this will be useful for the energy estimates in \cref{sec:linear_theory}.
	We record this here.
	Let us write $f'=f'(W^{\mathrm{sup}})$, and also $W^{\mathrm{s}}=W^{\mathrm{sup}}$ similarly for higher derivatives.
	Notice, that for some constant $c$, depending on the soliton velocities, we may write
	\begin{equation}\label{an:eq:leading_supercritical}
		P^a_{1,0}\mathfrak{E}^l=cf''\big(1-\underbrace{(\Delta+f')^{-1}f'}_{g}\big),
	\end{equation}
	where the first term is coming from the other solitons far field effect and the second one is from the correction \cref{an:eq:supercritical_correction_1}.
	Let us record the following computation
	\begin{nalign}\label{an:eq:supercritical_cancellation}
		&\int_{\R^3} \abs{\partial_i W^{\mathrm{s}}}^2f''\big(1-g\big)=\int_{\R^+} \abs{ W^{\mathrm{s}}{}'}^2f''\big(1-g\big)r^2=\int_{\R^+} W^{\mathrm{s}}{}'\partial_r(f')r^2-\int_{\R^+} W^{\mathrm{s}}{}'\partial_r(f') gr^2\\
		&=-\int_{\R^+} (\Delta W^{\mathrm{s}}) f'r^2+\int_{\R^+} (\Delta W^{\mathrm{s}}) f' g r^2+W^{\mathrm{s}}{}'f' g'r^2
		=-\int_{\R^+} (\Delta W^{\mathrm{s}}) f'r^2+\int_{\R^+} (\Delta W^{\mathrm{s}}) f' g r^2+W^{\mathrm{s}}{}'f' g'r^2\\
		&=\int_{\R^+}f f'r^2-\int_{\R^+} f f' g r^2-\partial_r (f) g'r^2=\int_{\R^+}f f'r^2-\int_{\R^+} f (f'+\Delta g) r^2=0.
	\end{nalign}

	\subsection{Using outgoing radiation}\label{an:sec:outgoing}
	The discussion in the present section builds upon the explicit formulas found in section 3 \cite{stoll_harmonic_2016}.
	
	In this section, we show how to construct solitons with well chosen outgoing radiation, in particular, how to overcome the difficulties presented by the scaling mode $\Lambda W$. Notice, that the first problem that we encountered constructing an ansatz was the $1/\jpns{t}$ error terms on the faces $F_a$ coming from the leading order interactions. To counteract such error terms, it is sufficient to add in well chosen radiation to cancel the kernel elements thereby rendering modulation unnecessary.
	First, let us recall from \cite{kadar_scattering_2024} how $N_\sigma$ is related to the Laplace-Beltrami operator on hyperbolic space.
	
	\begin{lemma}[Lemma 7.13 in \cite{kadar_scattering_2024}]\label{an:lemma:hyperbolic_space}
		Setting $\tilde{\rho}=\frac{1-\sqrt{1-\rho^2}}{\rho}$ and $h_\sigma(\tilde{\rho})=(\frac{1+\tilde{\rho}^2}{1-\tilde{\rho}^2})^{1+\sigma}$ we get
		\begin{equation}\label{eq:I0normal_operator}
			\begin{gathered}
				\frac{1}{h_\sigma(\tilde{\rho})}	N_\sigma h_\sigma(\tilde{\rho})f(\tilde{\rho})=\frac{(\tilde{\rho}^2+1)^2}{4}\Big(\Delta_{\mathbb{H}}-4(\sigma^2-1)\Big)\\
				\Delta_{\H}=(1-\tilde{\rho}^2)\partial^2_{\tilde{\rho}}+\frac{2(1-\tilde{\rho}^2)}{\tilde{\rho}}\partial_{\tilde{\rho}}+\frac{\slashed{\Delta}}{\tilde{\rho}^2}.
			\end{gathered}
		\end{equation}
	\end{lemma}
	
	Next, we use this correspondence, to show that we can prescribe radiation fields, so that the leading term of the corresponding solution to the linear wave equation has a particular leading order term at the location of the solitons.
	\begin{lemma}\label{an:lemma:outgoing radiation}
		a) Let $z_a\in B$ be a finite set of points and $\mu^{0}_a\in\R$ associated real numbers. For any $\sigma\geq0$ there exists $f\in\mathcal{C}^\infty(S^2)$ such that the solution $g$ to
		\begin{equation}\label{eq:outgoing radiation only}
			\begin{cases}
				N_\sigma g=0\\
				((1-\rho)^{\sigma}g)|_{\partial B_1}=f
			\end{cases}
			\sigma>0
			\qquad
		\begin{cases}
			N_0 g=0\\
			(\log(1-\rho)g)|_{\partial B_1}=f 
		\end{cases}
	\sigma=0
		\end{equation}
		satisfies $g(z_a)=\mu^{0}_a$ for all $a$.
		
		b) Fix furthermore $\mu^{1}_a\in T^\star_{z_a}$ cotangent vectors in $\H^3$ corresponding to $z_a$.
		For any $\sigma\geq 0$, there exists $f\in\mathcal{S}^2$ such that the solution to \cref{eq:outgoing radiation only} satisfies $g(z_a)=\mu^{0}_a, d g|_{z_a}=\mu^{1}_a$.
	\end{lemma}

	\begin{proof}
		\emph{a)}
		Let us start with the case $\sigma>0$. It suffices to study the above problem on hyperbolic space via \cref{an:lemma:hyperbolic_space}.
		Let us also recall that \emph{outgoing} and \emph{no outgoing} radiation solutions with asymptotics $g(\rho)\sim (1-\abs{\rho})^{-\sigma},(1-\abs{\rho})^{0}$ respectively correspond to $(1-\abs{\rho})^{-(\sigma-1)},(1-\abs{\rho})^{\sigma+1}$ asymptotics on $\H^3$ respectively.
		The metric on hyperbolic space takes the form $g_{\H^3}=\frac{4\d x_i\d x_j}{(1-\abs{x}^2)^3}$. 
		The Green's function is
		\begin{equation}
			G_{z_a}=\frac{1-\mathfrak{r}^2}{2\mathfrak{r}}\Big(\frac{1-\mathfrak{r}}{1+\mathfrak{r}}\Big)^{\sigma},\qquad \frac{\mathfrak{r}^2}{1-\mathfrak{r}^2}=d(z,z_a)=\frac{\norm{z-z_a}^2}{(1-\norm{z}^2)(1-\norm{z_a}^2)}.
		\end{equation}
		In particular it suffices to prove that there exists $g\in\O^{\sigma-1}_{B}$ and $f\in\C^{\infty}(S^2)$ such that
		\begin{equation}\label{eq:outgoing_rad_H3}
			\begin{gathered}
				(\Delta_{\H^3}-4(\sigma^2-1))g=0\\
				(1-\rho)^{\sigma-1}g|_{\partial B}=f
			\end{gathered}
		\end{equation}
		with $g(z_a)=\mu_a$.
		Given any $f$, the existence follows from \cref{an:lemma:model_operator_on_i_+}.
		We write the second condition as
		\begin{equation}\label{eq:Poisson kernel computation}
			\begin{gathered}
				\mu_a=\jpns{\delta(z-z_a),g}_{L^2(\H^3)}=\jpns{(\Delta-\alpha)G^\alpha(z-z_a),g}_{L^2(\H^3)}\\
				=\lim_{\rho\to 1^+}\int_{B_\rho}g(\Delta-\alpha)G^\alpha(z-z_a)=\lim_{\rho\to 1^+} \int_{\partial B_\rho} (g\nabla G(z-z_a)-G(z-z_a)\nabla g)\cdot \d S.
			\end{gathered}
		\end{equation}
		Using the fact that both $g$ and $G(z-z_a)$ are polyhomogeneous, we can write them as
		\begin{equation}
			g(z)=(1-\abs{z})^{1-\sigma}f(\omega)+\O^{2-\sigma}_{B},\quad G(z-z_a)=G_a(\omega)(1-\abs{z})^{1+\sigma}+\O^{2+\sigma,-1}_{\dot{B}_a},
		\end{equation}
		for some $G_a\in C^\infty(S^2)$. The measure on $\partial B_\rho$ is $\frac{4\slashed{g}_{S^2}}{(1-\rho^2)^2}$ and the unit outward derivative is $\frac{(1-\rho^2)}{2}\partial_\rho$. Putting this into the boundary integral and passing to the limit, we get
		\begin{equation}
			\begin{gathered}
				\mu_a=\lim_{\rho\to1^+}\int_{S^2}\frac{2}{(1-\rho^2)}\big(g(\omega\rho)\partial_\rho G(\omega\rho-z_a)-G(z-z_a)\partial_\rho g(\rho\omega)\big)\\
				=\lim_{\rho\to1^+}\int_{S^2}\frac{1}{1-\rho}\big(g(\omega\rho)\partial_\rho G(\omega\rho-z_a)-G(z-z_a)\partial_\rho g(\rho\omega)\big)=-2\sigma\int_{S^2}f(\omega)G_a(\omega).
			\end{gathered}
		\end{equation}
		If $G_a(\omega)$ are linearly independent as functions on $S^2$, than it follows that we can choose $f$ such that all the constraint ($g(z_a)=\mu_a$) are satisfied.
		Assume contrary, that is $G_a(\omega)=\sum_{b\neq a}c_b G_b(\omega)$.
		More specifically, let us assume that the equality holds on at least one spherical harmonic, $l$.
		
		For $h=P^{S^2}_lh$, a function over $S^2$,
		consider the boundary value problem 
		\begin{equation}\label{eq:subleading_rad_H3}
			\begin{gathered}
				(\Delta_{\H^3}-4(\sigma^2-1))g=0\\
				(1-\rho)^{-\sigma-1}g|_{\partial B}=h.
			\end{gathered}
		\end{equation}
		Using Frobenius method --treating $\Delta_{\H^3}-4(\sigma^2-1)$ as a regular singular ODE in $\rho$ variable-- we can find a unique local solution to \cref{eq:subleading_rad_H3} in  $\O^{(1-\sigma,0)}_{\B}(X)+\Hb^{3/2-\sigma}(X)$ where $X=(\{\abs{\rho}\in(\rho_0,1)\})$ and $\rho_0$ depends on $h,\sigma$.
		Applying this with $h=P^{S^2}_l\big(G_a(\omega)-\sum_{b\neq a}c_b G_b(\omega)\big)=0$ we get that
		\begin{equation}
				\big(G(z-z_a)-\sum_{b\neq a}c_b G(z-z_b)\big)=0, \qquad z\in\{\abs{\rho}\in(\rho_0,1)\}.
		\end{equation}
		Extending the left hand side analytically and acting by the operator $\Delta_{\H^3}-4(\sigma^2-1)$ we get the required contradiction.
		
		The only step that changes for $\sigma=0$, is that we have $g(z)=f(\omega)(1-z)\log(1-\abs{z})+\O_{B}^{(1,0)}$ and $G(z-z_a)=G_{a}(\omega)(1-\abs{z})+\O_{B}^{2}$. Computing \cref{eq:Poisson kernel computation} for $\sigma=0$, the logarithmic terms cancel and we get
		\begin{equation}
			\begin{gathered}
				\mu_a=\int_S^2f(\omega)G_a(\omega).
			\end{gathered}
		\end{equation}
		The rest of the proof is identical.
		
		\emph{b)}
		We proceed as before. We compute for some vector $X\in T_{z_a}$
		\begin{equation}\label{eq:outgoing radiation l=1}
			\begin{gathered}
				X(g(z))|_{z=z_a}=X(\jpns{\delta(\cdot-z),g(\cdot)})_{L^2}|_{z=z_a}=\jpns{X(\Delta-\alpha)G(\cdot-z),g(\cdot)}_{L^2}|_{z=z_a}\\
				=\int_{\lim\rho\to1^+}\int_{\partial B_\rho}\big(g(\bar{z})X_{z}\nabla_{\bar{z}} G(\bar{z}-z)-X_{{z}}G(\bar{z}-z)\nabla_z g(z)\big).
			\end{gathered}
		\end{equation}
		Using that $G(z,z_a)=$ only depends on $z,z_a$ through $d(z,z_a)$, and vanishes at the rate $(1-\abs{z})^{\sigma+1}$, we get that $G(z,z_a)=d(z,z_a)^{-(\sigma+1)}+l.o.t$. Differentiating this with respect to $z_a$, we get a function that still vanishes at $(1-\abs{z})^{\sigma+1}$ rate. Therefore, the limit in \cref{eq:outgoing radiation l=1} is well defined and only depends on $X,z_a$. We get that the functions $G_{a,X}$ for different $a$ are linearly independent on the sphere by the same argument as in \textit{a)}. Furthermore, $G_{a,X}$ are also linearly independent from $G_a$, using e.g. spherical decomposition around $z_a$.
		The result follows.
	\end{proof}
	
	\begin{remark}
		The above construction is also called the Poisson kernel, for more details, see \cite{stoll_harmonic_2016}.
	\end{remark}
	
	We use this lemma to remove kernel elements from error terms
	\begin{lemma}\label{an:lemma:outgoing_correction}
		Let $f\in\O^{5,N+3,(N,j)}_{\scri}$.
		
		a) For $N\geq1$ and $P^{S^2}_1P^a_{N,j}f=0,\, \forall a$ there exists $g\in\O^{1,(N,j),(N,j)}_{\scri}$ such that
		\begin{equation}
			P^a_{N,j}\Big(f-(\Box+\Vbold)g\Big)\perp\{\Lambda W^{\un}_a,\nabla W^{\un}_a\}.
		\end{equation}
		
		b) For $N\geq2$ there exists $g\in\O^{1,(N-1,j),(N,j)}_{\scri}$ such that
		\begin{equation}
			P^a_{N,j}\Big(f-(\Box+\Vbold)g\Big)\perp\{\Lambda W^{\un}_a,\nabla W^{\un}_a\}.
		\end{equation}
	\end{lemma}
	\begin{proof}
		\emph{a)} We set $\mu^{(0)}_a(V^{\un}_a,\Lambda W^{\un}_a)=(\Lambda W^{\un}_a,P^a_{N,j}f)$ and $\mu^{(1)}=0$, and take $g$ from \cref{an:lemma:outgoing radiation} with $\sigma=N-1$.
		This solves the required orthogonality condition.
		
		\emph{b)} Let us first note, that from an explicit computation, we have
		\begin{equation}
			(x_i V,\partial_j W)_{L^2}=\delta_{ij}C,\qquad C\neq0.
		\end{equation}
		We set $\mu^{0}_a=0$ and pick $\mu^{1}_a$ such that for $\sigma=N-2$ and $g$ as in \cref{an:lemma:outgoing radiation} we have for some constant $c(z_a)\neq0$
		\begin{equation}
			P^{S^2}_1P^a_{N,j}f-\Vbold \underbrace{P^{S^2}_1P^a_{N,j}g}_{c(z_a)\mu_a^1\cdot x}\perp\{\nabla W^{\un}\}.
		\end{equation}
		Next, we cancel the error terms for the spherically symmetric part as in \emph{a)}.
	\end{proof}

	Using these two lemmas, we can create an ansatz $\phi^a$ by setting each $z^{i,j}=c^{i,j}_{\nabla,a}=c^{i,j}_{\Lambda,a}=0$.

	\begin{proof}[Proof of \cref{an:item:main2}]
		The proof is essentially the same as for \cref{an:item:main1}, we detail only the differences.
		
		\emph{Step 1) start:}
		$f_{\mathrm{start}}$ from \cref{an:eq:f_start} has leading order term $1/t$ on each face $F_a$:
		\begin{equation}
			f_{\mathrm{start}}=\frac{c_aV_a}{t_a}+\O^{0,1}_a\O^{5,1}_{\loc}.
		\end{equation}
		As we cannot invert this in general, we add $g^{1,0}_+\in\O^{1,(1,0),(1,0)}_\scri$ to $\tilde{\phi}$ from \cref{an:lemma:outgoing_correction} creating the orthogonality.
		Since $(\Box+\Vbold) g^{1,0}_+=\Vbold g^{1,0}_+$,  we not only have orthogonality, but already
		\begin{equation}
			\Box\bar{\phi}+\bar{\phi}^5\in\O^{5,(2,0)}_{\loc},\qquad \mathfrak{E}^{\lin}[\bar{\phi}]\in\O^{5,(2,0)}_{\loc}.
		\end{equation}
		
		\emph{Step 2) iteration:}
		We iterate \cref{an:lemma:improving_F,an:lemma:improving_+} and use \cref{an:lemma:outgoing_correction} instead of modulating via $c^{i,j}_{\Lambda,a}$ and $z^{i,j}_a$.
	\end{proof}
	
	\subsection{Modified conservation laws}\label{an:sec:conservation_laws}
	The current section plays no vital role in the rest of the paper, we merely include it to describe the approximate solution $\bar{\phi}$ better. 
	We analyse the conserved quantities associated to $\bar{\phi}$ constructed in \cref{an:thm:existence_of_ansatz,an:thm:supercritical}.
	
	The main result of the current section is
	\begin{prop}\label{lemma:ansatz_with_conservation_laws}
		Let $\bar{\phi}$ be a solution constructed in the proof of \cref{an:thm:existence_of_ansatz} for some $N$.
		There exist modifications $\epsilon_{\nabla,a}\in\O^{2,1}_a$ and $\epsilon_{\Lambda,a}\in\O^{2,1}_a$ such that in $\{\abs{y}_a\leq \delta_4 y_a\}$ we have
		\begin{equation}\label{an:eq:modified_conservation}
			\begin{gathered}
				(\Box+\Vbold+\mathfrak{Err}^{\lin}[\bar{\phi}])(\partial_i W_a+\epsilon_{\Delta,a})=\O^{4,(2,0)}_{\loc}\\
				(\Box+\Vbold+\mathfrak{Err}^{\lin}[\bar{\phi}])(\Lambda W_a+\epsilon_{\Lambda,a})=\O^{4,2}_{\loc}.
			\end{gathered}
		\end{equation}
	\end{prop}
	
	\subsubsection{Scaling}
	We start with \cref{an:item:main1}.
	For the rest of this section, we restrict to $\{\abs{y}_a\leq \delta_4 y_a\}$ and ignore the weights towards $F_b$ for $b\neq0$.
	Note, that the linearised operator around $\bar{\phi}$ does not admits exact stationary solutions such as $\Lambda W_1,\partial W_1$.
	For $\phi$ solution to $(\Box+5\bar{\phi}^4)\phi=0$ we can expand the potential term in a neighbourhood of the first soliton as
	\begin{equation}
		\begin{gathered}
			\bar{\phi}^4=W_1^4+4W_1^3\frac{\Lambda W_1}{t}\big(c^{2,1}_{\Lambda,1}\log t+c^{2,0}_{\Lambda,1}\big)+\O^{4,2}_{\loc}
		\end{gathered}
	\end{equation}
	Therefore, we compute the correction to the conservation law for $\Lambda W_1$
	\begin{equation}
		\begin{gathered}
			(\Box+5\bar{\phi}^4)\Lambda W_1=\frac{(\Lambda W_1)^2W_1^3}{t}(\log tc_{\Lambda,1}^{2,1}+c_{\Lambda,1}^{2,0})+\O^{4,2}_{\loc}.
		\end{gathered}
	\end{equation}
	where the decay rate $4$ towards $I^+$ comes from the logarithmic change of the path of the soliton and the induced error from $\Box$.
	It is important to remove the logarithmic correction in the above, so we use the orthogonality \cref{M1} and write
	\begin{equation}
		\begin{gathered}
			(\Box+5\bar{\phi}^4)\Big(\Lambda W_1-\frac{G_1\log t}{t}c_{\Lambda,1}^{2,1}-\frac{G_1}{t}c^{2,0}_{\Lambda,1}\Big)=\O^{4,2}_{\loc}.
		\end{gathered}
	\end{equation}
	We similarly have the higher decaying analogue
	\begin{equation}\label{eq:LambdaW correction}
		\begin{gathered}
			(\Box+5(\bar{\phi})^4)t^{-i}(\Lambda W_1-\frac{G_1\log t}{t}c^{2,1}_{\Lambda,1}-\frac{G_1}{t}c^{2,0}_{\Lambda,1})=\O^{i+4,i+2}_{\loc}.
		\end{gathered}
	\end{equation}
	
	For \cref{an:item:main2} we have that the leading $1/t$ contribution of $\mathfrak{E}^{\lin}[\bar{\phi}]$ coming from $g^{2,0}_+$ and the cross terms $W_a^4\sum_b W_b$ exactly cancel and so the higher conservation automatically holds.
	
	\paragraph{Higher order conservation}
	Let us assume that for some $ i\geq0$ we already found correction term $\epsilon_{\Lambda,1}\in\O^{2,1}_{\loc}$ such that 
	\begin{equation}
		(\Box+5\bar{\phi}^4)\Big(\Lambda W_1+\epsilon_{\Lambda,1}\Big)=\O^{5+i,i+2}_{\loc}.
	\end{equation}
	We can locally around $F_1$ repeat \cref{an:lemma:improving_+,an:lemma:improving_F} to get better conservation laws.
	
	We compute that for $c_1\in\R$ and $h_1:=h_1(\tilde{x})\in\O^{3}(\R^3)$
	\begin{nalign}
		(\Box+5\bar{\phi}^4)t^{-i}\Big(c_1\big(t\Lambda W_1-G_1(\log tc_{\Lambda,1}^{2,1}+c_{\Lambda,1}^{2,0})+i(i-1)g^{i}_{\Lambda,1}\big)+h_1\Big)=t^{-i}(\Delta_1+V_1)h_1\\
		+t^{-i}M^{\Lambda,i}_{1,1}c_1V_1\qquad\mod \O_{\loc}^{i+5,i+1}
	\end{nalign}
	The inclusion of logarithmic terms is done as in \cref{an:eq:Lamdba-W}.
	Picking $h_1,c_1$ appropriately, we can remove the leading order error from the conservation law close to $F_1$, without making the error term at $I^+$ worse.	 
	In conclusion, we can find an alternative $\epsilon'_{\Lambda,1}\in\O^{2,1}_{\loc}$ such that the 
	\begin{equation}
		(\Box+5\bar{\phi}^4)\Big(\Lambda W_1+\epsilon'_{\Lambda,1}\Big)=\O^{5+i,1+2}_{\loc}.
	\end{equation}
	
	Similarly, for $i\geq0$ and
	\begin{equation}\label{an:eq:conservation_Lambda_error2}
		(\Box+5\bar{\phi}^4)\Big(\Lambda W_1+\epsilon_{\Lambda,1}\Big)=\mathfrak{e}=\O^{4+i,i+2}_{\loc}
	\end{equation}
	we can invert the leading term on $I^+$.
	We ignore the logarithmic terms, as they are dealt with exactly as before. 
	For a function $h:\dot{B}\to\R$ and cutoff $\chi=\prod_a \bar{\chi}^c(\tilde{y}_a)$ we compute
	\begin{equation}
		(\Box+5\bar{\phi}^4)t^{-i-2}\chi h(\bar{x}/t)=\chi t^{-i-4}(N_{i+2}h)(\bar{x}/t)+\chi h(\bar{x}/t)\O^{i+5,i+2}_{\loc}.
	\end{equation}
	The projection of the error term onto $I^+$, $P_+^{0,0}(\mathfrak{e}t^{4+i})$, may have severe singularities at the punctures not at $F_1$, and these correspondingly make severe singularities for $h$, when inverting $N_{i+2}$.
	These are exactly cancelled by the extra fast decaying $t^{i+2}$ prefactor, so that solving $N_{i+2}h=P_+^{0,0}(\mathfrak{e}t^{4+i})$ yields
	\begin{equation}
		(\Box+5\bar{\phi}^4)t^{-i-2}\chi h(\bar{x}/t)=\mathfrak{e}+\O^{5+i,1}_{\loc}\O^{0,2+i}_a.
	\end{equation}
	By adding $t^{-i-2}\chi h$ to $\epsilon_{\Lambda,1}$ the following improvement holds
	\begin{equation}
		(\Box+5\bar{\phi}^4)\Big(\Lambda W_1+\epsilon'_{\Lambda,1}\Big)=\O^{5+i,1}_{\loc}\O^{0,2+i}_a.
	\end{equation}
	
	\subsubsection{Momentum}

	\paragraph{Improvement to arbitrary order}
	
	For the momentum, we only show very fast decay.
	Note that in \cref{an:sec:improvements}, we proved \cref{an:thm:existence_of_ansatz} for arbitrary $z^{0,0}_a$, so we construct a family of approximate solutions $\bar\phi(z^{0,0}_a)$. 
	Differentiating this family with respect to $z^{0,0}_a$ yield an almost conserved quantity to much higher decay.
	\begin{lemma}
		Let $\bar{\phi}(z^{0,0}_a)$ be the approximate solutions from \cref{an:thm:existence_of_ansatz} parametrised by the centers of the solitons. Then
		\begin{equation}
			\begin{gathered}
				(\Box+5\bar{\phi}^4)\partial_{z^{0,0}_a}\bar{\phi}=\mathcal{O}_{\loc}^{5,N,N}\\
				\partial_{z^{0,0}_a}\bar{\phi}-\nabla W_a=\mathcal{O}_{a}^{2,1}
			\end{gathered}
		\end{equation}
	\end{lemma}
	\begin{proof}
		The construction of $\bar\phi$ in \cref{an:sec:improvements} is smooth in $z^{0,0}_a$, and so is the error term.
		The first equality follows from differentiating \cref{an:eq:bar_phi}.
		
		The second equality follows by observing noting that $\bar{\phi}(z_a^{0,0})-\bar{\phi}(z'{}_a^{0,0})=\O^{2,1}_a$.
	\end{proof}
	
		\subsubsection{Eigenvalues}
	In the construction of approximate solution $\bar{\phi}$, the eigenvalue
	\begin{equation}
		(\Delta+5W^4)Y=\lamed Y
	\end{equation}
	played no role, as it is not an obstruction for the invertibility of  $(\Delta+5W^4)$.
	Nonetheless, it gives an obstruction to boundedness for solutions to the linearised equation \cref{i:eq:lin}.
	Therefore, we also need to modify it so that it has good conservation law.
	We notice that
	\begin{multline}\label{lin:eq:error_Y_equation}
		e^{-t\lamed}(\Box+5\bar{\phi}^4)e^{t\lamed}Y=(-\lamed^2+\Delta+V_1+\frac{20c_{\Lambda,1}^{1,1}\log t}{t}W_1^3\Lambda W_1)Y=\\
		\frac{20c_{\Lambda,1}^{1,1}\log t}{t}W_1^3\Lambda W_1 Y\quad\O^{\infty,(1,0)}_1.
	\end{multline}
	A naive attempt to improve the conservation law by modifying $Y\mapsto Y+\frac{\log t}{t}\bar{Y}$ for some smooth and exponentially decaying $\bar{Y}$, yields
	\begin{equation}
		e^{-t\lamed}(\Box+5\bar{\phi}^4)e^{t\lamed}( Y+\frac{\log t}{t}\bar{Y})=\frac{\log t}{t}(20c_{\Lambda,1}^{1,1}W_1^3\Lambda W_1Y+(-\lamed^2+\Delta+V_1)\bar{Y})\quad\O^{\infty,(1,0)}_1.
	\end{equation}
	Unless the error in \cref{lin:eq:error_Y_equation} is orthogonal to $Y$, we cannot find $\bar{Y}$ to improve the decay rate.
	We may compute
	\begin{multline}\label{an:eq:Y_inner_product}
		0=\frac{1}{2}\frac{\dd }{\dd \lambda}\int \abs{\nabla Y^\lambda}^2-5 (W^\lambda)^4 Y^2+\lamed^2\lambda^2(Y^\lambda)^2=\int \cancel{\nabla Y\cdot\nabla SY-5W^4 Y SY+\lamed^2 Y SY}\\
		-\int10 W^3 \Lambda W Y^2-\lamed^2 Y^2=0.
	\end{multline}
	Therefore, we must modulate the kernel element.
	By setting $\lamed$ time dependent, $\lamed^Y=\lamed+\frac{\log t}{t}c+\O(t^{-1})$, we obtain
	\begin{equation}
		e^{-t\lamed}(\Box+5\bar{\phi}^4)e^{t\lamed^Y(t)}Y=\frac{1}{t}(\log t\cdot20c_{\Lambda,1}^{1,1}W_1^3\Lambda W_1-2c\lamed)Y\quad\mod\O^{\infty,(1,0)}_1.
	\end{equation}
	Hence, to obtain correction at the right order, we must correct the scaling at higher order.
	Indeed, setting $\lamed^Y=\lamed+\frac{\log t}{t}c^{1,1}_{\Lambda,1}+\O(t^{-1}\log t)$, and using \cref{an:eq:Y_inner_product}, we get that the error term is orthogonal to $Y$ so there exists $\bar{Y}\in L^2$ such that
	\begin{equation}
		(-\lamed^2+\Delta+V_1)\bar{Y}=2c^{1,1}_{\Lambda,1}(10W_1^3\Lambda W_1-\lamed)Y.
	\end{equation}
	Using elliptic regularity and Agmon estimates (see \cite{duyckaerts_dynamics_2008} section 5.5), we get that $\bar{Y}\in C^\infty$ and decays exponentially.
	We conclude that
	\begin{equation}
		e^{-t\lamed}(\Box+5\bar{\phi}^4)e^{t\lamed^Y(t)}(Y-\frac{\log t}{t}\bar{Y})=0\quad\mod\O^{\infty,(1,0)}_1
	\end{equation}
	\begin{lemma}
		Let $\bar{\phi}$ be as in \cref{an:thm:existence_of_ansatz}.
		Then, there exists $\lamed_a^Y-\lamed\lambda_a\in\O_{\R}^{(1,2)}$ and $Y^{\mathrm{cor}}-Y\in \O^{\infty,(1,1)}$ such that 
		\begin{equation}
			(\Box+5\bar{\phi}^4)e^{t\lamed_a^Y}Y^{\mathrm{cor}}=0 \qquad\mod\O^{\infty,N-10}_1.
		\end{equation}
		Similar result holds for the exponentially decaying solution.
	\end{lemma}
	\begin{proof}
		We iterate the above construction.
	\end{proof}

	\newpage
	
	\section{Scattering solutions}\label{sec:scattering}
	
	In this section, we present the local existence results on which our scattering construction is built upon.
	We study solutions to
	\begin{equation}\label{scat:eq:main}
		\Box\phi=\mathcal{P}[\phi],\quad \mathcal{P}[\phi]=\mathcal{N}[\phi]+\mathtt{V} \phi+f
	\end{equation}
	where $\mathtt{V}\in\Hbloc^{5-,5-,0-}$ and $\mathcal{N}[\phi]$ contains nonlinear terms of the form
	\begin{equation}
		\mathcal{N}[\phi]=\sum_{2\leq i\leq 5}\phi^i\Hbloc^{5-i-,5-i-,0-}.
	\end{equation}
	We also assume that $f\in\Hbloc^{3,a,b}$.
	The following is a standard semi-global existence result:
		
	\begin{theorem}\label{scat:thm:existence_of_scattering_solution}
			Let $\tau\gg1$ based on $\mathcal{N},V$.
			Fix $k>4$ and an initial hypersurface  $\Sigma_0=\mathcal{R}_{\tau,\tau}$ or $\Sigma_\epsilon=(\partial\mathcal{R}_{\tau-\epsilon,\tau})\setminus \Sigma_0$.
			Let $\phi\in\Hbloc^{1/2,a_+,a_F;k}$ 
			be a solution to \cref{scat:eq:main} in $\mathcal{R}_{\tau-\epsilon_1,\tau}$ for some $\epsilon_1>0$ satisfying the no outgoing radiation condition:
		\begin{equation}\label{scat:eq:no_outgoing}
			\partial_u r\phi|_{\scri}=0.
		\end{equation}
			Furthermore, let us assume that
			\begin{nalign}
				\Sigma_\epsilon[-T\cdot\T^0[\Vb^{k-1}\phi]]\leq 1.
			\end{nalign}
			Then there exists $\epsilon_2$ depending only on $f$ such that there exists a unique scattering solution to \cref{scat:eq:main} with no outgoing radiation in $\mathcal{R}_{\tau-\epsilon-\epsilon_2,\tau}$.
	\end{theorem}
	\begin{proof}
		This follows from local existence theory and the scattering result from \cite{kadar_case_2024}.
	\end{proof}
	
	\newpage
	\section{Linear theory }\label{sec:linear_theory}

	The linearised equation that we study takes the form
	\begin{equation}\label{lin:eq:main linearised}
			(\Box+5\bar{\phi}^4)\phi=f
	\end{equation}
	where $\bar{\phi}$ is as in \cref{an:thm:modulated_scaling}.
	More precisely, we work with:
	\begin{definition}\label{lin:def:admissible}
		We say that $\bar{\phi}$ is an \emph{admissible for energy estimates} ansatz if for $W_a$ as in \cref{an:eq:alternative_ansatz}
		\begin{subequations}\label{lin:eq:assumptionPhibar}
			\begin{gather}
				\bar{\phi}=\sum W_a+\O_{\loc}^{1,1},\qquad \mathfrak{Err}^{\lin,a}\in\O_{loc}^{5,1}\O_a^{0,1}\label{lin:eq:assumptionPhibar_strong}\\
				(\Box+\bar{\phi}^4)\bar\phi=\O^{5,N,N}.
			\end{gather}
		\end{subequations}
	\end{definition}
	Unless otherwise stated, we will assume for the rest of this section that $\bar{\phi}$ is admissible for energy estimates.
	The prosaic version of the main result of the current section is the following energy boundedness statement
	\begin{prop}[Prosaic version of \cref{lin:prop:main,lin:prop:main_weak}]\label{lin:prop:prosaic}\ \\
		a) Fix $\bar{\phi}$ from \cref{an:thm:existence_of_ansatz}.
		Fix $\tau_1\in[\tau_2^\Delta,\tau_2]$ for $\tau_2\gg1$ and let $\phi$ be a solution to \cref{lin:eq:main linearised} in $\Region$ with controlled unstable mode.
		There exist coercive energy quantities at regularity $k$ in $\mathcal{R}^{\e}_{\tau_1,\tau_2},\mathcal{R}^{a}_{\tau_1,\tau_2}$ denoted by $\master_{\tau_1,\tau_2}$ and similar inhomogeneity norms $\inhom_{\tau_1,\tau_2}$ such that
		\begin{equation}\label{lin:eq:prosaic}
			\master_{\tau_1,\tau_2}[\phi]\lesssim_{k,R_2}
			\master_{\tau_2,\tau_2}[\phi]+
			\inhom_{\tau_1,\tau_2}[f].
		\end{equation}
		b) \cref{lin:eq:prosaic} also holds for $\bar{\phi}$ of the form \cref{lin:eq:assumptionPhibar_weak}.
	\end{prop}
	
	The assumption \cref{lin:eq:assumptionPhibar_strong} on the approximate solution $\bar{\phi}$ is stronger than
	\begin{equation}\label{lin:eq:assumptionPhibar_weak}
		\bar\phi=\sum W_a+\frac{c_{\Lambda,a}^{2,1}\log t_a +c_{\Lambda,a}^{2,0}}{t_a}\Lambda W_a+\O^{1,2}=\sum W_a+\O^{1,1}
	\end{equation}
	proved in \cref{an:thm:existence_of_ansatz}, or the one obtained in \cref{an:thm:supercritical} for generic soliton velocities.
	This is necessary to close the energy estimates obtained in \cref{lin:lemma:local_energy_estimate,lin:prop:Theta_conservation}.
	Therefore, \cref{lin:prop:main} only applies around the ansatz $\bar{\phi}$ constructed in \cref{an:item:main2} and for \cref{an:thm:supercritical} under admissibility assumption.
	We improve this result to treat \cref{an:item:main2} and \cref{an:thm:supercritical} for generic soliton velocities in \cref{lin:prop:main_weak,lin:prop:supercritical}.
	
	\begin{remark}
		Even though, we call \cref{lin:prop:prosaic} an energy boundedness statement, a more appropriate name would be limited growth. In particular, note that applying the proposition to the forward problem (assuming control of the unstable and zero modes), one can merely obtain that the energy grows at most polynomially and does not remain bounded.
	\end{remark}
	
	Just as for the ansatz, we also have the same result for the energy supercritical problem
	\begin{prop}[Prosaic version of \cref{lin:prop:supercritical}]\label{lin:prop:prosaic_supercritical}
		Let $\bar{\phi}$ be the solution constructed in \cref{an:thm:supercritical}, and let
		$\phi$ have controlled unstable mode(s) and solve
		\begin{equation}\label{lin:eq:supercritical_main}
			(\Box+\Vbold^{\mathrm{s}}+\mathfrak{Err}^{lin,\mathrm{s}}[\bar{\phi}])\phi=f.
		\end{equation}
		Let $\tau_1,\tau_2$ be as in \cref{lin:prop:prosaic}.
		Then
		\cref{lin:eq:prosaic} holds for some modified $\master^{\mathrm{s}}$, $\inhom^{\mathrm{s}}$.
	\end{prop}
	
	\subsection{Energy fluxes}
		
	Before we begin the analysis of the solution of \cref{lin:eq:main linearised}, we recall some properties of the energy fluxes used in the section.
	
	\begin{lemma}[Equivalence of energies]\label{lin:lemma:equivalence of energies}
		Let $\phi$ be a smooth function,
		and let  $\tau_2\gg1$, $\tau_1\in[\tau_2^\Delta,\tau_2]$.
		Then we have the following equivalence of energies
		\begin{equation}
			\begin{gathered}
				\mathcal{C}^a_{\tau_1,\tau_2}[J^a]\sim_{N}\mathcal{C}^a_{\tau_1,\tau_2}[J^1]\\
				\big(\Sigma^\g_{\tau^\Delta_2}\cap\mathcal{R}^{\mathrm{t}}_{\tau_2}\big)[J^a]\sim_{N}\big(\Sigma^\g_{\tau^\Delta_2}\cap\mathcal{R}^{\mathrm{t}}_{\tau_2}\big)[J^1]\\
				J^a=\sum_{\abs{\alpha}\leq N}T^a\cdot\bar{\T}_a[\Gamma_a^\alpha\phi],\quad
				\Gamma_a\in\{\Omega_a,S_a,L_a\}.
			\end{gathered}
		\end{equation} 
	\end{lemma}
	\begin{proof}
		This is standard and follows from the following observations:
		\begin{itemize}
			\item on $\mathcal{C}^a_{\tau_1,\tau_2}$ and $\Sigma^\g_{\tau^\Delta_2}\cap\mathcal{R}^{\mathrm{t}}_{\tau_2}$, all sets of vector fields $\Gamma_a$ span $\Diff_b$,
			\item on $\mathcal{C}^a_{\tau_1,\tau_2}$ the energy fluxes of $J^a$ control all tangent vectors coercively, and only these.
			\item on $\Sigma^\g_{\tau^\Delta_2}\cap\mathcal{R}^{\mathrm{t}}_{\tau_2}$ the energy fluxes of $J^a$ control all derivatives coercively.
		\end{itemize}
	\end{proof}
	
	\begin{lemma}[Coercivity in exterior region]\label{lin:lemma:coercivity_exterior}
		Let $\phi$ be a smooth function in $\Region$.
		Then, for any non-timelike hypersurface  $\Sigma\subset\Region$ that satisfies
		\begin{equation}
			\Sigma\cap\Big(\bigcup_a\{\abs{\tilde{y}_a}\leq R_0\}\Big)=\emptyset
		\end{equation}
		for $R_0$ sufficiently large
		\begin{equation}
			\Sigma[-T\cdot\bar{\T}_a^0[\phi]]\sim\Sigma[-T\cdot\bar{\T}_a^V[\phi]].
		\end{equation}
	\end{lemma}
	\begin{proof}
		This estimate is standard and follows from the fact that $\bar{\T}$ controls $(\partial,\jpns{y_a}^{-1})\phi$ terms, while the non-coercive potential $V$ falls of like $r^{-5}$.
	\end{proof}

	\subsection{Conservation laws}\label{lin:sec:conservation_laws}	
	We will use $T$ energy estimates to control the solution $\phi$, but as we noted before, this yields a non-coercive quantity around the solitons.
	Indeed, in this section, we will mostly work in the region $\Region^1$ and the result around the other solitons follow similarly.
	
	In order to remedy the non-coercivity, we need to project out the zero modes and the eigenfunctions. 
	We introduce the conservation laws associated to $\ker(\Delta+V)$ and the corresponding coercivity estimates for \cref{lin:eq:main linearised}.
	First, in \cref{lin:sec:basic} we recall the global projectors as were defined in \cite{kadar_scattering_2024}.
	These are sufficient to obtain coercive control using $T\cdot\T^{V}$ currents.
	However, they are not sufficiently conserved for applications in the present paper, due to the slow decay of $\mathfrak{E}^l$.
	We next localise the global projectors to $\Region$ in \cref{lin:sec:localisation} and then correct the conservation laws in \cref{lin:sec:eigenvalues,lin:sec:Theta_error}.
	
	\subsubsection{Basic constructions}\label{lin:sec:basic}

	In this subsection, we study the unperturbed problem
	\begin{equation}\label{lin:eq:local_unperturbed}
		(\Box+V)\phi=f,
	\end{equation}
	with $z^{i,j}_1=c^{i,j}_{\Lambda,1}=0$, i.e. $V=V^{\mathrm{un}}_1$.
	We present why the conservation laws used in \cite{kadar_scattering_2024} are not sufficient in the present context.
	Although the construction will be localised momentarily in \cref{lin:sec:localisation}, we present the discussion semi-globally between two slices $\Sigma^{1}_{\tau_1;\tau_2},\Sigma^{1}_{\tau_2;\tau_2}$ for $\tau_1\in[\tau_2^\Delta,\tau_2]$, which we call $\D$ in this subsection.
	We recall that the non-coercivity of $T$ energy follows from 
	\begin{equation}\label{lin:eq:energy-integral}
		\begin{gathered}
			\E^0_{\Sigma^{1}_{\tau;\tau_2}}[\phi]:=-\Sigma^{1}_{\tau;\tau_2}[T\cdot \T^w[\phi]]=\frac{1}{2}\int_{\Sigma^{1}_{\tau;\tau_2}\simeq\R^3}  \Big((1-(h_{R_2}')^2)(T\phi)^2+(X \phi)^2+\frac{\abs{\slashed{\nabla}\phi}^2}{r^2}-w\phi^2\Big)
		\end{gathered}
	\end{equation}
	being not positive definite.
	
	We recall the following definitions from \cite{kadar_scattering_2024}.
	\begin{definition}[Projections]\label{lin:def:uncorrected_Theta}
		We define the linearised energy momentum tensor
		\begin{equation}\label{lin:eq:linearised_T}
			\begin{gathered}
				\Tlin_{\mu\nu}[\phi]:=2\partial_{(\mu} W\partial_{\nu)}\phi-\eta_{\mu\nu}(\partial W\cdot\partial\phi-\Vlin\phi),\qquad
				\Vlin=W^5.
			\end{gathered}
		\end{equation}
		We define the uncorrected functionals 
		\begin{nalign}\label{lin:eq:tildeTheta_def}
			&\tilde{\Theta}_{\Sigma^{1}_{\tau;\tau_2}}^{m}[\phi]:=\Sigma^{1}_{\tau;\tau_2}[\tilde{J}^m_i],&& \tilde{J}^m_i=(\partial_i)\cdot \Tlin[\phi]\\
			&\tilde{\Theta}^{c}_{\Sigma^{1}_{\tau;\tau_2}}[\phi]:=\Sigma^{1}_{\tau;\tau_2}[\tilde{J}^c_i-\tau_1\tilde{J}^m]-\int_{\Sigma^{1}_{\tau;\tau_2}\cap\scri}\hat{x}_ir\phi,,&&
			\tilde{J}^c[\phi]=\big((t\partial_i+x_i\partial_t)\big)\cdot \Tlin[\phi]\\
			&\tilde{\Theta}^{\Lambda}_{\Sigma^{1}_{\tau;\tau_2}}[\phi]:=\Sigma^1_{\tau;\tau_2}[\tilde{J}^\Lambda],&& \tilde{J}^\Lambda	=(\partial_t)\cdot \T^{V_1}[\phi,t\Lambda W].
		\end{nalign}
		Here, we kept the $i$ dependence of $\tilde{\Theta}^m,\tilde{\Theta}^c$ implicit and we treat them as a 3-tuple.
		Note, that for solutions with no outgoing radiation, the integral along $\scri$ for $\tilde{\Theta}^c$ vanishes.
		All $\tilde{\Theta}^{\bullet}$ can be written by explicit integral of $\phi,T\phi$ on $\Sigma^{1}_{\tau;\tau_2}$ as
		\begin{nalign}\label{lin:eq:tildeTheta_integral_form}
			&\tilde{\Theta}_{\Sigma^{1}_{\tau;\tau_2}}^m[\phi]=-\int_{\Sigma^{1}_{\tau;\tau_2}}h'\Big(X^{\r}_{\star}W(X_{\star}\phi)+\hat{x}\Vlin \phi\Big)+X_{\star}WT\phi(1-h'^2)\\
			&\tilde{\Theta}_{\Sigma^{1}_{\tau;\tau_2}}^c[\phi]=\int_{\Sigma^{1}_{\tau;\tau_2}}\hat{x}_i\bigg(R_2hX^\r_{\star}W T\phi (1-h'^2)
			+\phi\underbrace{\Big(X^\r_{\star}W\frac{R_2hh'-r}{r}+X^\r_{\star}WX^\r_{\star}\frac{R_2hh'}{r}-2R_2hh'\Vlin\Big)}_{\mathcal{O}(r^{-3})}\bigg)\\
			&\tilde{\Theta}_{\Sigma^{1}_{\tau;\tau_2}}^\Lambda[\phi]=\int_{\Sigma^{1}_{\tau;\tau_2}}(1-h'^2)T\phi\Lambda W-X^\r_{\star}\phi \underbrace{X^\r_{\star}(h\Lambda W)}_{\sim r^{-3}}+V\phi h\Lambda W
		\end{nalign}
	\end{definition}
	This follows from the computation in \cref{app:sec:flux calculation}, see also Section 5.1 of \cite{kadar_scattering_2024}.
	Let us introduce the projections onto (un)stable modes:
	\begin{equation}\label{lin:eq:unstable_definition_tilde}
		\tilde{\alpha}_{\pm}[\phi](\tau):=e^{\pm\lamed \tau }\Sigma_{\tau}[T\cdot\T^{V}[\chi_{R_1}(x)e^{\mp t\lamed}Y(x),\phi]],
	\end{equation}
	for $R_1=c_1\log\tau$
	We compute these explicitly below
	\begin{lemma}
		Let $\phi$ be a smooth function in $\D$ solving \cref{lin:eq:local_unperturbed}. 
		For $\tau$ sufficiently large (depending only on $\delta_2$), we can compute
		\begin{equation}\label{eq:linear:unstable_derivative}
			\tilde{\alpha}_{\pm}[\phi]=\frac{-\lamed}{2}\int_{\Sigma_\tau}\chi^c_{R_1}(y_a)Y(y_a)(\lamed\pm T_a)\phi+\B^{\pm}_{s},\quad
			\partial_\tau \tilde{\alpha}_{\pm}[\phi]=\pm\tilde{\alpha}_{\pm}[\phi]+\B_{b}^\pm
		\end{equation}
		where the error terms for $\chi=\chi_{R_1/d}(x)\chi^c_{dR_1}(x)$ are given by
		\begin{equation}
			\abs{\B_s^\pm}\lesssim\int_{\Sigma_\tau}e^{-R_1\lamed/d}\chi\phi,\quad \abs{\B_b^\pm}=\int_{\Sigma_\tau}\chi e^{-\lamed R_1/d}\cdot(\phi,\partial\phi)+\mathcal{O}^{\infty,0}_1\cdot f
		\end{equation}
	\end{lemma}
	\begin{proof}
		We start with \cref{eq:linear:unstable_derivative} and only compute $\alpha_{+}$, with the opposite sign following trivial changes.
		We write $g=\chi_{R_1}(x)e^{-t_a\lamed}Y(x)$ and compute
		\begin{nalign}
			\int_{\Sigma_\tau}\T^{V_a}[g,\phi](T,T)=\int_{\Sigma_\tau}\frac{1}{2}\Big(TgT\phi+\nabla g\cdot\nabla \phi-Vg\phi\Big)=\int_{\Sigma_\tau}\frac{1}{2}\Big(TgT\phi-\phi(\Delta+V) g\Big)\\
			=\int_{\Sigma_\tau}\frac{1}{2}\Big(- \lamed gT\phi-\phi\lamed^2 g\Big)+e^{-\lamed\tau}e^{-R_1\lamed/d}\O_1^{0,\infty}\cdot(\phi).
		\end{nalign}
		Multiplying by $e^{\lamed t}$ yields the result.
		
		Next, let us prove \cref{eq:linear:unstable_derivative}.
		We use the divergence theorem, to compute
		\begin{nalign}
			\big[\Sigma_\tau[T_1\cdot\T^{V}[g,\phi]]\big]^{\tau_2}_{\tau_1}=\int_{R_{\tau_1,\tau_2}}T_1 g(\Box+V)\phi+T_1\phi(\Box+V)g\\
			=e^{-\lambda\tau}\int_{R_{\tau_1,\tau_2}}\chi_{R_1/3}^c(y_a)e^{-\lamed R_1}\mathcal{O}^{\infty,0}_1\cdot(\phi,\partial_t\phi)+\mathcal{O}^{\infty,0}_1\cdot f.
		\end{nalign}
		Taking a limit in $\tau_1\to\tau_2$ yields that $\partial_\tau e^{-\lamed\tau}\tilde{\alpha}_{+}$ is bounded by $e^{-\lamed\tau}\mathcal{B}^+_b$.
		Multiplying by the exponential factor yields the result.
	\end{proof}
	
	The importance of $\alpha_\pm$ projectors is that they control the $L^2$ projection of the eigenvalue $Y$.
	Indeed, we have
	\begin{equation}\label{lin:eq:alpha_tilde_Y_control}
		\abs{\Sigma_\tau[T\cdot\T[\phi,Y]]}\lesssim \abs{\alpha_\pm[\phi]}+e^{-R_1/d}(\E^0[\phi])^{1/2}.
	\end{equation}
	
	To show why $\tilde{\Theta}^\bullet$ are useful quantities we recall the coercivity estimate:
	\begin{lemma}[\cite{duyckaerts_solutions_2016} Proposition 3.6]\label{lin:lemma:coercivity_cited}
		For $E_i\in \mathcal{C}^\infty_0(\R^4)$ with 
		\begin{equation}\label{key}
			\begin{gathered}
				(E_i,V_j)_{L^2}=\delta_{ij}\qquad (E_i,Y)=0,\qquad \text{ for } V_j\in\ker \Delta+V
			\end{gathered}
		\end{equation}
		the following holds. There exists $\mu>0$ depending on $E_i$  such that 
		\begin{equation}
			\begin{gathered}
				(-(\Delta+V)f,f)+\frac{1}{\mu}\Big(\sum (E_i,f)_{L^2}^2+(Y,f)_{L^2}^2\Big)\geq \mu \norm{f}_{\dot{H}^1}^2.
			\end{gathered}
		\end{equation}
	\end{lemma}
	
	We need to modify this to our global projectors.
	\begin{lemma}[Coercivity]\label{lin:lemma:tilde_coercivity}
		For a smooth function in $\D$ we have
		\begin{equation}
			\abs{\tilde{\Theta}^{c}[\phi]}^2,\abs{\tilde{\Theta}^{\Lambda}[\phi]}^2\lesssim R_2 \E^0[\phi],\quad \abs{\tilde{\Theta}^m[\phi]}^2\lesssim\E^0[\phi]
		\end{equation}
		and moreover they satisfy for $C_c,C_m,C_\Lambda\neq0$ the following:
		\begin{subequations}\label{lin:eq:Thete_evaluations}
			\begin{gather}
				\tilde{\Theta}^{c}[\partial_i W]=\delta_{ij}C_c,\quad\tilde{\Theta}^\Lambda[\Lambda W]=C_\Lambda,\quad\tilde{\Theta}^m[t\partial_i W]=\delta_{ij}C_m\label{lin:eq:Theta_tilde_explicit}\\
				\tilde{\Theta}^c[\Lambda W]=\tilde{\Theta}^\Lambda[\partial_i W]=\tilde{\Theta}^\Lambda[t\partial_i W]=0\label{lin:eq:Theta_tilde_explicit0}.
			\end{gather}
		\end{subequations}
		Furthermore, there exists a $C^{Y}>0$ constant such that
		\begin{nalign}\label{lin:eq:coerciveity_global}
			\bar{\E}^0\lesssim \bar{\E}^V+C^{Y}\abs{\tilde{\alpha}_\pm}^2+e^{-\lamed R_1/d}\bar{\E}^0[\phi]+\sum_{\bullet\in\{c,\Lambda\}}\abs{\tilde{\Theta}^\bullet}^2.
		\end{nalign}
	\end{lemma}
	\begin{proof}
		The fact that $\tilde{\Theta}$ are well defined and the explicit computations \cref{lin:eq:Thete_evaluations} are in section 5.1 of \cite{kadar_scattering_2024}.
		
		The estimate \cref{lin:eq:coerciveity_global} was already used in \cite{kadar_scattering_2024} Corollary 3.3.2.
		We include a proof for sake of completeness.
		
		\emph{Step 1: semi-definite.} 
		
		We first prove that for $\Sigma_\tau[T\cdot\T[\phi,Y]]=0$ we have $\bar{\E}^V\geq0$.
		Using the explicit form \cref{lin:eq:energy-integral}, it already follows from \cref{lin:lemma:coercivity_cited} that $\E^V\geq0$.
		This is a consequence of the fact that if positive were not to hold, there would be a positive eigenvalue of $(\Delta+V)$ orthogonal to $Y$.
		Similarly, we also have $\tilde{\E}^V\geq0$.
				
		\emph{Step 2: definite.} Next we prove that $\E^V[\phi]>0$ whenever $\Sigma_\tau[T\cdot\T[\phi,Y]]=\tilde{\Theta}^c=\tilde{\Theta}^\Lambda=0$.
		Assume $\E^V[\phi]=0$.
		Then, from $\E^V\geq0$ it follows that $\Sigma^1_{\tau;\tau_2}[T\cdot\T^V[\phi,\psi]]=0$ for all $\psi$.
		Hence $(\Delta+V)\phi=0$.
		Using \cref{lin:eq:Theta_tilde_explicit}, we get a contradiction.
		Exactly the same way $\bar{\E}^V[\phi]>0$.
		
		\emph{Step 3: coercivity} 
		Now, a standard compactness argument yields that for $\Sigma_\tau[T\cdot\T[\phi,Y]]=\tilde{\Theta}^c=\tilde{\Theta}^\Lambda=0$ we have $\bar{\E}^V[\phi]\gtrsim \E^0[\phi]$.
		Indeed, assuming the contrary, we can have a sequence $\phi_n$ such that $\bar{\E}^V[\phi-n]\leq  \bar{\E}^0[\phi_n]/n=1/n$ and $\Sigma_\tau[T\cdot\T[\phi_n,Y]]=\tilde{\Theta}^c[\phi_n]=\tilde{\Theta}^\Lambda[\phi_n]=0$.
		Passing to a subsequence we obtain the existence of $\phi$ with $\Sigma_\tau[T\cdot\T[\phi,Y]]=\tilde{\Theta}^c[\phi]=\tilde{\Theta}^\Lambda[\phi]=0$ and $\bar{\E}^V[\phi]=0$.
		A contradiction.
		
		\emph{Step 4: shift} Consider $\Sigma_\tau[T\cdot\T[\phi,Y]],\tilde{\Theta}[\phi]\neq0$.
		We set $\psi=\phi+e^{\mp\tau\lamed}a_\pm Y e^{\pm t\lamed_1}+b\Lambda W+c\cdot\partial W$ so that  $\alpha_\pm[\psi],\tilde{\Theta}[\psi]=0$.
		We know that $\abs{a_\pm}\lesssim\abs{\alpha_\pm}$ and $\abs{b,c}\lesssim\abs{\tilde{\Theta}^c[\phi],\tilde{\Theta}^\Lambda[\phi]}$.
		Using the previous step, we also get
		\begin{nalign}
			\E^V[\psi]\gtrsim \E^0[\psi]\gtrsim\E^0[\phi]-a_\pm^2-(\abs{\tilde{\Theta}^c[\phi]}^2+\abs{\tilde{\Theta}^\Lambda[\phi]}^2).
		\end{nalign}
		Summing these, we obtain \cref{lin:eq:coerciveity_global}.
	\end{proof}
	
	However, for applications, we need $\tilde{\Theta}$ to be approximately conserved.
	The perturbations $\mathfrak{E}^l$ are an obstruction to this.
	We overcome these obstructions in \cref{lin:sec:Theta_error}.
	For now, let us record the conservation laws satisfied by $\tilde{\Theta}$:
	
		\begin{lemma}
		Let $\phi$ be a scattering solution to \cref{lin:eq:local_unperturbed} in $\D$ with no outgoing radiation.
		Then for $\tau^\Delta_2\leq\tau\leq\tau'\leq\tau_2$
		\begin{subequations}
			\begin{align}
				\tilde{\Theta}^{m}_{\Sigma^{\ell}_{\tau;\tau_2}}&=\tilde{\Theta}^{m}_{\Sigma^{\ell}_{\tau';\tau_2}}+\mathcal{B}^{m}\\
				\tilde{\Theta}^{c}_{\Sigma^{\ell}_{\tau;\tau_2}}&=\tilde{\Theta}^{c}_{\Sigma^{\ell}_{\tau';\tau_2}}+(\tau_2-\tau_1)\tilde{\Theta}^{m}_{\tau_2}
				+\mathcal{B}^{c}\label{lin:eq:tildeTheta_conservation_c}\\
				\tilde{\Theta}^{\Lambda}_{\Sigma^{\ell}_{\tau;\tau_2}}&=\tilde{\Theta}^{\Lambda}_{\Sigma^{\ell}_{\tau';\tau_2}}+\mathcal{B}^{\Lambda}\label{lin:eq:tildeTheta_conservation_L}
			\end{align}
		\end{subequations}
		where
		\begin{equation}
			\begin{gathered}
				\mathcal{B}^{m}=\int_{\D}f\O_1^{2,0},\quad
				\mathcal{B}^{c}=\int_{\D}f\O_1^{1,-1},\quad
				\mathcal{B}^{c}=\int_{\D}f\O_1^{1,0}
			\end{gathered}
		\end{equation}
	\end{lemma}
	\begin{proof}
		These estimates are contained in \cite{kadar_scattering_2024} section 5, and follow from an application of the divergence theorem.
	\end{proof}

	In the case of non-zero outgoing radiation, the conservation laws for the centre of mass and the scaling, \cref{lin:eq:tildeTheta_conservation_c,lin:eq:tildeTheta_conservation_L} respectively, obtain extra contributions
	\begin{equation}
		\int_{\scri\cap\D}\hat{x}_i\partial_u r\phi|_{\scri},\quad (\scri\cap\D)[\tilde{J}^\Lambda]= \int_{\scri\cap\D}\partial_u r\phi|_{\scri}.
	\end{equation}
	The first can also be written as a current $(\scri\cap\D)[T\cdot\T^0[\phi,r\hat{x}_i]]$.
	The square of these contributions are \emph{not} bounded by the outgoing energy, only by $(\tau_2-\tau_1)\E^0_{\scri\cap\D}[\phi]$.
	
	We highlight, that if we include $\mathfrak{E}^l \phi$ in the inhomogeneity $f$, the error terms in $\B^c,\B^m,\B^\Lambda$ are not controllable with the rudimentary energy estimates that we use in the present context.
	To remedy this, we instead, upgrade our conservations laws to weaken the error terms in $\B^\bullet$.
	We define the modified $\Theta$ in \cref{lin:sec:Theta_error}.
	
	\subsubsection{Localisation}\label{lin:sec:localisation}
	We need to prove energy boundedness statement in $\Region^1$ region, which is located strictly away from $\scri$.
	In order to achieve this, we need to localise the constructions from \cref{lin:sec:basic}.
	We work on a single slice $\Sigma^{1}_{\tau;\tau_2}$, and derive a localised version of \cref{lin:lemma:tilde_coercivity}.
	
	We start, by relating the $\E^V$ energy of functions that are shifted by the kernel elements $\Lambda W, \partial_iW$.
	\begin{lemma}\label{lin:lemma:EV_shift}
		For $\psi\in\{\Lambda W,\partial_i W\}$ it holds that
		\begin{equation}\label{lin:eq:EV_shift}
			\abs{\E^V_{\Sigma^{1,\delta_3}_{\tau;\tau_2}}[\phi+c\psi]-\E^V_{\Sigma^{1,\delta_3}_{\tau;\tau_2}}[\phi]}\lesssim c^2\tau^{-1}+c\tau^{-1/2}(\E^0_{\Sigma^{1,\delta_3}_{\tau;\tau_2}\setminus\Sigma^{1,\delta_3/2}_{\tau;\tau_2}}[\phi])^{1/2}
		\end{equation}
	\end{lemma}
	\begin{proof}
		We prove the estimate for $\Lambda W$, the $\partial_i W$ follows similarly.
		First of all, we notice that $\E^V$ is a quadratic form, so we can write it as a sum of 3 terms corresponding to $(\phi,\phi)$, $(\phi,\Lambda W)$ and $(\Lambda W,\Lambda W)$ arguments.
		We first estimate
		\begin{nalign}
			\abs{\E_{\Sigma^{1,\delta_3}_{\tau;\tau_2}}^V[\Lambda W]}=\bigg|\underbrace{\E_{\Sigma_{\tau;\tau_2}}^V[\Lambda W]}_{=0}-\E_{\Sigma_{\tau;\tau_2}\setminus\Sigma^{1,\delta_3}_{\tau;\tau_2}}^V[\Lambda W]\bigg|\lesssim\tau^{-1}.
		\end{nalign}
		Next, we use a cutoff function $\chi$ localising to $\Sigma^{1,\delta_3}_{\tau;\tau_2}$ to write
		\begin{nalign}
			\Sigma^{1,\delta_3}_{\tau;\tau_2}[-T\cdot\T^V[\Lambda W,\phi]]\lesssim\underbrace{\Sigma^{1,\delta_3/2}_{\tau;\tau_2}[-T\cdot\T^V[\Lambda W,\chi \phi]]}_{=0}+\tau^{-1/2}(\E^0_{\Sigma^{1,\delta_3}_{\tau;\tau_2}\setminus\Sigma^{1,\delta_3/2}_{\tau;\tau_2}}[\phi])^{1/2}.
		\end{nalign}
	\end{proof}
	
	Next, we modify \cref{lin:lemma:tilde_coercivity} to include an extra weight for the kernel elements
	\begin{lemma}[Localising estimate]\label{lin:lemma:localisation}
		a) Fix $\delta_3>\delta>0$. For a smooth $\phi$ in $\D$ we have
		\begin{equation}\label{lin:eq:coercivity_degenerate1}
			\bar{\E}^0_{\Sigma^{}_{\tau;\tau_2}\setminus\Sigma^{\delta}_{\tau;\tau_2}}\lesssim_\delta\bar{\E}^V_{\Sigma^{}_{\tau;\tau_2}}+C^{Y}\abs{\tilde{\alpha}_\pm}^2+e^{-\lamed R_1/d}\bar{\E}^0_{\Sigma^{}_{\tau;\tau_2}}+\frac{1}{\tau}\sum_{\bullet\in\{c,\Lambda\}}\abs{\tilde{\Theta}^\bullet}^2.
		\end{equation}
	
		b)
		There exists $R$ sufficiently large, such that for all $R_3>\max(R_2,R)$ it holds that
		\begin{nalign}\label{linear:eq:coerciveity_local_theta}
			\bar{\E}^0_{\Sigma^{}_{\tau;\tau_2}}\lesssim\bar{\E}^V_{\Sigma^{}_{\tau;\tau_2}}+C^{Y}\abs{\tilde{\alpha}_\pm}^2+e^{-\lamed R_1/d}\bar{\E}^0_{\Sigma^{}_{\tau;\tau_2}}+R_2\sum_{\bullet\in\{c,\Lambda\}}\abs{\tilde{\Theta}_{\leq R_3}^\bullet}^2.
		\end{nalign}
		where $\tilde{\Theta}^\bullet_{\leq R_3}$ denotes the integrals evaluated only in the region $r\leq R_3$.
		
		c) For $\delta_3>\delta>0$ and $R_1/d>10\log\tau$ we have 
			\begin{equation}
				\bar{\E}^0_{\Sigma^{1,\delta_3}_{\tau;\tau_2}\setminus \Sigma^{\delta}_{\tau;\tau_2}}+\frac{1}{R_2\tau}\bar{\E}^0_{\Sigma^{1,\delta_3}_{\tau;\tau_2}}\lesssim_\delta\bar{\E}_{\Sigma^{1,\delta_3}_{\tau;\tau_2}}^V+\bar{\E}^0_{\Sigma^{1,\delta_3}_{\tau;\tau_2}\setminus\Sigma^{1,\delta_3/2}_{\tau;\tau_2}}+C^{Y}\abs{\tilde{\alpha}_{\pm}}^2+\frac{1}{\tau}\sum_{\bullet\in\{c,\Lambda\}}\abs{\tilde{\Theta}_{\Sigma^{1,\delta_3}_{\tau;\tau_2}}^\bullet}^2.
			\end{equation}
	\end{lemma}
	\begin{proof}
		a)
		We assume $\tilde{\Theta}^c[\phi]=\alpha_\pm[\phi]=0$, its inclusion follows similarly.
		Let us write $\psi=\phi+a\bar{\chi}^c(\delta\tau)(y_1)\Lambda W$ such that $\tilde{\Theta}[\psi]=0$.
		Using that $\Lambda W\sim 1/r$ and $\E^V[\Lambda W]=0$ we get that
		\begin{equation}
			\abs{\bar{\E}_{\Sigma^{}_{\tau;\tau_2}}^V[\psi]-\bar{\E}_{\Sigma^{}_{\tau;\tau_2}}^V[\phi]}\lesssim a(\bar{\E}_{\Sigma^{}_{\tau;\tau_2}}^0[\psi])^{1/2}\tau^{-1/2}+a^2\tau^{-1}.
		\end{equation}
		Using the coercivity \cref{lin:lemma:tilde_coercivity}, we also have
		\begin{equation}
			\bar{\E}_{\Sigma^{}_{\tau;\tau_2}}^V[\psi]\gtrsim\bar{\E}_{\Sigma^{}_{\tau;\tau_2}}^0[\psi]\gtrsim\bar{\E}_{{\Sigma^{}_{\tau;\tau_2}}\setminus {\Sigma^{\delta}_{\tau;\tau_2}}}^0[\psi]=\bar{\E}_{\Sigma^{}_{\tau;\tau_2}\setminus{\Sigma^{\delta}_{\tau;\tau_2}}}^0[\phi].
		\end{equation}
		Combining the above two yields the estimate.
		
		b) 	We expand \cref{lin:eq:coerciveity_global} and use that $\Lambda\sim 1/r$ (similar to \cref{lin:lemma:EV_shift}) to conclude
		\begin{nalign}
			\abs{\tilde{\Theta}^\bullet-\tilde{\Theta}^\bullet_{\leq R_3}}^2\leq\frac{1}{R_3}\bar{\E}_{\Sigma^{}_{\tau;\tau_2}}^0[\phi].
		\end{nalign}
		Taking $R_3$ sufficiently large, we absorb this error term.
		
		c) We first cut off $\tilde{\alpha}$ similarly as for $\tilde{\Theta}$, but since the cutoff is localised to the region $\geq \delta\tau$, we can use the control provided by $\bar{\E}^0_{\Sigma^{}_{\tau;\tau_2}\setminus {\Sigma^{\delta}_{\tau;\tau_2}}}$.
		For the localisation in $\E$, we use a simple cutoff function.
	\end{proof}
	
	For most applications the previous lemma suffices.
	However, in \cref{lin:eq:proof_lin_inhom1}, we will need the following refinement
	\begin{lemma}\label{lin:lemma:localisation_improved}
		Let $\phi$ be as in \cref{lin:lemma:localisation}.
		Then
		\begin{equation}
			\int_{\Sigma_{\tau;\tau_2}\setminus\Sigma^{\delta}_{\tau_1,\tau_2}}\big( (X_\star\phi)^2+\phi^2\jpns{y_a^{\tau_2}}^{-2}\big)\frac{\jpns{y_a^{\tau_2}}}{\tau}\lesssim_\delta \log\tau\Big(\bar{\E}^V_{\Sigma^{}_{\tau;\tau_2}}+C^{Y}\abs{\tilde{\alpha}_\pm}^2+e^{-\lamed R_1/d}\bar{\E}^0_{\Sigma^{}_{\tau;\tau_2}}+\frac{1}{\tau}\sum_{\bullet\in\{c,\Lambda\}}\abs{\tilde{\Theta}^\bullet}^2\Big)
		\end{equation}
	\end{lemma}
		\begin{proof}
			We note that keeping the $R=\delta\tau$ dependence explicitly in \cref{lin:eq:coercivity_degenerate1} we have for $R>R_2$
			\begin{equation}
				\bar{\E}^0_{\Sigma^{}_{\tau;\tau_2}\setminus\{\abs{y_1^{\tau_2}}>R\}}\lesssim_\delta\bar{\E}^V_{\Sigma^{}_{\tau;\tau_2}}+C^{Y}\abs{\tilde{\alpha}_\pm}^2+e^{-\lamed R_1/d}\bar{\E}^0_{\Sigma^{}_{\tau;\tau_2}}+\frac{1}{R}\sum_{\bullet\in\{c,\Lambda\}}\abs{\tilde{\Theta}^\bullet}^2.
			\end{equation}
			Using $R=2^n$ and summing the estimates we get
			\begin{equation}
				\sum_{n\leq \log\tau}\E^0_{\Sigma_{\tau;\tau_2}\setminus\{\abs{y_1^{\tau_2}}<2^n\}}\frac{2^n}{\tau}\lesssim \log\tau\Big(\bar{\E}^V_{\Sigma^{}_{\tau;\tau_2}}+C^{Y}\abs{\tilde{\alpha}_\pm}^2+e^{-\lamed R_1/d}\bar{\E}^0_{\Sigma^{}_{\tau;\tau_2}}+\frac{1}{\tau}\sum_{\bullet\in\{c,\Lambda\}}\abs{\tilde{\Theta}^\bullet}^2\Big).
			\end{equation}
			This is precisely the claimed estimate.
		\end{proof}

	\subsection{Eigenvalues}\label{lin:sec:eigenvalues}
	We move back to studying solutions to the perturbed equation \cref{lin:eq:main linearised}.
	We start, by studying the modifications to $\tilde{\alpha}$.
	As this will be localised to the region $\mathcal{R}^{a}_{\tau_1,\tau_2}$, we only give the computations around the first soliton, the others following similarly.
	
	Let us introduce the projections onto (un)stable modes. 
	\begin{equation}\label{lin:eq:unstable_definition}
		\alpha_{a,\pm}[\phi](\tau):=e^{\pm\lamed_a\tau/\gamma_a }\Sigma_{\tau}[{T}_a\cdot\T^{V_a}[\chi^c_{s_1\log t_a}(\tilde{y}_a)e^{\mp t_a\lamed_a}Y(\tilde{y}_a),\phi]]
	\end{equation}
	We compute these explicitly below
	
	\begin{lemma}\label{lin:lemma:unstable_computations}
		Let $\phi$ be a smooth solution to \cref{lin:eq:main linearised} with \cref{lin:eq:assumptionPhibar_weak}.
		For $\tau$ sufficiently large (depending only on $\delta_1$), $\chi=\chi_{s_1\log t_ad}(\tilde{y}_a)\chi^c_{s_1\log t_a/d}(\tilde{y}_a)$ and $\tilde{\Sigma}_\tau=\Sigma_\tau\cap\{\abs{y^{\tau_2}_a}<dR_1\}$, we can compute 
		\begin{subequations}
			\begin{align}
				&{\alpha}_{a,\pm}[\phi]=\frac{-\lamed_a/\gamma_a}{2}\int_{\tilde{\Sigma}_\tau}\chi^c_{s_1\log t_a}(y_a)Y(y_a)(\lamed_a\pm T_a)\phi+\O_a^{\infty,1}\cdot(\phi,\partial\phi)+ t_a^{-s_1\lamed_a/d}\cdot(\phi,\partial),\label{lin:eq:unstable_int}\\
				&\partial_\tau \alpha_{a,\pm}[\phi]=\pm\lamed_a\gamma_a^{-1} \alpha_{a,\pm}+\int_{\tilde{\Sigma}_\tau}\chi t_a^{-s_1\lamed_a/d}\mathcal{O}^{\infty,0}_a\cdot(\phi,\partial_t\phi)+\mathcal{O}^{\infty,1}_a\cdot(\phi,\partial_t\phi,f)\label{lin:eq:unstable_der}
			\end{align}
		\end{subequations}
	\end{lemma}
	\begin{proof}
		We prove the estimates for $a=1$ as $a\neq1$ follows by using that in a $\mathcal{R}_a\cap\{\abs{y_a}/t\leq \delta \}$, the foliation $\Sigma_\tau$ is simply $\gamma_at_a$ constant hypersurfaces.
		
		We start with \cref{lin:eq:unstable_int} and only compute ${\alpha}_{1,+}$, with the opposite sign following trivial changes.
		We write $g=\chi^c_{s_1t_a}(\tilde{y}_a)e^{-t_a\lamed}Y(\tilde{y}_a)$ and compute
		\begin{nalign}
			\int_{\tilde{\Sigma}_\tau}\T^{V_1}[g,\phi](T,T)=\int_{\tilde{\Sigma}_\tau}\frac{1}{2}\Big(TgT\phi+\nabla h\cdot\nabla \phi-V_1g\phi\Big)=\int_{\tilde{\Sigma}_\tau}\frac{1}{2}\Big(TgT\phi-\phi(\Delta+V_1) g\Big)\\
			=\int_{\tilde{\Sigma}_\tau}\frac{1}{2}\Big(- \lamed gT\phi-\phi\lamed^2 g\Big)+e^{-\lamed\tau}\O^{\infty,1}_1\cdot(T\phi)+e^{-\lamed\tau}t_a^{-s_1\lamed_1/d}\O_1^{\infty,0}\cdot(\phi).
		\end{nalign}
		Multiplying by $e^{\lamed t}$ yields the result.
		
		Next, we prove \cref{lin:eq:unstable_der}.
		We use the divergence theorem, to compute
		\begin{nalign}
			\big[\Sigma_\tau[T_1\cdot\T^{V_1}[h,\phi]]\big]^{\tau_2}_{\tau_1}=\int_{R_{\tau_1,\tau_2}\cap\{\abs{y^{\tau_2}_1}<dR_1\}}T_1 h(\Box+V_1)\phi+T_1\phi(\Box+V_1)h+e^{-\lamed\tau}\mathcal{O}^{\infty,1}_1\cdot\partial\phi\\
			=e^{-\lamed\tau}\int_{R_{\tau_1,\tau_2}\cap\{\abs{y^{\tau_2}_1}<dR_1\}}t^{-s_1 \lamed/d }\mathcal{O}^{\infty,0}_1\cdot(\phi,\partial\phi)+\mathcal{O}^{\infty,1}_1\cdot(\phi,\partial_t\phi,f).
		\end{nalign}
		Taking a limit in $\tau_1\to\tau_2$ yields that $\partial_\tau e^{-\lamed\tau}{\alpha}_{1,+}$ is bounded by the integral term on the right hand side of \cref{lin:eq:unstable_der}.
		Multiplying by the exponential factor yields the result.
	\end{proof}
	
	Next, we study the evolution of $\alpha^a_+[\phi]$.
	Since it decays exponentially towards the past, up to error terms, the estimate will be straightforward.
	\begin{lemma}\label{lin:lemma:stable_mode_estimate}
		Let $\phi$ be as in \cref{lin:lemma:unstable_computations}.
		Then for some implicit constant we have
		\begin{nalign}
			\abs{\alpha_+^a[\phi](\tau_1)}\lesssim\abs{\alpha_+^a[\phi](\tau_2)}+ \sup_{\tau\in(\tau_1,\tau_2)} \int_{\Sigma_\tau} t_a^{-s_1\lamed_a/d}\mathcal{O}^{\infty,0}_a\cdot(\phi,\partial_t\phi)+\mathcal{O}^{\infty,1}_a\cdot(\phi,\partial_t\phi,f)
		\end{nalign}
	\end{lemma}
	\begin{proof}
		This follows from \cref{lin:eq:unstable_definition} using the integrating factor $e^{-\tau\lamed_a\gamma_a^{-1}}$ .
	\end{proof}
	
	We introduce the following unstable bootstrap assumptions
	\begin{subequations}
		\begin{align}
			\abs{\alpha^a_-[{T^k\phi}]}&\leq e^{-\lamed R_1/2}\abs{\tau}^{-k}\epsilon\label{lin:eq:higher unstable bootstrap},\\
			\forall j\leq k\qquad\abs{\alpha^a_-[{T^j\phi}]}&\leq e^{-\lamed R_1/2}\abs{\tau}^{-j}\epsilon\label{lin:eq:higher unstable bootstrap_all_k}
		\end{align}
	\end{subequations}
	We obtain control of the lower order unstable modes \cref{lin:eq:higher unstable bootstrap_all_k}, provided that we control the energy of $\phi$ and \cref{lin:eq:higher unstable bootstrap}.
	\begin{lemma}\label{lin:lemma:unstable_recovery}
		Let $\phi$ and $\tilde{\Sigma}_\tau$ be as in \cref{lin:lemma:unstable_computations}.
		Then for $j<k$ we have 
		\begin{nalign}
			\abs{\alpha^a_-[T^j\phi]}\lesssim \abs{\alpha^a_-[T^k\phi]}+\sum_{l=j}^{k-1}\int_{\tilde{\Sigma}_\tau}\O_a^{\infty,1}\big((1,\partial)T^l\phi+T^lf\big)+t_a^{-c_1\lamed /d} T^l\phi
		\end{nalign}
	\end{lemma}
	\begin{proof}
		We prove the result for $a=1$.
		Using \cref{lin:eq:unstable_int}, we have
		\begin{nalign}
			\alpha^1_-[T^{j+1}\phi]=\partial_\tau\alpha^1_-[T^{j}\phi]+\int_{\tilde{\Sigma}_\tau}\chi\O_1^{\infty,1}\big((1,\partial)T^j\phi+T^jf\big)\\
			=-\lamed \alpha^1_-[T^{j}\phi]+\int_{\tilde{\Sigma}_\tau}\O_a^{\infty,1}\big((1,\partial)T^j\phi+T^jf\big)+t_1^{-c_1\lamed /d}\chi T^j\phi.
		\end{nalign}
		The lemma follows by induction.
	\end{proof}
	\subsection{Kernel element}\label{lin:sec:Theta_error}
	The main reason in \cite{kadar_scattering_2024} to introduce the projection operators $\tilde{\Theta}$ is that they are connected to conserved quantities of \cref{lin:eq:main linearised} when $\mathfrak{E}^l,f,z_1^{i,j}=0$.
	All of these quantities being nonzero introduces error terms in the conservation laws.
	We turn our attention to these.
	
	\paragraph{Coercivity}
		\begin{definition}\label{lin:def:corrected_Theta}
		Let $\bar{\phi}$ be from \cref{lin:def:admissible} and introduce the linearised energy momentum tensor
		\begin{equation}
			\Tmod_{\mu\nu}[\phi]:=2\partial_{(\mu}\bar{\phi}\partial_{\nu)}\phi-\eta_{\mu\nu}(\partial\bar{\phi}\cdot\partial\phi-\bar{\phi}^5\phi)=\T^{\bar\phi^4}[\phi,\bar{\phi}].
		\end{equation}
		Let us define corrected functionals as
		\begin{subequations}\label{lin:eq:Theta_def}
			\begin{align}
				&\Theta_{\Sigma^{1,\delta_3}_{\tau_1,\tau_2}}^{a,m}[\phi]:={\Sigma^{1,\delta_3}_{\tau_1,\tau_2}}[J^m],&& J^m[\phi]:=(X^a)\cdot \Tmod[\phi]\\
				&{\Theta}^{a,c}_{\Sigma^{1,\delta_3}_{\tau_1,\tau_2}}[\phi]:={\Sigma^{1,\delta_3}_{\tau_1,\tau_2}}[J^c_i-\tau_1J^m]-\int_{\partial{\Sigma^{1,\delta_3}_{\tau_1,\tau_2}}}\hat{x}_i\phi,&&
				J^c[\phi]:=\big((t_aX^a+y_aT^a)\big)\cdot \Tmod[\phi]\label{lin:eq:Theta_def_c}\\
				&\Theta_{\Sigma^{1,\delta_3}_{\tau_1,\tau_2}}^{a,\Lambda}[\phi]:={\Sigma^{1,\delta_3}_{\tau_1,\tau_2}}[J^\Lambda],&& J^\Lambda[\phi]:=\tilde{T}^{a,\mathrm{c}}\cdot \T^{V_a}[\phi,t\Lambda W_a],
			\end{align}
		\end{subequations}
		where we used $\tilde{T}^{a,\mathrm{c}}=\partial_{t_a}|_{y^\mathrm{c}_a}$ with $y^{\mathrm{c}}_a$ defined in \cref{not:eq:modulated_solitons}.
	\end{definition}

	For the rest of this section, we drop the label $a$ specifying the soliton and work only in $\Region^1$.
	The result extend to all regions.
	
	We observe, that these corrected projections are just perturbations of the original ones.
	However, we must take into consideration that the soliton is constantly moving, so we introduce
	\begin{equation}\label{lin:eq:shifted_tilde_Theta}
		\tilde{\Theta}^{\bullet;\tau_2}_{\Sigma^{1,\delta_3}_{\tau_1,\tau_2}}[\phi]=\tilde{\Theta}^{\bullet}_{\Sigma^{1,\delta_3}_{\tau_1,\tau_2}+\hat{y}_a\cdot z_a^{1,0}\log\tau_1}[\phi(\cdot+z_a^{1,0}\log\tau_1)]
	\end{equation}
	where the second projection is evaluated on a shifted hypersurface.
	Equivalently, we could have moved the soliton in the definition of $\tilde{\Theta}$.
	Notice, that $\bar{\phi}$ do not shift the solitons on the hyperboloidal foliation associated to $\Sigma_{\tau_1,\tau_2}$.
	However, locally, we can approximate their location
	\begin{nalign}\label{lin:eq:Wt-Wt_star}
		W_a=W^{\mathrm{un}}_a(y_a-z_a^{1,0}\log t_a)=W^{\mathrm{un}}_a(y_a-z_a^{1,0}\log( t_{a,\star}^{\tau_2}-R_2h^{\tau_2}_a))\\
		=W^{\mathrm{un}}_a(y_a-z^{1,0}_a\log t_{a,\star}^{\tau_2})\quad\mod\O_1^{2,1}.
	\end{nalign}
	We will call this new approximate $W_a^\star=\sigma_aW^{\lambda_a}(y_a-z^{1,0}_a\log t_{a,\star}^{\tau_2})$.
	
	\begin{lemma}\label{lin:lemma:Theta_differences}
		Fix $\tau_1\in[\tau_2^\Delta,\tau_2]$ and let $h=h^{\tau_2}_1(y_{1,\tau_2})$.
		For a smooth function $\phi$ in $\Region^1$ and $\bar{\phi}$ satisfying \cref{lin:eq:assumptionPhibar_weak}, we have
		\begin{subequations}\label{lin:eq:Theta_integral_form}
			\begin{multline}\label{lin:eq:Theta_integral_form_m}
				\Theta_{\Sigma^{{1,\delta_3}}_{\tau_1,\tau_2}}^{1,m}[\phi]=\tilde{\Theta}_{\Sigma^{1,\delta_3}_{\tau_1,\tau_2}}^{m;\tau_1}[\phi]-\int_{\Sigma^{1,\delta_3}_{\tau_1;\tau_2}}X_\star \phi TW_1^\star\cdot\big(1-(h')^2\big)+(1-h'^2)T\phi TW_1^\star
				\\
				+\int_{\Sigma^{1,\delta_3}_{\tau_1;\tau_2}}(X\phi)(X\O^{2,1})+\phi\O^{6,1}+T\phi\O^{2,1}(1-(h')^2)
			\end{multline}
			\begin{multline}\label{lin:eq:Theta_integral_form_c}
				{\Theta}_{\Sigma^{1,\delta_3}_{\tau_1;\tau_2}}^{1,c}[\phi]=\tilde{\Theta}_{\Sigma^{1,\delta_3}_{\tau_1;\tau_2}}^{c;\tau_2}[\phi]+\int_{\Sigma^{1,\delta_3}_{\tau_1;\tau_2}}\frac{1}{2}T\phi TW_1^\star\bigg(\hat{x}h(1-h'^2)h'+x(1-h'^2)\bigg)-h(1-h'^2)T W_1^\star X_\star\phi\\
				+\int_{\Sigma^{1,\delta_3}_{\tau_1;\tau_2}}(X\phi)X\O^{1,1}+\phi\O^{5,1}+T\phi\O^{1,1}(1-h'^2)
			\end{multline}
			\begin{multline}\label{lin:eq:Theta_integral_form_l}
				{\Theta}_{\Sigma^{1,\delta_3}_{\tau_1;\tau_2}}^{1,\Lambda}[\phi]=\tilde{\Theta}_{\Sigma^{1,\delta_3}_{\tau_1;\tau_2}}^{\Lambda;\tau_2}[\phi]+\tau_1\Sigma_{\tau_1,\tau_2}[\tilde{T}\cdot \T^V[\phi,\Lambda W]]+\\
				+\int_{\Sigma^{1,\delta_3}_{\tau_1;\tau_2}}(X\phi)X\O^{1,1}+\phi\O^{5,1}+T\phi\O^{1,1}(1-h'^2)
			\end{multline}
		\end{subequations}
	We also have 
	\begin{equation}
		\abs{\Sigma_{\tau_1,\tau_2}[\tilde{T}\cdot \T^V[\phi,\Lambda W_1]]}\lesssim_\epsilon\tau_1^{-3/2+\epsilon}(\E_{\Sigma_{\tau_1,\tau_2}}^0[\phi])^{1/2}.
	\end{equation}
	\end{lemma}

	\begin{proof}

		\emph{Step 1:}
		$\Theta^m$ is a straightforward computation.
		We already notice, that the part of $\Tmod$, that contains $\tilde{\phi}\in\O^{2,1}$, i.e. $\T^{4\tilde{\phi}\bar{\phi}^3}[\phi,\bar{\phi}]$ and $\T^{4\bar{\phi}^4}[\phi,\tilde{\phi}]$, is already of the form \cref{lin:eq:Theta_integral_form_m}.
		Therefore, it is sufficient to consider $W^a$ in place of $\bar{\phi}$.
		Now, we use the observation \cref{lin:eq:Wt-Wt_star} to also replace $W_a$ with $W_a^\star$.
		The form of \cref{lin:eq:Theta_integral_form_m} follows from evaluate a simply current, see \cref{app:eq:mom}.
		We note, that the difference from \cref{lin:eq:tildeTheta_integral_form} depending on $W_\star$ comes from time derivatives, that are of course not present when $z^{0,1}_a=0$.
		
		\emph{Step 2:}
		To compute $\Theta^c$, we first note that the associated current from \cref{lin:eq:Theta_def_c} has no time dependence, only an extra spatial weight.
		Therefore, the analysis for $\Theta^m$ applies verbatim, with loss of this extra spatial weight.
		
		\emph{Step 3:} For $\Theta^\Lambda$ we present a globalised argument, i.e set $\delta_3=\infty$. Localisation follows as in \cref{lin:lemma:localisation}. 
		We use the splitting 
		\begin{equation}\label{lin:eq:proof_Theta_diff}
			\Theta^{c}_{\Sigma^1_{\tau_1,\tau_2}}=\tau_1\Sigma^1_{\tau_1,\tau_2}[\tilde{T}\cdot \T^V[\phi,\Lambda W_1]]+\Sigma^1_{\tau_1,\tau_2}[J^\Lambda[\phi]-\tau_1\tilde{T}\cdot \T^V[\phi,\Lambda W_1]].
		\end{equation}
		The latter has no explicit $t$ weight, so we can already use the substitutions as done for $\Theta^c$.
		Notice also, that the first component exactly vanishes when $z^{0,1}=0$.
		
		We now prove that the first term on the right hand side of \cref{lin:eq:proof_Theta_diff} is appropriately bounded.
		Let us first compute on $\tilde{\Sigma}^1=\{t=const\}$ hypersurface, that has the same flat part as $\Sigma^1_{\tau_1,\tau_2}$
		\begin{nalign}
			\tilde{\Sigma}^1[\tilde{T}\cdot\T^V[\phi,\Lambda W_1]]&=\int_{\tilde{\Sigma}^1} T\phi T\Lambda W_1+\nabla\phi\cdot\nabla \Lambda W_1-V\Lambda W_1\phi
			+\frac{z^{0,1}\cdot \big(X\phi T\Lambda W_1+X\Lambda W_1 T\phi\big)}{t}\\
			&=\int_{\tilde{\Sigma}^1}T\phi\tilde{T}\Lambda W_1+t^{-1}T\Lambda W_1 z^{0,1}\cdot X\phi=\int_{\tilde{\Sigma}^1}t^{-1}T\Lambda W_1 z^{0,1}\cdot X\phi\lesssim t^{-2}(\E^0_{\tilde{\Sigma}^1}[\phi])^{1/2}.
		\end{nalign}
		Let's call $\D$ the region between $\Sigma^1_{\tau_1,\tau_2}$ and $\tilde{\Sigma}^1$.
		Next, we consider $\phi$ to be the scattering solution with no outgoing radiation to $(\Box+V_1)\phi=0$ in $\D$  with data posed on $\Sigma^1_{\tau_1,\tau_2}$.
		Using the divergence theorem, we compute
		\begin{multline}\label{lin:eq:proof_Theta_diff2}
			\tilde{\Sigma}^1[\tilde{T}\cdot\T^V[\phi,\Lambda W_1]]-\Sigma^1_{\tau_1,\tau_2}[\tilde{T}\cdot\T^V[\phi,\Lambda W_1]]=\int_{\D}\tilde{T}\phi(\Box+V)\Lambda W_1+\tilde{T}\Lambda W_1(\Box+V)\phi\\
			+t^{-2}\big(T\phi z^{0,1}X\Lambda W_1+T\Lambda W_1 z^{0,1}X\phi\big).
		\end{multline}
		Using an exterior Morawetz estimate\footnote{see equation (7) in \cite{yang_global_2013-2} }, we also have the estimate
		\begin{equation}\label{lin:eq:proof_Theta_diff3}
			\E^0_{\Sigma^1_{\tau_1,\tau_2}}[\phi]\gtrsim_\epsilon\int_{\D} \big((T\phi)^2+\abs{X\phi}^2\big)r^{-1-\epsilon}.
		\end{equation}
		Combining \cref{lin:eq:proof_Theta_diff2,lin:eq:proof_Theta_diff3}, we obtain
		\begin{equation}
			\abs{\tilde{\Sigma}^1[\tilde{T}\cdot\T^V[\phi,\Lambda W]]-\Sigma^1_{\tau_1,\tau_2}[\tilde{T}\cdot\T^V[\phi,\Lambda W]]}\lesssim_\epsilon\tau^{-3/2+\epsilon}(\E^0[\phi])^{1/2}.
		\end{equation}
	\end{proof}
	
	\begin{cor}\label{lin:cor:Theta_differences}
		Let $\phi,\bar\phi$ be as in \cref{lin:lemma:Theta_differences}.
		Then, we have
		\begin{nalign}\label{lin:eq:Theta-tilde_Theta}
			\abs{\Theta_{\Sigma^{1,\delta_3}_{\tau_1,\tau_2}}^{m}[\phi]-\tilde{\Theta}_{\Sigma^{1,\delta_3}_{\tau_1,\tau_2}}^{m;\tau_2}[\phi]}\lesssim\tau^{-1}(\E^0_{{\Sigma^{1,\delta_3}_{\tau_1,\tau_2}}}[\phi])^{1/2}\\
			\abs{\Theta_{\Sigma^{1,\delta_3}_{\tau_1,\tau_2}}^{c}[\phi]-\tilde{\Theta}_{\Sigma^{1,\delta_3}_{\tau_1,\tau_2}}^{c;\tau_2}[\phi]}\lesssim\tau^{-1/2}(\E^0_{{\Sigma^{1,\delta_3}_{\tau_1,\tau_2}}}[\phi])^{1/2}\\
			\abs{\Theta_{\Sigma^{1,\delta_3}_{\tau_1,\tau_2}}^{\Lambda}[\phi]-\tilde{\Theta}_{\Sigma^{1,\delta_3}_{\tau_1,\tau_2}}^{\Lambda;\tau_2}[\phi]}\lesssim\tau^{-1/2}(\E^0_{{\Sigma^{1,\delta_3}_{\tau_1,\tau_2}}}[\phi])^{1/2}.
		\end{nalign}
	\end{cor}
	\begin{proof}
		All, the estimates follow from Cauchy-Schwarz and the explicit form given in \cref{lin:eq:Theta_integral_form}.
		For instance, we bound
		\begin{equation}
			\int_{\Sigma^{1,\delta_3}_{\tau_1,\tau_2}}X\phi\O^{3,1}\lesssim (\E^{0}_{\Sigma^{1,\delta_3}_{\tau_1,\tau_2}}[\phi])^{1/2}\int_{\Sigma^{1,\delta_3}_{\tau_1,\tau_2}}\O^{6,2}\lesssim(\E^{0}_{\Sigma^{1,\delta_3}_{\tau_1,\tau_2}}[\phi])^{1/2}\tau_1^{-1}.
		\end{equation}
	\end{proof}
	
	\begin{cor}[$L^2$ coercivity estimate]\label{lin:cor:coercivity_estimate}
		Let $\Theta^\bullet$ be defined as in \cref{lin:def:corrected_Theta}.
		Then for a smooth function $\phi$ in $\Region$ and $\bar\phi$ satisfying \cref{lin:eq:assumptionPhibar_weak}, $\tau$ sufficiently large, we have 
		\begin{nalign}\label{lin:eq:coercivity_estimate}
			\bar{\E}^0_{\Sigma^{1,\delta_3}_{\tau;\tau_2}\setminus\Sigma^{1,\delta_2}_{\tau;\tau_2}}+\tau_2^{-1}\bar{\E}^0_{\Sigma^{1,\delta_3}_{\tau;\tau_2}}\lesssim \bar{\E}^V_{\Sigma^{1,\delta_3}_{\tau;\tau_2}}+\bar{\E}^0_{\Sigma^{1,\delta_3}_{\tau;\tau_2}\setminus\Sigma^{1,\delta_3/2}_{\tau;\tau_2}}+\bar{\E}^0_{\Sigma^{1,\delta_3}_{\tau;\tau_2}\setminus\Sigma^{1,\delta_3/2}_{\tau;\tau_2}}+\abs{\alpha_{\pm}}^2
			+\frac{1}{\tau_2}\sum_{\bullet\in\{c,\Lambda\}}\abs{\Theta^{\bullet}_{\Sigma^{1,\delta_3}_{\tau';\tau_2}}}^2
		\end{nalign}
	\end{cor}

	\paragraph{Conservation laws}
	Even though, the fluxes are perturbations of the original, the corresponding bulk terms are significantly improved. 

	\begin{prop}\label{lin:prop:Theta_conservation}
		Let $\bar{\phi}$ be as in \cref{lin:def:admissible} and $\phi$ be a solution to \cref{lin:eq:main linearised} in $\Region^1$.		
		Then for $\tau_1\in[\tau_2^\Delta,\tau_2]$
		\begin{subequations}\label{lin:eq:Theta_conservation}
			\begin{align}
				{\Theta}^{m}_{\Sigma^{1,\delta_3}_{\tau_1;\tau_2}}[\phi]&={\Theta}^{m}_{\Sigma^{1,\delta_3}_{\tau_2;\tau_2}}[\phi]+\C_{\tau_1,\tau_2;\tau_2}\big[J^m[\phi]\big]+\mathcal{B}^{m}\label{lin:eq:Theta_conservation_m}\\
				{\Theta}^{c}_{\Sigma^{1,\delta_3}_{\tau_1;\tau_2}}[\phi]&={\Theta}^{c}_{\Sigma^{1,\delta_3}_{\tau_2;\tau_2}}[\phi]+\C_{\tau_1,\tau_2;\tau_2}\big[J^c[\phi]+J^{c,r}[\phi]\big]+\tau_2{\Theta}^{m}_{\tau_2}[\phi]-\tau_1{\Theta}^{m}_{\tau_1}[\phi]
				+\mathcal{B}^{c}\label{lin:eq:Theta_conservation_c}\\
				{\Theta}^{\Lambda}_{\Sigma^{1,\delta_3}_{\tau_1;\tau_2}}[\phi]&={\Theta}^{\Lambda}_{\Sigma^{1,\delta_3}_{\tau_2;\tau_2}}[\phi]+\C_{\tau_1,\tau_2;\tau_2}\big[J^\Lambda[\phi]\big]+\mathcal{B}^{\Lambda}\label{lin:eq:Theta_conservation_l}
			\end{align}
		\end{subequations}
		where the current $J^{c,r}=T\cdot\T^0[\phi,r\hat{x}_i]$
		and the error terms are
		\begin{subequations}
			\begin{align}
				\mathcal{B}^{m}&=\int_{\Region^1}f\O_1^{2,0}+
				T\phi \mathcal{O}_1^{N,N}\\
				\mathcal{B}^{c}&=\int_{\Region^1}f\O_1^{1,-1}+T\phi \mathcal{O}_1^{N-1,N-1}\\
				\mathcal{B}^{\Lambda}&=\int_{\Sigma^{1,\delta_3}_{\tau_1;\tau_2}}\O^{3,1}\phi+\int_{\Sigma^{1,\delta_3}_{\tau_2;\tau_2}}\O^{3,1}\phi+\int_{\C_{\tau_1;\tau_2}}\O^{3,1}\phi+\int_{\Region^1} f\O_1^{1,0}+(\jpns{y_1^{\tau_2}}^{-1}\phi,T\phi)\O^{3,2}_1
			\end{align}
		\end{subequations}
	\end{prop}
	\begin{remark}
		The condition \cref{lin:eq:assumptionPhibar_strong} is only used for \cref{lin:eq:Theta_conservation_l}.
		In particular, weakening the condition to \cref{lin:eq:assumptionPhibar_weak}, only changes the $\O^{3,2}$ factor to $\O^{3,1}$ for the linear in $\phi$ error term.
	\end{remark}

	\begin{proof}
		\emph{Step 1: $k=0$.}
		We begin with $\Theta^m$.
		We compute
		\begin{equation}
			\partial^\mu\Tmod_{\mu\nu}[\phi]=\partial_\nu\phi(\Box\bar{\phi}+\bar{\phi}^5)+\partial_\nu\bar{\phi}(\Box+5\bar{\phi}^4)\phi=\partial_\nu\phi \mathcal{O}_{\scri}^{5,N,N}+f\O^{2,0}_1
		\end{equation}
		Contracting with $\partial_t$ and applying the divergence theorem yields the result for $\Theta^m$.
		Similarly, we can contract with $(x_i\partial_t+t\partial_i)$ to get the result for $\Theta^c$
		
		For $\Theta^\Lambda$, we first compute the divergence of $J^\Lambda$.
		This is composed of three terms
		\begin{subequations}\label{lin:eq:proof:LambdaW}
			\begin{align}
				\begin{multlined}
					(\tilde{T}^c)^\nu\partial^\mu\T_{\mu\nu}[\phi,t\Lambda W_1]=\tilde{T}^c(\phi)(\Box+V)t\Lambda W_1+\tilde{T}^c( t\Lambda W_1)(\Box+V)\phi+\phi t\Lambda W_1\tilde{T}^c(V)\\
					=\tilde{T}^c(\phi)\O^{3,1}_1+\Lambda W_1(-\mathfrak{E}^l\phi+f)
				\end{multlined}\\
				\T^{\mu\nu}[\phi,t\Lambda W_1]\partial_\mu \tilde{T}^c_\nu=\T^{ti}[\phi,t\Lambda W_1]\frac{z^{1,0}_i}{t^2}=\frac{z^{1,0}_i}{t}T(\phi)\partial_i(\Lambda W_1)+\O^{3,2}\partial\phi
			\end{align}
		\end{subequations}
		We improve the control for the least decaying terms, the ones with $T(\phi)$ factor, by using the divergence theorem once again.
		For instance, we write 
		\begin{equation}\label{lin:eq:proof:div0}
			\O_1^{3,(1,0)}T(\phi)=\mathrm{div} \Big(\dd t \phi\O_1^{3,(1,0)}\Big)-\phi\O^{4,(2,0)}_1.
		\end{equation}
		The integral of the first term can be bounded as
		\begin{equation}
			\int_{\Region^1}\mathrm{div} \Big( \dd t\phi\O^{3,(1,0)}\Big)\lesssim\tau_1^{-1/2}(\bar{\E}_{\Sigma_{\tau_1;\tau_2}}^0[\phi]+
			 \bar{\E}_{\Sigma_{\tau_1;\tau_2}}^0[\phi]+\bar{\E}_{\C_{\tau_1;\tau_2}}^0[\phi]).
		\end{equation}
		Summing the terms in \cref{lin:eq:proof:LambdaW}, and using the improvement from \cref{lin:eq:proof:div0},  we get
		\begin{nalign}\label{lin:eq:scaling_critical}
			\partial\cdot J^\Lambda=\phi\O_1^{4,2}+\partial\phi\O_1^{3,2} +f\O_1^{1,0}-\boxed{\phi\Lambda W \mathfrak{E}^l}+\mathrm{div}(\phi \O_1^{3,1}).
		\end{nalign}
		Using the strong condition \cref{lin:eq:assumptionPhibar_strong}, we get the result.
	\end{proof}
	
	For future use, we note that the proof of \cref{lin:eq:Theta_conservation_l} also implies that for $J=\tilde{T}\cdot\T^V[\Lambda W,\phi]$
	\begin{nalign}\label{lin:eq:Theta_conservation_stupid}
		&\Sigma^{1,\delta_3}_{\tau_1,\tau_2}[J]=\Sigma^{1,\delta_3}_{\tau,\tau_2}[J]+\mathcal{B}^{\Lambda'}+\C_{\tau,\tau_1;\tau_2}[J]\\
		&\B^{\Lambda'}:=\int_{\Sigma_{\tau_1;\tau_2}}\O^{4,2}\phi+\int_{\Sigma_{\tau_2;\tau_2}}\O^{4,2}\phi+\int_{\C_{\tau_1;\tau_2}}\O^{4,2}\phi+\int_{\Region} f\O_1^{1,0}+(\jpns{y^1_{\tau_2}}^{-1}\phi,T\phi)\O^{4,3}_1.
	\end{nalign}

	Next, let us study the contribution of the currents on the cones.
	\begin{lemma}[Incoming modulation]\label{lin:lemma:outgoing_modulation}
		Let $\bar{\phi}$ satisfy \cref{lin:eq:assumptionPhibar_weak} and let $\phi$ be a smooth function in $\Region^1$.
		For $J^m,J^c,J^\Lambda,J^{c,r}$ as in \cref{lin:prop:Theta_conservation}, we have the bound on the outgoing modulation currents 
		\begin{subequations}\label{lin:eq:Theta_in_current}
			\begin{align}
				&\abs{\C^a_{\tau_1,\tau_2}\big[J^m[\phi]\big]}^2 \lesssim \tau_2^{-1}\bar{\E}^0_{\C^a_{\tau_1,\tau_2}}[\phi]\label{lin:eq:Theta_in_current_m} \\
				&\abs{\C^a_{\tau_1,\tau_2}\big[J^c[\phi]+J^{c,r}[\phi]\big]}^2 +\abs{\C^a_{\tau_1,\tau_2}\big[J^\Lambda[\phi]\big]}^2\lesssim(\tau_2-\tau_1) \bar{\E}^0_{\C^a_{\tau_1,\tau_2}}[\phi].\label{lin:eq:Theta_in_current_cL}
			\end{align}
		\end{subequations}
		We estimate their derivatives
		\begin{subequations}\label{lin:eq:Theta_in_current_der}
			\begin{align}
				&	\abs{\partial_{\tau_1}\C_{\tau_1,\tau_2}[J^m[\phi]]}^2\lesssim\tau_2^{-2}(\bar{\E}^0_{\C_{\tau_1,\tau_2}}[T\phi])^{1/2}(\bar{\E}^0_{\C_{\tau_1,\tau_2}}[\phi])^{1/2}\label{lin:eq:Theta_in_current_der_m}\\
				&\abs{\partial_{\tau_1}\C^a_{\tau_1,\tau_2}\big[J^c[\phi]+J^{c,r}[\phi]\big]}^2 +\abs{\partial_{\tau_1}\C^a_{\tau_1,\tau_2}\big[J^\Lambda[\phi]\big]}^2\lesssim(\bar{\E}^0_{\C_{\tau_1,\tau_2}}[T\phi])^{1/2}(\bar{\E}^0_{\C_{\tau_1,\tau_2}}[\phi])^{1/2}
			\end{align}
		\end{subequations}
	\end{lemma}
	\begin{proof}
		\emph{Proof of \cref{lin:eq:Theta_in_current}:}
		Let us start with $J^m$.
		The integrand on $\C^a_{\tau_1,\tau_2}$ is
		\begin{nalign}\label{lin:eq:proof_in_current_1}
			(T-X^{\mathrm{r}})^\mu(X_i)^\nu\Tmod_{\mu\nu}[\phi]=(T-X^{\mathrm{r}})^\mu(X_i)^\nu\T^0_{\mu\nu}[\phi,\bar{\phi}]+\hat{x}_i(\bar{\phi}^5\phi).
		\end{nalign}
		The estimate now follows from Cauchy-Schwarz and the fact that $\bar\phi\in\O^{(1,0),0}_{\loc}$:
		\begin{nalign}
			\abs{\C^a_{\tau_1,\tau_2}\big[J^m[\phi]\big]}^2\leq \abs{\C^a_{\tau_1,\tau_2}\big[T\cdot\T[\phi,\bar{\phi}]\big]}^2+\Big(\int_{\C^a_{\tau_1,\tau_2}}\abs{(\bar{\phi}^5\phi)}\Big)^2\\
			\lesssim \abs{\C^a_{\tau_1,\tau_2}\big[T\cdot\T[\phi]\big]}\abs{\C^a_{\tau_1,\tau_2}\big[T\cdot\T[\bar{\phi}]\big]}+\int_{\C^a_{\tau_1,\tau_2}}\phi^2/r^2\int_{\C^a_{\tau_1,\tau_2}}\bar{\phi}^{10}r^2
			\lesssim\tau_2^{-1} \abs{\C^a_{\tau_1,\tau_2}\big[T\cdot\bar{\T}[\phi]\big]}.
		\end{nalign}
		For $J^{c}$, the loss of $\tau$ follows from the extra weight in the multiplier.
		For $J^{c,r}$, we use Cauchy-Schawrz:
		\begin{equation}
			\C^a_{\tau_1,\tau_2}\big[J^{c,r}[\phi]\big]\lesssim\Big(\int_{\C_{\tau,\tau';\tau_2}}\abs{\partial_u\phi}\Big)^2\lesssim(\tau_2-\tau_1)\int_{\C_{\tau,\tau';\tau_2}}\abs{\partial_u\phi}^2.
		\end{equation}
		Finally, for $J^\Lambda$, the extra $\tau$ factor follows as for $J^c$, but the weight is not in the multiplier, but in $t\Lambda W$.
		
		\emph{Proof of \cref{lin:eq:Theta_in_current_der}:}
		We again use the form \cref{lin:eq:proof_in_current_1}, but this time, we only need to integrate it over a sphere $S_\tau=\Sigma_{\tau,\tau_2}\cap\C_{\tau_1,\tau_2}$
		\begin{nalign}
			\partial_{\tau}\C_{\tau,\tau_2}[J^m[\phi]]=\int_{S_\tau}\Big((T-X^{\mathrm{r}})^\mu(X_i)^\nu\T^0_{\mu\nu}[\phi,\bar{\phi}]+\hat{x}_i(\bar{\phi}^5\phi)\Big) r^2\dd_{g_{S^2}}
		\end{nalign}
		
		We bound this using Sobolev inequality.
		Let $\Gamma\in\{T-X^\r,r^{-1}\Omega,r^{-1}\}$, stand for any one of the tangential derivatives on the cone $\C_{\tau_1,\tau_2}$.
		Then, using that $\bar{\phi}\in\O^{(1,0),0}_1$, we can bound
		\begin{multline}\label{lin:eq:proof_derivative}
			\sup_{\tau\in(\tau_1,\tau_2)}\int_{S_\tau}(T-X^{\mathrm{r}})^\mu(X_i)^\nu\T^0_{\mu\nu}[\phi,\bar{\phi}]r^2\dd_{g_{S^2}}\lesssim
			\sup_{\tau\in(\tau_1,\tau_2)}\int_{S_\tau}\abs{\Gamma\phi}\dd_{g_{S^2}}\\
			\lesssim\Bigg(\int_{\tau=\tau_1}^{\tau_2}\Big(\partial_\tau\int_{S_\tau}\Gamma\phi\dd_{g_{S^2}}\Big)^2\Bigg)^{1/4}\Bigg(\int_{\tau=\tau_1}^{\tau_2}\Big(\int_{S_\tau}\Gamma\phi\dd_{g_{S^2}}\Big)^2\Bigg)^{1/4}\lesssim\tau_2^{-1}(\bar{\E}^0_{\C_{\tau_1,\tau_2}}[T\phi])^{1/4}(\bar{\E}^0_{\C_{\tau_1,\tau_2}}[\phi])^{1/4}.
		\end{multline}
		We proceed similarly for the other currents.
		We only show the computation for $J^{c,r}$, as this capture the worst decaying term.
		We first of all notice that
		\begin{equation}
			\abs{\partial_{\tau}\C_{\tau,\tau_2}[J^{c,r}[\phi]]}\lesssim \int_{S_\tau} r\abs{\Gamma\phi}\dd g_{S^2}.
		\end{equation}
		Therefore, we use Sobolev embedding the same way as before and see that as $r\sim\tau$ in $\C_{\tau_1,\tau_2}$, we get an extra factor of $\tau$ compared to \cref{lin:eq:proof_derivative}.
		
	\end{proof}

	\paragraph{Higher order norm}

	\begin{prop}\label{lin:prop:recover_Theta}
		Fix $\epsilon_1>0$ arbitrary.
		Let $\bar{\phi}$ satisfy \cref{lin:eq:assumptionPhibar_weak} and let $\phi$ be a smooth solution to \cref{lin:eq:main linearised} in $\Region^1$.
		Then, for $k\geq1$ we have 
		\begin{subequations}
			\begin{multline}
				\tau^{-\epsilon_1}\abs{\Theta^{m}_{\Sigma^{1,\delta_3}_{\tau_1;\tau_2}}[T^k\phi]}^2\lesssim  \tau^{-1}\Big(\tau_1^{-2}\bar{\E}_{\Sigma^{1,\delta_3}_{\tau_1,\tau_2}\setminus\Sigma^{\delta_3/d^2}_{\tau_1,\tau_2}}^0[T^{k-1}\phi]+\bar{\E}_{\Sigma^{1,\delta_3}_{\tau_1,\tau_2}\setminus\Sigma^{\delta_3/d^2}_{\tau_1,\tau_2}}^0[T^{k}\phi]\Big)
				+\tau^{-2}\bar{\E}^0_{\Sigma^{1,\delta_3}_{\tau_1,\tau_2}}[T^{k}\phi]\\
				+\tau^{-4}\bar{\E}^0_{\Sigma^{1,\delta_3}_{\tau_1,\tau_2}}[T^{k-1}\phi]+\tau_2^{-2}(\bar{\E}^0_{\C_{\tau_1,\tau_2}}[T^k\phi])^{1/2}(\bar{\E}^0_{\C_{\tau_1,\tau_2}}[T^{k-1}\phi])^{1/2},
			\end{multline}
			\begin{multline}
				\tau^{-\epsilon_1}\abs{\Theta^{c}_{\Sigma^{1,\delta_3}_{\tau_1;\tau_2}}[T^k\phi]}^2\lesssim\abs{\Theta^{m}_{\Sigma_{T^k\tau_1;\tau_2}}[T^{k-1}\phi]}^2+\tau^{-1}\Big(\tau_1^{-2}\bar{\E}_{\Sigma^{1,\delta_3}_{\tau_1,\tau_2}\setminus\Sigma^{\delta_3/d^2}_{\tau_1,\tau_2}}^0[T^{k-1}\phi]+\bar{\E}_{\Sigma^{1,\delta_3}_{\tau_1,\tau_2}\setminus\Sigma^{\delta_3/d^2}_{\tau_1,\tau_2}}^0[T^{k}\phi]\Big)\\
				+\tau^{-1}(\bar{\E}_{\Sigma^{1,\delta_3}_{\tau_1,\tau_2}}^0[T^{k}\phi]+\tau^{-2}\bar{\E}^0_{\Sigma^{1,\delta_3}_{\tau_1,\tau_2}}[T^{k-1}\phi])+(\bar{\E}^0_{\C_{\tau_1,\tau_2}}[T^{k}\phi])^{1/2}(\bar{\E}^0_{\C_{\tau_1,\tau_2}}[T^{k-1}\phi])^{1/2}
			\end{multline}
			\begin{multline}
				\tau^{-\epsilon_1}\abs{\Theta^{\Lambda}_{\Sigma^{1,\delta_3}_{\tau_1;\tau_2}}[T^k\phi]}^2\lesssim\tau^{-1}\Big(\tau_1^{-2}\bar{\E}_{\Sigma^{1,\delta_3}_{\tau_1,\tau_2}\setminus\Sigma^{\delta_3/d^2}_{\tau_1,\tau_2}}^0[T^{k-1}\phi]+\bar{\E}_{\Sigma^{1,\delta_3}_{\tau_1,\tau_2}\setminus\Sigma^{\delta_3/d^2}_{\tau_1,\tau_2}}^0[T^{k}\phi]\Big)\\
				+\tau^{-1}(\bar{\E}_{\Sigma^{1,\delta_3}_{\tau_1,\tau_2}}^0[T^{k}\phi]+\tau^{-2}\bar{\E}^0_{\Sigma^{1,\delta_3}_{\tau_1,\tau_2}}[T^{k-1}\phi])+(\bar{\E}^0_{\C_{\tau_1,\tau_2}}[T^{k}\phi])^{1/2}(\bar{\E}^0_{\C_{\tau_1,\tau_2}}[T^{k-1}\phi])^{1/2}
			\end{multline}
		\end{subequations}
	\end{prop}
	\begin{proof}
		We only show the result for $k=1$, as higher derivatives follow by definition.
		We also ignore the possible logarithmic corrections in error terms of the form $\O^{2,1}$, as they will be included in $\tau^{\epsilon_1}$ prefactor.
		In particular, we treat $\O^{2,1}$ decay as $t^{-1}\jpns{y_a^{\tau_2}}^{-1}$.
		
		\emph{Step 1:}
		We start with $\Theta^m$.
		We will make the following approximations qualitative below:
		\begin{equation}
			\Theta^m_{\Sigma^{1,\delta_3}_{\tau_1,\tau_2}}[T\phi]\approx\tilde{\Theta}^{m,\tau_1}_{\Sigma^{1,\delta_3}_{\tau_1,\tau_2}}[T\phi]\approx 	\partial_{\tau_1}\tilde{\Theta}^{m;\tau_1}_{\Sigma^{1,\delta_3}_{\tau_1,\tau_2}}[\chi\phi]\approx\partial_{\tau_1}\Theta^m_{\Sigma^{1,\delta_3}_{\tau_1,\tau_2}}[\phi].
		\end{equation}
		First, we apply a cutoff function $\chi=\bar{\chi}^c_{\tau\delta_3/d}(y_1)$ to bound 
		\begin{multline}
			\abs{\Theta^m_{\Sigma^{1,\delta_3}_{\tau_1,\tau_2}}[T\phi]-\Theta^m_{\Sigma^{1,\delta_3}_{\tau_1,\tau_2}}[T\chi\phi]}=\abs{\Theta^m_{\Sigma^{1,\delta_3}_{\tau_1,\tau_2}}[\chi^cT\phi]+\Theta^m_{\Sigma^{1,\delta_3}_{\tau_1,\tau_2}}[[\chi,T]\phi]}\\
			\lesssim\Big(\tau_1^{-2}\bar{\E}_{\Sigma^{1,\delta_3}_{\tau_1,\tau_2}\setminus\Sigma^{\delta_3/d^2}_{\tau_1,\tau_2}}^0[\phi]+\bar{\E}_{\Sigma^{1,\delta_3}_{\tau_1,\tau_2}\setminus\Sigma^{\delta_3/d^2}_{\tau_1,\tau_2}}^0[T\phi]\Big)^{1/2}\tau^{-1/2}.
		\end{multline}
		Using the integral form \cref{lin:eq:Theta_integral_form_m}, we get
		\begin{equation}
			\abs{\Theta^m_{\Sigma^{1,\delta_3}_{\tau_1,\tau_2}}[T\chi\phi]-\tilde{\Theta}^{m,\tau_1}_{\Sigma^{1,\delta_3}_{\tau_1,\tau_2}}[T\chi\phi]}\lesssim\tau^{-1}(\bar{\E}_{\Sigma^{1,\delta_3}_{\tau_1,\tau_2}}^0[T\phi]+\tau^{-2}\bar{\E}^0_{\Sigma^{1,\delta_3}_{\tau_1,\tau_2}}[\phi])^{1/2}.
		\end{equation}
		
		Now, we use that since $\chi\phi$ is supported away from the boundary of $\Sigma_{\tau_1,\tau_2}^{\delta_3}$, and in this region, $\Sigma_{\tau_1,\tau_2}^{\delta_3}$ is defined by shifting along $T$, we can take the derivative out of the integral expressions
		for
		\begin{nalign}\label{lin:eq:proof_Theta_Tk1}
			\partial_{\tau_1}\tilde{\Theta}^{m;\tau_1}_{\Sigma^{1,\delta_3}_{\tau_1,\tau_2}}[\chi\phi]=\int_{\Sigma^{1,\delta_3}_{\tau_1,\tau_2}}T\Bigg(h'\Big(X^{\r}_{\star}
			W_1^\star(X_{\star}\phi)+\hat{x}(W_1^\star)^5 \phi\Big)+X_{\star}W_1^\star T\phi(1-h'^2)\Bigg)\\
			=\tilde{\Theta}^{m;\tau_1}_{\Sigma^{1,\delta_3}_{\tau_1,\tau_2}}[T\chi\phi]+\int_{{\Sigma^{1,\delta_3}_{\tau_1,\tau_2}}} (X_\star \phi,X_\star T\phi)\O^{3,1}_1+(T\phi,T^2\phi)(1-h'^2)\O^{2,1}_1.
		\end{nalign}
		We can again go back to $\partial_\tau \Theta^m$ via the same cutoffs.
		Using the integral form \cref{lin:eq:Theta_integral_form_m}, we conclude that
		\begin{multline}
			\abs{\Theta^m_{\Sigma^{1,\delta_3}_{\tau_1,\tau_2}}[T\phi]-\partial_{\tau_1}\Theta^m_{\Sigma^{1,\delta_3}_{\tau_1,\tau_2}}[\phi]}\lesssim \tau^{-1/2}\Big(\tau_1^{-2}\bar{\E}_{\Sigma^{1,\delta_3}_{\tau_1,\tau_2}\setminus\Sigma^{\delta_3/d^2}_{\tau_1,\tau_2}}^0[\phi]+\bar{\E}_{\Sigma^{1,\delta_3}_{\tau_1,\tau_2}\setminus\Sigma^{\delta_3/d^2}_{\tau_1,\tau_2}}^0[\tau_2^{-1}\Vb\phi]\Big)^{1/2}\\
			+\tau^{-1}(\bar{\E}_{\Sigma^{1,\delta_3}_{\tau_1,\tau_2}}^0[\tau_2^{-1}\Vb\phi]+\tau^{-2}\bar{\E}^0_{\Sigma^{1,\delta_3}_{\tau_1,\tau_2}}[\phi])^{1/2}.
		\end{multline}
		We can also compute the derivative using the conservation law \cref{lin:eq:Theta_conservation_m}:
		\begin{nalign}\label{lin:eq:proof_Theata_Tk}
			\partial_{\tau_1}\Theta^m_{\Sigma^{1,\delta_3}_{\tau_1,\tau_2}}[\phi]=(\partial_{\tau_1}\C_{\tau_1,\tau_2;\tau_2}\big[J^m[\phi]\big])_{\tau_1=\tau_2}+\int_{\Sigma^{1,\delta_3}_{\tau_1,\tau_2}}f\O_1^{2,0}+
			T\phi \mathcal{O}_1^{N,N}.
		\end{nalign}
		Using \cref{lin:eq:Theta_in_current_der} yields the result.
		
		\emph{Step 2:}
		Let us move onto $\Theta^c$.
		We use the same cutoff function $\chi$ as for $\Theta^m$ together with the slower decay rates in the region far region from \cref{lin:eq:Theta_integral_form_c} to obtain
		\begin{equation}
			\abs{\Theta^c_{\Sigma^{1,\delta_3}_{\tau_1,\tau_2}}[T\phi]-\Theta^c_{\Sigma^{1,\delta_3}_{\tau_1,\tau_2}}[T\chi\phi]}
			\lesssim\Big(\tau_1^{-2}\bar{\E}_{\Sigma^{1,\delta_3}_{\tau_1,\tau_2}\setminus\Sigma^{\delta_3/d^2}_{\tau_1,\tau_2}}^0[\phi]+\bar{\E}_{\Sigma^{1,\delta_3}_{\tau_1,\tau_2}\setminus\Sigma^{\delta_3/d^2}_{\tau_1,\tau_2}}^0[T\phi]\Big)^{1/2}\tau^{-1/2},
		\end{equation}
		and
		\begin{equation}
			\abs{\Theta^c_{\Sigma^{1,\delta_3}_{\tau_1,\tau_2}}[T\chi\phi]-\tilde{\Theta}^{c,\tau_1}_{\Sigma^{1,\delta_3}_{\tau_1,\tau_2}}[T\chi\phi]}\lesssim\tau^{-1/2}(\bar{\E}_{\Sigma^{1,\delta_3}_{\tau_1,\tau_2}}^0[T\phi]+\tau^{-2}\bar{\E}^0_{\Sigma^{1,\delta_3}_{\tau_1,\tau_2}}[\phi])^{1/2}.
		\end{equation}
		In conclusion, we get
		\begin{multline}\label{lin:eq:proof_ThetaTk2}
			\abs{\Theta^c_{\Sigma^{1,\delta_3}_{\tau_1,\tau_2}}[T\phi]-\partial_{\tau_1}\Theta^c_{\Sigma^{1,\delta_3}_{\tau_1,\tau_2}}[\phi]}\lesssim \tau^{-1/2}\Big(\tau_1^{-2}\bar{\E}_{\Sigma^{1,\delta_3}_{\tau_1,\tau_2}\setminus\Sigma^{\delta_3/d^2}_{\tau_1,\tau_2}}^0[\phi]+\bar{\E}_{\Sigma^{1,\delta_3}_{\tau_1,\tau_2}\setminus\Sigma^{\delta_3/d^2}_{\tau_1,\tau_2}}^0[T\phi]\Big)^{1/2}\\
			+\tau^{-1/2}(\bar{\E}_{\Sigma^{1,\delta_3}_{\tau_1,\tau_2}}^0[T\phi]+\tau^{-2}\bar{\E}^0_{\Sigma^{1,\delta_3}_{\tau_1,\tau_2}}[\phi])^{1/2}.
		\end{multline}
		In order to apply the divergence theorem for the partial derivative, we use \cref{lin:eq:Theta_conservation_c},
		and get an extra term from  $\Theta^m$ on the right hand side
		\begin{equation}
			\abs{\partial_{\tau_1}\Theta^c_{\Sigma^{1,\delta_3}_{\tau_1,\tau_2}}[\phi]}\lesssim\abs{\tau_2\partial_{\tau_1}\Theta^m_{\Sigma^{1,\delta_3}_{\tau_1,\tau_2}}[\phi]+\Theta^m_{\Sigma^{1,\delta_3}_{\tau_1,\tau_2}}[\phi]+\partial_{\tau_1} \C_{\tau_1,\tau_2}[J^c[\phi]+J^{c,r}[\phi]]}+\int f\O^{1,-1}+T\phi\O^{N-1,N-1}_1.
		\end{equation}
		We use \cref{lin:eq:proof_Theata_Tk,lin:eq:Theta_in_current_der} to obtain the result.
		
		\emph{Step 3:}
		We finish with $\Theta^\Lambda$.
		Until \cref{lin:eq:proof_ThetaTk2} the proof is identical to $\Theta^c$.
		However, to obtain over $\partial_\tau\Theta^\Lambda_{\Sigma^{1,\delta_3}_{\tau_1,\tau_2}}$, we need to bound the second term in \cref{lin:eq:Theta_integral_form_l}:
		\begin{equation}
			\partial_{\tau_1}(\tau_1\Sigma_{\tau_1,\tau_2}[\tilde{T}\cdot \T^V[\phi,\Lambda W]])=\Sigma_{\tau_1,\tau_2}[\tilde{T}\cdot \T^V[\phi,\Lambda W]]+\tau_1\partial_{\tau_1}\Sigma_{\tau_1,\tau_2}[\tilde{T}\cdot \T^V[\phi,\Lambda W]].
		\end{equation}
		We bound the second term with a divergence theorem estimate similar to \cref{lin:eq:proof_Theata_Tk}, but using \cref{lin:eq:Theta_conservation_stupid}, and that the boundary term at the sphere $\C_{\tau_1,\tau_2}\cap\Sigma_{\tau_1,\tau_2}$ is bounded as \cref{lin:eq:Theta_in_current_der_m}
		\begin{equation}
			\abs{\partial_{\tau_1}\Sigma_{\tau_1,\tau_2}[\tilde{T}\cdot \T^V[\phi,\Lambda W]]}\lesssim\tau_2^{-5/2}
			 (\bar{\E}^0_{\Sigma_{\tau_1,\tau_2}}[\phi])^{1/2}+\tau_2^{-2}(\bar{\E}^0_{\C_{\tau_1,\tau_2}}[T\phi])^{1/2}(\bar{\E}^0_{\C_{\tau_1,\tau_2}}[\phi])^{1/2}.
		\end{equation}
		Summing the terms from \cref{lin:eq:proof_ThetaTk2} and the above two equations yields the result.
	\end{proof}
	
	We summarise the above estimates in the coercivity lemma below.
	\begin{lemma}[Higher order coercivity]\label{lin:lemma:coercivity_higher}
		Fix $k\geq1$ and  $\epsilon_2,\epsilon_1>0$ such that $\epsilon_1+\epsilon_2<1$.
		Let $\bar{\phi}$ satisfy \cref{lin:eq:assumptionPhibar_weak} and $\phi$ be a smooth solution to \cref{lin:eq:main linearised} in $\Region^1$.
		There exists $\tau_2$ sufficiently large such that
		\begin{multline}
			\sum_{j\leq k}\tau_1^{2j(1-\epsilon_2-\epsilon_1)}\big(\bar{\E}_{\Sigma^{1,\delta_3}_{\tau_1;\tau_2}\setminus\Sigma^{1,\delta_2}_{\tau_1;\tau_2}}^0[T^j\phi]+\tau_1^{-1+(1-\delta_{j0})\epsilon_2}\bar{\E}_{\Sigma^{1,\delta_3}_{\tau_1;\tau_2}}^0[T^j\phi]\big)\lesssim\tau_1\abs{\Theta^m_{\Sigma^{1,\delta_3}_{\tau_1;\tau_2}}[\phi]}^2\\
			+\frac{R_2}{\tau_1}\Big( \abs{\Theta_{\Sigma^{1,\delta_3}_{\tau_1;\tau_2}}^c[\phi]}^2+\abs{\Theta_{\Sigma^{1,\delta_3}_{\tau_1;\tau_2}}^\Lambda[\phi]}^2\Big)+\tau^{2k}\abs{\alpha_\pm[T^k\phi]}^2\\
			+\sum_{j\leq k}\tau_1^{2j(1-\epsilon_1-\epsilon_2)}\Big(\bar{\E}_{\Sigma^{1,\delta_3}_{\tau_1;\tau_2}\setminus\Sigma^{1,\delta_3/2}_{\tau_1;\tau_2}}^0[T^j\phi]+\bar{\E}_{\Sigma^{1,\delta_3}_{\tau_1;\tau_2}}^V[T^j\phi]+\E^0_{\C^1_{\tau_1,\tau_2}}[T^j\phi]\Big),
		\end{multline}
		where $\delta_{j0}$ is the Kronecker delta indicating that the improvement only happens for $j>0$.
	\end{lemma}
	\begin{remark}
		We cannot prove a full $\tau$ gain for each $T$ commutator.
		The reason is two fold.
		Firstly, \cref{lin:prop:recover_Theta} already has a small loss.
		Secondly, in the energy estimate part, we will need to have stronger than $\tau^{-1}\E_{\Sigma^{1,\delta_3}_{\tau_1,\tau_2}}^0[T\phi]$ control to close the energy estimates, see the proof of \cref{lin:prop:main}.
	\end{remark}
	\begin{proof}
		We prove the estimate by induction.
		The base case follows from \cref{lin:cor:coercivity_estimate}.
		
		Next, we notice, that using \cref{lin:lemma:unstable_recovery} we can recover the control over $\alpha_\pm[T^j\phi]$, provided that $\epsilon_1+\epsilon_2<1$.
		
		For $k\geq2$, we use \cref{lin:prop:recover_Theta} to obtain
		\begin{multline}
			\tau^{-\epsilon_1}\abs{\Theta^{c}_{\Sigma_{\tau_1;\tau_2}}[T^k\phi]}^2\lesssim\abs{\Theta^{m}_{\Sigma_{T^k\tau_1;\tau_2}}[T^{k-1}\phi]}^2+\tau^{-1}\Big(\tau_1^{-2}\bar{\E}_{\Sigma^{1,\delta_3}_{\tau_1,\tau_2}\setminus\Sigma^{\delta_3/d^2}_{\tau_1,\tau_2}}^0[T^{k-1}\phi]+\bar{\E}_{\Sigma^{1,\delta_3}_{\tau_1,\tau_2}\setminus\Sigma^{\delta_3/d^2}_{\tau_1,\tau_2}}^0[T^{k}\phi]\Big)\\
			+\tau^{-1}(\E_{\Sigma^{1,\delta_3}_{\tau_1,\tau_2}}^0[T^{k}\phi]+\tau^{-2}\E^0_{\Sigma^{1,\delta_3}_{\tau_1,\tau_2}}[T^{k-1}\phi])+(\bar{\E}^0_{\C_{\tau_1,\tau_2}}[T^{k}\phi])^{1/2}(\bar{\E}^0_{\C_{\tau_1,\tau_2}}[T^{k-1}\phi])^{1/2}\\
			\lesssim\tau^{-1}\Big(\tau_1^{-2}\bar{\E}_{\Sigma^{1,\delta_3}_{\tau_1,\tau_2}\setminus\Sigma^{\delta_3/d^2}_{\tau_1,\tau_2}}^0[T^{k-2}\phi]+\bar{\E}_{\Sigma^{1,\delta_3}_{\tau_1,\tau_2}\setminus\Sigma^{\delta_3/d^2}_{\tau_1,\tau_2}}^0[T^{k-1}\phi]\Big)
			+\tau^{-2}\E^0_{\Sigma^{1,\delta_3}_{\tau_1,\tau_2}}[T^{k-1}\phi]\\
			+\tau^{-4}\E^0_{\Sigma^{1,\delta_3}_{\tau_1,\tau_2}}[T^{k-2}\phi]+\tau^{-1}\E_{\Sigma^{1,\delta_3}_{\tau_1,\tau_2}}^0[T^{k}\phi]+\tau^{-1}\bar{\E}_{\Sigma^{1,\delta_3}_{\tau_1,\tau_2}\setminus\Sigma^{\delta_3/d^2}_{\tau_1,\tau_2}}^0[T^{k}\phi]\\
			+\tau_2^{-3}\bar{\E}^0_{\C_{\tau_1,\tau_2}}[T^{k-2}\phi]+\tau_1^{1-\epsilon_2-\epsilon_1}\bar{\E}^0_{\C_{\tau_1,\tau_2}}[T^{k}\phi]+\tau_1^{-1+\epsilon_2+\epsilon_1}\bar{\E}^0_{\C_{\tau_1,\tau_2}}[T^{{k-1}}\phi].
		\end{multline}
		We notice the following structure for terms on $\Sigma_{\tau_1,\tau_2}$: whenever we commute with one less $T$ derivative, we either gain 2 weights , or gain 1 weight and localise to the exterior.
		For the radiation through $\C_{\tau_1,\tau_2}$, we can only obtain weaker than $\tau^1$ weight for the top order quantity if we weaken the the improvement over the one less commuted one.
		For $k=1$, we have an extra term from $\Theta^m$.
		A similar estimate also works for $\Theta^{\Lambda}$.
		
		Next, we use \cref{lin:cor:coercivity_estimate} for $T^k\phi$, with $\tau^{-1+\epsilon_2-\epsilon_1}$ loss for $\Theta$ to obtain
		\begin{multline}
			\bar{\E}^0_{\Sigma^{1,\delta_3}_{\tau;\tau_2}\setminus\Sigma^{1,\delta_2}_{\tau;\tau_2}}[T^k\phi]+\tau_2^{-1+\epsilon_2}\bar{\E}^0_{\Sigma^{1,\delta_3}_{\tau;\tau_2}}[T^k\phi]\lesssim \bar{\E}^V_{\Sigma^{1,\delta_3}_{\tau;\tau_2}}[T^k\phi]+\bar{\E}^0_{\Sigma^{1,\delta_3}_{\tau;\tau_2}\setminus\Sigma^{1,\delta_3/2}_{\tau;\tau_2}}[T^k\phi]+\bar{\E}^0_{\Sigma^{1,\delta_3}_{\tau;\tau_2}\setminus\Sigma^{1,\delta_3/2}_{\tau;\tau_2}}[T^k\phi]\\
			+\abs{\alpha_{\pm}[T^k\phi]}^2
			+\tau^{\epsilon_1+\epsilon_2-1}\Bigg(\tau^{-1}\Big(\tau_1^{-2}\bar{\E}_{\Sigma^{1,\delta_3}_{\tau_1,\tau_2}\setminus\Sigma^{\delta_3/d^2}_{\tau_1,\tau_2}}^0[T^{k-2}\phi]+\bar{\E}_{\Sigma^{1,\delta_3}_{\tau_1,\tau_2}\setminus\Sigma^{\delta_3/d^2}_{\tau_1,\tau_2}}^0[T^{k-1}\phi]\Big)
			+\tau^{-2}\E^0_{\Sigma^{1,\delta_3}_{\tau_1,\tau_2}}[T^{k-1}\phi]\\
			+\tau^{-4}\E^0_{\Sigma^{1,\delta_3}_{\tau_1,\tau_2}}[T^{k-2}\phi]+\tau^{-1}\E_{\Sigma^{1,\delta_3}_{\tau_1,\tau_2}}^0[T^{k}\phi]+\tau_2^{-3}\bar{\E}^0_{\C_{\tau_1,\tau_2}}[T^{k-2}\phi]	+\tau_1^{1-\epsilon_2-\epsilon_1}\bar{\E}^0_{\C_{\tau_1,\tau_2}}[T^{k}\phi]\\
			+\tau_1^{-1+\epsilon_2+\epsilon_1}\bar{\E}^0_{\C_{\tau_1,\tau_2}}[T^{{k-1}}\phi]\Bigg)
		\end{multline}
		The estimate follows after summing the above inequalities.
	\end{proof}
	
	\paragraph{Transition}
	At the end of a dyadic iteration, we need to restrict the evaluation of $\Theta$ to smaller hypersurfaces.
	
	\begin{lemma}[Transition]\label{lin:lemma:transition}
		For $\bar{\phi}$ satisfying \cref{lin:eq:assumptionPhibar_weak} and $\phi$ a smooth function in $\Region^1$ we have 
		\begin{equation}
			\abs{\Theta^{\bullet}_{\Sigma^{1,\delta_3}_{\tau_2^\Delta;\tau_2^\Delta}}[\phi]}^2\lesssim
			\abs{\Theta^{\bullet}_{\Region^1\cap\Sigma^{}_{\tau_2^\Delta;\tau_2^\Delta}}[\phi]}^2+\tau_2^{-\kappa}\bar{\E}^0_{\Region^1\cap\Sigma^{}_{\tau_2^\Delta;\tau_2^\Delta}\setminus\Sigma^{1,\delta_3}_{\tau_2^\Delta;\tau_2^\Delta}}[\phi]
		\end{equation}
		where
		\begin{equation}\label{linear:eq:kappa}
			\kappa(m)=\begin{cases}
				1& \bullet=m\\
				0& \text{else}
			\end{cases}
		\end{equation}
	\end{lemma}
	\begin{proof}
		This follows from the fast fall-off of the projection operators expressed in \cref{lin:eq:tildeTheta_integral_form} as in the proof of \cref{lin:lemma:EV_shift}.
		See \cref{not:fig:local_regions} for the regions where the estimate is applied.
	\end{proof}

	\subsection{Energy estimates}\label{lin:sec:energy_est}
	In order to control the solution $\phi$, we also apply standard energy estimates.
	
	\begin{lemma}[Interior energy estimate]\label{lin:lemma:local_energy_estimate}
		Let $\bar{\phi}$ satisfy \cref{lin:eq:assumptionPhibar_strong} and $\phi$ be a smooth solution to \cref{lin:eq:main linearised} in $\Region^1$.
		Then for
		\begin{equation}
			J^k=J^{\E}[\tilde{T}^k\phi]:=T\cdot(\T^{5\bar{\phi}^4}[\tilde{T}^k\phi]+\tilde{\T}^{5\bar{\phi}^4}[\tilde{T}^k\phi])
		\end{equation}
		we have
		\begin{nalign}
			\Sigma^{1,\delta_3}_{\tau;\tau_2}[J^k]\leq\Sigma^{1,\delta_3}_{\tau';\tau_2}[J^k]+\C^{\delta_3}_{\tau',\tau}[J^k]+\B^{\E,\ell,k}
		\end{nalign}
		with
		\begin{nalign}
			 ,
		\end{nalign}
		\begin{nalign}
			\B^{\E,\ell,k}=\int_{\Region^1}\O^{5,2+1/4} ((T^{k+1}\phi)^2+\jpns{y_1^{\tau_2}}^{-2}(T^k\phi)^2)+\O^{2,2}_1T\tilde{T}^k\phi(\partial,\jpns{y_1^{\tau_2}}^{-1})\tilde{T}^k\phi\\
			+\O_1^{-3,-2-1/4}(T^kf)^2+t^{-2k-2+1/4}\Big((t^{1}T,1,\partial)^{k-1}\phi\Big)^2.
		\end{nalign}
	\end{lemma}
	\begin{proof}
		\textit{Step 1: $k=0$.}
		We recall $S_\nu$ from \cref{not:eq:twisted_conservation}.
		We compute the divergences of $J^\E[\phi]$ coming from $\tilde{\T}$ and $\T$
		\begin{subequations}
			\begin{align}
				&\partial^\mu \T_{\mu\nu}^{5\bar{\phi}^4}[\phi]=\partial_\nu\phi(\Box+5\bar{\phi}^4)\phi+10\phi^2\bar{\phi}^3\partial_\nu\bar{\phi},\\
				&\partial^\mu \tilde{\T}_{\mu\nu}^{5\bar{\phi}^4}[\phi]=\tilde{\partial}_\nu\phi(\Box+5\bar{\phi}^4)\phi+10\phi^2\bar{\phi}^3\tilde{\partial}_\nu\bar{\phi}+S_\nu[\phi]\\
				&\overline{\T}^{5\bar{\phi}^4}_{it}[\phi]\frac{1}{t^2}z_1^{0,1}=\frac{1}{t^2}T\phi z^{0,1}\cdot\big( X\phi+\jpns{y_a^{\tau_2}}^{-1}X \jpns{y_a^{\tau_2}}\phi\big)
			\end{align}
		\end{subequations}
		Therefore, using \cref{lin:eq:assumptionPhibar_strong}, i.e. that $\mathfrak{E}^l\in\O^{5,2}_1$, we get
		\begin{nalign}\label{lin:eq:energy_critical}
			\partial\cdot J^{\E}\lesssim \boxed{\phi^2\O_1^{6,3}}+\O^{2,2}_1T\phi(\partial,\jpns{y_1^{\tau_2}}^{-1})\phi +f(\tilde{T},\jpns{y_1^{\tau_2}}^{-1})\phi\O_1^{0,0}\\
			\lesssim ((T\phi)^2+\jpns{y_1^{\tau_2}}^{-2}\phi^2)\O_1^{3,2+\epsilon}+f\O_1^{-3,-2-\epsilon}.
		\end{nalign}
		Using $\epsilon=1/4$ yields the result.
		
		\textit{Step 2, $k\geq1$:}	
		We commute with $\tilde{T}$ to get
		\begin{nalign}
			(\Box+5\bar{\phi}^4)\tilde{T}\phi=[\tilde{T},5\bar{\phi}^4]\phi+[\tilde{T},\Box]\phi+\tilde{T}f=t^{-2}\tilde{T}(z^{0,1}_1\cdot X)\phi+\partial^2\phi\O^{3,3}_1+\phi \O^{6,3}+\tilde{T}f.
		\end{nalign}
		By induction, we similarly get
		\begin{equation}\label{lin:eq:tildeT_commutation}
			(\Box+kt^{-2}(z\cdot X)+5\bar{\phi}^4)\tilde{T}^k\phi=\tilde{T}^kf+\sum_{j\leq k-1}\O_1^{5+j,3+j}\tilde{T}^{k-1-j}(\partial,1)^{j+2}\phi.
		\end{equation}
		Taking divergence of the current, we get
	\begin{multline}
		\partial\cdot J^k
		\gtrsim \O^{5,2+\epsilon} ((\tilde{T}^{k+1}\phi)^2+\jpns{y_1^{\tau_2}}^{-2}(\tilde{T}^k\phi)^2)+\O^{2,2}_1T\tilde{T}^k\phi(\partial,\jpns{y_1^{\tau_2}}^{-1})\tilde{T}^k\phi\\
		+\O_1^{-3,-2-\epsilon}(T^kf)^2+\sum_{j\leq k-1}\O_1^{9+2j,4-\epsilon+2j}(\tilde{T}^{k-1-j}(\partial,1)^{j+2}\phi)^2.
	\end{multline}
	\end{proof}

	We also have energy estimate in the exterior
	\begin{lemma}[Exterior energy estimate]\label{lin:lemma:ext_energy_estimate}
		Let $\phi$ be a scattering solution to \cref{lin:eq:main linearised} with no outgoing radiation from $\scri$ in $\mathcal{R}=\mathcal{R}^{\mathrm{ext}}_{\tau_1,\tau_2}$ with $\tau_1\in\tau_2[1-\delta,1]$.
		Then, for $k\geq 0$ and $J^k=J^{\E,\e,1}[\Vb^k\phi]:=t T\cdot(\T[\Vb^k\phi]+\tilde{\T}[\Vb^k\phi])$ we have
		\begin{equation}
			\Sigma^{\mathrm{ext}}_{\tau_1}[J^{k}]+\sum_a \C^a_{\tau_1,\tau_2}[J^{k}]=\Sigma^{\mathrm{ext}}_{\tau_1}[J^{k}]+\int_{\Region^\e}(\Vb^kf)^2\O_{all}^{-2}.
		\end{equation}
	\end{lemma}
	\begin{proof}
		We first prove the result for $k=0$.
		We notice, that for $\phi$ a solution to \cref{lin:eq:main linearised}, the divergence of $J^{\E,\e,1}$ satisfies
		\begin{equation}\label{lin:eq:positive_bulk_control}
			\partial\cdot J^{\E,\e,1}\gtrsim \big((T\phi)^2+\abs{\nabla\phi}^2+\jpns{x}^{-2}\phi^2\big)+\phi^2\O_{all}^{4}+f\partial\phi\O^{-1}_{all}\gtrsim f^2\O^{-2}_{all}.
		\end{equation}
		The lemma follows form an application of divergence theorem and using that for no outgoing radiation $J^{\E,\e,1}$ has vanishing flux through $\scri$, see \cref{scat:thm:existence_of_scattering_solution}.
		
		To obtain higher order estimates, we commute with the symmetries of $\Box$.
		The error terms obtained from commuting $5\bar{\phi}^4$ with these symmetries can be incorporated in $f$.
		Observing that $\bar{\phi}^4\O^{-2}_{all}=\O^{2}_{all}$, these errors will take the form $(\Vb^{k-1}\phi)^2\O^{2}_{all}$ .
		All of these are controlled by the positive bulk as in \cref{lin:eq:positive_bulk_control}.
	\end{proof}

	In the region $\mathcal{R}^{\mathrm{t},a}_{\tau_2}$, we can proceed similarly as in the exterior region, as all the currents give coercive contributions
	
	\begin{lemma}[Transitional energy estimate]\label{lin:lemma:transitional_energy_estimate}
		Let $\phi$ be a solution to \cref{lin:eq:main linearised} in $\mathcal{R}=\mathcal{R}^{\mathrm{t},1}_{\tau_2}$.
		Then, for $k\geq 0$ and $J^k=J^{\E,\e,1}[\Vb^k\phi]$ we have
		\begin{equation}
			\Sigma^{\mathrm{t},1}_{\tau_2}[J^{k}] =\C^a_{\tau_1,\tau_2}[J^{k}]+\Big(\Sigma^{1,\delta_3}_{\tau_2;\tau_2}\cap\mathcal{R}^{\mathrm{t},a}_{\tau_2}\Big) [J^{k}]+\int_{\Region^{\mathrm{t},1}}(\Vb^kf)^2\O_{all}^{-2}.
		\end{equation}
	\end{lemma}

	Similarly between the two cones $\C^{\delta_3}_{\tau_1,\tau_2},\C^{\delta_3/2}_{\tau_1,\tau_2}$
	\begin{lemma}\label{lin:lemma:local_midcone_energy_estimate}
		Let $\phi$ be a solution to \cref{lin:eq:main linearised} in $\Region^a$. Then, for $k\geq 0$ and $J^k=J^{\E,\e,1}[\Vb^k\phi]$ we have
		\begin{equation}
			(\Sigma_{\tau_1,\tau_2}^{a,\delta_3}\setminus\Sigma_{\tau_1,\tau_2}^{a,\delta_3/2})[J^k]=\C^{a,\delta_3}_{\tau_1,\tau_2}[J^k]+(\Sigma_{\tau_2,\tau_2}^{a,\delta_3}\setminus\Sigma_{\tau_2,\tau_2}^{a,\delta_3/2})[J^k]+\int_{\Region^a\setminus \D^{\mathrm{c},a,\delta_3/2}_{\tau_2}}(\Vb^kf)^2\O_{all}^{-2}.
		\end{equation}
	\end{lemma}

	\subsubsection{Elliptic estimates}
	
	\begin{lemma}\label{lin:lemma:elliptic_estimate}
		Let $\phi$ be a smooth solution to \cref{lin:eq:main linearised} in $\Region^{a,\delta_4}$.
		Then for $k\geq1$ and $c>1$
		\begin{nalign}			
			\Big(\Sigma^{a,\delta_2}_\tau\cap\{r>cR_1\}\Big)[-T_a\cdot \bar{\T}[\Vb^k\phi]] \lesssim_{c,k}\Big(\Sigma^{a,c\delta_2}_\tau\cap\{r>R_1\}\Big)\big[-T_a\cdot\bar{\T}[(\tau T_a,1)^k\phi]\big]\\
			+ \int_{\Sigma^{a,c\delta_2}_\tau\cap\{r>R_1\}} (\Vb^{k-1}f)^2.
		\end{nalign}
	\end{lemma}
	\begin{proof}
		This is standard estimate and follows from Lemma 5.12 in \cite{kadar_scattering_2024}, with including $\mathfrak{E}^{l}\phi$ as extra inhomogeneity.
		We provide some details.
		Consider the case $k=1$.
		We use \cref{not:eq:wave__operator_in_t_star} to write \cref{lin:eq:main linearised} as 
		\begin{equation}
			\Delta_{x}|_{t_\star^1}\phi=\Big(
			-(1-h'^2)T^2-(2h'T)X^\r_\star-R_2^{-1}h''T+\frac{2}{r}(X^\r_\star-h'T)+\mathfrak{E}^l\Big)\phi+f.
		\end{equation}
		Introducing a cutoff function $\chi$ equal to 1 on $\Sigma^{a\delta_2}_\tau\cap\{r>cR_1\}$ and supported on $\Sigma^{a,c\delta_2}_\tau\cap\{r>R_1\}$, we get the result from elliptic regularity:
		\begin{nalign}\label{lin:eq:elliptic_proof1}
			\norm{\jpns{y_a^{\tau_2}}^{1/2}(\jpns{y_a^{\tau_2}}^{-1},\partial)\phi}_{\Hb^k(\Sigma^{a\delta_2}_\tau\cap\{r>cR_1\})}\lesssim_{c,k}\norm{\jpns{y_a^{\tau_2}}^{3/2}\Delta_x|_{\tau_\star^1}\phi}_{\Hb^{k-1}(\Sigma^{ac\delta_2}_\tau\cap\{r>R_1\})}\\
			\lesssim \Big(\Sigma^{a,c\delta_2}_\tau\cap\{r>R_1\}\Big)\big[-T_a\cdot\bar{\T}[(\tau T_a,1)^1\Vb^{k-1}\phi]\big]		+ \int_{\Sigma^{a,c\delta_2}_\tau\cap\{r>R_1\}} r(\Vb^{k-1}f)^2.
		\end{nalign}
		The extra $1/2$ weight comes from the natural weight in the energy integrals.
		
		For $k\geq2$, we proceed by induction.
		Commuting the equation by $T$, we already have
		\begin{nalign}
			\norm{(1,\partial)\tau T\phi}_{\Hb^k(\Sigma^{a\delta_2}_\tau\cap\{r>cR_1\})}\lesssim_{c,k}\lesssim\Big(\Sigma^{a,c\delta_2}_\tau\cap\{r>R_1\}\Big)\big[-T_a\cdot\bar{\T}[(\tau T_a,1)^2\Vb^{k-1}\phi]\big]\\
			+ \int_{\Sigma^{a,c\delta_2}_\tau\cap\{r>R_1\}} r(\tau T\Vb^{k-1}f)^2.
		\end{nalign}
		We recover the rest of the top order derivatives via \cref{lin:eq:elliptic_proof1}.
	\end{proof}
	
	\subsection{Combining the estimates}\label{lin:sec:combining}

	Before stating the estimate, let us introduce the norms that measure what is controlled using the estimates derived in the previous sections. 
	
	\begin{definition}[Norms for solution]\label{lin:def:norms_master}
		Fix $\epsilon\in(0,1/10)$.
		For  $\tau_1\in[\tau^\Delta_2,\tau_2]$ and $k\geq0$ we define the non coercive local quantities
		\begin{multline}
			\master^{\ell,V}_{\tau_1,\tau_2}[\phi]:=\sup_{\tau\in[\tau_1,\tau_2]}\sum_{j\leq k,a}\tau_2^{-1}\abs{\Theta^{a,c}_{\Sigma^{a,\delta_3}_{\tau;\tau_2}}[\phi]+\Theta^{a,\Lambda}_{\Sigma^{a,\delta_3}_{\tau;\tau_2}}[\phi]}^2+\tau_2^{1-2\epsilon}\abs{\Theta^{a,m}_{\Sigma^{a,\delta_3}_{\tau;\tau_2}}[\phi]}\\
			+\tau^{1/2+2k(1-\epsilon)}\abs{\alpha^{a}_{+}[T^k\phi](\tau)}^2+\tau_2^{2j(1-\epsilon)}\Sigma^{a,\delta_3,}_{\tau;\tau_2}\Big[J^{\E,a}[({\tilde{T}}^a)^j\phi]\Big],
		\end{multline}
		and their coercive counterparts locally, as well as exterior norms
		\begin{subequations}
				\begin{align}
					&\begin{multlined}
						\master^{\ell,0}_{\tau_1,\tau_2}[\phi]:=\sup_{\tau\in[\tau_1,\tau_2]}\sum_{a}\tau_2^{-1}\bar{\E}^0_{\Sigma^{a,\delta_3}_{\tau;\tau_2}}[\phi]+\tau_2^{-1+\epsilon/2}\bar{\E}^0_{\Sigma^{a,\delta_3}_{\tau;\tau_2}}[\Vbe^{k-1}\{t^{1-\epsilon}T^a,t^{-\epsilon}\jpns{y_a^{\tau_2}}X^a\}\phi]\\
						+\bar{\E}^0_{\Sigma^{a,\delta_3,}_{\tau;\tau_2}\setminus \Sigma^{a,\delta_2}_{\tau;\tau_2}}[\Vbe^k\phi]\label{lin:eq:master_local0},
					\end{multlined}\\
					&\master^{\e}_{\tau_1,\tau_2}[\phi]:=\tau_2^{-1}\Big(\sup_{\tau\in[\tau_1,\tau_2]}\Sigma^{\e}_{\tau;\tau_2}[J^{\E,\e,1}[\Vbe^k\phi]]+\sum_a \C^{a,\delta_4}_{\tau_1,\tau_2}[J^{\E,\e,1}[\Vbe^k\phi]]\Big)\label{lin:eq:master_ext0},
				\end{align}
		\end{subequations}
	where we introduced the lossy $\mathrm{b}$ derivatives, which are spanned by $\{t^{1-\epsilon}T^a,t^{-\epsilon}\jpns{y_a^{\tau_2}}X^a,1\}$ in $\Region^a$ and by $\{u^{-\epsilon}\partial_u,u^{-\epsilon}v\partial_v,1\}$ in $\Region^\e$.
	We write $\master^V_{\tau_1,\tau_2}:=\master^{\ell,V}_{\tau_1,\tau_2}+\master^{\e}_{\tau_1,\tau_2}$ and $\master^0_{\tau_1,\tau_2}:=\master^{\ell,0}_{\tau_1,\tau_2}+\master^{\e}_{\tau_1,\tau_2}$.
	\end{definition}

	We already compute what pointwise rates these energy estimates correspond to 
	\begin{lemma}\label{lin:lemma:master_to_Hb}
		Fix a smooth function $\phi$, $\tau_1\in[\tau_2^\Delta,\tau_2]$, $k\geq3$ and $\master^0_{\tau_1,\tau_2}[\phi]\leq1$.
		Then uniformly in $\tau_2$ we have
		\begin{equation}
			\phi\in\Hbloc^{1-3\epsilon,-3\epsilon,-1/2-3\epsilon;3}\subset\Hbloc^{7/10,-3/10,-8/10;3}
		\end{equation}
	\end{lemma}
	\begin{proof}
		In the exterior region, we have
		\begin{equation}
			\master_{\tau_1,\tau_2}^0[\phi]\gtrsim\int_{\tau_1}^{\tau_2}\frac{d \tau}{\tau}\int_{\Sigma^{\mathrm{g}}_{\tau}}\tau^{-1}t\frac{(\Vbe^3\phi)^2}{r^2}=\int_{\mathcal{R}^{\g}_{\tau_1,\tau_2}}\frac{d \mu}{\tau r^3}\frac{t}{r}(\frac{r}{\tau^{1/2}}\Vbe^3\phi)^2\sim\norm{\Vbe^3\phi}_{\Hb^{1,1/2,\infty;0}(\D^{\g}\cap\mathcal{R}^{\e}_{\tau_1,\tau_2})}.
		\end{equation}
		In the interior region, we only have weak control close to the soliton
		\begin{equation}
			\master_{\tau_1,\tau_2}^0[\phi]\gtrsim\int_{\tau_1}^{\tau_2}\frac{d \tau}{\tau}\int_{\Sigma^{\mathrm{c},a}_{\tau;\tau_2}}\frac{(\Vbe^3\phi)^2\big(\tau_2^{-1}\bar{\chi}^c_{\delta_2\tau}(\tilde{y})+\bar{\chi}_{\delta_2\tau_2}(\tilde{y})\big)}{\jpns{\tilde{y}_a}^2}\sim\norm{\Vbe^3\phi}_{\Hbloc^{\infty,0,-1/2;0}(\D^\g\cap\mathcal{R}^a_{\tau_1,\tau_2})}.
		\end{equation}
		
	\end{proof}
	
	We also define the corresponding norms suitable for the inhomogeneity
	
	\begin{definition}[Norms for inhomogeneity]\label{lin:def:norms_inhom}
		For a smooth function $f,\phi$, we define
		\begin{subequations}
			\begin{align}
				&\inhom_{\tau_1,\tau_2}[f]:=\int_{\Region^{\e}}(\Vb^kf)^2\O_{all}^{-2}+\sum_a\int_{\Region^a}(\Vb^kf)^2\O_{all}^{-3}+\sup_{\tau\in[\tau_1,\tau_2]}\int_{\Sigma^{a,\delta_3}_{\tau_1,\tau_2}}\O^{-2,-2}_a(\Vb^kf)^2\\
				&\begin{multlined}
					\inhom^{\mathrm{\lin}}_{\tau_1,\tau_2}[\phi]:=\sum_a\int_{\Region^{a}}\O^{5,2+1/4}_a\big((t^{1-\epsilon}T_a,1)^k\phi\big)^2+\boxed{\O^{2,2}_1T_a(t^{1-\epsilon}\tilde{T}_a,1)^k\phi(\partial,\jpns{y_1^{\tau_2}}^{-1})(t^{1-\epsilon}\tilde{T}_a,1)^k\phi}\\
					+\tau_2^{-1}\Big(\sum_a\int_{\Region^a}
					(T_a\phi,\jpns{y_a^{\tau_2}}^{-1})\phi\O^{3,2}_1\Big)^2
				\end{multlined}\label{lin:eq:inhom_lin}	
			\end{align}
		\end{subequations}
	\end{definition}
	
	\begin{lemma}\label{lin:lemma:inhom_Hb_relation}
		Fix a smooth function $f$, $\tau_1\in[\tau_2^\Delta,\tau_2]$ satisfying $\norm{f}_{\Hbloc^{5/2,3,3/2;k}}\leq1$ and $
		\tau_2^2\master_{\tau_1,\tau_2}^0[f]\leq1$.
		Then uniformly in $\tau_2$ we have
		\begin{equation}
			\inhom_{\tau_1,\tau_2}[f]\lesssim1.
		\end{equation}
	\end{lemma}
	\begin{proof}
		This is standard with extra $3/2,2,1/2$ weights appearing at $\scri,I^+,F_a$ coming from the difference between $\dd\mu$ and $\dd\mu_\b$.
	\end{proof}

	Let us first note, that the coercive and non-coercive estimates are equivalent.
	\begin{cor}
		Let $\phi$ be a function in $\Region$ and assume that \cref{lin:eq:higher unstable bootstrap} holds with $\epsilon=\master^{l,V}_{\tau_1,\tau_2}[\phi]$.
		Then
		\begin{equation}\label{lin:eq:master_coercivity}
			\master^{l,V}_{\tau_1,\tau_2}[\phi]+\inhom_{\tau_1,\tau_2}[f]\gtrsim_{R_2} \master^{l,0}_{\tau_1,\tau_2}[\phi].
		\end{equation}
	\end{cor}
	\begin{proof}
		Using the elliptic estimate, we can control all tangential derivatives.
		Note, that we commuted with $\tilde{T}$ and $T$ for the energy estimate and kernel elements respectively.
		To still be able to use the coercivity estimate \cref{lin:lemma:coercivity_higher}, we use \cref{lin:lemma:elliptic_estimate} to bound the extra $t^{-1}X$ terms present in $\tilde{T}$.
		
		We first apply \cref{lin:lemma:local_midcone_energy_estimate}.
		Then, \cref{lin:eq:master_coercivity} follows from the coercivity and elliptic estimates from \cref{lin:lemma:coercivity_higher,lin:lemma:elliptic_estimate}.
	\end{proof}

	We also know that the energy stays bounded.
	\begin{prop}\label{lin:prop:main}
		Let $\tau_2$ be sufficiently large and $\tau_1\in[\tau_2^\Delta,\tau_2]$, $k\geq1$.
		Let $\phi$ be a solution to \cref{lin:eq:main linearised} in $\Region$ satisfying \cref{lin:eq:higher unstable bootstrap}.
		Then we have the estimate
		\begin{equation}\label{lin:eq:boundedness_non-coercive}
			\master^{V}_{\tau_1,\tau_2}[\phi]\lesssim_{k} \master^{V}_{\tau_2,\tau_2}[\phi]+\inhom_{\tau_1,\tau_2}[f]+\inhom^{\mathrm{\lin}}_{\tau_1,\tau_2}[\phi].
		\end{equation}
		Furthermore, 
		there exists a constant  such that 
		\begin{equation}\label{lin:eq:boundedness}
			\master^{V}_{\tau_1,\tau_2}[\phi]\lesssim_{k} \master^{V}_{\tau_2,\tau_2}[\phi]+\inhom_{\tau_1,\tau_2}[f]
		\end{equation}
	\end{prop}
	\begin{proof}
		The proof follows from the work so far.
		We apply \cref{lin:lemma:ext_energy_estimate} to obtain control of the solution in $\Region^{\e}$ and the cones $\C^a_{\tau_1,\tau_2}$.
		We use \cref{lin:lemma:equivalence of energies} to control the incoming fluxes to the local regions.
		We use these fluxes in $\Region^a$ via \cref{lin:lemma:local_energy_estimate,lin:prop:Theta_conservation} to obtain \cref{lin:eq:boundedness_non-coercive}.
		
		For the second estimate, we need to bound the linear part $\inhom^{\mathrm{\lin}}$.
		We start with the wort decaying term, the boxed in \cref{lin:eq:inhom_lin} .
		We first notice that the lowest in regularity term can be bounded using \cref{lin:lemma:localisation_improved}
		\begin{nalign}\label{lin:eq:proof_lin_inhom1}
			\int_{\Region^1}\O^{2,2}_1T\phi\partial\phi\lesssim \tau_2^{-1}\Big(\int_{\Region^1}\O^{0,1}_1(\partial\phi)^2\Big)^{1/2}\Big(\int_{\Region^1}\O_1^{2,1}(T\phi)^2\Big)^{1/2}\\
			\lesssim(\log\tau_2)^2\prod_{j\in\{0,1\}}\Big(\bar{\E}^V[T^j\phi]+\abs{\alpha_\pm[T^j\phi]}^2+\tau_2^{-1}\big(\Theta^c[T^j\phi]+\Theta^m[T^j\phi]\big)\Big).
		\end{nalign}
		Since, we have that $T$ commuted quantities decay with an extra $t^{1-\epsilon}$ weight, we have an extra smallness factor to obtain
		\begin{equation}
			\int_{\Region^1}\O^{2,2}_1T\phi\partial\phi\lesssim\master^{\ell,V}_{\tau_1,\tau_2}[\phi]+\inhom^{}_{\tau_1,\tau_2}[f].
		\end{equation}
		Notice already, that the same idea works to bound higher higher order boxed terms from \cref{lin:eq:inhom_lin}, but not at top order.
		There, we must use that we have an improved local energy control with weight  $\tau^{-1+\epsilon/2}$ instead $\tau^{-1}$ from \cref{lin:eq:master_local0}.
		In conclusion, we have
		\begin{equation}
			\int_{\Region^1}\O^{2,2}_1T_a(t^{1-\epsilon}\tilde{T}_a,1)^k\phi(\partial,\jpns{y_1^{\tau_2}}^{-1})(t^{1-\epsilon}\tilde{T}_a,1)^k\phi\lesssim\master^{\ell,V}_{\tau_1,\tau_2}[\phi]+\inhom^{}_{\tau_1,\tau_2}[f].
		\end{equation}
		
		We bound the rest of the terms as
		\begin{nalign}\label{lin:eq:linear_inhomogeneity_absorbtion}
			\sum_a\tau_2^{-2-1/8}\int_{\Region^{a}}\O^{3-1/8,1/8}_a\big((t^{1-\epsilon}T,1)^k\big)^2
			+\tau_2^{-1}\Big(\sum_a\int_{\Region^a}
			(T\phi,\jpns{y_a^{\tau_2}}^{-1})\phi\O^{3,2}\Big)^2\\ \lesssim\sum_a\tau_2^{-1-1/8}\int_{\tau_1}^{\tau_2}\master^0[\phi]
		\end{nalign}
		Using \cref{lin:eq:master_coercivity,lin:eq:boundedness_non-coercive} implies \cref{lin:eq:boundedness}.
	\end{proof}
	
	We finish with the transition of $\mathcal{X}^V$ from one slab to the next.
	\begin{prop}\label{lin:prop:transition}
		Let $\tau_2$ be sufficiently large, $\tau_1=\tau_2^\Delta$ and
		let $\phi$ be a solution to \cref{lin:eq:main linearised} in $\Region$ satisfying \cref{lin:eq:higher unstable bootstrap}.
		Then 
		\begin{equation}
			\master^V_{\tau_2^\Delta+1,\tau_2^\Delta+1}\lesssim_{R_2,k} \master^{V}_{\tau_2,\tau_2}[\phi]+\inhom_{\tau_1,\tau_2}[f].
		\end{equation}
	\end{prop}
	\begin{proof}
		First of all, we note that using \cref{not:lemma:spacetime_regions} we get that the domain of definition for $\master^V_{\tau_2^\Delta+1/2,\tau_2^\Delta+1}$ is in $\Region$.
		The energy estimate for the exterior region follows from \cref{lin:lemma:transitional_energy_estimate}.
		
		For the interior energy estimate, let's consider $a=1$.
		We pick $\tau\in[\tau_2^\Delta,\tau_2]$ such that the flat part of $\Sigma^1_{\tau,\tau_2}$ is the same as $\Sigma^1_{\tau_2^{\Delta},\tau_2^{\Delta}}$.
		We apply an energy estimate in the region with boundaries $\Sigma^{1,\delta_3}_{\tau,\tau_2}, \Sigma^1_{\tau_2^{\Delta},\tau_2^{\Delta}}\cap \D^{\mathrm{c},1,\delta_3}_{\tau_2}$ and part of $\C^{1,\delta_3}_{\tau_2^{\Delta},\tau_2}$.
		Since this is an exterior region, we can use the vectorfield $J^{\E,\e,1}$ from \cref{lin:lemma:ext_energy_estimate}.
		
		We propagate the kernel elements by their respective fluxes in the same interior region.
		Finally, we truncated the evaluation of the kernel elements via \cref{lin:lemma:transition}.
	\end{proof}

	\subsection{Cancellation for error terms}\label{lin:sec:error_term_cancellation}
	
	In this section so far, we assumed that $\bar{\phi}$ is admissible for energy estimate according to \cref{lin:def:admissible}.
	This was necessary, as the left hand side \cref{lin:eq:linear_inhomogeneity_absorbtion} would be impossible to bound if we had \cref{lin:eq:assumptionPhibar_weak} instead of \cref{lin:eq:assumptionPhibar_strong}.
	More precisely, we used the stronger assumptions only in the boxed terms in \cref{lin:eq:scaling_critical,lin:eq:energy_critical}.
	
	In this section, we show that a certain class of perturbations $\bar{\phi}$ resulting in $\mathfrak{E}^l\in\O^{5,1}$ term is allowed, provided that these satisfy an appropriate orthogonality condition.
	In this section, we will study how the error terms change when we relax \cref{lin:eq:assumptionPhibar_strong} to
	\begin{equation}\label{lin:eq:assumptionPhibar_structured}
		\bar{\phi}=\sum W_a+\frac{\log t_a}{t_a}\Lambda W_a+\O_{\loc}^{2,2}.
	\end{equation}
	Including another $t^{-1}$ decaying term would also be admissible.
	We leave this out from the discussion.
	Alternatively to \cref{lin:eq:assumptionPhibar_structured}, we could take the scaling parameters $\Lambda$ to depend on time as in \cref{an:thm:modulated_scaling}, however we find it helpful to work with the explicit form \cref{lin:eq:assumptionPhibar_structured}.
	
	Let us show on the example \cref{lin:eq:energy_critical}, why the estimates from \cref{lin:eq:energy_critical} are borderline insufficient.
	The corresponding extra boxed error term takes the form
	\begin{equation}
		\int_{\Region^1} \frac{\log t}{t^2}\phi^2\Lambda W_1 W_1^3\lesssim\int_{\Region^1} \frac{c_1\log t}{t}\master^0_{\tau_1,\tau_2}.
	\end{equation}
	The $\log^2$ term does not allow us to close the estimate via Gronwall inequality.
	Note, that we cannot use any improvement that is obtained by commuting with $T$ as in \cref{lin:eq:proof_lin_inhom1}, as this failure happens in an undifferentiated term.
	
	We present the extra cancellation now.
	We note that \cref{M1} has a localised version \cite{Mathematica}
	\begin{equation}\label{lin:eq:localised_M1}
		\int_{\abs{x}\leq R}W^3(\Lambda W)^3=\frac{9R^3(45-6R^2+5R^4)}{40(3+R^2)^5}\lesssim R^{-3}.
	\end{equation}
	We also have a similar cancellation for the other kernel element
	\begin{equation}\label{lin:eq:localised_M1_com}
		\int_{\abs{x}\leq R}W^3(\Lambda W)(\partial_i W)^2=\frac{9R^5}{10(3+R^2)^5}\lesssim R^{-5}.
	\end{equation}
	Indeed, this orthogonality will be the sufficient requirement to be able to absorb the $\O_1^{1,(1,0)}$ term of $\bar{\phi}$.
	
	\begin{lemma}
		Let $\phi$ be a smooth solution in $\mathcal{R}_{\tau_1,\tau_2}^1$ for $\tau_1\in[\tau_2^\Delta,\tau_2]$.
		Assume that
		\begin{equation}\label{lin:eq:control_bad_error_energy}
			\masterZero^V_{\tau_1,\tau_2}[\phi]\leq1
		\end{equation}	
		Then
		\begin{equation}
			\int_{\mathcal{R}_{\tau_1,\tau_2}^1} \frac{\log t_a}{(t_a)^2}\phi^2(\Lambda W_1) W_1^3\lesssim_{R_2}\frac{\log\tau_2}{\tau_2^{1/2}},\qquad \int_{\mathcal{R}_{\tau_1,\tau_2}^1} \frac{\log t_a}{(t_a)^2}\phi(\Lambda W_1)^2 W_1^3\lesssim_{R_2}\frac{\log\tau_2}{\tau_2^{1/2}}.
		\end{equation}
	\end{lemma}
	\begin{proof}
		The proof is almost identical for the two integrals.
		We only present it for the first one.
		
		Note, that if we controlled $\Theta^\Lambda,\Theta^c$ with the same bound as for $\E^V$, the estimate would be trivial, as $\E^0$ control would be integrable in time.
		We write the error term as
		\begin{multline}
			\int_{\mathcal{R}_{\tau_1,\tau_2}^1} \frac{\log t_a}{(t_a)^2}\phi^2(\Lambda W_1) W_1^3=\int_{\mathcal{R}_{\tau_1,\tau_2}^1} \phi^2(\Lambda W_1) W_1^3\Big(\frac{\log t_a}{(t_a)^2}-\frac{\log t_{a,\star}}{(t_{a,\star})^2}\Big)+\int_{\mathcal{R}_{\tau_1,\tau_2}^1} \frac{\log t_{a,\star}}{(t_{a,\star})^2}\phi^2(\Lambda W_1) W_1^3
		\end{multline}
		The integrand multiplying $\phi^2$ in the first integral is in $\O^{8,3}$, which can be controlled with the techniques from the previous sections.
		Similarly, we can change $W_1$ to $W_1^\star$ as in \cref{lin:eq:Wt-Wt_star}, with the other error controlled in $\O^{6,3}_1$.
		In conclusion, it is sufficient to control
		\begin{equation}
			\int_{\mathcal{R}_{\tau_1,\tau_2}^1} \frac{\log t_{a,\star}}{(t_{a,\star})^2}\phi^2(\Lambda W^\star_1) (W^\star_1)^3.
		\end{equation}
		In turn, this will follow from the estimate on a single time slice $\Sigma=\Sigma^{1,\delta_3}_{\tau,\tau_2}$:
		\begin{multline}\label{lin:eq:improve_proof2}
			\int_{\Sigma^{1,\delta_3}_{\tau,\tau_2}}\phi^2(\Lambda W_1^\star) (W^\star_1)^3\lesssim_{R_2}\bar{\E}_\Sigma^V[\phi]+\bar{\E}^V_{\Sigma^{1,\delta_3}_{\tau,\tau_2}\setminus\Sigma^{1,\delta_3/2}_{\tau,\tau_2}}[\phi]+\abs{\alpha_\pm[\phi]}^2\\
			+\sum_{\bullet\in{\Lambda,c}}(\bar{\E}_\Sigma^V[\phi]+\bar{\E}^V_{\Sigma^{1,\delta_3}_{\tau,\tau_2}\setminus\Sigma^{1,\delta_3/2}_{\tau,\tau_2}}[\phi])^{1/2}\abs{\Theta_{\Sigma^{1,\delta_3}_{\tau,\tau_2}}^\bullet[\phi]}+\abs{\Theta_{\Sigma^{1,\delta_3}_{\tau,\tau_2}}^\bullet[\phi]}^2\tau_2^{-1/2}\lesssim\tau_2^{1/2}\masterZero^V_{\tau_1,\tau_2}.
		\end{multline}
		Note, that this is a $\tau_2^{-1/2}$ improvement compared to using the $\tau^{-1}\bar{\E}^0[\phi]$ control provided by $\masterZero^V_{\tau_1,\tau_2}[\phi]$.
		We write $\phi=\psi+c_{\Lambda}\Theta_{\Sigma}^{\Lambda}[\phi]\Lambda W^\star_1+c_{\nabla}\Theta^{c}_{\Sigma}[\phi]\cdot\nabla W^\star_1$, where $c_{\Lambda},c_{\nabla}$ are normalisation constants, so that we have
		$\Theta^\bullet_{\Sigma}[\psi]=0$.
		We apply \cref{lin:lemma:EV_shift} to control $\E^V[\psi]$ in terms of $\E^V[\phi]$ and $\Theta^{\Lambda},\Theta^{c}$
		\begin{equation}\label{lin:eq:improve_proof1}
			\abs{\E_\Sigma^V[\psi]-\E_\Sigma^V[\phi]}\lesssim \sum_\bullet \tau^{-1}(\Theta_{\Sigma}^{\bullet}[\phi])^2+\tau^{-1/2}(\E^V_{\Sigma^{1,\delta_3}_{\tau,\tau_2}\setminus\Sigma^{1,\delta_3/2}_{\tau,\tau_2}}[\phi])^{1/2}\abs{\Theta_{\Sigma}^{\bullet}[\phi]}.
		\end{equation}
		Using this decomposition together with the estimates \cref{lin:eq:localised_M1,lin:eq:localised_M1_com}
		we compute
		\begin{multline}
			\int_{\Sigma^{1,\delta_3}_{\tau,\tau_2}}\phi^2(\Lambda W_1^\star) (W^\star_1)^3\lesssim \int_{\Sigma}(\psi^2)(\Lambda W_1^\star) (W^\star_1)^3+\abs{\Theta^{\Lambda}_\Sigma[\phi]}^2\tau^{-3}+\abs{\Theta^{c}_\Sigma[\phi]}^2\tau^{-5}\\
			+(\abs{\Theta^\Lambda_\Sigma[\phi]}+\abs{\Theta^c_\Sigma[\phi]})\int_{\Sigma}\abs{\psi}\jpns{r}^{-7}.
		\end{multline}
		Using Cauchy-Schwarz on the linear in $\psi$ term together with the coercivity estimate \cref{lin:cor:coercivity_estimate}, we get
		\begin{multline}
			\int_{\Sigma^{1,\delta_3}_{\tau,\tau_2}}\phi^2(\Lambda W_1^\star) (W^\star_1)^3\lesssim \bar{\E}_\Sigma^V[\psi]+\bar{\E}^V_{\Sigma^{1,\delta_3}_{\tau,\tau_2}\setminus\Sigma^{1,\delta_3/2}_{\tau,\tau_2}}[\phi]+(\abs{\Theta^{\Lambda}_\Sigma[\phi]}^2+\abs{\Theta^{c}_\Sigma[\phi]}^2)\tau^{-3}+\\
			(\bar{\E}_\Sigma^V[\psi]+\bar{\E}^V_{\Sigma^{1,\delta_3}_{\tau,\tau_2}\setminus\Sigma^{1,\delta_3/2}_{\tau,\tau_2}}[\phi])^{1/2}(\abs{\Theta^{\Lambda}_\Sigma[\phi]}+\abs{\Theta^{c}_\Sigma[\phi]}).
		\end{multline}
		Finally, we use \cref{lin:eq:improve_proof1} to obtain \cref{lin:eq:improve_proof2}.
	\end{proof}

	Similarly, we have the higher order estimate
	\begin{lemma}\label{lin:lemma:cancellation}
		Let $\phi$ be a smooth solution in $\mathcal{R}_{\tau_1,\tau_2}^1$ for $\tau_1\in[\tau_2^\Delta,\tau_2]$.
		Assume that
		\begin{equation}
			\inhom[f]_{\tau_1,\tau_2}+\master^V_{\tau_1,\tau_2}[\phi]\leq1
		\end{equation}	
		Then for all $j,j'\leq k$ we have
		\begin{nalign}
			\tau_2^{(j+j')(1-\epsilon)}\int_{\mathcal{R}_{\tau_1,\tau_2}^1} \frac{\log t_a}{(t_a)^2}T^j\phi T^{j'}\phi (\Lambda W_1) W_1^3\lesssim_{R_2}\frac{\log\tau_2}{\tau_2^{1/2}},\\ \tau_2^{j(1-\epsilon)}\int_{\mathcal{R}_{\tau_1,\tau_2}^1} \frac{\log t_a}{(t_a)^2}T^j\phi(\Lambda W_1)^2 W_1^3\lesssim_{R_2}\frac{\log\tau_2}{\tau_2^{1/2}}.
		\end{nalign}
	\end{lemma}
	\begin{proof}
		The proof is identical to the previous case, but we need to use \cref{lin:eq:master_coercivity} to obtain coercive control.
	\end{proof}

	\begin{prop}\label{lin:prop:main_weak}
		Let $\bar{\phi}$ satisfy
		\begin{equation}
			\bar{\phi}=\sum W_a+\frac{c^{2,1}_{\Lambda,a}\log t_a}{t_a}(\Lambda W_a)+\frac{c^{2,0}_{\Lambda,a}}{t_a}(\Lambda W_a)+\O_{\loc}^{2,2}.
		\end{equation}
		Furthermore, let $\tau_2$ be sufficiently large and $\tau_1\in[\tau_2^\Delta,\tau_2]$.
		For $\phi$ a solution to \cref{lin:eq:main linearised} in $\Region$ satisfying \cref{lin:eq:higher unstable bootstrap}, we have that \cref{lin:eq:boundedness,lin:eq:boundedness_non-coercive} hold.
	\end{prop}
	\begin{proof}
		The proof is essentially the same as for \cref{lin:prop:main}, except that we need to take care of the extra $\frac{\log t_a}{t_a}W_a^3\Lambda W_a$ term.
		This is only problematic in \cref{lin:eq:scaling_critical,lin:eq:energy_critical}.
		For these, we use the improved estimate \cref{lin:lemma:cancellation}, instead of $\master^0$ control.
	\end{proof}
	
	\paragraph{Super critical problems}
	
	We also note, that the solutions constructed in \cref{an:thm:supercritical} also have $\mathfrak{E}^l=\O_{\loc}^{2,1}$.
	For these, we also have an orthogonality condition as shown in \cref{an:eq:supercritical_cancellation}:
	\begin{equation}\label{lin:eq:supercritical_cancellation}
		\int\abs{\partial_i W_1^{\mathrm{sup}}}^2P^a_{1,0}\mathfrak{E}^{l}=0.
	\end{equation}
	Observing that $P^a_{1,0}\mathfrak{E}^{l}\in\O^{(3,0)}_{\R^3}$, we get that the integrand in \cref{lin:eq:supercritical_cancellation} is in $\O^{(7,0)}_{\R^3}$.
	Therefore, when performing the sequence of integrations by parts in \cref{an:eq:supercritical_cancellation}, we get
	\begin{equation}\label{lin:eq:supercritical_cancellation_loc}
		\int_{\abs{x}\leq R}\abs{\partial_i W_1^{\mathrm{sup}}}^2P^a_{1,0}\mathfrak{E}^{l}\in\O_{\R}^{4,0}.
	\end{equation}
	This is even stronger than \cref{lin:eq:localised_M1}, which was used to prove \cref{lin:prop:main_weak}.
	Unlike for \cref{i:eq:lin}, the linearised problem for supercritical equations may admit any number of finitie unstable modes.
	Calling these $Y_{i}$ and the corresponding projections $\alpha_{\pm}^{a,i,\mathrm{s}}$ defined analogously as \cref{lin:eq:unstable_definition}, we can use the bootstrap assumption
	\begin{equation}\label{lin:eq:higher unstable bootstrap_sup}
		\abs{\alpha^{a,i,\mathrm{s}}_-[{T^k\phi}]}\leq \abs{\tau}^{-k-1/10}
	\end{equation}
	to control the unstable modes.
	
	By same same method, we conclude
	\begin{prop}\label{lin:prop:supercritical}
		Let $g^{1,0}_a$ be as in \cref{an:eq:supercritical_correction_1} and $\bar{\phi}$ satisfy
		\begin{equation}
			\bar{\phi}=\sum_a W_a^{\mathrm{sup}}-g^{1,0}_a+\O^{2,2}_{\loc}.
		\end{equation}
		Furthermore, let $\tau_2$ be sufficiently large and $\tau_1\in[\tau_2^{\Delta},\tau_2]$.
		Let $\master^{\mathrm{s}}$ be defined as in \cref{lin:def:norms_master}, with the exclusion of $\Theta^{a,\Lambda}$ projections \textit{and} the $\Theta^{a,c,\mathrm{s}} ,\Theta^{a,m,\mathrm{s}}$ defined mutatis mutandis.
		For $\phi$ a solution to \cref{lin:eq:supercritical_main}, under the assumption \cref{lin:eq:higher unstable bootstrap_sup},  \cref{lin:eq:boundedness,lin:eq:boundedness_non-coercive} hold.
	\end{prop}
	
	\newpage
	\section{Non-linear theory}\label{sec:non-linear}
	In this section, we finish the proof of \cref{i:thm:prosaic}.
	We introduce the class of solutions for which, our non-linear theory applies.
	\begin{definition}\label{non:def:admissible}
		We say that the equation \cref{lin:eq:main linearised}
		is admissible for \emph{nonlinear estimates} if $\bar{\phi}\in\O_{\loc}^{(1,0),(0,0)}$, and \cref{lin:eq:boundedness_non-coercive,lin:eq:boundedness} hold.
	\end{definition}
	For the rest of this section, we assume that \cref{lin:eq:main linearised} is admissible for nonlinear estimates.
	We study nonlinear solutions to 
	\begin{equation}\label{non:eq:main}
		\begin{gathered}
			(\Box+5\bar{\phi}^4)\phi=f-\mathcal{N}[\phi;\bar{\phi}]
		\end{gathered}
	\end{equation}
	where  $\mathcal{N}$ as in \cref{an:def:nonlinear_operators}
	
	Before we study solutions to \cref{non:eq:main}, let us show a few embedding for $\inhom$ and $\master^0$ spaces.
	\begin{lemma}[Nonlinear estimate]\label{non:lemma:embedding}
		Let $k\geq3$ and $\phi$ be a smooth function in $\Region$ with $\master^0_{\tau_1,\tau_2}[\phi]\leq\epsilon$.
		Then
		\begin{equation}
			\inhom_{\tau_1,\tau_2}[\mathcal{N}[\phi;\bar{\phi}]]\lesssim\epsilon^2\tau_2^6+\epsilon^5\tau_2^9
		\end{equation}
	\end{lemma}
	\begin{proof}
		We bound the quadratic and the quintic in $\phi$ terms, the others following similarly.
		We use \cref{lin:lemma:inhom_Hb_relation,lin:lemma:master_to_Hb} to get
		\begin{nalign}
			\inhom_{\tau_1,\tau_2}[\bar{\phi}^3\phi^2]\lesssim\norm{\bar{\phi}^3\phi^2}_{\Hbloc^{5/2,3,3/2;k}(\Region)}^2+\tau_2^2\master^0[\bar{\phi}^3\phi^2]\\
			\lesssim\norm{\phi^2}_{\Hbloc^{-1/2,0,3/2;k}(\Region)}^2+\tau_2^{4}\master^0[\phi]^2\lesssim\epsilon^2\tau_2^{6}.
		\end{nalign}
		Similarly, for the quintic, we have
		\begin{equation}
			\inhom_{\tau_1,\tau_2}[\bar{\phi}^3\phi^2]\lesssim\norm{\phi^5}_{\Hbloc^{5/2,3,3/2;k}(\Region)}^2+\tau_2^2\master^0[\phi^5]\lesssim\epsilon^5\tau_2\norm{1}_{\Hbloc^{-5/2,1/2,4;k}(\Region)}^2+\tau_2^7\master^0[\phi]^5\lesssim\epsilon^5\tau_2^{9}.
		\end{equation}
		
	\end{proof}

	\subsection{Energy boundedness in a slab}\label{sec:non-linear_slab}
	The main goal of the present section is to prove existence in a large but bounded region provided the unstable bootstrap assumption holds.
	Then, we propagate this bound to the next region of spacetime.
	
	\begin{cor}\label{non:cor:slab}
		Let $k>3$,$\tau_2\gg1$ and $\tau_1\in[\tau_2^\Delta,\tau_2]$.
		Let $\phi$ be a solution to \cref{non:eq:main} in $\Region$ satisfying \cref{lin:eq:higher unstable bootstrap}
		with $\epsilon<\tau_2^{-7}$.
		We have for $\master_{\tau_2,\tau_2}^V,\inhom_{\tau_1,\tau_2}[f]\leq\epsilon$  that
		\begin{equation}
			\master_{\tau_1,\tau_2}^V\lesssim_{k,R_2}\master_{\tau_2,\tau_2}^V+\inhom_{\tau_1,\tau_2}[f].
		\end{equation}
	\end{cor}
	\begin{proof}
		Under the bootstrap assumption, we simply apply \cref{lin:prop:main} for the solution $\phi$ to get
		\begin{equation}
			\master_{\tau_1,\tau_2}^V\lesssim_{k,R_2}\master_{\tau_2,\tau_2}^V+\inhom_{\tau_1,\tau_2}[f]+\inhom_{\tau_1,\tau_2}[\mathcal{N}[\phi;\bar{\phi}]].
		\end{equation}
		We apply \cref{non:lemma:embedding,lin:cor:coercivity_estimate} to get
		\begin{equation}
			\master_{\tau_1,\tau_2}^V\lesssim_{k,R_2}\master_{\tau_2,\tau_2}^V+\inhom_{\tau_1,\tau_2}[f]+(\master^V_{\tau_1,\tau_2})^2\tau_2^6.
		\end{equation}
		We use the assumption that $\epsilon<\tau_2^{-7}$ to absorb the right most term to the left hand side and obtain the result.
	\end{proof}
	
	\begin{cor}\label{non:cor:transition}
		Let $k>3$,$\tau_2\gg1$ and $\tau_1\in[\tau_2^\Delta,\tau_2]$.
		Let $\phi$ be a solution to \cref{non:eq:main} in $\Region$ satisfying \cref{lin:eq:higher unstable bootstrap}
		with $\epsilon<\tau_2^{-7}$.
		We have for $\master_{\tau_2,\tau_2}^V,\inhom_{\tau_1,\tau_2}[f]\leq\epsilon$  that
		\begin{equation}
			\master_{\tau^\Delta_2+1,\tau^\Delta_2+1}^V\lesssim_{k,R_2} \master_{\tau_2,\tau_2}^V+\inhom_{\tau_1,\tau_2}[f].
		\end{equation}
	\end{cor}
	\begin{proof}
		We use \cref{lin:prop:transition} instead \cref{lin:prop:main}, and also use that the nonlinear terms are already controlled via \cref{non:cor:slab}.
	\end{proof}
	
	\subsection{Dyadic iteration}\label{sec:non-linear_dyadic}
	
	In this section, we use the method developed by Dafermos-Holzegel-Rodnianski-Taylor in \cite{dafermos_quasilinear_2022} in a robust way to construct scattering solutions for equations where only a weak form of energy boundedness is known. 
	Let $C^{\mathrm{dyad}}(k,R_2)$ be the implicit constant in \cref{non:cor:slab,non:cor:transition}.
	
	Let's first define the sequence of times where the foliation ends.
	For $\tau>0$ we recursively define $\tau_{1}:=\tau$, $\tau_{m+1}:=\tau_{m}^\Delta+1$.
	Next, we introduce the associated dyadic norm for a smooth function $\phi$
	\begin{equation}
		\master_{\tau,m}^{V;q}[\phi]=\sup_{n\in\N_{\leq m}}\master^V_{\tau_{n}^\Delta,\tau_n}[\phi]\tau_n^{2q}.
	\end{equation}
	We also introduce the minimum decay rate $q_0$ satisfying $C^{\mathrm{dyad}}(1-\delta_3/8)^{q_0}=1/2$.
	We finally introduce the bootstrap assumption
	\begin{equation}\label{non:eq:global_unstable_bootstrap}
		\norm{\alpha^a_-[{T^k\phi}](\tau)}_{l^2(a)}< \abs{\tau}^{-k-q-1/10}.
	\end{equation}
	\begin{remark}
		Note, that the decay rate of $\alpha_-$ is slightly faster than the decay rate of the energy quantities.
		The reason is, that in the cutoff region is localised at a logarithmically growing region, so we have an improved control.
	\end{remark}

	The first result, is that solution to \cref{non:eq:main} can only fail to be bounded if the bootstrap assumption fails.
	\begin{lemma}[Uniform boundedness]\label{non:lemma:uniform_boundedness}
		Fix $\tau_{\mathrm{last}}$ sufficiently large, $q>(q_0,7),k>3$ and $\phi$ a smooth solution to \cref{non:eq:main} satisfying \cref{non:eq:global_unstable_bootstrap}.
		Pick $\tau>\tau_{\mathrm{last}}$ and $\tau_m>\tau_{\mathrm{last}}$.
		Assume furthermore that
		\begin{equation}\label{non:eq:condition_on_f_bulk}
			\sup_{n\in\N_{\leq m}}C^{\mathrm{dyad}}\tau_{n+1}^{2q}\inhom_{\tau^{\Delta}_n,\tau_n}[f],\master^V_{\tau,\tau}[\phi]\tau^{2q}<1/2.
		\end{equation}
		Then, we have 
		\begin{equation}
			\master^{V;q}_{\tau,m}[\phi]<2 C^{\mathrm{dyad}}.
		\end{equation}
	\end{lemma}
	\begin{proof}
		As long as $\master^V_{\tau_n,\tau_n}[\phi],\inhom_{\tau_n^\Delta,\tau_n}[f]\leq \tau_n^{-7}$, we can apply \cref{non:cor:transition} to obtain
		\begin{equation}
			\master^V_{\tau_{n+1},\tau_{n+1}}[\phi]\leq C^{\mathrm{dyad}}\Big(\master^V_{\tau_n,\tau_n}[\phi]+\inhom_{\tau^{\Delta}_n,\tau_n}[f]\Big).
		\end{equation}
		Multiplying both sides by $\tau_{n+1}^{q}$, we get
		\begin{equation}
			\tau_{n+1}^{q}\master^V_{\tau_{n+1},\tau_{n+1}}[\phi]\leq C^{\mathrm{dyad}}(1-\delta_3/8)^q\tau_{n}^{q}\master^V_{\tau_n,\tau_n}[\phi]+C^{\mathrm{dyad}}\tau_{n+1}^{q}\inhom_{\tau^{\Delta}_n,\tau_n}[f].
		\end{equation}
		Therefore, for $C^{\mathrm{dyad}}(1-\delta_3/8)^q<1/2$, and 
		\begin{equation}
			\sup_{n\in\N_{\leq m}}C^{\mathrm{dyad}}\tau_{n+1}^{q}\inhom_{\tau^{\Delta}_n,\tau_n}[f]<1/2,
		\end{equation}
		we obtain by induction that $\tau_{n+1}^{q}\master^V_{\tau_{n+1},\tau_{n+1}}[\phi]\leq1$.
		The estimate on the dyadic norm follows after an application of \cref{non:cor:slab} in each slab.
	\end{proof}
	
	\subsection{Unstable mode}
	Let us introduce the initial data set we are interested in.
	For $c\in\R^d$, where $d=\abs{A}$ is the number of solitons, let's consider initial data
	\begin{nalign}\label{non:eq:modified_initial_data}
		\phi|_{\Sigma^{a}_{\tau,\tau}}=c_aY(\tilde{y}_a)\bar{\chi}^c_{2R_1} \tau^{-q-k-1/10}\\
		T_a\phi|_{\Sigma^{a}_{\tau,\tau}}=-c_a\gamma_a\lamed Y(\tilde{y}_a)\bar{\chi}^c_{2R_1}\tau^{-q-k-1/10}
	\end{nalign}
	In order to start our solution, and apply the local existence result from \cref{scat:thm:existence_of_scattering_solution}, we need to exhibit a solution on an open neighbourhood $\mathcal{R}_{\tau-\epsilon,\tau}$.
	Obtaining this for \cref{non:eq:main} is difficult in general due to the characteristic nature of the problem.
	We resolve this issue, by truncating the inhomogeneity $f$ with a cutoff $\chi=\bar{\chi}(t^\g_\star/\tau-1)$ and study solutions to
	\begin{equation}\label{non:eq:main_cutoff}
		(\Box+\Vbold+\mathfrak{Err}^{\lin}[\bar{\phi}])\phi=\chi f-\mathcal{N}[\phi;\bar{\phi}].
	\end{equation}
	Since \cref{non:eq:main_cutoff} has no inhomogeneity in a neighbourhood of $\mathcal{R}_{\tau,\tau}$ and has trivial data outside compact subsets of the flat parts, it follows by local existence and domain of dependence that a solution with no outgoing radiation to \cref{non:eq:main_cutoff} with data \cref{non:eq:modified_initial_data} exists in $\mathcal{R}_{\tau-\epsilon,\tau}$ for $\epsilon>0$.

	We call $\phic$ the scattering solution to \cref{non:eq:main_cutoff} with initial data given by \cref{non:eq:modified_initial_data} and no outgoing radiation through $\scri$.
	From \cref{scat:thm:existence_of_scattering_solution}, we know that a solution exists in the region $\mathcal{R}_{\tau-\epsilon,\tau}$ for some $\epsilon>0$.
	We compute
	\begin{nalign}\label{non:eq:starting_alpha}
		\tau^{q+k+1/10}T_a^j\phic|_{\Sigma^1_{\tau,\tau}\cap\{\abs{\tilde{y}}\leq R_1\}}=c_a(-\lamed)^kY(\tilde{y}_a)+\mathfrak{F}^a(c),\\
		\implies \tau^{q+k}\alpha_-^a[T_a^k\phic]=(-\lamed)^kc_a+\mathfrak{F}^a(c)
	\end{nalign}
	where $\mathfrak{F}^a(c)\in\C^{\infty}(\R^3\to\R)$ is a smooth polynomial functional of $c$ changing from line to line, satisfying $\norm{\mathfrak{F}^a(c)}_{l^2_aL^\infty}\lesssim_c \tau^{-1}$.
	
	\begin{lemma}\label{non:lemma:unstable}
		Let $q>\max(q_0,7),k>3$ and let $f$ be such that \cref{non:eq:condition_on_f_bulk} hold and $\tau_{\mathrm{start}}$ sufficiently large.
		Fix $\tau>\tau_{\mathrm{last}}$ and $m$ such that $\tau_m>\tau_{\mathrm{last}}$.
		Then, there exists a choice $c(\tau)\in\R^d$ with $\abs{c(\tau)}\lesssim 1$ independent of $\tau$ such that the scattering solution $\phic$ exists in $\bigcup_{n\leq m} \mathcal{R}_{\tau^\Delta_n,\tau_n}$ with
		\begin{equation}\label{non:eq:global_phic_bound}
			\master^{V;q}_{\tau,m}[\phic]< 2C^{\mathrm{dyad}}.
		\end{equation}
	\end{lemma}

	\begin{proof}
		We take $\tau_{\mathrm{start}}$ sufficiently large, so that \cref{non:lemma:uniform_boundedness} holds.
		Let's define
		\begin{equation}
			\mathcal{R}^{\mathrm{con}}_{\tau'}:=\bigcup_{n:\tau'<\tau^\Delta_n}\mathcal{R}_{\tau^{\Delta}_n,\tau_n}\cup  \bigcup_{n:\tau'\in[\tau^\Delta_n,\tau_n]}\mathcal{R}_{\tau',\tau_n}.
		\end{equation}
		For any choice of $c$, we know that there exists a unique local solution in some $\mathcal{R}^{\mathrm{con}}_{\tau'}$.
		Let's define the function $\beta:\R^d\to\R$
		\begin{equation}\label{non:eq:def_brouwer_function}
			\beta(c):=\max\Big(\min\big(\inf\{\tau':\exists\phic \text{ in }\mathcal{R}_{\tau'}^{\mathrm{con}} \text{ satisfying } \cref{non:eq:main,non:eq:global_unstable_bootstrap,non:eq:global_phic_bound}\},\tau\big),\tau_m\Big).
		\end{equation}
		Let us furthermore define $\tilde{\beta}:\R^d\to\R$ to be the minimum value of $t^{\g}_\star$ in $\mathcal{R}^{\mathrm{con}}_{\tau'}$ for the same region as in \cref{non:eq:def_brouwer_function}.
		By construct $\beta$ and $\tilde{\beta}$ are related by a continuous function.
		
		\textit{Boundary values:}		
		Let's fix some $C_1$ such that for $c\in \partial B^{C_1}$ where $B^{C_1}:=\{\abs{x}\leq C_1\}$, we have 
		\begin{equation}
			\tau^{k+q+1/10}\norm{\alpha^a_{-}[T^k\phic](\tau)}_{l_a^2}\in(1.9,2.1).
		\end{equation}
		The existence of such a $C_1$ follows from \cref{non:eq:starting_alpha}.
		We choose $\tau_{\mathrm{start}}$ sufficiently large, so that for all $c\in B^{C_1}$
		\begin{equation}\label{non:eq:control_frakF}
			\norm{\mathfrak{F}^a(c)}_{l^2_aL^\infty}\leq 1/10.
		\end{equation}
		Indeed, we take 
		Henceforth, let's only consider $\beta|_{B^{C_1}}:B^{C_1}:\to\R$ and write $\beta$ for this restriction.
		We already note that $\beta|_{\partial B^{C_1}}=\tau$
		
		\textit{Continuity:} We claim that $\beta:B^{C_1}\to[\tau_m,\tau]$ is continuous.
		Let $c\in B^{C_1}$ such that $\beta(c)=\tau_m$ and there exists $\phic$ in $\mathcal{R}^{\mathrm{con}}_{\tau_m}$ as in \cref{non:eq:def_brouwer_function}. By Cauchy stability $\beta$ is constant in a neighbourhood of $c$.
		
		Let $c\in B^{C_1}$ be otherwise and set $\tilde{\tau}=\tilde{\beta}(c)$. 
		Then, using \cref{non:lemma:uniform_boundedness} it follows that
		\begin{equation}\label{non:eq:alpha_saturates}
			\tilde{\tau}^{k+q+1/10}\norm{\alpha^a_-[{T^k\phi}](\tilde{\tau})}_{l^2(a)}=1.
		\end{equation} 
		Furthermore, we may use \cref{eq:linear:unstable_derivative} to get
		\begin{nalign}\label{eq:nonlinear:unstable_derivative}
			\begin{multlined}
			\frac{1}{2}\partial_\tau \norm{\alpha^a_-[{T^k\phi}](\tau)}^2_{l^2(a)}|_{\tau=\tilde{\tau}}=-\lamed\sum_{a}\abs{\alpha^a_-[\phic](\tilde{\tau})}^2\gamma_a+\\
				\sum_a\alpha^{a}_-[\phic](\tilde{\tau})\int_{\Sigma_\tau}\chi_{R_1/3}(\tilde{y}_a)e^{-\lamed R_1}\mathcal{O}^{\infty,0}_a(1,\partial_t)T^k\phic+\mathcal{O}^{\infty,1}_a(,\partial_t)T^k\phic\\+\mathcal{O}_a^{\infty,1}T^k(\mathcal{N}[\phic;\bar{\phi}]+f)
			\end{multlined}
		\end{nalign}
		For $\tau_{\mathrm{last}}\gg1$, with implicit constant depending on the functions in the integral in \cref{eq:nonlinear:unstable_derivative}, we may use \cref{non:eq:global_phic_bound} to get
		\begin{equation}
			\frac{1}{2}\partial_\tau (\alpha^{l_2}_-[\phic](\tau))^2|_{\tau=\tilde{\tau}}\lesssim \tilde{\tau}^{-2(q+k+1/10)}.
		\end{equation}
		By Cauchy stability, $\partial_\tau (\alpha^{l_2}_-[\phic](\tau))^2|_{\tau=\tilde{\tau}}$ is nonzero in a neighbourhood of $c$. Thus $\beta$ is continuous.

		\textit{Topological argument:} Now, we claim that there exists $c\in B^{C_1}$ such that $\beta(c)=\tau_m$. 
		Assume otherwise. 
		Let us define
		\begin{equation}
			\Phi(c)=\frac{\tau^{q+k}\alpha^{a}_-[\phic](\tilde{\beta}(c))-\mathfrak{F}(c)}{\norm{\tau^{q+k}\alpha^{a}_-[\phic](\tilde{\beta}(c))-\mathfrak{F}(c)}_{l^2(a)}}\in S^{d-1}
		\end{equation}
		 which is well-defined by \cref{non:eq:control_frakF} and the assumption that \cref{non:eq:alpha_saturates} holds for all $c\in B^{C_1}$.
		 Let us also compute for $c\in\partial B^{C_1}$
		 \begin{equation}
		 	\Phi(c)=\frac{c_a}{\norm{c_a}_{l^2(a)}}.
		 \end{equation}
		 Rescaling the spheres, we get a retraction on the unit ball.
		 By Brouwer fixed point theorem we get a contradiction.
	\end{proof}

	\begin{theorem}[Existence of scattering solution]\label{non:thm:existence of scattering solution}
		Let $q>max(q_0,6),k>3$ and let f be such that \cref{non:eq:condition_on_f_bulk} hold.
		Then, there exists $\tau$ sufficiently large, such that \cref{non:eq:main} admits a scattering solution with no outgoing radiation in $t_\star^{\g}\geq\tau$ satisfying
		\begin{equation}
			\phi\in\Hbloc^{1/2,q,q-1;k-1}.
		\end{equation}
	\end{theorem}
	\begin{proof}
		We take a sequence $\tau_n\to\infty$ and construct solution as in \cref{non:lemma:unstable}.
		We notice that \cref{non:eq:global_phic_bound} implies that the solutions constructed embed compactly into $\Hbloc^{1/2,q,q-1;k-1}$.
		We obtain a convergent subsequence by compactness.
	\end{proof}
	
	\begin{proof}[Proof of \cref{i:thm:existence_of_scattering_solution}]
		We pick $q$ such that \cref{non:thm:existence of scattering solution} holds.
		Then, we use \cref{an:thm:existence_of_ansatz,an:thm:supercritical} to find an ansatz that satisfies the decay condition required by $f$ in \cref{non:thm:existence of scattering solution}.
	\end{proof}

	\newpage
	\appendix

	\section{Divergence theorem}

	\subsection{Flux computations}\label{app:sec:flux calculation}
	\paragraph{Energy}
	We compute the $T$ energy through a hypersurface $\tilde{\Sigma}=\{t=-h(x)\}$ for some $h\in\mathcal{C}^\infty$ with $\norm{\nabla h}\leq1$. As a simple application of this computation we will find \cref{not:eq:energy integral}. Let us use coordinates $x,s=t+h(x)$ in a neighbourhood of $\tilde{\Sigma}$. We remark the following computations
	\begin{equation}
		\begin{gathered}
			\partial_i|_t=\partial_i|_s-h_iT,\quad ds=dt+\partial_ih\cdot dx_i.
		\end{gathered}
	\end{equation}
	The induced measure on $\tilde{\Sigma}$ with the above coordinates is $\det(\delta_{ij}-\partial_ih\partial_jh)^{1/2}=(1-\abs{\partial h}^2)^{1/2}=\abs{\d s}$. 
	Thus, we find the induced energy to be
	\begin{equation}
		\begin{gathered}
			\int_{\Sigma_\tau}\T^w[\phi]\bigg(\frac{ds}{\abs{ds}},dt\bigg)=\int_{\Sigma_\tau}\frac{1}{2}\big((T\phi)^2+\partial_i|_t\phi\cdot\partial_i|_t\phi-w\phi^2\big)+\partial_ih\partial_i|_{t}\phi T\phi\\
			=\int_{\Sigma_\tau}\frac{1}{2}\Big((1-\partial h\cdot\partial h)(T\phi)^2+\partial_i|_s\phi\cdot\partial_i|_s\phi-w\phi^2\Big).
		\end{gathered}
	\end{equation}
	
	We can similarly calculate the bilinear energy content, with the difference that we need to replace $\T^w[\phi](ds,dt)$ with
	\begin{equation}
		\begin{gathered}
			\T^w[\phi,\phi](dt,ds)=(1-\abs{\nabla h}^2)T\phi T\phi+\partial_i|_s\phi\cdot\partial_i|_s\phi-w\phi\phi.
		\end{gathered}
	\end{equation}
	yielding
	\begin{equation}\label{app:en:energy:fg}
		\begin{gathered}
			\int_{\Sigma_\tau}\T^w[\phi,\phi]\bigg(\frac{ds}{\abs{ds}},dt\bigg)
			=\int_{\R^3}\frac{1}{2}\Big((1-\partial h\cdot\partial h)(T\phi)(T\phi)+\partial_i|_s\phi\cdot\partial_i|_s\phi-w\phi\phi\Big)
		\end{gathered}
	\end{equation}
	
	\paragraph{Momentum}
	Similarly, we can compute the flux given by contraction with the $X=\partial_x$ Killing vector.
	For this, we use that $X=X_\star-\hat{x}h'T$ and $g^{-1}[(\dd t_\star)]=-(1-h'^2)T-h'X^\r_\star$, where we used $g^{-1}$ as a raising operator.
	Therefore, we get
	\begin{nalign}
		\T^w[\phi](X,g^{-1}[\dd t_\star])=(X_\star-\hat{x}h'T)\phi \big(-h'X^\r_\star-(1-h'^2_T)\big)\phi\\
		-\frac{1}{2}h'\hat{x}\Big(X_\star\phi\cdot X_\star\phi-2h'T\phi X^\r_\star\phi-(1-h'^2)(T\phi)^2-w\phi^2\Big)
	\end{nalign}
	We integrate this and get
	\begin{nalign}\label{app:eq:mom}
		\int_{\Sigma_\tau}\T^w[\phi]\bigg(\frac{ds}{\abs{ds}},d x\bigg)=\frac{-1}{2}\int_{\Sigma_\tau}\hat{x}(1-h'^2)h'(T\phi)^2+2(1-h'^2)(T\phi)X_\star\phi+2h'X^\r_\star\phi X_\star\phi\\
		-\hat{x}h'\big(X_\star\phi\cdot X_\star\phi-w\phi^2\big)
	\end{nalign}
	We can use the polarization identity to find the corresponding bilinear functional.
	
	\paragraph{Center of mass}
	Finally, we may combine the previous two computation, to find the flux corresponding to the Lorentz boost. 
	\begin{nalign}\label{app:eq:com}
		-\T^w[\phi]\bigg(g^{-1}\frac{ds}{\abs{ds}}, hX+xT\bigg)=\frac{1}{2}(T\phi)^2\bigg(\hat{x}h(1-h'^2)h'+x(1-h'^2)\bigg)-(1-h'^2)h(T\phi)X\phi\\
		-hh' X_\star^\r\phi X_\star\phi+\frac{1}{2}(h'h\hat{x}-x)\hat{x}\abs{X_\star\phi}^2+\frac{1}{2}\phi^2\big(xw+\hat{x}hwh'\big)
	\end{nalign}
	
	\section{Compactification}\label{app:compactification}
	
	We present extra detail for the compactification in \cref{not:def:compactification}.
	Let $\rho_a,\rho_\scri,\rho_+$ be as there.
	We construct the compactification $\D^g$ via coordinate charts.
	From \cref{not:lemma:spacetime_regions}, we know that $\{\rho_a<\delta_4\}$ are each disjoint from one another.
	
	First concentrate on the region $\abs{\bar{y}_a}<1$.
	Using $\bar{y}_a=\tilde{y}_a$ and $1/t$ as coordinates, we define the manifold with boundary by extending $1/t$ to 0.
	Next, in the region $\{\abs{\bar{y}_a}>1/2\}\cap\{\rho_a<\delta_4\}$, we use coordinates $\omega=\bar{y}_a/\abs{\bar{y}_a},1/\abs{\bar{y}_a},\abs{\bar{y}_a}/t$, and define the manifold with boundaries by attaching $\{1/\abs{\bar{y}_a}=0\}$ and $\{\abs{\bar{y}_a}/t=0\}$.
	Finally, we glue together the two regions via the coordinate changes in \cref{not:def:coordinates}.
	
	Note, that in  $\cup_a\{\rho_a>\delta_3\}$ we have $\bar{x}=\gamma_a(\bar{y}_a+t_a)$ for any $a$.
	We define a compactification in this region, by using $\omega=x/\abs{x},u^{-1},u/t$ as coordinates, and attaching the $u^{-1}=0$ and $u/t=0$ boundaries.
	\pagebreak
	\printbibliography
	
\end{document}